\pdfoutput=1
\documentclass[11pt]{amsart}
\usepackage{lmodern}
\usepackage{ifthen}
\usepackage{amsfonts}
\usepackage{amsxtra}
\usepackage{amssymb}
\usepackage{array}
\usepackage[margin=0.9in]{geometry}
\usepackage{xcolor}
\definecolor{cite}{rgb}{0.30,0.60,1.00}
\definecolor{url}{rgb}{0.00,0.00,0.80}
\definecolor{link}{rgb}{0.40,0.10,0.20}
\usepackage[colorlinks,linkcolor=link,urlcolor=url,citecolor=cite,pagebackref,breaklinks]{hyperref}
\usepackage{bbm}
\usepackage{mathtools}
\usepackage{mathrsfs}
\usepackage{appendix}
\usepackage[all]{xy}
\usepackage[lite,abbrev,msc-links,alphabetic]{amsrefs}
\usepackage{enumitem}

\usepackage[OT2,T1]{fontenc}
\DeclareSymbolFont{cyrletters}{OT2}{wncyr}{m}{n}
\DeclareMathSymbol{\Sha}{\mathalpha}{cyrletters}{"58}

\numberwithin{equation}{section}

\theoremstyle{plain}
\newtheorem{proposition}{Proposition}[section]
\newtheorem{conjecture}[proposition]{Conjecture}
\newtheorem{corollary}[proposition]{Corollary}
\newtheorem{lem}[proposition]{Lemma}
\newtheorem{theorem}[proposition]{Theorem}

\theoremstyle{definition}
\newtheorem{definition}[proposition]{Definition}

\theoremstyle{remark}
\newtheorem{remark}[proposition]{Remark}

\renewcommand{\b}[1]{\mathbf{#1}}
\renewcommand{\c}[1]{\mathcal{#1}}
\renewcommand{\d}[1]{\mathbb{#1}}
\newcommand{\f}[1]{\mathfrak{#1}}
\renewcommand{\r}[1]{\mathrm{#1}}
\newcommand{\s}[1]{\mathscr{#1}}
\renewcommand{\sf}[1]{\mathsf{#1}}
\renewcommand{\(}{\left(}
\renewcommand{\)}{\right)}
\newcommand{\res}{\mathbin{|}}
\newcommand{\ol}{\overline}
\newcommand{\ul}{\underline}
\renewcommand{\leq}{\leqslant}
\renewcommand{\geq}{\geqslant}
\newcommand{\sig}{\r{sig}}

\newcommand{\bphi}{\boldsymbol{\phi}}
\newcommand{\bblambda}{\boldsymbol{\lambda}}
\newcommand{\bbeta}{\boldsymbol{\eta}}
\newcommand{\bbf}{\boldsymbol{f}}
\newcommand{\bbi}{\boldsymbol{i}}
\newcommand{\bbA}{\boldsymbol{A}}
\newcommand{\bbB}{\boldsymbol{B}}
\newcommand{\bbP}{\boldsymbol{P}}
\newcommand{\bbX}{\boldsymbol{X}}

\newcommand{\bA}{\b A}
\newcommand{\bB}{\b B}

\newcommand{\bG}{\b G}
\newcommand{\bH}{\b H}

\newcommand{\bK}{\b K}

\newcommand{\bV}{\b V}
\newcommand{\bW}{\b W}

\newcommand{\bj}{\b j}

\newcommand{\bq}{\b q}

\newcommand{\bs}{\b s}
\newcommand{\bt}{\b t}

\newcommand{\cA}{\c A}

\newcommand{\cC}{\c C}
\newcommand{\cD}{\c D}
\newcommand{\cE}{\c E}
\newcommand{\cF}{\c F}

\newcommand{\cI}{\c I}
\newcommand{\cJ}{\c J}

\newcommand{\cL}{\c L}
\newcommand{\cM}{\c M}
\newcommand{\cN}{\c N}
\newcommand{\cO}{\c O}
\newcommand{\cP}{\c P}
\newcommand{\cQ}{\c Q}

\newcommand{\cS}{\c S}
\newcommand{\cT}{\c T}

\newcommand{\cV}{\c V}
\newcommand{\cW}{\c W}
\newcommand{\cX}{\c X}

\newcommand{\cZ}{\c Z}

\newcommand{\dC}{\d C}

\newcommand{\dF}{\d F}
\newcommand{\dG}{\d G}

\newcommand{\dL}{\d L}

\newcommand{\dN}{\d N}

\newcommand{\dP}{\d P}
\newcommand{\dQ}{\d Q}
\newcommand{\dR}{\d R}

\newcommand{\dZ}{\d Z}

\newcommand{\fa}{\f a}

\newcommand{\fd}{\f d}

\newcommand{\fg}{\f g}

\newcommand{\fj}{\f j}

\newcommand{\fm}{\f m}

\newcommand{\fp}{\f p}
\newcommand{\fq}{\f q}

\newcommand{\fu}{\f u}

\newcommand{\rB}{\r B}
\newcommand{\rC}{\r C}
\newcommand{\rD}{\r D}
\newcommand{\rE}{\r E}
\newcommand{\rF}{\r F}
\newcommand{\rG}{\r G}
\newcommand{\rH}{\r H}
\newcommand{\rI}{\r I}
\newcommand{\rJ}{\r J}
\newcommand{\rK}{\r K}
\newcommand{\rL}{\r L}
\newcommand{\rM}{\r M}

\newcommand{\rO}{\r O}
\newcommand{\rP}{\r P}
\newcommand{\rQ}{\r Q}
\newcommand{\rR}{\r R}
\newcommand{\rS}{\r S}
\newcommand{\rT}{\r T}
\newcommand{\rU}{\r U}
\newcommand{\rV}{\r V}
\newcommand{\rW}{\r W}

\newcommand{\rZ}{\r Z}

\newcommand{\ra}{\r a}
\newcommand{\rb}{\r b}
\newcommand{\rc}{\r c}
\newcommand{\rd}{\r d}
\newcommand{\re}{\r e}

\newcommand{\rh}{\r h}

\renewcommand{\rm}{\r m}

\newcommand{\rt}{\r t}

\newcommand{\sC}{\s C}

\newcommand{\sE}{\s E}

\newcommand{\sH}{\s H}

\newcommand{\sL}{\s L}

\newcommand{\sO}{\s O}

\newcommand{\sR}{\s R}
\newcommand{\sS}{\s S}

\newcommand{\sfL}{\sf L}

\newcommand{\sfh}{\sf h}

\newcommand{\tT}{\mathtt{T}}

\newcommand{\tc}{\mathtt{c}}

\newcommand{\tp}{\mathtt{p}}
\newcommand{\tq}{\mathtt{q}}
\newcommand{\ts}{\mathtt{s}}
\newcommand{\tu}{\mathtt{u}}
\newcommand{\tv}{\mathtt{v}}
\newcommand{\tw}{\mathtt{w}}

\newcommand{\ac}{\r{ac}}

\newcommand{\alg}{\r{alg}}
\newcommand{\can}{\r{can}}
\newcommand{\CF}{\mathbbm{1}}
\newcommand{\cl}{\r{cl}}

\newcommand{\cusp}{\r{cusp}}

\newcommand{\disc}{\r{disc}}
\newcommand{\dr}{\r{dR}}

\newcommand{\et}{\acute{\r{e}}\r{t}}
\newcommand{\even}{\r{even}}

\newcommand{\id}{\r{id}}

\newcommand{\Iw}{\r{Iw}}
\newcommand{\loc}{\r{loc}}

\newcommand{\odd}{\r{odd}}

\newcommand{\pr}{\r{pr}}

\newcommand{\red}{\r{red}}

\newcommand{\spl}{\r{spl}}
\newcommand{\ssl}{\r{ss}}

\newcommand{\ur}{\r{nr}}

\DeclareMathOperator{\AJ}{AJ}
\DeclareMathOperator{\Alb}{Alb}
\DeclareMathOperator{\As}{As}
\DeclareMathOperator{\Aut}{Aut}

\DeclareMathOperator{\CH}{CH}
\DeclareMathOperator{\CZ}{CZ}

\DeclareMathOperator{\End}{End}

\DeclareMathOperator{\Fil}{Fil}

\DeclareMathOperator{\Gal}{Gal}

\DeclareMathOperator{\FJ}{FJ}
\DeclareMathOperator{\GL}{GL}

\DeclareMathOperator{\Hom}{Hom}
\DeclareMathOperator{\IH}{IH}
\DeclareMathOperator{\IM}{Im}

\DeclareMathOperator{\Ind}{Ind}

\DeclareMathOperator{\Ker}{ker}
\DeclareMathOperator{\length}{length}
\DeclareMathOperator{\Lie}{Lie}

\DeclareMathOperator{\Mat}{Mat}

\DeclareMathOperator{\Mor}{Mor}
\DeclareMathOperator{\Mp}{Mp}

\DeclareMathOperator{\Nm}{N}

\DeclareMathOperator{\Orb}{Orb}

\DeclareMathOperator{\Pic}{Pic}
\DeclareMathOperator{\RE}{Re}
\DeclareMathOperator{\Res}{Res}
\DeclareMathOperator{\Sch}{Sch}

\DeclareMathOperator{\Sh}{Sh}

\DeclareMathOperator{\Sp}{Sp}

\DeclareMathOperator{\Spec}{Spec}
\DeclareMathOperator{\Spf}{Spf}

\DeclareMathOperator{\Tr}{Tr}
\DeclareMathOperator{\tr}{tr}
\DeclareMathOperator{\UG}{U}
\DeclareMathOperator{\val}{val}
\DeclareMathOperator{\vol}{vol}

\begin{document}

\title[Fourier--Jacobi cycles and arithmetic relative trace formula]
{Fourier--Jacobi cycles and arithmetic relative trace formula (with an appendix by Chao~Li and Yihang~Zhu)}

\author{Yifeng Liu}
\address{Institute for Advanced Study in Mathematics, Zhejiang University, Hangzhou 310058, China}
\email{liuyf0719@zju.edu.cn}

\date{\today}

\begin{abstract}
  In this article, we develop an arithmetic analogue of Fourier--Jacobi period integrals for a pair of unitary groups of equal rank. We construct the so-called Fourier--Jacobi cycles, which are algebraic cycles on the product of unitary Shimura varieties and abelian varieties. We propose the arithmetic Gan--Gross--Prasad conjecture for these cycles, which is related to the central derivatives of certain Rankin--Selberg $L$-functions, and develop a relative trace formula approach toward this conjecture.
\end{abstract}

\maketitle

\tableofcontents

\section{Introduction}
\label{ss:1}

\subsection{Fourier--Jacobi cycles and the arithmetic Gan--Gross--Prasad conjecture for $\rU(n)\times\rU(n)$}
\label{ss:fjcycle}

We first recall the classical notion of Fourier--Jacobi periods for $\rU(n)\times\rU(n)$ and their relation with $L$-functions (see \cite{GGP}*{Section~24} for more details). Let $E/F$ be a quadratic extension of number fields with the nontrivial Galois involution $\tc$ and the associated quadratic character $\mu_{E/F}\colon F^\times\backslash\bA_F^\times\to\{\pm1\}$. Let $\rV$ be a (non-degenerate) hermitian space over $E$ of rank $n\geq 1$ with respect to $\tc$, with the unitary group $\rU(\rV)$. Consider two irreducible cuspidal automorphic representations $\pi_1$ and $\pi_2$ of $\rU(\rV)(\bA_F)$. To define the Fourier--Jacobi periods for $\pi_1\times\pi_2$, we need an auxiliary conjugate symplectic automorphic character $\mu$ of $\bA_E^\times$, that is, an automorphic character of $\bA_E^\times$ whose restriction to $\bA_F^\times$ coincides with $\mu_{E/F}$. The character $\mu$ (together with a nontrivial additive character of $(E+\bA_F)\backslash\bA_E$) will give us a Weil representation of $\rU(\rV)(\bA_F)$, with natural automorphic realization via theta series $\theta_\mu^\phi$ attached to certain Schwartz functions $\phi$. We define the \emph{Fourier--Jacobi period integral} for $f_1\in\pi_1$, $f_2\in\pi_2$, and $\phi$ to be
\[
\cP_\mu(f_1,f_2;\phi)\coloneqq\int_{\rU(\rV)(F)\backslash\rU(\rV)(\bA_F)}f_1(g)f_2(g)\theta_\mu^\phi(g)\;\rd g,
\]
where $\rd g$ is the Tamagawa measure on $\rU(\rV)(\bA_F)$. Readers may realize that the above formula is very close to the Rankin--Selberg integral for $\GL(n)\times\GL(n)$ in which the role of the theta series is replaced by a mirabolic Eisenstein series (see \cite{Liu14}*{Section~3} for a systematic discussion). In particular, it is not surprising that $\cP_\mu(f_1,f_2;\phi)$ is related to $L$-values. In fact, if we assume that $\pi_1$ and $\pi_2$ are both tempered, then as a special case of the global Gan--Gross--Prasad (GGP) conjecture, one expects an Ichino--Ikeda type relation
\begin{align}\label{eq:ichino_ikeda}
|\cP_\mu(f_1,f_2;\phi)|^2\sim L(\tfrac{1}{2},\pi_1\times\pi_2\otimes\mu)\cdot\alpha(f_1,f_2;\phi),
\end{align}
where $\sim$ means that the two sides are differed by an explicit nonzero factor which depends only on $\pi_1$, $\pi_2$, and $\mu$; $L(s,\pi_1\times\pi_2\otimes\mu)$ denotes the complete Rankin--Selberg $L$-function (of symplectic type) centered at $s=\tfrac{1}{2}$; and $\alpha(f_1,f_2;\phi)$ is some explicit period integral of local matrix coefficients. See \cite{Xue16}*{Conjecture~1.1.2} for a precise conjecture. Suppose that the central $\epsilon$-factor $\epsilon(\tfrac{1}{2},\pi_1\times\pi_2\otimes\mu)$ is $1$. By the refined local GGP conjecture, which is known in this case by \cite{GI16}, if we consider the entire Vogan $L$-packet of the triple $(\rV,\pi_1,\pi_2)$, then there is a unique member for which $\alpha$ is a nonzero functional. Thus, the global GGP conjecture asserts that the global period $\cP_\mu$ vanishes on the entire Vogan $L$-packet if and only if $L(\tfrac{1}{2},\pi_1\times\pi_2\otimes\mu)=0$. Note that the $L$-function depends only on the Vogan $L$-packet.

Now suppose that $\epsilon(\tfrac{1}{2},\pi_1\times\pi_2\otimes\mu)=-1$. Then the local GGP conjecture already forces $\cP_\mu$ to be zero; and the first possible nonzero term in the Taylor expansion of $L(s,\pi_1\times\pi_2\otimes\mu)$ at $s=\tfrac{1}{2}$ is $L'(\tfrac{1}{2},\pi_1\times\pi_2\otimes\mu)$. Thus, it is curious to find a replacement of $\cP_\mu$ that encodes information about the first central derivative. This is the main goal of this article. In fact, the same question can be asked for all types of periods in the global GGP conjecture, namely,
\begin{enumerate}
  \item $\rO(m)\times\rO(n)$ with $n-m$ odd, which is of Bessel type,

  \item $\rU(m)\times\rU(n)$ with $n-m$ odd, which is of Bessel type,

  \item $\rU(m)\times\rU(n)$ with $n-m$ even, which is of Fourier--Jacobi type,

  \item $\Sp(2m)\times\Mp(2n)$, which is of Fourier--Jacobi type.
\end{enumerate}
A replacement of the period integral (under certain assumptions on the field $E/F$ and archimedean components of the representations) is only known before for Case (1) with $|m-n|=1$ and $m,n\geq 2$, and Case (2) with $|m-n|=1$ and $m,n\geq 1$. They are both realized as height pairings of certain diagonal cycles. See \cite{GGP}*{Section~27} for more details. For example, the celebrated Gross--Zagier formula \cite{GZ86} is responsible for $\rO(2)\times\rO(3)$; see \cite{YZZ} for a generalization. Now we give a formulation for Case (3) with $n=m\geq 2$.\\

In what follows, we will assume that $E/F$ is a CM extension, and $n\geq 2$. We first state a result concerning the Albanese variety of a unitary Shimura variety. Let $\bV$ be a totally positive definite incoherent hermitian space over $\bA_E$ of rank $n$. We have the associated system of Shimura varieties $\{\Sh(\bV)_K\}_K$ indexed by sufficiently small open compact subgroups $K\subseteq\rU(\bV)(\bA_F^\infty)$, each being smooth of dimension $n-1$ over $\Spec E$. Let $X_K$ be the canonical (smooth) toroidal compactification of $\Sh(\bV)_K$ (which is just $\Sh(\bV)_K$ if it is already proper). Put $X_\infty\coloneqq\varprojlim_K X_K$. Let $A_K$ be the Albanese variety of $X_K$; see Section \ref{ss:a}. Put $A_\infty\coloneqq\varprojlim_K A_K$, which is an abelian group pro-scheme over $E$. To every conjugate symplectic automorphic character $\mu$ of $\bA_E^\times$ of weight one (Definition \ref{de:conjugate2}), we associate a number field $M_\mu\subseteq\dC$, and an abelian variety $A_\mu$ over $E$ with a CM action $i_\mu\colon M_\mu\to\End_E(A_\mu)_\dQ$, unique up to isogeny, in Subsection \ref{ss:motives}. In particular, $A_\mu$ has dimension $[M_\mu:\dQ]/2$; and the set $\Omega(\mu)\coloneqq\Hom_E(A_\infty,A_\mu)_\dQ$ is naturally an $M_\mu[\rU(\bV)(\bA_F^\infty)]$-module depending only on $\mu$.

\begin{theorem}[Theorem \ref{th:cm_albanese}, Corollary \ref{co:cm_albanese}]\label{th:main1}
Let the notation be as above. There is an isomorphism
\[
\Omega(\mu)\otimes_{M_\mu}\dC\simeq\bigoplus_{\varepsilon}\bigoplus_{\chi}\omega(\mu,\varepsilon,\chi)
\]
of $\dC[\rU(\bV)(\bA_F^\infty)]$-modules. Here, $\{\omega(\mu,\varepsilon,\chi)\}_{\varepsilon,\chi}$, introduced in Definition \ref{de:oscillator_triple}, is a certain collection of Weil representations of $\rU(\bV)(\bA_F^\infty)$ that are isomorphic to the finite part of the Weil representations appearing in the definition of $\cP_\mu$. We refer to Theorem \ref{th:cm_albanese} for the precise statement. Moreover, for every sufficiently small open compact subgroup $K$ of $\rU(\bV)(\bA^\infty_F)$, there is an isogeny decomposition
\[
A_K\sim\prod_\mu A_\mu^{d(\mu,K)},\qquad
\text{resp. }A_K^{\r{end}}\sim\prod_\mu A_\mu^{d(\mu,K)}
\]
over $E$ when $n\geq 3$ (resp.\ $n=2$), where
\begin{itemize}
  \item the product is taken over representatives of $\Gal(\dC/\dQ)$-orbits of all conjugate symplectic automorphic characters of $\bA_E^\times$ of weight one,

  \item $d(\mu,K)\coloneqq\sum_{\varepsilon}\sum_{\chi}\dim_\dC\omega(\mu,\varepsilon,\chi)^K$ with the same index set for $\varepsilon,\chi$, and

  \item $A_K^{\r{end}}$ is the endoscopic part of $A_K$ when $n=2$, defined in \eqref{eq:endoscopic_albanese}.
\end{itemize}
\end{theorem}

The above theorem suggests that if we want to replace $\cP_\mu$ by algebraic cycles, then the Albanese variety should be involved.

\begin{definition}\label{de:relevant}
We say that a complex representation $\Pi$ of $\GL_n(\bA_E)$ is \emph{relevant} if\footnote{Here, the notion of relevant representation is more general than the one in \cite{L5} as we allow $\Pi$ to be isobaric.}
\begin{enumerate}
  \item $\Pi=\boxplus_{i=1}^{s(\Pi)}\Pi_i$ is an isobaric sum of $s(\Pi)$ irreducible cuspidal automorphic representations $\{\Pi_1,\dots,\Pi_{s(\Pi)}\}$, which we call \emph{cuspidal factors} of $\Pi$, satisfying $\Pi_i\circ\tc\simeq\Pi_i^\vee$ for every $1\leq i\leq s(\Pi)$,

  \item for every archimedean place $v$ of $E$, $\Pi_v$ is isomorphic to the (irreducible) principal series representation induced by the characters $(\arg^{1-n},\arg^{3-n},\dots,\arg^{n-3},\arg^{n-1})$, where $\arg\colon\dC^\times\to\dC^\times$ is the \emph{argument character} defined by the formula $\arg(z)\coloneqq z/\sqrt{z\ol{z}}$.
\end{enumerate}
Note that (2) implies that the cuspidal factors $\Pi_1,\dots,\Pi_{s(\Pi)}$ in (1) are mutually non-isomorphic.
\end{definition}

Now we fix our (tempered) Vogan $L$-packet by choosing two relevant representations $\Pi_1$ and $\Pi_2$ of $\GL_n(\bA_E)$. We also fix a conjugate symplectic automorphic character $\mu$ of $\bA_E^\times$ of weight one. Let $\bV$ be a totally positive definite incoherent hermitian space over $\bA_E$ of rank $n$. Consider irreducible admissible representations $\pi_1^\infty$ and $\pi_2^\infty$ of $\rU(\bV)(\bA_F^\infty)$ whose base change to $\GL_n(\bA_E^\infty)$ are $\Pi_1^\infty$ and $\Pi_2^\infty$, respectively.

Take a level subgroup $K\subseteq\rU(\bV)(\bA_F^\infty)$. Let $\alpha_K\colon X_K\to A_K$ be ``the Albanese morphism'' sending the zero-dimensional cycle $D_K^{n-1}$ to zero,\footnote{This is not exactly what we do. Here, we state it in this ideally correct but technically wrong way only for simplicity and for emphasizing the main idea. See Subsection \ref{ss:construction_cycles} for the rigorous construction.} where $D_K$ is the canonical extension of the Hodge divisor. For test functions $f_1,f_2\in\sC^\infty_c(K\backslash\rU(\bV)(\bA_F^\infty)/K,\dC)$ for $\pi_1^\infty$ and $\pi_2^\infty$, respectively (Definition \ref{de:test_function}), and a homomorphism $\phi\colon A_K\to A_\mu$, we define a Chow cycle
\[
\FJ(f_1,f_2;\phi)_K\coloneqq|\pi_0((X_K)_{E^\ac})|\cdot(\tT_K^{f_1}\otimes\tT_K^{f_2}\otimes\tT_\mu^{\r{can}})^*
(\id_{X_K\times X_K}\times(\phi\circ\alpha_K))_*\Delta^3 X_K
\]
on $X_K\times X_K\times A_\mu$, where $\tT_K^{f_i}$ denotes the normalized Hecke correspondence on $X_K$ attached to $f_i$; $\tT_\mu^{\r{can}}$ is a specific correspondence on $A_\mu$ (Definition \ref{de:mu_canonical}); and $\Delta^3 X_K\subseteq X_K^3$ is the diagonal cycle. For $i\in\dZ$, put
\[
\CH^i(X_\infty\times X_\infty\times A_\mu)_\dC^0
\coloneqq\varinjlim_K\CH^i(X_K\times X_K\times A_\mu)_\dC^0,
\]
and denote by $\CH^i_\mu(X_\infty\times X_\infty)^0_\dC$ the subspace of $\CH^i(X_\infty\times X_\infty\times A_\mu)_\dC^0$ on which $M_\mu$ acts via the inclusion $M_\mu\hookrightarrow\dC$, which depends only on $\mu$.

\begin{theorem}[Subsections \ref{ss:construction_cycles} and \ref{ss:aggp}]\label{th:main2}
The Chow cycle $\FJ(f_1,f_2;\phi)_K$ is homologically trivial, compatible under pullbacks when changing $K$, hence defines an element
\[
\FJ(f_1,f_2;\phi)\in\CH^{n-1+[M_\mu:\dQ]/2}(X_\infty\times X_\infty\times A_\mu)_\dC^0.
\]
If we assume the conjecture on the injectivity of the $\ell$-adic Abel--Jacobi map, then the assignment $(f_1,f_2,\phi)\mapsto\FJ(f_1,f_2;\phi)$ induces a complex linear map
\begin{align*}
\FJ_\varepsilon&\colon\pi_1^\infty\otimes_\dC\pi_2^\infty\otimes_\dC\Omega(\mu,\varepsilon)
\to\Hom_{\dC[\rU(\bV)(\bA_F^\infty)\times\rU(\bV)(\bA_F^\infty)]}
\(\pi_1^\infty\otimes_\dC\pi_2^\infty,\CH^{n-1+[M_\mu:\dQ]/2}_\mu(X_\infty\times X_\infty)_\dC^0\),
\end{align*}
which is invariant under the diagonal action of $\rU(\bV)(\bA_F^\infty)$ on the left-hand side, for every given $\mu$-admissible collection $\varepsilon$ (Definition \ref{de:admissible_collection}). Here, $\Omega(\mu,\varepsilon)$ is the sum of the factors of $\Omega(\mu)\otimes_{M_\mu}\dC$ in the decomposition in Theorem \ref{th:main1} corresponding to $\varepsilon,\chi$ with $\chi$ arbitrary.
\end{theorem}

We propose the following unrefined version of the \emph{arithmetic Gan--Gross--Prasad conjecture} for $\rU(n)\times\rU(n)$.

\begin{conjecture}[Conjecture \ref{co:ggp_unrefined}]\label{co:main1}
Let the notation be as above. Then for every given $\mu$-admissible collection $\varepsilon$ (Definition \ref{de:admissible_collection}), the following three statements are equivalent:
\begin{enumerate}[label=(\alph*)]
  \item We have $\FJ_\varepsilon\neq 0$.

  \item We have $\FJ_\varepsilon\neq 0$, and
      \[
      \dim_\dC\Hom_{\dC[\rU(\bV)(\bA_F^\infty)\times\rU(\bV)(\bA_F^\infty)]}
      \(\pi_1^\infty\otimes_\dC\pi_2^\infty,\CH^{n-1+[M_\mu:\dQ]/2}_\mu(X_\infty\times X_\infty)_\dC^0\)=1.
      \]

  \item We have $L'(\tfrac{1}{2},\Pi_1\times\Pi_2\otimes\mu)\neq 0$, and
      \[
      \Hom_{\dC[\rU(\bV)(\bA_F^\infty)]}(\pi_1^\infty\otimes_\dC\pi_2^\infty\otimes_\dC\Omega(\mu,\varepsilon),\dC)\neq\{0\}.
      \]
\end{enumerate}
\end{conjecture}

In view of the local GGP conjecture, the above conjecture implies that if $\epsilon(\tfrac{1}{2},\Pi_1\times\Pi_2\otimes\mu)=-1$, then $L'(\tfrac{1}{2},\Pi_1\times\Pi_2\otimes\mu)\neq 0$ if and only if one can find a triple $(\bV,\pi_1^\infty,\pi_2^\infty)$ as above such that $\FJ_\varepsilon\neq 0$. Moreover, if this is the case, then such triple is uniquely determined for every fixed $\varepsilon$. See Remark \ref{re:ggp_unrefined} for more details. In fact, in the actual discussion in Subsection \ref{ss:aggp}, we replace $\CH^{n-1+[M_\mu:\dQ]/2}_\mu(X_\infty\times X_\infty)_\dC^0$ by its quotient by the common kernel of $\ell$-adic Abel--Jacobi maps for all $\ell$, in order to avoid p and make the discussion unconditional. Moreover, we also track the rationality of the functional $\FJ_\varepsilon$ via replacing $\dC$ by a certain subfield of $\dC$.

We now propose the following refined version of the \emph{arithmetic Gan--Gross--Prasad conjecture} for $\rU(n)\times\rU(n)$, which is an explicit formula relating the Beilinson--Bloch--Poincar\'{e} heights (See Subsection \ref{ss:bbp}) of the cycles $\FJ(f_1,f_2;\phi)_K$ with the central derivative of $L(s,\Pi_1\times\Pi_2\otimes\mu)$.

\begin{conjecture}[Conjecture \ref{co:ggp_refined}]\label{co:main2}
Let the notation be as above. For test functions $f_1$, $f_1^\vee$, $f_2$, $f_2^\vee$ for $\pi_1^\infty$, $(\pi_1^\infty)^\vee$, $\pi_2^\infty$, $(\pi_2^\infty)^\vee$, respectively, and $\phi\in\Hom_E(A_K,A_\mu,\varepsilon)$, $\phi_\tc\in\Hom_E(A_K,A_{\mu^\tc},-\varepsilon)$, the identity
\begin{align*}
&\vol(K)^2\cdot\langle\FJ(f_1,f_2;\phi)_K,\FJ(f_1^\vee,f_2^\vee;\phi_\tc)_K\rangle_{X_K\times X_K,A_\mu}^{\r{BBP}}\\
&=\frac{\prod_{i=1}^nL(i,\mu_{E/F}^i)}{2^{s(\Pi_1)+s(\Pi_2)}}\cdot\frac{L'(\tfrac{1}{2},\Pi_1\times\Pi_2\otimes\mu)}
{L(1,\Pi_1,\As^{(-1)^n})\cdot L(1,\Pi_2,\As^{(-1)^n})}\cdot\beta(f_1,f_1^\vee,f_2,f_2^\vee,\phi,\phi_\tc)
\end{align*}
holds. Here,
\begin{itemize}
  \item $\mu^\tc\coloneqq\mu\circ\tc$ is the $\tc$-conjugation of $\mu$; and we may identify $A_{\mu^\tc}$ with $A_\mu^\vee$ (Proposition \ref{pr:cm_data});

  \item $\Hom_E(A_K,A_\mu,\varepsilon)$ (resp.\ $\Hom_E(A_K,A_{\mu^\tc},-\varepsilon)$) is the intersection of $\Hom_E(A_K,A_\mu)$ (resp.\ $\Hom_E(A_K,A_{\mu^\tc})$) and $\Omega(\mu,\varepsilon)$ (resp.\ $\Omega(\mu^\tc,-\varepsilon)$);

  \item $\vol(K)$ is the normalized volume of $K$ (Definition \ref{de:hodge_divisor});

  \item $\langle\;,\;\rangle_{X_K\times X_K,A_\mu}^{\r{BBP}}$ is a variant of the (conjectural) Beilinson--Bloch height pairing, which we call the Beilinson--Bloch--Poincar\'{e} height pairing, which is a bilinear map
      \[
      \CH^{n-1+[M_\mu:\dQ]/2}(X_K\times X_K\times A_\mu)_\dC^0\times\CH^{n-1+[M_\mu:\dQ]/2}(X_K\times X_K\times A_\mu^\vee)_\dC^0\to\dC;
      \]

  \item $s(\Pi_i)$ has appeared in Definition \ref{de:relevant};

  \item $\r{As}^\pm$ stand for the two Asai representations (see, for example, \cite{GGP}*{Section~7}); and

  \item $\beta$ is a certain normalized matrix coefficient integral defined immediately after Conjecture \ref{co:ggp_refined}.
\end{itemize}
\end{conjecture}

In order to transfer the height pairing in the above conjecture to some other pairing without $A_\mu$, we introduce a variant of the cycle $\FJ(f_1,f_2;\phi)_K$ via replacing the diagonal $\Delta^3X_K$ by a modified diagonal $\Delta^3_z X_K$, which we denote by $\FJ(f_1,f_2;\phi)_K^z$. It is actually equal to $\FJ(f_1,f_2;\phi)_K$ as elements in $\CH^{n-1+[M_\mu:\dQ]/2}(X_K\times X_K\times A_\mu)_\dC^0$ if the injectivity of the $\ell$-adic Abel--Jacobi map is granted. Thus, we also formulate a variant of the above refined arithmetic Gan--Gross--Prasad conjecture as Conjecture \ref{co:ggp_refined_bis}.

\begin{remark}
We expect that the Fourier--Jacobi cycles can also be used to bound Selmer groups for the Rankin--Selberg motive associated to $\Pi_1\times\Pi_2\otimes\mu$, just as what we have done for $\rO(3)\times\rO(4)$ \cite{Liu16} and for $\rU(n)\times\rU(n+1)$ \cite{L5} using diagonal cycles.
\end{remark}

\subsection{A relative trace formula approach}

For the case of central $L$-values for $\rU(n)\times\rU(n)$, namely the relation \eqref{eq:ichino_ikeda}, the author developed a relative trace formula approach in \cite{Liu14} generalizing the Jacquet--Rallis relative trace formula, which was later carried out by Hang Xue \cites{Xue14,Xue16}. Thus, it is natural to expect a relative trace formula approach toward Conjecture \ref{co:main2} as well, similar to what Wei Zhang did for the case $\rU(n)\times\rU(n+1)$ \cite{Zha12}. However, our situation is much more complicated due to both the construction of the cycle $\FJ(f_1,f_2;\phi)$ and the height pairing itself. Nevertheless, we still find such an approach after several reduction steps for the height pairing in Conjecture \ref{co:main2}, or rather its variant Conjecture \ref{co:ggp_refined_bis}. In order to avoid extra technical difficulty, in this article, we only discuss the relative trace formula for the case where $\Sh(\bV)_K$ is already proper, which we will now assume.

The first reduction step is the following theorem, which we refer as the \emph{doubling formula for CM data}.

\begin{theorem}[Proposition \ref{pr:doubling}, \eqref{eq:calibration1}, and Proposition \ref{pr:intersection}]\label{th:main3}
Let the notation be as in Conjecture \ref{co:main2} (or rather Conjecture \ref{co:ggp_refined_bis}). For $i=1,2$, let $\bbf_i\coloneqq f_i^\rt\ast f_i^\vee$ be the convolution of the transpose of $f_i$ and $f_i^\vee$. If we write $\bbf_1=\sum_s d_s\CF_{g_s^{-1}K\cap Kg_s^{-1}}$ as a finite sum with $d_s\in\dC$ and $g_s\in\rU(\bV)(\bA_F^\infty)$, then the identity
\begin{align*}
\vol(K)^2\cdot\langle\FJ(f_1,f_2;\phi)^z_K,\FJ(f_1^\vee,f_2^\vee;\phi_\tc)^z_K\rangle_{X_K\times X_K,A_\mu}^{\r{BBP}}
=\sum_sd_s\cdot\cI^zp_{K_s}(\rL_{g_s}\bbf_2,\bphi_s)
\end{align*}
holds, in which
\begin{itemize}
  \item $K_s\coloneqq K\cap g_sKg_s^{-1}$,

  \item $\rL_{g_s}\bbf_2$ is the left translation of $\bbf_2$ by $g_s$,

  \item $\bphi_s\in\sS(\bV(\bA_E^\infty))^{K_s}$ is a Schwartz function determined by $(\phi,g_s\phi_\tc)$ via \eqref{eq:doubling0}, and

  \item we put, for general $K\subseteq\rU(\bV)(\bA_F^\infty)$, $\bbf\in\sC^\infty_c(K\backslash\rU(\bV)(\bA_F^\infty)/K,\dC)$, and $\bphi\in\sS(\bV(\bA_E^\infty))^K$,
     \begin{align*}
     \cI^z_K(\bbf,\bphi)\coloneqq\langle\tp_{135}^*\Delta^3_zX_K,
     (\Delta X_K\times\tT^{\bbf}_K\times Z_K^\heartsuit).\tp_{246}^*\Delta^3_zX_K\rangle_{X_K^6}^{\r{BB}},
     \end{align*}
     where $Z_K^\heartsuit$ is a (formal sum of) divisor on $X_K\times X_K$ such that its restriction to the diagonal $\Delta X_K$ is Kudla's generating series of special divisors associated to $\bphi$ (Definition \ref{de:generating_series}).
\end{itemize}
Moreover, if $\bbf\otimes\bphi$ is regularly supported at some nonarchimedean place $v$ of $F$ (Definition \ref{de:regular_unitary}), then the cycles $\tp_{135}^*\Delta^3X_K$ and $(\Delta X_K\times\tT^{\bbf}_K\times Z_K^\heartsuit).\tp_{246}^*\Delta^3X_K$ have empty intersection on $X_K^6$.
\end{theorem}

Thus, it suffices to study the functional $\cI^z_K(\bbf,\bphi)$. We now assume that $\bbf\otimes\bphi$ is regularly supported at some nonarchimedean place $v$ of $F$. Then the definition of the Beilinson--Bloch height pairing provides us with a decomposition
\[
\cI^z_K(\bbf,\bphi)=\sum_{u}\cI^z_K(\bbf,\bphi)_u
\]
into local heights over all places $u$ of $E$. In what follows, we will study an approximation $\cI_K(\bbf,\bphi)_u$ of the local term $\cI^z_K(\bbf,\bphi)_u$ at certain places $u$ by ignoring $z$.

To continue the discussion, we need some notation. For integers $r,s\geq 1$, denote by $\Mat_{r,s}$ the scheme over $\dZ$ of $r$-by-$s$ matrices. For $n\geq 1$, we put $\rM_n\coloneqq\Mat_{n,1}\times\Mat_{1,n}$; and let $\rS_n$ be the $O_F$-subscheme of $\Res_{O_E/O_F}\Mat_{n,n}$ consisting of matrices $g$ satisfying $g\cdot g^\tc=\rI_n$, known as the symmetric space. In view of the relative trace formula developed in \cite{Liu14}, we are looking for test functions $\tilde\bbf\in\sS(\rS_n(\bA_F))$ and $\tilde\bphi\in\sS(\rM_n(\bA_F))$ such that $\cI^z_K(\bbf,\bphi)$ can be compared to another functional $\cJ(\tilde\bbf,\tilde\bphi)$ which encodes the right-hand side of Conjecture \ref{co:main2}. In this article, we only discuss the term $\cI_K(\bbf,\bphi)_\fp$ and local components $\tilde\bbf_\fp,\tilde\bphi_\fp$ when $\fp$ is a good inert prime of $F$ (Definition \ref{de:good_inert}), also regarded as a place of $E$.

Let $\fp$ be a good inert prime. Then $X_K$ has a canonical integral smooth model $\cX_K$ over $O_{E_\fp}$; the Hecke operator $\tT^{\bbf}_K$ extends naturally to $\cX_K$ by taking Zariski closure; and we also have a natural extension of $Z_K^\heartsuit$ to a (formal sum of) divisor $\cZ_K^\heartsuit$ on $\cX_K\times_{O_{E_\fp}}\cX_K$. We define the \emph{local arithmetic invariant functional} at $\fp$ to be
\[
\cI_K(\bbf,\bphi)_\fp\coloneqq 2\log|O_F/\fp|\cdot\chi\(\cO(\tp_{135}^*\Delta^3\cX_K)\otimes^\dL_{\cO_{\cX_K^6}}
\cO((\Delta\cX_K\times\tT^{\bbf}_K\times\cZ_K^\heartsuit).\tp_{246}^*\Delta^3\cX_K)\)
\]
as an intersection number of algebraic cycles on $\cX_K^6$, the sixfold self fiber product of $\cX_K$ over $O_{E_\fp}$, where $\chi$ denotes the Euler--Poincar\'{e} characteristic. The following result provides an orbital decomposition of $\cI_K(\bbf,\bphi)_\fp$, which is the key for the comparison of relative trace formulae.

\begin{theorem}[Theorem \ref{th:orbital}]\label{th:main4}
Let $K,\bbf,\bphi$ be as above such that $\bbf\otimes\bphi$ is regularly supported at some nonarchimedean place $v$ of $F$. Then for a good inert prime $\fp$, the identity
\[
\cI_K(\bbf,\bphi)_\fp=2\log|O_F/\fp|\cdot\sum_{(\bar\xi,\bar{x})\in[\rU(\bar\rV)(F)\times\bar\rV(E))]_{\r{rs}}}
\re^{-2\pi\cdot\Tr_{F/\dQ}(\bar{x},\bar{x})_{\bar\rV}}
\Orb(\bar\bbf^\fp,\bar\bphi^\fp;\bar\xi,\bar{x})\cdot\chi\(\cO_{\Gamma_{\bar\xi}}\otimes^\dL_{\cO_{\cN^2}}\cO_{\Delta\cZ(\bar{x})}\)
\]
holds, where
\begin{itemize}
  \item $\bar\rV$ is a hermitian space over $E$ satisfying $\bar\rV\otimes_F\bA_F^\fp\simeq\bV\otimes_{\bA_F}\bA_F^\fp$,

  \item the orbital integral is defined as
    \[
    \Orb(\bar\bbf^\fp,\bar\bphi^\fp;\bar\xi,\bar{x})\coloneqq\int_{\rU(\bar\rV)({\bA_F^\infty}^{,\fp})}
    \bar\bbf^\fp(\bar{g}^{-1}\bar\xi\bar{g})\bar\bphi^\fp(\bar{g}^{-1}\bar{x})\;\rd\bar{g},
    \]

  \item $\chi\(\cO_{\Gamma_{\bar\xi}}\otimes^\dL_{\cO_{\cN^2}}\cO_{\Delta\cZ(\bar{x})}\)$ is a certain intersection number defined on a relative Rapoport--Zink space.
\end{itemize}
We refer to Theorem \ref{th:orbital} for the precise meaning of all the notation.
\end{theorem}

The term $\chi\(\cO_{\Gamma_{\bar\xi}}\otimes^\dL_{\cO_{\cN^2}}\cO_{\Delta\cZ(\bar{x})}\)$ is the one that is related to the derivative of $L$-function, more precisely, to the derivative of local orbital integrals at $\fp$ in the decomposition of $\cJ(\tilde\bbf,\tilde\bphi)$. The precise relation is the content of the \emph{arithmetic fundamental lemma for $\rU(n)\times\rU(n)$}, which we introduce in the next subsection.

\subsection{Arithmetic fundamental lemma for $\rU(n)\times\rU(n)$}
\label{ss:afl}

In this subsection, we introduce the arithmetic fundamental lemma for $\rU(n)\times\rU(n)$. Since the question is purely local, we will shift our notation slightly from the previous discussion. Moreover, we will allow $n$ to be an arbitrary positive integer since the discussion makes sense even for $n=1$.

Let $F$ be a finite extension of $\dQ_p$, with residue cardinality $q$. Let $E/F$ be an unramified quadratic extension, and $\breve{E}$ a completed maximal unramified extension of $E$ with $k$ its residue field.

We recall some definitions and facts from \cite{Liu14}*{Section~5.3}. We say that a pair $(\zeta,y)\in\rS_n(F)\times\rM_n(F)$ is \emph{regular semisimple} if the matrix $(y_2\zeta^{i+j-2}y_1)_{i,j=1}^n$ is non-degenerate, where we write $y=(y_1,y_2)\in\Mat_{n,1}(F)\times\Mat_{1,n}(F)$. If $(\zeta,y)$ is regular semisimple, we define its \emph{transfer factor} to be
\[
\omega(\zeta,y)\coloneqq\mu_{E/F}(\det(y_1,\zeta y_1,\dots,\zeta^{n-1}y_1)).
\]
The group $\GL_n(F)$ acts on $\rS_n(F)\times\rM_n(F)$ by the formula $(\zeta,y)g=(g^{-1}\zeta g,g^{-1}y_1,y_2g)$, which preserves regular semisimple elements. We denote by $[\rS_n(F)\times\rM_n(F)]_{\r{rs}}$ the set of regular semisimple $\GL_n(F)$-orbits.

Let $\rV_n^+$ (resp.\ $\rV_n^-$) be a hermitian space over $E$ of rank $n$ whose determinant has even (resp.\ odd) valuation. For $\delta=\pm$, we say that a pair $(\xi,x)\in\rU(\rV^\delta_n)(F)\times\rV^\delta_n(E)$ is \emph{regular semisimple} if $\{x,\xi x,\dots,\xi^{n-1}x\}$ are linearly independent. The group $\rU(\rV_n^\delta)(F)$ acts on $\rU(\rV_n^\delta)(F)\times\rV_n^\delta(E)$ by the formula $(\xi,x)g=(g^{-1}\xi g,g^{-1}x)$, which preserves regular semisimple elements. We denote by $[\rU(\rV_n^\delta)(F)\times\rV_n^\delta(E)]_{\r{rs}}$ the set of regular semisimple $\rU(\rV_n^\delta)(F)$-orbits. We say that $(\zeta,y)\in[\rS_n(F)\times\rM_n(F)]_{\r{rs}}$ and $(\xi,x)\in[\rU(\rV_n^\delta)(F)\times\rV_n^\delta(E)]_{\r{rs}}$ \emph{match} if $\zeta$ and $\xi$ have the same characteristic polynomial and
$y_2\zeta^iy_1=(\xi^ix,x)$ for $0\leq i\leq n-1$. The matching relation induces a bijection
\[
[\rS_n(F)\times\rM_n(F)]_{\r{rs}}\simeq[\rU(\rV_n^+)(F)\times\rV_n^+(E)]_{\r{rs}}\coprod[\rU(\rV_n^-)(F)\times\rV_n^-(E)]_{\r{rs}}.
\]
Denote by $[\rS_n(F)\times\rM_n(F)]_{\r{rs}}^\pm\subseteq[\rS_n(F)\times\rM_n(F)]_{\r{rs}}$ the subset corresponding to orbits in $[\rU(\rV_n^\pm)(F)\times\rV_n^\pm(E)]_{\r{rs}}$. Then a regular semisimple orbit $(\zeta,y)$ belongs to $[\rS_n(F)\times\rM_n(F)]_{\r{rs}}^\delta$ for $\delta=+$ (resp.\ $\delta=-$) if and only if the $\det((y_2\zeta^{i+j-2}y_1)_{i,j=1}^n)$ has even (resp.\ odd) valuation.

Now we introduce the relevant orbital integral. For a regular semisimple pair $(\zeta,y)\in\rS_n(F)\times\rM_n(F)$ and a pair of Schwartz functions $f\in\sS(\rS_n(F))$, $\phi\in\sS(\rM_n(F))$, we define
\begin{align*}
\Orb(s;f,\phi;\zeta,y)\coloneqq\int_{\GL_n(F)}f(g^{-1}\zeta g)\phi(g^{-1}y_1,y_2g)\mu_{E/F}(\det g)|\det g|_E^s\;\rd g,
\end{align*}
where $\rd g$ is the Haar measure under which $\GL_n(O_F)$ has volume $1$. It is clear that the product $\omega(\zeta,y)\Orb(0;f,\phi;\zeta,y)$ depends only on the $\GL_n(F)$-orbit of $(\zeta,y)$. We recall the following conjecture from \cite{Liu14}.

\begin{conjecture}[relative fundamental lemma for $\rU(n)\times\rU(n)$]\label{co:rfl_general}
For every regular semisimple orbit $(\zeta,y)\in[\rS_n(F)\times\rM_n(F)]_{\r{rs}}$, we have
\begin{enumerate}
  \item if $(\zeta,y)\in[\rS_n(F)\times\rM_n(F)]_{\r{rs}}^-$, then
     \[
     \omega(\zeta,y)\Orb(0;\CF_{\rS_n(O_F)},\CF_{\rM_n(O_F)};\zeta,y)=0;
     \]

  \item if $(\zeta,y)\in[\rS_n(F)\times\rM_n(F)]_{\r{rs}}^+$, then
     \begin{align}\label{eq:rfl}
     \omega(\zeta,y)\Orb(0;\CF_{\rS_n(O_F)},\CF_{\rM_n(O_F)};\zeta,y)=
     \int_{\rU(\rV_n^+)}\CF_{K_n}(g^{-1}\xi g)\CF_{\Lambda_n}(g^{-1}x)\;\rd g
     \end{align}
     where $(\xi,x)\in[\rU(\rV_n^+)(F)\times\rV_n^+(E)]_{\r{rs}}$ is the unique orbit that matches $(\zeta,y)$, $\Lambda_n$ is a self-dual lattice in $\rV_n^+$, $K_n$ is the stabilizer of $\Lambda_n$, and $\rd g$ is the Haar measure on $\rU(\rV_n^+)$ under which $K_n$ has volume $1$.
\end{enumerate}
\end{conjecture}

\begin{remark}\label{re:rfl}
Conjecture \ref{co:rfl_general}(1) is known by \cite{Liu14}*{Proposition~5.14}. Conjecture \ref{co:rfl_general}(2) is known for $p$ sufficiently large by \cite{Liu14}*{Theorem~5.15}.
\end{remark}

Now we describe the arithmetic fundamental lemma, where in \eqref{eq:rfl} we replace the left-hand side by its derivative and the right-hand side by a certain intersection number on a (relative) Rapoport--Zink space. We start by recalling the notion of relative Rapoport--Zink spaces. For an $O_{\breve{E}}$-scheme $S$, a \emph{unitary $O_F$-module of signature $(r,s)$} with integers $r,s\geq 0$ is a triple $(X,i,\lambda)$, in which
\begin{itemize}
  \item $X$ is a strict $O_F$-module over $S$ of dimension $r+s$ and $O_F$-height $2(r+s)$ over $S$,

  \item $i\colon O_E\to\End_S(X)$ is an action compatible with the $O_F$-module structure satisfying that for every $e\in O_E$ the characteristic polynomial of $i(e)$ on $\Lie_S(X)$ is given by $(T-a^\tc)^r(T-a)^s\in\cO_S[T]$,

  \item $\lambda\colon X\to X^\vee$ is a principal polarization such that the associated Rosati involution induces the conjugation on $O_E$.
\end{itemize}
We say that $(X,i,\lambda)$ is supersingular if $X$ is a supersingular strict $O_F$-module.

We fix a supersingular unitary $O_F$-module $(\bbX_0,\bbi_0,\bblambda_0)$ of signature $(1,0)$ over $O_{\breve{E}}$, which is unique up to isomorphism. For every integer $n\geq 1$, we also choose a supersingular unitary $O_F$-module $(\bbX_n,\bbi_n,\bblambda_n)$ of signature $(n-1,1)$ over $k$, which is unique up to $O_E$-linear isogeny preserving the polarization up to scalars. Let $\cN_n$ be the relative Rapoport--Zink space parameterizing quasi-isogenies of $(\bbX_n,\bbi_n,\bblambda_n)$ of height zero. More precisely, it is a formal scheme over $O_{\breve{E}}$ such that for every scheme $S$ over $O_{\breve{E}}$ on which $p$ is locally nilpotent, $\cN_n(S)$ is the set of isomorphism classes of quadruples $(X,i,\lambda;\rho)$, where
\begin{itemize}
  \item $(X,i,\lambda)$ is a unitary $O_F$-module over $S$ of signature $(n-1,1)$,

  \item $\rho\colon X\times_S S_k\to\bbX_n\times_k S_k$ is an $O_E$-linear quasi-isogeny (of height zero), such that $\rho^*\bblambda_n=\lambda$. Here, we put $S_k\coloneqq S\otimes_{O_{\breve{E}}}k$.
\end{itemize}
It is known that $\cN_n$ is formally smooth over $O_{\breve{E}}$ of relative dimension $n-1$. See \cite{Mih}*{Section~3.1} for more details.

We recall the notion of formal special divisors from \cite{KR11}. Put $\Lambda_n\coloneqq\Hom_k((\bbX_{0k},\bbi_{0k}),(\bbX_n,\bbi_n))$
and $\rV_n^-\coloneqq(\Lambda_n)_\dQ$. Then $\rV_n^-$ is an $E$-vectors space of rank $n$ equipped with a hermitian form
\[
(x,y)=\bbi_{0k}^{-1}\(\bblambda_{0k}^{-1}\circ y^\vee\circ\bblambda_n\circ x\)\in E.
\]
If we denote by $\Lambda_n^*$ the dual lattice of $\Lambda_n$ under the above hermitian form, then we have $\Lambda_n\subseteq\Lambda_n^*$ and that the length of the $O_E$-module $\Lambda_n^*/\Lambda_n$ is odd. In particular, the determinant of $\rV_n^-$ has odd valuation, justifying its notation.

\begin{definition}\label{de:special_cycle_formal}
For every $x\in\rV_n^-$ that is nonzero, we define $\cZ_n(x)$ to be the subfunctor of $\cN_n$ such that for every scheme $S$ over $O_{\breve{E}}$ on which $p$ is locally nilpotent, $\cZ_n(x)(S)$ consists of $(X,i,\lambda;\rho)$ satisfying that the composite homomorphism
\[
\bbX_{0k}\times_k S_k\xrightarrow{x}\bbX_n\times_kS_k\xrightarrow{\rho^{-1}}X\times_SS_k
\]
extends to an $O_E$-linear homomorphism $\bbX_0\times_{O_{\breve{E}}}S\to X$ over $S$.
\end{definition}

By \cite{RZ96}*{Proposition~2.9}, $\cZ_n(x)$ is a closed sub-formal scheme of $\cN_n$. For every $g\in\rU(\rV_n)(F)$, let $\rho_g\colon\bbX_n\to\bbX_n$ be the unique $O_E$-linear quasi-isogeny (of height zero) such that $gx=\rho_g\circ x$ for every $x\in\rV_n$; and, by abuse of notation, let $g\colon\cN_n\to\cN_n$ be the (auto)morphism sending $(X,i,\lambda;\rho)$ to $(X,i,\lambda;\rho_g\circ\rho)$. We denote by $\Gamma_g\subseteq\cN_n^2\coloneqq\cN_n\times_{O_{\breve{E}}}\cN_n$ the graph of $g$.

\begin{conjecture}[arithmetic fundamental lemma for $\rU(n)\times\rU(n)$]\label{co:afl_general}
For every regular semisimple orbit $(\zeta,y)\in[\rS_n(F)\times\rM_n(F)]_{\r{rs}}^-$, we have
\[
-\omega(\zeta,y)\left.\frac{\rd}{\rd s}\right|_{s=0}\Orb(s;\CF_{\rS_n(O_F)},\CF_{\rM_n(O_F)};\zeta,y)
=2\log q\cdot\chi\(\cO_{\Gamma_\xi}\otimes^\dL_{\cO_{\cN_n^2}}\cO_{\Delta\cZ_n(x)}\),
\]
where $(\xi,x)\in[\rU(\rV_n^-)(F)\times\rV_n^-(E)]_{\r{rs}}$ is the unique orbit that matches $(\zeta,y)$, and $\chi$ denotes the Euler--Poincar\'{e} characteristic.
\end{conjecture}

In Conjecture \ref{co:afl_general}, it follows from Conjecture \ref{co:rfl_general}(1), which is known, that the left-hand side depends only on the $\GL_n(F)$-orbit of $(\zeta,y)$.

\begin{remark}\label{re:fl_comparison}
During the referee process of this article, Wei Zhang \cite{Zha}*{Proposition~4.12 \& Remark~3.1} has shown that his arithmetic fundamental lemma for $\rU(n)\times\rU(n+1)$ is equivalent to our arithmetic fundamental lemma for $\rU(n)\times\rU(n)$ (with respect to the same field extension $E/F$) when the residue cardinality of $F$ is greater than $n$. In particular, we find
\begin{enumerate}
  \item Conjecture \ref{co:afl_general} holds when $n\leq 2$ and $q$ is odd, by \cite{Zha12}*{Theorem~2.10 \& Theorem~5.5}.

  \item Conjecture \ref{co:afl_general} holds when $F=\dQ_p$ with $p>n$, by \cite{Zha}*{Theorem~15.1}.
\end{enumerate}
\end{remark}

\begin{remark}\label{re:minuscule}
In Section \ref{ss:e}, Chao~Li and Yihang~Zhu proved Conjecture \ref{co:afl_general} (for arbitrary $E/F$) in the so-called minuscule case, similar to the case of $\rU(n)\times\rU(n+1)$. In the case of $\rU(n)\times\rU(n)$, we say that a regular semisimple pair $(\xi,x)\in\rU(\rV_n^-)(F)\times\rV_n^-(E)$ is \emph{minuscule} if the $O_E$-lattice $L_{\xi,x}$ generated by $\{x,\xi x,\cdots,\xi^{n-1}x\}$ satisfies $\varpi L_{\xi,x}^*\subseteq L_{\xi,x}\subseteq L_{\xi,x}^*$ where $\varpi$ is a uniformizer of $F$ and $L_{\xi,x}^*$ denotes the dual lattice.
\end{remark}

\subsection{Relation between $\rU(n)\times\rU(n)$ and $\rU(n)\times\rU(n+1)$}

In this subsection, we make an informal comparison between the two scenarios of $\rU(n)\times\rU(n)$ and $\rU(n)\times\rU(n+1)$, for both automorphic periods/central values and arithmetic periods/central derivatives.

The following diagram compares the automorphic periods and the relative trace formula approaches toward the global GGP conjectures in the two scenarios.
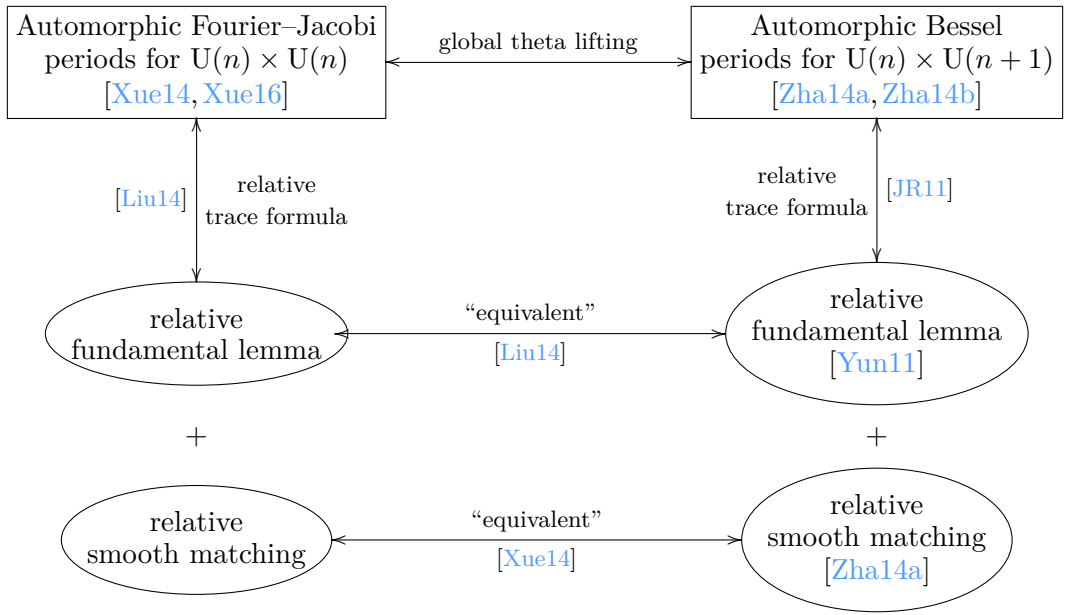
\[
\xymatrix{
*+[F]{\txt{Automorphic Fourier--Jacobi \\ periods for $\rU(n)\times\rU(n)$ \\ \cites{Xue14,Xue16}}}*
\ar@{<->}[rrrr]^-{\txt{\footnotesize global theta lifting}}
\ar@{<->}[dd]^-{\txt{\footnotesize relative \\ \footnotesize trace formula}}_-{\txt{\footnotesize \cite{Liu14}}} &&&&
*+[F]{\txt{Automorphic Bessel \\ periods for $\rU(n)\times\rU(n+1)$ \\ \cites{Zha13,Zha14}}}*
\ar@{<->}[dd]_-{\txt{\footnotesize relative \\ \footnotesize trace formula}}^-{\txt{\footnotesize \cite{JR11}}} \\ \\
*+++[o][F]{\txt{relative \\ fundamental lemma \\ }}*
\ar@{}[d]|{\txt{$+$}} \ar@{<->}[rrrr]^-{\txt{\footnotesize ``equivalent''}}_-{\txt{\footnotesize \cite{Liu14}}} &&&&
*+++[o][F]{\txt{relative \\ fundamental lemma \\ \cite{Yun11} }}* \ar@{}[d]|{\txt{$+$}} \\
*+++[o][F]{\txt{relative \\ smooth matching \\  }}*
\ar@{<->}[rrrr]^-{\txt{\footnotesize ``equivalent''}}_-{\txt{\footnotesize \cite{Xue14}}} &&&&
*+++[o][F]{\txt{relative \\ smooth matching \\ \cite{Zha13} }}*
}
\]

In the first line, if we assume Conjecture \ref{co:vanishing} below,\footnote{Recently, Dihua~Jiang and Lei~Zhang \cite{JZ} have confirmed this conjecture when $n\leq 4$. Of course, when $n\leq 2$, it was already known before.} then the method of global theta lifting should provide an equivalence between the two sides when all $n$ are considered. In fact, Xue \cites{Xue14,Xue16} has essentially verified the deduction for both directions starting from two stable tempered representations on $\rU(n)$ that satisfy Conjecture \ref{co:vanishing}.


\begin{conjecture}\label{co:vanishing}
Let $\rV$ be a hermitian space over $E$ of rank $n\geq 1$. Let $\pi$ be a tempered cuspidal automorphic representation of $\rU(\rV)(\bA_F)$. If $n$ is even (resp.\ odd), then there exists a conjugate orthogonal (resp.\ conjugate symplectic) automorphic character $\mu$ of $\bA_E^\times$ such that
\[
L(\tfrac{1}{2},\Pi\otimes\mu)\neq 0,
\]
where $\Pi$ is the standard base change of $\pi$ to $\GL_n(\bA_E)$.
\end{conjecture}

The following diagram compares the arithmetic periods and the relative trace formula approaches toward the arithmetic GGP conjectures in the two scenarios.
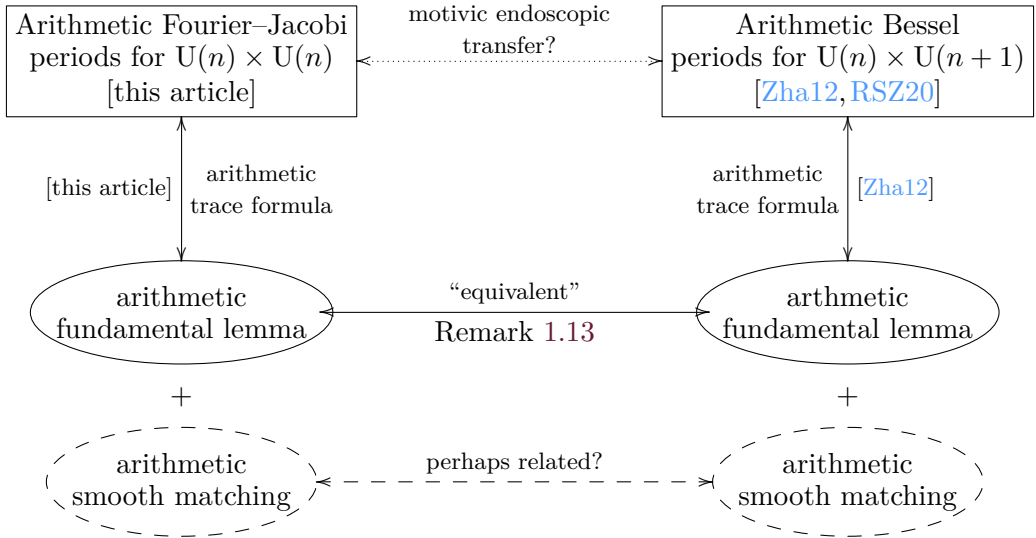
\[
\xymatrix{
*+[F]{\txt{Arithmetic Fourier--Jacobi \\ periods for $\rU(n)\times\rU(n)$ \\ [this article]}}*
\ar@{<..>}[rrrr]^-{\txt{\footnotesize motivic endoscopic \\ \footnotesize transfer?}}
\ar@{<->}[dd]^-{\txt{\footnotesize arithmetic \\ \footnotesize trace formula}}_-{\txt{\footnotesize [this article]}} &&&&
*+[F]{\txt{Arithmetic Bessel \\ periods for $\rU(n)\times\rU(n+1)$ \\ \cites{Zha12,RSZ}}}*
\ar@{<->}[dd]_-{\txt{\footnotesize arithmetic \\ \footnotesize trace formula}}^-{\txt{\footnotesize \cite{Zha12}}} \\ \\
*+++[o][F]{\txt{arithmetic \\ fundamental lemma }}*
\ar@{}[d]|{\txt{$+$}} \ar@{<->}[rrrr]^-{\txt{\footnotesize ``equivalent''}}_-{\txt{Remark \ref{re:fl_comparison}}} &&&&
*+++[o][F]{\txt{arthmetic \\ fundamental lemma }}* \ar@{}[d]|{\txt{$+$}} \\
*+++[o][F--]{\txt{arithmetic \\ smooth matching \\  }}*
\ar@{<-->}[rrrr]^-{\txt{\footnotesize perhaps related?}} &&&&
*+++[o][F--]{\txt{arithmetic \\ smooth matching }}*
}
\]

In the first line, the Tate conjecture over number fields predicts a motivic endoscopic lifting (or motivic theta lifting) that transfers algebraic cycles from one side to the other. Thus, we expect that our Fourier--Jacobi cycles should be related to the diagonal cycle considered in \cite{Zha12} in a certain way. However, at this moment, the motivic endoscopic lifting seems far out of reach. For the two dashed bubbles surrounding ``arithmetic smooth matching'', we do not how to formulate a precise conjecture in general. However, in some special cases for $\rU(n)\times\rU(n+1)$, there are some results \cites{RSZ17,RSZ18}.

\subsection{Relation with the arithmetic triple product formula}

In this subsection, we compare our arithmetic GGP conjecture for $n=2$ with the (conjectural) arithmetic triple product formula, which can be regarded as the arithmetic GGP for $\rO(3)\times\rO(4)$ in which $\rO(4)$ has trivial discriminant. Lots of progress has been made toward the arithmetic triple product formula; see, for example, \cites{GK92,GK93,YZZ1}.

We first make a quick review of the arithmetic triple product formula following the line of \cite{YZZ1}. Consider three irreducible cuspidal automorphic representations $\sigma_1,\sigma_2,\sigma_3$ of $\GL_2(\bA_F)$ of parallel weight $2$ such that the product of their central characters is trivial and $\epsilon(\tfrac{1}{2},\sigma_1\times\sigma_2\times\sigma_3)=-1$. Then the local dichotomy of triple product invariant functionals provides us with a totally definite incoherent quaternion algebra $\bB$ over $\bA_F$, unique up to isomorphism. Let $\{Y_U\}_U$ be the system of compactified Shimura curves over $F$ associated to $\bB$ indexed by open compact subsets $U\subseteq(\bB\otimes_{\bA_F}\bA_F^\infty)^\times$. For $i=1,2,3$, let $A_i$ be the abelian variety of strict $\GL(2)$-type over $F$ associated to $\sigma_i$. For morphisms $g_i\colon Y_U\to A_i$ for $i=1,2,3$, we have the Gross--Kudla--Schoen cycle, which is essentially
\[
\r{GKS}(g_1,g_2,g_3)_U\coloneqq(g_1\times g_2\times g_3)_*\Delta^3 Y_U\in\CH_1(A_1\times A_2\times A_3)_\dQ.
\]
The arithmetic triple product formula predicts a relation between the Beilinson--Bloch height of $\r{GKS}(g_1,g_2,g_3)_U$ and the central derivative $L'(\tfrac{1}{2},\sigma_1\times\sigma_2\times\sigma_3)$.

Now we discuss its connection with our case. Suppose that
\begin{enumerate}
  \item $\sigma_3$ is the theta lifting of $\mu^\alg\coloneqq\mu\cdot|\;|_E^{-1/2}$ for an automorphic character $\mu$ of $\bA_E^\times$p which is necessarily conjugate symplectic of weight one;

  \item for $i=1,2$, the base change of $\sigma_i$ to $\GL_2(\bA_E)$, denoted by $\Pi_i$, has trivial central character.
\end{enumerate}
Then $\Pi_1$ and $\Pi_2$ are both relevant; and we may take $A_\mu$ to be $A_3\otimes_FE$. Let $\bV$ be the unique up to isomorphism totally positive definite (incoherent) hermitian space over $\bA_E$ of rank $2$ such that for every nonarchimedean place $v$ of $F$, $\bV\otimes_{\bA_F}F_v$ is isotropic if and only if $\bB\otimes_{\bA_F}F_v$ is split. We recall our compactified unitary Shimura curve $\{X_K\}_K$ associated to $\bV$. For morphisms $f_i\colon X_K\to A_i\otimes_FE$ for $i=1,2$ and $\phi\colon X_K\to A_\mu=A_3\otimes_FE$, we have the Fourier--Jacobi cycle, which is essentially
\[
\FJ(f_1,f_2;\phi)_K\coloneqq(f_1\times f_2\times \phi)_*\Delta^3 X_K\in\CH_1((A_1\times A_2\times A_3)\otimes_FE)_\dQ.
\]
Conjecture \ref{co:main2} predicts a relation between the Beilinson--Bloch height of $\FJ(f_1,f_2;\phi)_K$ and the central derivative $L'(\tfrac{1}{2},\Pi_1\times\Pi_2\otimes\mu)$. It is possible to show \emph{a priori} that the height of $\r{GKS}(g_1,g_2,g_3)_U$ for some choice of $U,g_1,g_2,g_3$ is related to the height of $\FJ(f_1,f_2;\phi)_K$ for some choice of $K,f_1,f_2,\phi$. This is not surprising as in this case we have the equality
\[
L(s,\sigma_1\times\sigma_2\times\sigma_3)=L(s,\Pi_1\times\Pi_2\otimes\mu)
\]
between $L$-functions. In other words, our work in the special case where $n=2$ provides a relative trace formula approach toward the arithmetic triple product formula in the situation where (1) and (2) are satisfied. However, we point out that not all cases for $\rU(2)\times\rU(2)$ arise from the arithmetic triple product formula in this way since $\Pi_1$ and $\Pi_2$ are not necessarily base change from $\GL_2(\bA_F)$.

\subsection{Structure of the article}

The main part of the article contains five sections.

In Section \ref{ss:a}, we study the Albanese varieties. In Subsection \ref{ss:albanese_general}, we introduce the Albanese varieties of proper smooth varieties over a general base field, and study their polarizations. In Subsection \ref{ss:picard}, we generalize the construction of Picard motives using not necessarily ample divisors as cutting divisors, which will be used in Subsection \ref{ss:kunneth}.

In Section \ref{ss:2}, we make some preparation for algebraic cycles and height pairings for general varieties. In Subsection \ref{ss:cycles}, we review the notion of algebraic cycles and correspondences. In Subsection \ref{ss:bbp}, we review the construction of the Beilinson--Bloch height pairing and introduce our variant -- the Beilinson--Bloch--Poincar\'{e} height pairing. In Subsection \ref{ss:kunneth}, we discuss the construction of some K\"{u}nneth--Chow projectors for curves and surfaces, which will be used in the modified diagonal $\Delta^3_zX_K$ later.

In Section \ref{ss:3}, we construct Fourier--Jacobi cycles and state our main conjectures. In Subsection \ref{ss:motives}, we construct the category of CM data for a conjugate symplectic automorphic character $\mu$ of weight one. In Subsection \ref{ss:albanese_unitary}, we introduce our Shimura varieties and study their Albanese varieties; in particular, we prove Theorem \ref{th:main1}. In Subsection \ref{ss:construction_cycles}, we construct Fourier--Jacobi cycles and show that they are homologically trivial. In Subsection \ref{ss:aggp}, we prove the remaining part of Theorem \ref{th:main2}, and propose various versions of the arithmetic Gan--Gross--Prasad conjecture for $\rU(n)\times\rU(n)$.

In Section \ref{ss:4}, we discuss a relative trace formula approach toward the arithmetic GGP conjecture for $\rU(n)\times\rU(n)$. In Subsection \ref{ss:doubling}, we prove the doubling formula for CM data in Theorem \ref{th:main3}. In Subsection \ref{ss:invariant}, we introduce the global arithmetic invariant functional and its local version at good inert primes for which we perform some preliminary computation. In Subsection \ref{ss:orbital}, we prove Theorem \ref{th:main4}.

The article also contains four appendices.

In Section \ref{ss:e}, provided by Chao~Li and Yihang~Zhu, we confirm the arithmetic fundamental lemma in the minuscule case.

In Section \ref{ss:b}, we prove some results about global theta lifting for unitary groups, namely, Theorem \ref{th:pole} and its two corollaries. Those results are only used in the proof of Proposition \ref{pr:endoscopy_general}. Thus, if the readers are willing to admit these results from the theory of automorphic forms, they are welcome to skip the entire section except the very short Subsection \ref{ss:discrete} where we introduce some notation for the discrete automorphic spectrum.

In Section \ref{ss:c}, we summarize different versions of unitary Shimura varieties. In Subsection \ref{ss:appendix_isometry}, we recall Shimura varieties associated to isometry groups of hermitian spaces, which are of abelian type; we also introduce the Shimura varieties associated to incoherent hermitian spaces -- they are the main geometric objects studied in this article. In Subsection \ref{ss:appendix_similitude}, we recall the well-known PEL type Shimura varieties associated to groups of rational similitude of skew-hermitian spaces, and their integral models at good primes, after Kottwitz. These Shimura varieties are only for the preparation of the next subsection, which are not logically needed in the main part of the article. In Subsection \ref{ss:appendix_connection}, we summarize the connection of these two kinds of unitary Shimura varieties via the third one which possesses a moduli interpretation but is not of PEL type in the sense of Kottwitz, after \cites{BHKRY,RSZ}. In Subsection \ref{ss:integral_models}, we discuss integral models of the third unitary Shimura varieties at good inert primes and their uniformization along the basic locus. The last two subsections are crucial to the discussion in Subsections \ref{ss:invariant} and \ref{ss:orbital}.

In Section \ref{ss:d}, we compute the cohomology of Shimura curves associated to \emph{isometry} groups of hermitian spaces of rank $2$, as Galois--Hecke modules. In Subsection \ref{ss:local_theta}, we collect some results about local oscillator representations of unitary groups of general rank. In Subsection \ref{ss:setup}, we recall some facts and introduce some notation about cohomology of Shimura varieties in general.
These two subsections will be used both in Section \ref{ss:d} and in the main part of the article. The last two subsections concern the cohomology of unitary Shimura curves, for the statements and for the proof, respectively. These statements are only used in the proof of Theorem \ref{th:cm_albanese_pre} and Theorem \ref{th:cm_albanese} in the main part of the article, and are probably known to experts. However, we can not find any reference for the proofs or even for the statements themselves.

For readers' convenience, we summarize the logical dependence of the article in the following diagram.
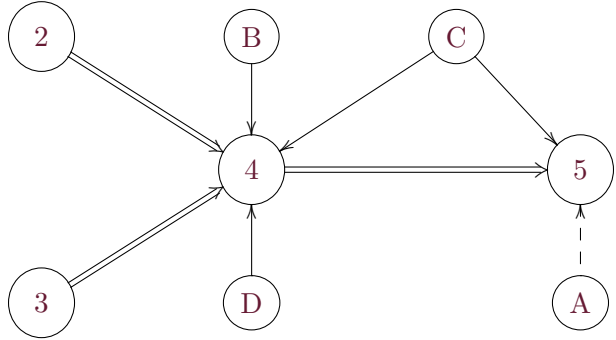
\[
\xymatrix{
*+++[o][F]{\txt{\ref{ss:a}}}* \ar@{=>}[rrd] &&
*++[o][F]{\txt{\ref{ss:b}}}* \ar[d] &&
*++[o][F]{\txt{\ref{ss:c}}}* \ar[dr]\ar[dll] \\
&& *+++[o][F]{\txt{\ref{ss:3}}}* \ar@{=>}[rrr] &&&
*+++[o][F]{\txt{\ref{ss:4}}}* \\
*+++[o][F]{\txt{\ref{ss:2}}}* \ar@{=>}[rru] &&
*++[o][F]{\txt{\ref{ss:d}}}* \ar[u] &&&
*++[o][F]{\txt{\ref{ss:e}}}* \ar@{-->}[u]
}
\]

\subsection{Notation and conventions}
\label{ss:notation}

\subsubsection*{General notation}

\begin{itemize}
  \item For a set $S$, we denote by $\CF_S$ the characteristic function of $S$.

  \item Suppose that we work in a category with finite products. Then
  \begin{itemize}
    \item for a finite collection $\{X_1,\dots,X_n\}$ of objects, the notation
        \[
        \tp_{abc\cdots}\colon X_1\times\cdots\times X_n\to X_a\times X_b\times X_c\times\cdots
        \]
        will, by default, stand for the projection to the factors labeled by the subset $\{a,b,c,\dots\}\subseteq\{1,\dots,n\}$;

    \item for an abject $X$ and an integer $r\geq 0$, we denote $\Delta^r\colon X\to X^r$ the diagonal morphism, and simply write $\Delta$ for $\Delta^2$.
  \end{itemize}

  \item All rings are commutative and unital; and ring homomorphisms preserve unity.

  \item For an abelian group $A$ and a ring $R$, we put $A_R\coloneqq A\otimes_\dZ R$ as an $R$-module.

  \item For a field $k$, we denote by $k^\ac$ an abstract algebraic closure of $k$.

  \item The \emph{bar} $\ol{\phantom{z}}$ only denotes the complex conjugation in $\dC$. For example, for an element $x\in\dC\otimes_\dQ E$, $\ol{x}$ is obtained by only applying conjugation to the first factor.

  \item We denote by $\dC^1$ the subgroup of $\dC^\times$ consisting of $z$ satisfying $z\ol{z}=1$.
\end{itemize}

\subsubsection*{Notation in number theory}

\begin{itemize}
  \item A reflex field is always a \emph{subfield} of $\dC$.

  \item Denote by $\bA^\infty\coloneqq\widehat\dZ_\dQ$ the ring of finite ad\`{e}les of $\dQ$, and put $\bA\coloneqq\dR\times\bA^\infty$.

  \item For a number field $k$, we put $\bA_k\coloneqq\bA\otimes_\dQ k$ and $\bA^\infty_k\coloneqq\bA^\infty\otimes_\dQ k$.
  \begin{itemize}
    \item Denote by $|\;|_\dQ\colon\dQ^\times\backslash\bA^\times\to\dR^\times_{>0}$ the character, uniquely determined by the properties that $|x|_\dQ=|x|$ for $x\in\dR^\times$ and that $|\;|_\dQ$ is trivial on $\widehat\dZ^\times$. For every $s\in\dC$, Put $|\;|_k^s\coloneqq|\;|^s_\dQ\circ\Nm_{k/\dQ}\colon k^\times\backslash\bA_k^\times\to\dR^\times_{>0}$.

    \item Denote by $\psi_\dQ\colon\dQ\backslash\bA\to\dC^\times$ the character, uniquely determined by the properties that $\psi_\dQ(x)=\exp(2\pi i x)$ for $x\in\dR$ and that $\psi_\dQ$ is trivial on $\widehat\dZ$. Put $\psi_k\coloneqq\psi_\dQ\circ\Tr_{k/\dQ}\colon k\backslash\bA_k\to\dC^\times$, which we call the \emph{standard additive character} for $k$.
  \end{itemize}

  \item In local or global class field theory, the Artin reciprocity map always sends a uniformizer at a nonarchimedean place $v$ to a geometric Frobenius element at $v$.
\end{itemize}

\subsubsection*{Notation in algebraic geometry}

\begin{itemize}
  \item For a scheme $S$ and a rational prime $p$, we denote by $\Sch_{/S}$ the category of schemes over $S$ and by p$\Sch'_{/S}$ the subcategory of those that are locally Noetherian. If $S=\Spec R$ is affine, then we simply replace $S$ by $R$ in the above notations.

  \item We denote by $\dG_\rm\coloneqq\Spec\dZ[T,T^{-1}]$ the multiplicative group scheme over $\dZ$. For integers $r,s\geq 1$, denote by $\Mat_{r,s}$ the scheme over $\dZ$ of $r$-by-$s$ matrices. For an integer $n\geq 1$, we put $\rM_n\coloneqq\Mat_{n,1}\times\Mat_{1,n}$.

  \item For a ring $R$, a scheme $X$ in $\Sch_{/R}$, and a ring $R'$ over $R$, we usually write $X_{R'}\in\Sch_{/R'}$ instead of $X\times_{\Spec R}\Spec S$.

  \item For a ring $R$, a scheme $X\in\Sch_{/R}$ that is locally of finite type, a homomorphism $\tau\colon R\to\dC$, and an abelian group $\Lambda$, we denote by $\rH^i_{\rB,\tau}(X,\Lambda)$ the degree $i$ singular cohomology of the underlying topological space of $X\otimes_{R,\tau}\dC$ with coefficients in $\Lambda$. When $R$ is a \emph{subring} of $\dC$ and $\tau$ is the inclusion, we suppress $\tau$ in the subscript.

  \item For a ring $R$, we denote by $\rD^\rb_{\r{fl}}(R)$ the bounded derived category of $R$-modules whose cohomology consists of $R$-modules of finite length. For $\cF\in \rD^\rb_{\r{fl}}(R)$, we have the Euler--Poincar\'{e} characteristic $\chi(\cF)\coloneqq\sum_{i}(-1)^i\length_R\rH^i\cF$. In general, for a (formal) scheme $X$ over $R$ and an element $\cF$ in the derived category of $\cO_X$-modules, we define its Euler--Poincar\'{e} characteristic $\chi(\cF)$ to be $\chi(\rR s_*\cF)$ (resp.\ $\infty$) if $\rR s_*\cF$ belongs to $\rD^\rb_{\r{fl}}(R)$ (resp.\ otherwise), where $s$ is the structure morphism.
\end{itemize}

\subsubsection*{Acknowledgements}

The author would like to thank Sungyoon~Cho, Benedict~Gross, Kai-Wen~Lan, Chao~Li, Xinyi~Yuan, Shouwu~Zhang, and Wei~Zhang for helpful comments and discussion, and thank Chao~Li and Yihang~Zhu for providing Section \ref{ss:e} for the proof of the arithmetic fundamental lemma in the minuscule case. He thanks the anonymous referee for careful reading and valuable comments. The research of the author is partially supported by NSF grant DMS--1702019 and a Sloan Research Fellowship.

\section{Albanese variety}
\label{ss:a}

In this section, we study the Albanese varieties. In Subsection \ref{ss:albanese_general}, we introduce the Albanese varieties of proper smooth varieties over a general base field, and study their polarizations. In Subsection \ref{ss:picard}, we generalize the construction of Picard motives using not necessarily ample divisors as cutting divisors.

Let $k$ be a field. We work in the category $\Sch_{/k}$.

\subsection{Albanese variety and its polarization}
\label{ss:albanese_general}

\begin{definition}\label{de:split}
Consider schemes $X,Y\in\Sch_{/k}$ of finite type.
\begin{enumerate}
  \item We denote by $\nabla X$ the smallest open and closed subscheme of $X\times X$ containing the diagonal $\Delta X$. For every morphism $u\colon Y\to X$, $u\times u$ restricts to a morphism $\nabla u\colon \nabla Y\to \nabla X$.

  \item We say that a field $k'$ over $k$ \emph{splits} $X$ if every connected component of $X_{k'}$ is geometrically connected. For such $k'$, we regard $\pi_0(X_{k'})$ as a scheme in $\Sch_{/k'}$, which induces a factorization of morphisms $X_{k'}\to\pi_0(X_{k'})\to\Spec k'$ in $\Sch_{/k'}$. In particular, giving an element in $X(\pi_0(X_{k'}))$ is equivalent to giving an element in $X_i(k')$ for every connected component $X_i$ of $X_{k'}$.

  \item Let $f\colon\nabla X\to Y$ be a morphism. For every field $k'$ over $k$ that splits $X$ and every element $x\in X(\pi_0(X_{k'}))$, we denote by
      \[
      f_x\colon X_{k'}\to Y_{k'}
      \]
      the morphism such that its restriction to a connected component $X_i$ is the restriction of $f_{k'}$ to $X_i\times_{k'}(x\cap X_i)\simeq X_i$.
\end{enumerate}
\end{definition}

The following proposition on the Albanese variety without rational base point is probably well-known. Since we can not find a precise reference for it, we include a proof.

\begin{proposition}
Let $X$ be a proper smooth scheme in $\Sch_{/k}$. Consider the functor $\underline\Alb_X$ on the category of abelian varieties $A$ over $k$ such that $\underline\Alb_X(A)$ is the set of morphisms $f\colon \nabla X\to A$ over $k$ such that $\Delta X$ is contained in $f^{-1}0_A$. Then $\underline\Alb_X$ is corepresentable.
\end{proposition}

\begin{proof}
Let $k'$ be a separable closure of $k$. Then $k'$ splits $X$ (Definition \ref{de:split}). Put $X'\coloneqq X_{k'}$. We first consider the problem for $X'$. Pick an element $x\in X(\pi_0(X'))$, which is possible as $X'$ is smooth over $k'$. By Serre's construction \cite{Ser59} of the Albanese variety (see \cite{Wit08}*{Appendix~A} for a version over separably closed field), we have a morphism $g_x\colon X'\to\Alb_{X'}$, universal among all morphisms $g\colon X'\to A$ to an abelian variety $A$ over $k'$ such that $g(x)=0_A$. Now it is easy to see that the composite morphism
\[
\nabla_{k'}X'\xrightarrow{g_x\times g_x}\Alb_{X'}\times_{k'}\Alb_{X'}\xrightarrow{-}\Alb_{X'}
\]
does not depend on the choice of $x$, and corepresents the functor $\underline\Alb_{X'}$. Here, $\nabla_{k'}X'$ is defined similarly as in Definition \ref{de:split}, but with the base field $k'$. As $(\nabla X)_{k'}\simeq \nabla_{k'}X'$, the statement for $X$ then follows by Galois descent.
\end{proof}

\begin{definition}\label{de:albanese}
Let $X$ be a proper smooth scheme in $\Sch_{/k}$. The abelian variety that corepresents the functor $\underline\Alb_X$ is called the \emph{Albanese variety} of $X$, denoted by $\Alb_X$. The canonical morphism, denoted by
\[
\alpha_X\colon \nabla X\to\Alb_X,
\]
is called the \emph{Albanese morphism}. For a morphism $u\colon Y\to X$ of proper smooth schemes over $k$, we have the induced morphism $\Alb_u\colon\Alb_Y\to\Alb_X$ by the universal property, which satisfies $\Alb_u\circ\alpha_X=\alpha_Y\circ\nabla u$.
\end{definition}

\begin{lem}\label{le:albanese}
Suppose that $k$ has characteristic zero. Then
\begin{enumerate}
  \item for every homomorphism $\tau\colon k\to\dC$, we have a canonical isomorphism $\rH^1_{\rB,\tau}(\Alb_X,\dQ)\simeq\rH^1_{\rB,\tau}(X,\dQ)$;

  \item for every prime $\ell$, we have a canonical isomorphism $\rH^1_{\et}((\Alb_X)_{k^\ac},\dQ_\ell)\simeq\rH^1_{\et}(X_{k^\ac},\dQ_\ell)$ of $\Gal(k^\ac/k)$-modules.
\end{enumerate}
\end{lem}

\begin{proof}
For (1), we pick an element $x\in X(\pi_0(X\otimes_{k,\tau}\dC))$, which induces a morphism
\[
(\alpha_X)_x\colon X\otimes_{k,\tau}\dC\to\Alb_X\otimes_{k,\tau}\dC
\]
from Definition \ref{de:albanese} and Definition \ref{de:split}. By the property of complex Albanese varieties, the induced map
\[
(\alpha_X)_x^*\colon\rH^1_{\rB,\tau}(\Alb,\dQ)\to\rH^1_{\rB,\tau}(X,\dQ)
\]
is an isomorphism; it is independent of the choice of $x$ since translation acts trivially on $\rH^1_{\rB,\tau}(\Alb,\dQ)$.

For (2), we extend the morphism $\alpha_X$ to $\alpha'_X\colon X\times X\to\Alb_X$ by letting $X\times X\setminus\nabla X$ map to $0_{\Alb_X}$.
By the K\"{u}nneth formula, we have the map
\[
(\alpha'_X)^*\colon\rH^1_{\et}((\Alb_X)_{k^\ac},\dQ_\ell)\to\rH^1_{\et}(X_{k^\ac},\dQ_\ell)\otimes_{\dQ_\ell}\rH^0_{\et}(X_{k^\ac},\dQ_\ell)
\]
of $\Gal(k^\ac/k)$-modules. Taking cup product, we obtain a map
\[
\alpha''_X\colon\rH^1_{\et}((\Alb_X)_{k^\ac},\dQ_\ell)\to\rH^1_{\et}(X_{k^\ac},\dQ_\ell)
\]
of $\Gal(k^\ac/k)$-modules. It suffices to show that this is an isomorphism. Since $k$ has characteristic $0$, by the Lefschetz principle and the comparison theorem for singular and \'{e}tale cohomology, $\alpha''_X\otimes_{\dQ_\ell}\dC$ via any embedding $\dQ_\ell\hookrightarrow\dC$ is isomorphic to the canonical map in (1). Thus, (2) follows.
\end{proof}

Now we study polarizations of $\Alb_X$. Let $k$ be an arbitrary field.

\begin{proposition}\label{pr:albanese_theta}
Let $X$ be a proper smooth scheme and $A$ an abelian variety, both over $k$. For every $f\in\underline{\Alb}_X(A)$ and every divisor $D$ on $X$, there is a unique homomorphism
\[
\theta_{f,D}\colon A^\vee\to A
\]
(over $k$) satisfying the following property: for every field $k'$ over $k$, every geometric point $a$ of $A^\vee(k')$ corresponding to a line bundle $L_a$ on $A'\coloneqq A_{k'}$, and every element $x\in X(\pi_0(X_{k'}))$, we have
\[
\theta_{f,D}(a)=\Sigma_{A'}\(c_1(L_a).f_{x*}(D^{\dim X-1})\),
\]
where $\Sigma_{A'}\colon\CH_0(A')\to A(k')$ is the (classical) Albanese map for $A'$. Moreover, $\theta_{f,D}$ is symmetric, depends only on the rational equivalence class of $D$, and satisfies $\theta_{f,nD}=[n^{\dim X-1}]_A\circ\theta_{f,D}$ for $n\in\dZ$.
\end{proposition}

This is previously known when $D$ is a hyperplane section. See, for example, \cite{Mur90}*{Section~2}.

\begin{proof}
The uniqueness is clear. Now we show the existence. We may assume that $k$ is separably closed. In fact, for every element $x\in X(\pi_0(X))$, we are going to define a homomorphism $\theta_{f,D,x}$ satisfying the requirement in the proposition. Then we will show that $\theta_{f,D,x}$ does not depend on the choice of $x$. Therefore, by Galois descent, we conclude for the general field $k$.

We start from the construction of $\theta_{f,D,x}$. Let $\cP$ be the Poincar\'{e} line bundle on $A^\vee\times A$. Consider the following diagram of projection homomorphisms
\begin{align}\label{eq:albanese_theta}
\xymatrix{
& A^\vee\times A^\vee\times A \ar[dl]_{\tp_{12}}\ar@<-3pt>[d]_-{\tp_{13}}\ar@<3pt>[d]^-{\tp_{23}}\ar[dr]^-{\tp_3} \\
A^\vee\times A^\vee & A^\vee\times A & A.
}
\end{align}
For every $z\in\CH_1(A)$, put
\[
D_z\coloneqq \tp_{12*}(\tp_{13}^*c_1(\cP).\tp_{23}^*c_1(\cP).\tp_3^*z),
\]
which belongs to $\CH^1(A^\vee\times A^\vee)$. Then we put $\cL_z\coloneqq\cO_{A^\vee\times A^\vee}(D_z)$. We show
\begin{enumerate}
  \item $\cL_z$ is symmetric, that is, $\cL_z$ is invariant under the obvious involution of $A^\vee\times A^\vee$;

  \item the restrictions of $\cL_z$ to $0_{A^\vee}\times A^\vee$ and $A^\vee\times 0_{A^\vee}$ are both trivial;

  \item for every point $a\in A^\vee(k)$, the restriction of $\cL_z$ to $a\times A^\vee$ corresponds to the point $\Sigma_A(c_1(L_a).z)$ under the canonical isomorphism $A^{\vee\vee}\simeq A$.
\end{enumerate}
Part (1) is straightforward from the definition. For (2), it suffices to show that the restricted line bundle $\cL_z\res{0_{A^\vee}\times A^\vee}$ is trivial by (1). However, this is a special case of (3).

Now we show (3). We expand the previous commutative diagram \eqref{eq:albanese_theta} to the following one
\[
\xymatrix{
& A^\vee\times A\simeq a\times A^\vee\times A \ar[r]^-{j}\ar[dl]_-{q_1}
& A^\vee\times A^\vee\times A \ar[dl]_{\tp_{12}}\ar@<-3pt>[d]_-{\tp_{13}}\ar@<3pt>[d]^-{\tp_{23}}\ar[dr]^-{\tp_3} \\
A^\vee\simeq a\times A^\vee \ar[r]^-{i} & A^\vee\times A^\vee & A^\vee\times A & A
}
\]
in which the parallelogram is Cartesian. By \cite{Ful98}*{Proposition~1.7}, we have
\begin{align}\label{eq:albanese_theta1}
\cL_z\res{a\times A^\vee}\simeq i^*\cL_z\simeq\cO_{A^\vee}\(q_{1*}j^*(\tp_{13}^*c_1(\cP).\tp_{23}^*c_1(\cP).\tp_3^*z)\).
\end{align}
We put $q_2\coloneqq\tp_3\circ j\colon A^\vee\times A\to A$ which is simply the projection to the second factor. Since $j^*\tp_{13}^*\cP$ is isomorphic to $q_2^*L_a$, we have
\[
\eqref{eq:albanese_theta1}\simeq \cO_{A^\vee}\(q_{1*}(q_2^*c_1(L_a).c_1(\cP).q_2^*z)\)
\simeq\cO_{A^\vee}\(q_{1*}(c_1(\cP).q_2^*(c_1(L_a).z))\).
\]
It remains to show that the line bundle $\cL'$ on $A^\vee$ corresponding to the point $\Sigma_A(c_1(L_a).z)$ is $\cO_{A^\vee}\(q_{1*}(c_1(\cP).q_2^*(c_1(L_a).z))\)$. Choose a representative $\sum_i m_ia_i$ of the $0$-cycle $c_1(L_a).z$; it has degree zero since $L_a$ is algebraically equivalent to zero. Then we have
\[
\cL'\simeq\bigotimes_i\cL_{a_i}^{\otimes m_i},
\]
where $\cL_{a_i}$ is the line bundle on $A^\vee$ corresponding to $a_i$ which, by the property of the Poincar\'{e} bundle, is isomorphic to $\cO_{A^\vee}(q_{1*}(c_1(\cP).q_2^*a_i))$. Thus, we have
\[
\cL'\simeq\bigotimes_i\cO_{A^\vee}(q_{1*}(c_1(\cP).q_2^*a_i))^{\otimes m_i}\simeq\cO_{A^\vee}\(q_{1*}(c_1(\cP).q_2^*(c_1(L_a).z))\);
\]
and (3) is proved.

By (1) and (2), the line bundle $\cL_z$ induces a symmetric homomorphism $\theta_{z}\colon A^\vee\to A$. Now taking $z=f_{x*}D^{\dim X-1}$, we obtain a symmetric homomorphism $\theta_{f,D,x}\colon A^\vee\to A$ satisfying the requirement in the proposition. To construct $\theta_{f,D}$, it suffices to show that $\theta_{f,D,x}=\theta_{f,D,y}$ for any other choice of $y$. This amounts to showing that
\begin{align}\label{eq:albanese_theta2}
\Sigma_A\(c_1(L_a).f_{x*}(D^{\dim X-1})\)=\Sigma_A\(c_1(L_a).f_{y*}(D^{\dim X-1})\).
\end{align}
Put $b\coloneqq f_x(y)\in A(k)$. Then we have $f_y=t_b\circ f_x$, where $t_b$ is the translation morphism on $A$ by $b$. Since $L_a$ is algebraically equivalent to zero, $c_1(L_a).f_{y*}(D^{\dim X-1})$ is a degree zero divisor. Thus, we have
\[
\Sigma_A\(c_1(L_a).f_{y*}(D^{\dim X-1})\)=\Sigma_A\(t_{-b*}c_1(L_a).f_{x*}(D^{\dim X-1})\).
\]
Again, since $L_a$ is algebraically equivalent to zero, we have $t_{-b*}c_1(L_a)=c_1(L_a)\in\CH^1(A)$. Thus, \eqref{eq:albanese_theta2} follows.

The last assertion of the proposition is already clear.
\end{proof}

In the case where $(A,f)=(\Alb_X,\alpha_X)$, we will simply write
\[
\theta_{X,D}\coloneqq\theta_{\alpha_X,D}\colon\Alb_X^\vee\to\Alb_X.
\]

\begin{remark}
If $\dim X=1$, then $\theta_{X,D}$ is the canonical polarization of $\Alb_X$ (which is simply the Jacobian of $X$), hence is an isomorphism and is independent of $D$.
\end{remark}

We have the following result on the functoriality of $\theta_{X,D}$.

\begin{proposition}\label{pr:albanese_functoriality}
Let $u\colon Y\to X$ be a generically finite dominant morphism of proper smooth schemes over $k$. Let $D$ be a divisor on $X$. Then we have
\[
[\deg u]_{\Alb_X}\circ\theta_{X,D}=\Alb_u\circ\theta_{Y,u^*D}\circ\Alb_u^\vee.
\]
Here, $\deg u$ is regarded as a function on $\pi_0(X)$ whose value on a connected component of $X$ is the total degree of $u$ over it; and if we write $X=\coprod X_i$, then $[\deg u]_{\Alb_X}$ is the endomorphism $\prod_i[(\deg u)(X_i)]_{\Alb_{X_i}}$ on $\Alb_X\simeq\prod_i\Alb_{X_i}$.
\end{proposition}

\begin{proof}
We may assume that $k$ is algebraically closed and that both $X$ and $Y$ are connected. Put $d\coloneqq\dim X=\dim Y$. Take points $a\in\Alb_X^\vee(k)$ and $y\in Y(k)$. Put $b\coloneqq\Alb_u^\vee(a)\in\Alb_Y^\vee(k)$ and $x\coloneqq u(y)\in X(k)$. Put $f\coloneqq\alpha_{X,x}\colon X\to\Alb_X$ and $g\coloneqq\alpha_{Y,y}\colon Y\to\Alb_Y$ for short. By the functoriality of Albanese morphisms, the following diagram
\[
\xymatrix{
Y \ar[r]^-{g}\ar[d]_-{u} & \Alb_Y \ar[d]^-{\Alb_u} \\
X \ar[r]^-{f} & \Alb_X
}
\]
commutes. To prove the proposition, it suffices to show that
\begin{align}\label{eq:albanese_functoriality}
[\deg u]_{\Alb_X}(\theta_{X,D}(a))=\Alb_u(\theta_{Y,u^*D}(b)).
\end{align}
By Proposition \ref{pr:albanese_theta} and the projection formula \cite{Ful98}*{Example~8.1.7}, the left-hand side of \eqref{eq:albanese_functoriality} equals
\begin{align}\label{eq:albanese_functoriality1}
[\deg u]_{\Alb_X}\(\Sigma_{\Alb_X}\(c_1(L_a).f_*(D^{d-1})\)\)=\Sigma_{\Alb_X}\(\deg u\cdot f_*\(f^*c_1(L_a).D^{d-1}\)\).
\end{align}
Again by the projection formula, we have
\[
\deg u\cdot f^*c_1(L_a).D^{d-1}=f^*c_1(L_a).u_*(u^*D^{d-1}).
\]
Repeatedly applying the projection formula, we have
\begin{align*}
\eqref{eq:albanese_functoriality1}&=\Sigma_{\Alb_X}\(f_*\(f^*c_1(L_a).u_*(u^*D^{d-1})\)\)\\
&=\Sigma_{\Alb_X}\(f_*\(u_*\(u^*f^*c_1(L_a).(u^*D)^{d-1}\)\)\)\\
&=\Sigma_{\Alb_X}\(\Alb_{u*}g_*\(g^*\Alb_u^*c_1(L_a).(u^*D)^{d-1}\)\)\\
&=\Alb_u\(\Sigma_{\Alb_Y}\(g_*\(g^*c_1(L_b).(u^*D)^{d-1}\)\)\)\\
&=\Alb_u\(\Sigma_{\Alb_Y}\(c_1(L_b).g_*(u^*D)^{d-1}\)\)\\
&=\Alb_u(\theta_{Y,u^*D}(b)).
\end{align*}
The proposition follows.
\end{proof}

\begin{definition}\label{de:almost_ample}
We say that a divisor $D$ on a proper smooth scheme $X$ over $k$ is \emph{almost ample} if there exists $m\in\dZ_{>0}$ such that $|mD|$ is base point free and that the induced morphism $\phi_{mD}\colon X\to\dP(|mD|)$ is a generically finite morphism onto its image.
\end{definition}

\begin{proposition}\label{pr:almost_ample}
Suppose that $k$ has characteristic zero. Let $X$ be a proper smooth scheme in $\Sch_{/k}$ and $D$ a divisor on $X$ such that $D$ is almost ample. Then the symmetric homomorphism $\theta_{X,D}\colon\Alb_X^\vee\to\Alb_X$ is a polarization.
\end{proposition}

\begin{proof}
Since $k$ has characteristic zero, by the Lefschetz principle, we may assume that $k$ is embeddable into $\dC$. To check whether $\theta_{X,D}$ is a polarization, we may assume $k=\dC$ and that $X$ is connected. Since $D$ is almost ample, by replacing $D$ by $mD$ for some $m\in\dZ_{>0}$, we may assume that $|D|$ is base point free and that the induced morphism $\phi_D\colon X\to\dP(|D|)$ is a generically finite morphism onto its image.

Put $A\coloneqq\Alb_X$, $d\coloneqq\dim X$, and $h\coloneqq\dim A$ for short. We choose a point $x\in X(\dC)$, and put $f\coloneqq\alpha_{X,x}\colon X\to A$. We have canonical isomorphisms
\[
A^\vee(\dC)\simeq\rH^1(A,\cO_A)/\rH^1(A,\dZ),\qquad
A(\dC)\simeq\rH^h(A,\Omega_A^{h-1})/\rH^{2h-1}(A,\dZ)
\]
of complex manifolds. From the construction, the following diagram
\[
\xymatrix{
\rH^1(A,\cO_A) \ar[r]^-{\wedge f_*c_1(D)^{d-1}}\ar[d] & \rH^h(A,\Omega_A^{h-1}) \ar[d] \\
\rH^1(A,\cO_A)/\rH^1(A,\dZ) \ar[r]^-{\theta_{X,D}} & \rH^h(A,\Omega_A^{h-1})/\rH^{2h-1}(A,\dZ)
}
\]
commutes, where the vertical arrows are quotient maps. Here, $c_1(D)$ is regarded as the Chern class in $\rH^1(X,\Omega_X)$. Then the symmetric homomorphism $\theta_{X,D}$ is a polarization if and only if for every nonzero $\ol\partial$-closed smooth $(0,1)$-form $\omega$ on $A$, we have
\[
\int_{A(\dC)}\omega\wedge\ol\omega\wedge f_*c_1(D)^{d-1}>0.
\]
By the property that $D$ satisfies, we may find a smooth hermitian metric $\|\;\|_D$ on $\cO_X(D)$ such that its Chern $(1,1)$-form $c_1(\|\;\|_D)$ is semi-positive on $X(\dC)$ and strictly positive on a Zariski dense open subset. Therefore,
\begin{align*}
\int_{A(\dC)}\omega\wedge\ol\omega\wedge f_*c_1(D)^{d-1}
&=\int_{X(\dC)}f^*\omega\wedge\ol{f^*\omega}\wedge c_1(D)^{d-1} \\
&=\int_{X(\dC)}f^*\omega\wedge\ol{f^*\omega}\wedge c_1(\|\;\|_D)^{d-1}
>0.
\end{align*}
The proposition follows.
\end{proof}

\begin{remark}\label{re:biproduct}
There is a byproduct in proof of Proposition \ref{pr:almost_ample}: For an almost ample divisor $D$ on a proper smooth scheme $X$ over a field $k$ of characteristic zero, the degree of the top intersection $\deg D^{\dim X}$ is strictly positive on every irreducible component of $X$.
\end{remark}

\begin{remark}
We are curious whether one can find an algebraic proof of Proposition \ref{pr:almost_ample}, and whether the proposition holds for an arbitrary field $k$ or a weaker condition on $D$. Note that if $D$ is a hyperplane, then it is previously known that $\theta_{X,D}$ is an isogeny for an arbitrary field $k$.
\end{remark}

\subsection{Picard motives via almost ample divisors}
\label{ss:picard}

Let $k$ be a field of characteristic zero. Let $X$ be a proper smooth scheme in $\Sch_{/k}$ of pure dimension $d\geq 1$. For every almost ample divisor $D$ on $X$, we now define a correspondence $e_{X,D}\in\CH^d(X\times X)_\dQ$ such that the induced endomorphism
\[
\cl_\dr^*(e_{X,D})\colon\bigoplus_{i=0}^{2d}\rH^i_\dr(X/k)\to\bigoplus_{i=0}^{2d}\rH^i_\dr(X/k)
\]
on the de Rham cohomology of $X$ is the projection onto $\rH^1_\dr(X/k)$. In particular, when $X$ is projective, $(X,e_{X,D})$ is a Grothendieck motive, which is a Picard motive for $X$. The construction generalizes the one in \cite{Mur90}*{Section~3}. We use such construction only in Subsection \ref{ss:kunneth} when the Shimura variety is a non-proper surface; so the readers may choose to skip this subsection for now.

Let $\theta\coloneqq\theta_{X,D}\colon\Alb_X^\vee\to\Alb_X$ be the polarization obtained from Proposition \ref{pr:almost_ample}. Let $\vartheta\colon\Alb_X\to\Alb_X^\vee$ be an isogeny such that $\theta\circ\vartheta=[n]_{\Alb_X}$ for some integer $n\geq 1$. We obtain a morphism
\[
\beta\coloneqq(\vartheta\circ\alpha_X)\times\alpha_X\colon\nabla X\times\nabla X\to\Alb_X^\vee\times\Alb_X.
\]
Let $\cP$ be the Poincar\'{e} line bundle on the target. We put
\[
E_{X,D}\coloneqq \tp_{24*}\(\beta^*c_1(\cP).(D^d\times X\times D^d\times D^{d-1})\)\in\CH^d(X\times X)_\dQ,
\]
where the intersection is taken in $X\times X\times X\times X$, and
\[
e_{X,D}\coloneqq\frac{1}{n(\deg D^d)^2}E_{X,D}\in\CH^d(X\times X)_\dQ,
\]
where $\deg D^d$ is understood as a function on $\pi_0(X)$. We leave the readers an easy exercise to show that $e_{X,D}$ does not depend on the choice of $\vartheta$.

\begin{proposition}\label{pr:picard_functoriality}
Let $X$ be a proper smooth scheme in $\Sch_{/k}$ of pure dimension $d\geq 1$, and $D$ an almost ample divisor on $X$.
\begin{enumerate}
  \item The map $\cl^*_\dr(e_{X,D})$ coincides with the projection to $\rH^1_\dr(X/k)$.

  \item Let $u\colon Y\to X$ be a generically finite dominant morphism of proper smooth schemes over $k$. Then $u^*D$ is an almost ample divisor on $Y$, and we have
      \[
      (\r{id}_Y\times u)_*e_{Y,u^*D}=(u\times\r{id}_X)^*e_{X,D}
      \]
      in $\CH^d(Y\times X)_\dQ$.
\end{enumerate}
\end{proposition}

\begin{proof}
For both assertions, we may assume that $k$ is algebraically closed and $X$ is connected.

For (1), recall that for every $x\in X(k)$, we have the induced morphism $(\alpha_X)_x\colon X\to\Alb_X$ by restriction. Now take two arbitrary points $x,x^\vee\in X(k)$. We have the induced morphism
\[
(\vartheta\circ(\alpha_X)_{x^\vee})\times(\alpha_X)_x\colon X\times X\to\Alb_X^\vee\times\Alb_X.
\]
Put $E\coloneqq((\vartheta\circ(\alpha_X)_{x^\vee})\times(\alpha_X)_x)^*c_1(\cP).(X\times D^{d-1})\in\CH^d(X\times X)_\dQ$. It suffices to show that the induced map $\cl_\dr^*(E)$ on the de Rham cohomology of $X$ is the projection onto $\rH^1_\dr(X/k)$ multiplied by $n$.

As $\cl_\dr(c_1(\cP))\in\rH^1_\dr(\Alb_X^\vee/k)\otimes_k\rH^1_\dr(\Alb_X/k)$, we have $\cl_\dr(E)\in\rH^1_\dr(X/k)\otimes_k\rH^{2d-1}_\dr(X/k)$, which implies that $\cl_\dr^*(E)\res{\rH^i_\dr(X/k)}=0$ unless $i=1$. It remains to show that $\cl_\dr^*(E)$ acts on $\rH^1_\dr(X/k)$ via the multiplication by $n$. By Lemma \ref{le:albanese} and the comparison theorem, it suffices to show that the correspondence
\[
(\vartheta\times\r{id}_{\Alb_X})^*c_1(\cP).(\Alb_X\times(\alpha_X)_{x*}D^{d-1})\in\CH^h(\Alb_X\times\Alb_X)_\dQ
\]
induces the multiplication by $n$ on $\rH^1_\dr(\Alb_X/k)$, where $h$ is the dimension of $\Alb_X$. This in turn is equivalent to that the correspondence
\[
(\theta\times\r{id}_{\Alb_X})^*(\vartheta\times\r{id}_{\Alb_X})^*c_1(\cP).(\Alb_X^\vee\times(\alpha_X)_{x*}D^{d-1})
\in\CH^h(\Alb_X^\vee\times\Alb_X)_\dQ
\]
induces the map $n\cdot\theta^*\colon\rH^1_\dr(\Alb_X/k)\to\rH^1_\dr(\Alb_X^\vee/k)$. However, we have $\theta\circ\vartheta=[n]_{\Alb_X}$, which implies $\vartheta\circ\theta=[n]_{\Alb_X^\vee}$, hence
\[
(\theta\times\r{id}_{\Alb_X})^*(\vartheta\times\r{id}_{\Alb_X})^*c_1(\cP)=([n]_{\Alb_X^\vee}\times\r{id}_{\Alb_X})^*c_1(\cP)=n\cdot c_1(\cP).
\]
On the other hand, the construction of $\theta$ in Proposition \ref{pr:albanese_theta} implies that the correspondence $c_1(\cP).(\Alb_X^\vee\times(\alpha_X)_{x*}D^{d-1})$ exactly induces the restriction $\theta^*\colon\rH^1_\dr(\Alb_X/k)\to\rH^1_\dr(\Alb_X^\vee/k)$. Thus, (1) is proved.

For (2), the assertion that $u^*D$ is almost ample follows directly from Definition \ref{de:almost_ample}. For the rest, we may assume that $k(Y)/k(X)$ is Galois. In fact, by the resolution of singularity, we can always find another generically finite dominant morphism of connected proper smooth schemes $v\colon Z\to Y$ such that $k(Z)/k(X)$ is Galois. Now if (2) holds for $v$ and $u\circ v$, then it holds for $u$. Thus, we may assume that $k(Y)/k(X)$ is Galois with the Galois group $\Gamma$.

Choose two arbitrary points $y,y^\vee\in Y(k)$, and put $x\coloneqq u(y)$ and $x^\vee\coloneqq u(y^\vee)$. Put
\[
\beta_X\coloneqq(\alpha_X)_{x^\vee}\times(\alpha_X)_x\colon X\times X\to\Alb_X\times\Alb_X,
\]
and similarly for $\beta_Y$. We choose a homomorphism $\vartheta_X\colon\Alb_X\to\Alb_X^\vee$ (resp.\ $\vartheta_Y\colon\Alb_Y\to\Alb_Y^\vee$) such that $\theta_{X,D}\circ\vartheta_X=[n_X]_{\Alb_X}$ (resp.\ $\theta_{Y,u^*D}\circ\vartheta_Y=[n_Y]_{\Alb_Y}$). Put
\begin{align*}
E_X&\coloneqq\beta_X^*(\r{id}_{\Alb_X}\times\vartheta_X)^*c_1(\cP_X).(X\times D^{d-1}),\\
E_Y&\coloneqq\beta_Y^*(\r{id}_{\Alb_Y}\times\vartheta_Y)^*c_1(\cP_Y).(Y\times u^*D^{d-1}),
\end{align*}
where $\cP_X$ (resp.\ $\cP_Y$) is the Poincar\'{e} line bundle on $\Alb_X\times\Alb_X^\vee$ (resp.\ $\Alb_Y\times\Alb_Y^\vee$). Then the formula in (2) follows from the symmetry of Poincar\'{e} bundles, and the identity
\begin{align*}
(\r{id}_Y\times u)_*E_Y=\frac{n_Y}{n_X}\cdot(u\times\r{id}_X)^* E_X
\end{align*}
in $\CH^d(Y\times X)_\dQ$. By the projection formula, this in turn follows from
\begin{align}\label{eq:picard_functoriality}
(\r{id}_Y\times u)_*\beta_Y^*(\r{id}_{\Alb_Y}\times\vartheta_Y)^*c_1(\cP_Y)=
\frac{n_Y}{n_X}\cdot (u\times\r{id}_X)^*\beta_X^*(\vartheta_X\times\r{id}_{\Alb_X})^*c_1(\cP_X).
\end{align}
Consider the following diagram
\[
\xymatrix{
Y\times Y \ar[r]^-{\beta_Y}\ar[d]_-{\r{id}_Y\times u} & \Alb_Y\times\Alb_Y \ar@{-->}[rr]^-{\r{id}_{\Alb_Y}\times\vartheta_Y}\ar[d]^-{\r{id}_{\Alb_Y}\times\Alb_u}
&& \Alb_Y\times\Alb_Y^\vee \\
Y\times X \ar[r]^-{\beta}\ar[d]_-{u\times\r{id}_X} &\Alb_Y\times\Alb_X
\ar[rr]^-{\r{id}_{\Alb_Y}\times\vartheta_X}\ar[d]^-{\Alb_u\times\r{id}_{\Alb_X}}
&& \Alb_Y\times\Alb_X^\vee \ar[u]_-{\r{id}_{\Alb_Y}\times\Alb_u^\vee} \ar[d]^-{\Alb_u\times\r{id}_{\Alb_X^\vee}}   \\
X\times X \ar[r]^-{\beta_X} & \Alb_X\times\Alb_X
\ar[rr]^-{\r{id}_{\Alb_X}\times\vartheta_X} && \Alb_X\times\Alb_X^\vee
}
\]
where $\beta\coloneqq(\alpha_Y)_y\times(\alpha_X)_{x^\vee}$. Note that squares involving the dash arrow do not necessarily commute. By the isomorphism $(\Alb_u\times\r{id}_{\Alb_X^\vee})^*\cP_X\simeq(\r{id}_{\Alb_Y}\times\Alb_u^\vee)^*\cP_Y$, \eqref{eq:picard_functoriality} is equivalent to
\begin{align}\label{eq:picard_functoriality1}
(\r{id}_Y\times u)_*\beta_Y^*(\r{id}_{\Alb_Y}\times\vartheta_Y)^*c_1(\cP_Y)=\frac{n_Y}{n_X}\cdot
\beta^*(\r{id}_{\Alb_Y}\times\vartheta_X)^*(\r{id}_{\Alb_Y}\times\Alb_u^\vee)^*c_1(\cP_Y).
\end{align}
By the projection formula, \eqref{eq:picard_functoriality1} is equivalent to that
\[
\deg u\cdot n_X\cdot(\r{id}_{\Alb_Y}\times\vartheta_Y)^*c_1(\cP_Y) -
n_Y\cdot(\r{id}_{\Alb_Y}\times(\Alb_u^\vee\circ\vartheta_X\circ\Alb_u))^*c_1(\cP_Y)
\]
is contained in the kernel of $(\r{id}_Y\times u)_*\circ\beta_Y^*$. Now the Galois group $\Gamma$ acts on $\Alb_Y$ via the homomorphisms $\gamma\colon\Alb_Y\to\Alb_Y$ for $\gamma\in\Gamma$. We have a similar action on $\Alb_Y^\vee$ by duality, and the homomorphism $\vartheta_Y$ is $\Gamma$-equivariant since the divisor $u^*D$ is $\Gamma$-invariant. For a line bundle $\cL$ on $\Alb_Y\times\Alb_Y$, we have the trace line bundle
\[
\cL_\Gamma\coloneqq\bigotimes_{\gamma\in\Gamma}(\r{id}_{\Alb_Y}\times\gamma)^*\cL.
\]
Moreover, if $\cL_\Gamma$ is torsion, then $c_1(\cL)$ is in the kernel of $(\r{id}_Y\times u)_*\circ\beta_Y^*$. We define similarly $\cL_\Gamma$ for line bundles $\cL$ on $\Alb_Y\times\Alb_Y^\vee$.

In all, \eqref{eq:picard_functoriality1} will follow from the following claim: For $\cP\coloneqq\cP_Y$ on $\Alb_Y\times\Alb_Y^\vee$, the line bundle
\[
(\r{id}_{\Alb_Y}\times\vartheta_Y)^*\cP_\Gamma^{\otimes n_X}\otimes
(\r{id}_{\Alb_Y}\times(\Alb_u^\vee\circ\vartheta_X\circ\Alb_u))^*\cP^{\otimes-n_Y}
\]
is torsion. An easy diagram chasing implies that the claim will follow if we can show that the two homomorphisms
\begin{align}\label{eq:picard_functoriality2}
\Alb_u^\vee\circ\vartheta_X\circ\Alb_u\circ[n_Y]_{\Alb_Y},\qquad
\sum_{\gamma\in\Gamma}\gamma^\vee\circ\vartheta_Y\circ[n_X]_{\Alb_Y}
\end{align}
from $\Alb_Y$ to $\Alb_Y^\vee$ coincide. However, this can be checked on the level of $k$-points as the base field is algebraically closed of characteristic zero. Then we have a homomorphism $u_*\colon\Alb_Y^\vee(k)\to\Alb_X^\vee(k)$ induced by pushforward of divisors along $u\colon Y\to X$ as we have $\Alb_X^\vee\simeq\Pic^0_X$ and $\Alb_Y^\vee\simeq\Pic^0_Y$. By the definition of pushforward, the diagram
\[
\xymatrix{
& \Alb_X^\vee(k) \ar[dr]^-{\Alb_u^\vee} \\
\Alb_Y^\vee(k) \ar[ur]_-{u_*}\ar[rr]^-{\sum_{\gamma\in\Gamma}\gamma^\vee} && \Alb_Y^\vee(k)
}
\]
commutes; and by the projection formula, the diagram
\[
\xymatrix{
\Alb_Y^\vee(k) \ar[rr]^-{\theta_{Y,u^*D}}\ar[d]_-{u_*} && \Alb_Y(k) \ar[d]^-{\Alb_u} \\
\Alb_X^\vee(k) \ar[rr]^-{\theta_{X,D}} && \Alb_X(k)
}
\]
commutes as well. The two diagrams imply the coincidence of the two homomorphisms in \eqref{eq:picard_functoriality2}. Thus, the claim hence (2) follow.
\end{proof}

\section{Algebraic cycles and height pairings}
\label{ss:2}

In this section, we make some preparation for algebraic cycles and height pairings for general varieties. In Subsection \ref{ss:cycles}, we review the notion of algebraic cycles and correspondences. In Subsection \ref{ss:bbp}, we review the construction of the Beilinson--Bloch height pairing and introduce our variant -- the Beilinson--Bloch--Poincar\'{e} height pairing. In Subsection \ref{ss:kunneth}, we discuss the construction of some K\"{u}nneth--Chow projectors for curves and surfaces, which will be used in the modified diagonal $\Delta^3_zX_K$ later.

Let $k$ be a field of characteristic zero. We work in the category $\Sch_{/k}$.

\subsection{Cycles and correspondences}
\label{ss:cycles}

Consider a proper smooth scheme $X\in\Sch_{/k}$ of pure dimension $d$. Let $\rZ^i(X)$ (resp.\ $\CH^i(X)$) be the abelian group of algebraic cycles (resp.\ Chow cycles) on $X$ of codimension $i$, with a natural surjective map $\rZ^i(X)\to\CH^i(X)$. For example, we have the diagonal cycle
\[
\Delta^rX\in\rZ^{(r-1)d}(X^r)
\]
for $r\geq 1$ as the image of the diagonal morphism $\Delta^r\colon X\to X^r$. We write $\Delta X$ for $\Delta^2 X$ for simplicity.

We have the de Rham cycle class map
\[
\cl_\dr\colon\CH^i(X)_\dQ\to\rH^{2i}_\dr(X/k),
\]
whose kernel we denote by $\CH^i(X)^0_\dQ$. By various comparison theorems, $\CH^i(X)^0_\dQ$ coincides with the kernel of the Betti cycle class map
\[
\cl_{\rB,\tau}\colon\CH^i(X)_\dQ\to\rH^{2i}_{\rB,\tau}(X,\dQ)
\]
for every embedding $\tau\colon k\hookrightarrow\dC$, and the $\ell$-adic cycle class map
\[
\cl_\ell\colon\CH^i(X)_\dQ\to\rH^{2i}_{\et}(X_{k^\ac},\dQ_\ell(i))
\]
for every rational prime $\ell$. Moreover, by the Hochschild--Serre spectral sequence, we obtain the $\ell$-adic Abel--Jacobi map
\[
\AJ_\ell\colon\CH^i(X)_\dQ^0\to\rH^1(k,\rH^{2i-1}_{\et}(X_{k^\ac},\dQ_\ell(i))).
\]

\begin{definition}\label{de:aj_kernel}
We put $\CH^i(X)_\dQ^1\coloneqq\bigcap_\ell\Ker\AJ_\ell$ as a subspace of $\CH^i(X)_\dQ^0$, where the intersection is taken over all rational primes $\ell$p, and
\begin{align*}
\CH^i(X)_R^0\coloneqq\CH^i(X)_\dQ^0\otimes_\dQ R,\qquad
\CH^i(X)_R^\natural\coloneqq(\CH^i(X)_\dQ^0/\CH^i(X)_\dQ^1)\otimes_\dQ R
\end{align*}
for every ring $R$ containing $\dQ$. We call elements in $\CH^i(X)_R^\natural$ \emph{natural cycles} (of codimension $i$).
\end{definition}

\begin{remark}\label{re:beilinson}
In \cite{Bei87}, Beilinson conjectures that $\Ker\AJ_\ell=\{0\}$ for every rational prime $\ell$ if $k$ is a number field and $X$ is projective, which implies  $\CH^i(X)_R^0=\CH^i(X)_R^\natural$.
\end{remark}

We introduce the following definition, which will be used in Section \ref{ss:4}.

\begin{definition}\label{de:chow_convergent}
We say that a formal series $\sum_j c_jZ_j$ with $c_j\in\dC$ and $Z_j\in\rZ^i(X)$ is \emph{Chow convergent} if the image of $\{Z_j\}_j$ in $\CH^i(X)_\dC$ generates a finite dimensional subspace, and the induced formal series in this finite dimensional space is absolutely convergent. We denote by $\CZ^i(X)$ the set of Chow convergent formal series in $\rZ^i(X)$, which is a complex vector space and admits a natural complex linear map $\CZ^i(X)\to\CH^i(X)_\dC$.
\end{definition}

Now we recall the notation of correspondences. A \emph{(Chow self-)correspondence of $X$} is an element $z\in\CH^d(X\times X)$. It induces a graded map
\[
z^*\colon\bigoplus_{i=0}^d\CH^i(X)\to\bigoplus_{i=0}^d\CH^i(X)
\]
sending $\alpha$ to $\tp_{1*}(z.\tp_2^*\alpha)$, where $\tp_i\colon X\times X\to X$ is the projection to the $i$-th factor, a convention recalled from Subsection \ref{ss:notation}. On the level of various cohomology, it induces graded maps
\begin{align*}
\cl_\dr^*(z)&\colon\bigoplus_{i=0}^{2d}\rH^i_\dr(X/k)\to\bigoplus_{i=0}^{2d}\rH^i_\dr(X/k), \\
\cl_{\rB,\tau}^*(z)&\colon\bigoplus_{i=0}^{2d}\rH^i_{\rB,\tau}(X,\dQ)\to\bigoplus_{i=0}^{2d}\rH^i_{\rB,\tau}(X,\dQ),
\end{align*}
and
\begin{align*}
\cl_\ell^*(z)\colon\bigoplus_{i=0}^{2d}\rH^i_{\et}(X_{k^\ac},\dQ_\ell(j))\to\bigoplus_{i=0}^{2d}\rH^i_{\et}(X_{k^\ac},\dQ_\ell(j))
\end{align*}
for every prime rational $\ell$ and $j\in\dZ$. They are compatible with each other under various
comparison theorems and cycle class maps. When we regard the diagonal $\Delta X\subseteq X\times X$ as a correspondence, we usually write it as $\id_X$.

\subsection{Beilinson--Bloch--Poincar\'{e} height pairing}
\label{ss:bbp}

We review the theory of height pairing between cycles of Beilinson and Bloch. Now suppose that $k$ is a number field. Consider a projective smooth scheme $X\in\Sch_{/k}$ of pure dimension $d$. Beilinson \cite{Bei87} and Bloch \cite{Blo84} have defined, via two approaches, a bilinear pairing
\[
\langle\;,\;\rangle_X^{\r{BB}}\colon\CH^i(X)_\dQ^0\times\CH^{d+1-i}(X)_\dQ^0\to\dC.
\]
However, both approaches relies on some hypotheses that are still unknown even today.

We review briefly Beilinson's construction: For $z_1\in\CH^i(X)_\dQ^0$ and $z_2\in\CH^{d+1-i}(X)_\dQ^0$, we choose their representatives $Z_1\in\rZ^i(X)_\dQ$ and $Z_2\in\rZ^{d+1-i}(X)_\dQ$ that have disjoint support. For every place $v$ of $k$, there is a local index $\langle Z_1,Z_2\rangle_{X_v}$ on $X_v\coloneqq X_{k_v}$. For $v$ archimedean, this is defined in \cite{Bei87}*{Section~3} via the potential theory on K\"{a}hler manifolds; it is unconditional. For $v$ nonarchimedean such that $X_v$ has good reduction, this is defined via intersection theory on an arbitrary smooth model of $X_v$ over $O_{k_v}$. For $v$ nonarchimedean in general, the definition of $\langle Z_1,Z_2\rangle_{X_v}$ is conditional: Choose a rational prime $\ell$ not underlying $v$ and an isomorphism $\iota_\ell\colon\dC\xrightarrow{\sim}\dQ_\ell^\ac$ such that $\rH^{2i}_{\et}(X_{k_v^\ac},\dQ_\ell)$ satisfies the weight-monodromy conjecture, which implies that the cycle class of $Z_1$ in the absolute \'{e}tale cohomology $\rH^{2i}_{\et}(X_v,\dQ_\ell(i))$ vanishes (same for $Z_2$). Then one can define $\langle Z_1,Z_2\rangle_{X_v}$ as a ``link pairing'' valued in $\dQ_\ell$ followed by the map $\iota_\ell^{-1}$. See \cite{Bei87}*{Section~2.1} for more details. We then define
\begin{align}\label{eq:bb_decompose}
\langle Z_1,Z_2\rangle_X^{\r{BB}}\coloneqq\sum_v r(v)\cdot\langle Z_1,Z_2\rangle_{X_v},
\end{align}
where the sum is taken over all places $v$ of $k$, and $r(v)$ equals $1$, $2$, and $\log q_v$ when $v$ is real, complex, and nonarchimedean (with $q_v$ the residue cardinality of $k_v$), respectively.

For every intermediate ring $\dQ\subseteq R\subseteq\dC$, we obtain a pairing
\begin{align*}
\langle\;,\;\rangle_X^{\r{BB}}&\colon\CH^i(X)_R^0\times\CH^{d+1-i}(X)_R^0\to\dC
\end{align*}
via $R$-bilinear extension. Beilinson conjectures that the pairing $\langle\;,\;\rangle_X^{\r{BB}}$ is independent of $\ell$ and the isomorphism $\iota_\ell$.

\if false

Let $\ell$ be a rational prime such that $X$ has proper smooth reduction at every place above $\ell$. Choose an isomorphism $\iota_\ell\colon\dC\xrightarrow{\sim}\dQ_\ell^\ac$, we define globally a pairing
\[
\langle\;,\;\rangle_X^{\r{BB},\iota_\ell}\colon \CH^i(X)_\dQ^0\times \CH^{d+1-i}(X)_\dQ^0\to\dC
\]
via the formula
\[
\langle Z_1,Z_2\rangle_X^{\r{BB},\iota_\ell}\coloneqq\sum_v r(v)\cdot\iota_\ell^{-1}\langle Z_1,Z_2\rangle_{X_v}
\]
where $r(v)$ is some elementary factor determined by $v$. It is easy to see that this is a finite sum.

\fi

\begin{remark}\label{re:height_pairing}
As we have mentioned, if $X_v$ satisfies the weight-monodromy conjecture for every nonarchimedean place $v$ of $k$ (for example, when $X$ is a product of curves, surfaces, or abelian varieties), then the Beilinson--Bloch height pairing $\langle\;,\;\rangle_X^{\r{BB}}$ is unconditionally defined (but may \emph{a priori} depend on the choices of $\ell$ and $\iota_\ell$). When $X$ is a curve, the Beilinson--Bloch height pairing coincides with the N\'{e}ron--Tate height pairing up to $-1$. When $X$ is an abelian variety, the Beilinson--Bloch height pairing coincides with the pairing defined in \cite{Kun01}. In particular, in these two cases, the independence of $\ell$ and $\iota_\ell$ is known.
\end{remark}

\begin{lem}\label{le:divisor_coh}
Suppose that the Beilinson--Bloch height pairing is defined for $X$. Take $Z\in\CH^1(X)_R^0$ for some intermediate ring $\dQ\subseteq R\subseteq\dC$. Then we have
\[
\langle Z_1,Z.Z_2\rangle_X^{\r{BB}}=0
\]
for every $Z_1\in\CH^i(X)_R^0$ and $Z_2\in\CH^{d-i}(X)_R^0$.
\end{lem}

\begin{proof}
We fix an embedding $k\hookrightarrow\dC$. Since $Z$ is homologically trivial, it is algebraically equivalent to zero; so is $Z.Z_2$. By \cite{Bei87}*{Lemma~4.0.7}, it suffices to show that the image of $Z.Z_2$ under the complex Abel--Jacobi map $\CH^{d-i+1}(X)_R^0\to\rJ^{d-i+1}(X_\dC)_R$ is zero, where $\rJ^{d-i+1}(X_\dC)$ is the $(d-i+1)$-th intermediate Jacobian of $X_\dC$ (as an abelian group). We replace $Z$ and $Z_2$ by their representatives in $\rZ^1(X)_R$ and $\rZ^{d-i}(X)_R$ with proper intersection. Since $Z_2$ is homologically trivial, we may choose a (singular) chain $C_{Z_2}$ of (real) dimension $2i+1$ with boundary $Z_2$. Then $Z.C_{Z_2}$ is a chain of dimension $2i-1$ with boundary $Z.Z_2$. It suffices to show that
\[
\int_{Z.C_{Z_2}}\omega=0
\]
for every closed differential form $\omega$ whose class belongs to the Hodge filtration $\Fil^i\rH^{2i-1}_\rB(X,\dC)$. Since (the underlying cycle of) $Z$ is homologous to zero, we can take a $(1,0)$-form $\eta$ such that $\rd\eta$ is the class represented by $Z$ by the $\partial\ol\partial$-lemma from Hodge theory. Thus,
\[
\int_{Z.C_{Z_2}}\omega=\int_{C_{Z_2}}\rd\eta\wedge\omega=\int_{C_{Z_2}}\rd(\eta\wedge\omega)=
\int_{Z_2}\eta\wedge\omega=0,
\]
in which the last equality follows as we may take a representative of $\omega$ as a sum of $(p,2i-1-p)$-forms with $p\geq i$. The lemma follows.
\end{proof}

Recall that if $A$ is an abelian variety over $k$ of dimension $h\geq 1$, and $\dQ\subseteq R\subseteq\dC$ is an immediate ring, then we also have the N\'{e}ron--Tate (bilinear) height pairing
\[
\langle\;,\;\rangle_A^{\r{NT}}\colon A(k)_R\times A^\vee(k)_R\to\dC.
\]
Composing with the Albanese maps $\CH^h(A)_R^0\to A(k)_R$ and $\CH^h(A^\vee)_R^0\to A^\vee(k)_R$, we may regard the above pairing as a map
\begin{align}\label{eq:nt}
\langle\;,\;\rangle_A^{\r{NT}}\colon \CH^h(A)_R^0\times \CH^h(A^\vee)_R^0\to\dC.
\end{align}

\begin{remark}\label{re:bbnt}
The N\'{e}ron--Tate height pairing \eqref{eq:nt} is related to the Beilinson--Bloch height pairing via the following commutative diagram:
\[
\xymatrix{
\CH^h(A)_R^0  \ar@{}[r]|{\bigtimes}\ar@{=}[d] & \CH^1(A)_R^0 \ar[d]\ar[rr]^-{\langle\;,\;\rangle_A^{\r{BB}}} && \dC \ar@{=}[d]\\
\CH^h(A)_R^0  \ar@{}[r]|{\bigtimes} & \CH^h(A^\vee)_R^0 \ar[rr]^-{-\langle\;,\;\rangle_A^{\r{NT}}} && \dC \\
}
\]
in which $\CH^1(A)_R^0\to\CH^h(A^\vee)_R^0$ is the tautological map.
\end{remark}

Now we will combine the Beilinson--Bloch height pairing on $X$ and the N\'{e}ron--Tate height pairing on $A$ to give a height pairing
\[
\langle\;,\;\rangle_{X,A}^{\r{BBP}}\colon\CH^{h+i}(X\times A)_R^0\times\CH^{h+d-i}(X\times A^\vee)_R^0\to\dC
\]
using the Poincar\'{e} bundle, for every intermediate ring $\dQ\subseteq R\subseteq\dC$. The process is easy: Let $\cP$ be the Poincar\'{e} line bundle on $A\times A^\vee$. We have projection morphisms
\[
\tp_{12}\colon X\times A\times A^\vee\to X\times A,\quad \tp_{13}\colon X\times A\times A^\vee\to X\times A^\vee,
\]
and recall the Fourier--Mukai transform
\[
\wp\colon\CH^{h+d-i}(X\times A^\vee)_R^0\to\CH^{d+1-i}(X\times A)_R^0
\]
sending $z$ to $\tp_{12*}((X\times c_1(\cP)).\tp_{13}^*z)$. We then define
\[
\langle z_1,z_2\rangle_{X,A}^{\r{BBP}}\coloneqq\langle z_1,\wp(z_2)\rangle_{X\times A}^{\r{BB}}.
\]

\begin{definition}
We call $\langle\;,\;\rangle_{X,A}^{\r{BBP}}$ the \emph{Beilinson--Bloch--Poincar\'{e} height pairing} for $(X,A)$.
\end{definition}

\begin{remark}
The Beilinson--Bloch--Poincar\'{e} height pairing is unconditionally defined if $X_v$ satisfies the weight-monodromy conjecture for every nonarchimedean place $v$ of $k$ (but may \emph{a priori} depend on the choices of $\ell$ and $\iota_\ell$). When $X=\Spec k$ (resp. $h=1$, that is, $A$ is an elliptic curve, hence is canonically isomorphic to $A^\vee$), the Beilinson--Bloch--Poincar\'{e} height pairing for $(X,A)$ reduces to the N\'{e}ron--Tate height pairing \eqref{eq:nt} for $A$ up to $-1$ (resp. the Beilinson--Bloch height pairing for $X\times A$).
\end{remark}

\begin{remark}\label{re:bbp}
The Beilinson--Bloch--Poincar\'{e} height pairing can be defined more generally for an abelian scheme $\ul{A}$ of relative dimension $h\geq 1$ over $X$ as a pairing
\[
\langle\;,\;\rangle_{\ul{A}}^{\r{BBP}}\colon\CH^{h+i}(\ul{A})_R^0\times\CH^{h+d-i}(\ul{A}^\vee)_R^0\to\dC
\]
such that $\langle z_1,z_2\rangle_{\ul{A}}^{\r{BBP}}=\langle z_1,\tp_{1*}(c_1(\ul{\cP}).\tp_2^*z_2)\rangle_{\ul{A}}^{\r{BB}}$, where $\tp_1\colon\ul{A}\times_X\ul{A}^\vee\to\ul{A}$ and $\tp_2\colon\ul{A}\times_X\ul{A}^\vee\to\ul{A}^\vee$ are projection morphisms, and $\ul{\cP}$ is the relative Poincar\'{e} bundle on $\ul{A}\times_X\ul{A}^\vee$.
\end{remark}

\subsection{K\"{u}nneth--Chow projectors}
\label{ss:kunneth}

In this subsection, we will construct some K\"{u}nneth--Chow projectors, which will be used in Subsection \ref{ss:aggp}. The readers may skip it at this moment.

Consider a proper smooth scheme $X\in\Sch_{/k}$ of pure dimension $d$. Put
\[
\rH^\even_\dr(X/k)\coloneqq\bigoplus_{i\text{ even}}\rH^i_\dr(X/k),\qquad
\rH^\odd_\dr(X/k)\coloneqq\bigoplus_{i\text{ odd}}\rH^i_\dr(X/k).
\]

\begin{definition}\label{de:projector}
We say that a correspondence $z\in\CH^d(X\times X)_\dQ$ is an \emph{even} (resp.\ \emph{odd}) \emph{projector} if the map $\cl_\dr^*(z)$ is the projection map to $\rH^\even_\dr(X/k)$ (resp.\ $\rH^\odd_\dr(X/k)$).
\end{definition}

We introduce the following convention: for a zero cycle $D$ on $X$, we regard its degree $\deg D$ as a function on $\pi_0(X)$.

\begin{lem}\label{le:easy_projector}
We have
\begin{enumerate}
  \item Suppose that $d=1$. Let $D\in\CH^1(X)_\dQ$ be a cycle such that $\deg D$ is nonzero on every connected component of $X$. Then
    \[
    z_{X,D}\coloneqq\Delta X-\frac{1}{\deg D}(X\times D + D\times X)
    \]
    is an odd projector for $X$.

  \item Suppose that $d=2$. Let $D\in\CH^1(X)_\dQ$ be a cycle that is an almost ample divisor (Definition \ref{de:almost_ample}). Then
    \[
    z_{X,D}\coloneqq e_{X,D}+e_{X,D}^\rt
    \]
    is an odd projector for $X$, where $e_{X,D}$ is the correspondence in Proposition \ref{pr:picard_functoriality} and $e_{X,D}^\rt$ is its transpose.
\end{enumerate}
\end{lem}

\begin{proof}
Part (1) is obvious. Part (2) follows from Proposition \ref{pr:picard_functoriality}(1).
\end{proof}

\begin{lem}\label{le:triple_product}
Let $z$ be an odd projector for $X$. Then
\begin{enumerate}
  \item the image of the induced map $z^*\colon\CH^i(X)_\dQ\to\CH^i(X)_\dQ$ is contained in $\CH^i(X)_\dQ^0$;

  \item the cycle
    \[
    z\times z\times z+z\times(\Delta X-z)\times(\Delta X-z)+(\Delta X-z)\times z\times(\Delta X-z)+(\Delta X-z)\times(\Delta X-z)\times z
    \]
    is an odd projector for $X\times X\times X$.
\end{enumerate}
\end{lem}

\begin{proof}
For (1), since $\cl_\dr(\IM z^*)\subseteq\IM(\cl_\dr^*(z))=\rH^\odd_\dr(X/k)$, we know that the image of $z^*$ is contained in $\CH^i(X)_\dQ^0$.

For (2), note that if $z$ is an odd projector, then $\Delta X-z$ is an even projector. Thus, (2) follows from the K\"{u}nneth decomposition for the algebraic de Rham cohomology.
\end{proof}

\begin{definition}\label{de:triple_product}
Let $z$ be an odd projector for $X$. We define
\[
\pr^{[3]}_z\colon\CH^i(X\times X\times X)_\dQ\to\CH^i(X\times X\times X)_\dQ^0
\]
to be the map induced by the odd projector for $X\times X\times X$ as in Lemma \ref{le:triple_product}(2).
\end{definition}

\if false

Note that $\CH^d(X\times_k X)_\dQ$ is a $\dQ$-algebra with the multiplication given by composition of correspondences. The following lemma will be used later.

\begin{lem}\label{le:projector}
Let $u\colon Y\to X$ be a generically finite dominant morphism of proper smooth schemes over $k$. Let $\sH_X$ (resp.\ $\sH_Y$) be a commutative $\dQ$-subalgebra of $\CH^d(X\times_k X)_\dQ$ (resp.\ $\CH^d(Y\times_k Y)_\dQ$) such that
\begin{itemize}
  \item the supports of $\rH^\even_\dr(Y/k)$ and $\rH^\odd_\dr(Y/k)$ as $\sH_Y$-modules are disjoint;

  \item the image of $\sH_Y$ under the map $(u\times u)_*\colon\CH^d(Y\times_k Y)_\dQ\to\CH^d(X\times_k X)_\dQ$ is contained in $\sH_X$.
\end{itemize}
Let $z_X$ (resp.\ $z_Y$) be an odd projector in the commutator of $\sH_X$ (resp.\ $\sH_Y$) in $\CH^d(X\times_k X)_\dQ$ (resp.\ $\CH^d(Y\times_k Y)_\dQ$). Then we have
\[
u^*\circ z_X^*=z_Y^*\circ u^*\colon\CH^i(X)_\dQ\to\CH^i(Y)_\dQ^\natural
\]
for every $i$ (see Definition \ref{de:aj_kernel}).
\end{lem}

\begin{proof}
First by Lemma \ref{le:triple_product}, the images of the maps $u^*\circ z_X^*$ and $z_Y^*\circ u^*$ are contained in $\CH^i(Y)_\dQ^0$. We need to show that the induced map
\[
u^*\circ z_X^* - z_Y^*\circ u^*\colon\CH^i(X)_\dQ\to\CH^i(Y)_\dQ^0/\Ker\AJ_\ell
\]
vanishes for every $\ell$. First, since pullback via correspondences is compatible with cycle class map, the previous map $u^*\circ z_X^* - z_Y^*\circ u^*$ factors through a map, which we denote
\[
\zeta_\ell\colon\CH^i(X)_\dQ/\Ker\AJ_\ell\to\CH^i(Y)_\dQ^0/\Ker\AJ_\ell.
\]
We prove by contradiction. If $\zeta_\ell$ is not zero, then the composite map
\[
\zeta'_\ell\coloneqq\AJ_\ell\circ\zeta_\ell\colon\CH^i(X)_\dQ/\Ker\AJ_\ell\to\rH^1(k,\rH^{2i-1}_{\et}(X\otimes_kk^\ac,\dQ_\ell(i)))
\]
is nonzero. By the comparison theorem, the difference $u^*\circ z_X^* - z_Y^*\circ u^*$ induces the zero map from $\rH^1(k,\rH^{2i-1}_{\et}(X\otimes_kk^\ac,\dQ_\ell(i)))$ to $\rH^1(k,\rH^{2i-1}_{\et}(Y\otimes_kk^\ac,\dQ_\ell(i)))$. Therefore, $\Ker\zeta'_\ell$ contains $\CH^i(X)_\dQ^0/\Ker\AJ_\ell$. In other words, we have an induced map
\[
\zeta''_\ell\colon \CH^i(X)_\dQ/\CH^i(X)_\dQ^0\to\rH^1(k,\rH^{2i-1}_{\et}(Y\otimes_kk^\ac,\dQ_\ell(i))).
\]
Regard $\CH^i(X)_\dQ/\CH^i(X)_\dQ^0$ as a module over $\sH_Y$ via $(u\times u)_*$. As $(u\times u)_*\sH_Y\subseteq\sH_X$ by the second assumption, $s_X^*$ commutes with the action of $\sH_Y$ on $\CH^i(X)_\dQ/\CH^i(X)_\dQ^0$. Therefore, $\zeta''_\ell$ is a map of $\sH_Y$-modules. Note that the quotient $\sH_Y$-module $\CH^i(X)_\dQ/\CH^i(X)_\dQ^0$ is a submodule of $\rH^{2i}_\dr(X/k)$, whose support is contained in the support of $\rH^{2i}_\dr(Y/k)$. However, since the support of $\rH^1(k,\rH^{2i-1}_{\et}(Y\otimes_kk^\ac,\dQ_\ell(i)))$ is contained in the support of $\rH^{2i-1}_{\et}(Y\otimes_kk^\ac,\dQ_\ell)$, hence of $\rH^{2i-1}_\dr(Y/k)$, by the comparison theorem, we know that $\zeta''_\ell$ must be zero by our first assumption on $\sH_Y$. The lemma follows.
\end{proof}

\fi

\section{Fourier--Jacobi cycles and derivative of $L$-functions}
\label{ss:3}

In this section, we construct Fourier--Jacobi cycles and state our main conjectures. In Subsection \ref{ss:motives}, we construct the category of CM data for a conjugate symplectic automorphic character $\mu$ of weight one. In Subsection \ref{ss:albanese_unitary}, we introduce our Shimura varieties and study their Albanese varieties. In Subsection \ref{ss:construction_cycles}, we construct Fourier--Jacobi cycles and show that they are homologically trivial. In Subsection \ref{ss:aggp}, we propose various versions of the arithmetic Gan--Gross--Prasad conjecture for $\rU(n)\times\rU(n)$.

Let $F$ be a totally real number field of degree $d\geq 1$, and $E/F$ a totally imaginary quadratic extension. We denote by
\begin{itemize}
  \item $\tc$ the nontrivial Galois involution of $E$ over $F$,

  \item $E^-$ the subgroup of $E$ consisting of $e$ satisfying $e+e^\tc=0$, and $E^1$ the subgroup of $E^\times$ consisting of $e$ satisfying $ee^\tc=1$,

  \item by $\mu_{E/F}\colon F^\times\backslash\bA_F^\times\to\dC^\times$ the quadratic character associated to $E/F$ via the global class field theory,

  \item $E_v$ the base change $E\otimes_FF_v$ for every place $v$ of $F$,

  \item $\Phi_F$ the set of real embeddings of $F$, $\Phi_E$ the set of complex embeddings of $E$, and $\pi\colon\Phi_E\to\Phi_F$ the projection map given by restriction.
\end{itemize}
Recall that a CM type (of $E$) is a subset $\Phi$ of $\Phi_E$ such that $\pi$ induces a bijection from $\Phi$ to $\Phi_F$.

In this section, we work in the category $\Sch_{/E}$.

\subsection{Motives for CM characters}
\label{ss:motives}

In this subsection, we generalize some constructions in \cite{Den89}*{Section~2}.

\begin{definition}\label{de:conjugate}
We say that an automorphic character $\mu\colon E^\times\backslash\bA_E^\times\to\dC^\times$ is \emph{conjugate self-dual} if $\mu$ is trivial on $\Nm_{\bA_E/\bA_F}\bA_E^\times$. We say that $\mu$ is \emph{conjugate orthogonal} (resp.\ \emph{conjugate symplectic}) if $\mu\res{\bA_F^\times}=1$ (resp.\ $\mu\res{\bA_F^\times}=\mu_{E/F}$).
\end{definition}

\begin{remark}
A conjugate self-dual automorphic character is necessarily strictly unitary (Definition \ref{de:strictly_unitary}). It is either conjugate orthogonal or conjugate symplectic, but not both.
\end{remark}

For a conjugate symplectic (resp.\ conjugate orthogonal) automorphic character $\mu$, there exist a CM type $\Phi_\mu$ and a unique tuple $\underline{\tw}_\mu=(\tw_\tau)_{\tau\in\Phi_F}$ of odd (resp.\ even) nonnegative integers such that for every $\tau\in\Phi_F$, the component $\mu_\tau\colon(E\otimes_{F,\tau}\dR)^\times\to\dC^\times$ is the character
\[
z\mapsto\arg(z)^{-\tw_\tau},
\]
where we have identified $(E\otimes_{F,\tau}\dR)^\times$ with $\dC^\times$ via the unique element $\tau'\in\Phi_\mu$ above $\tau$. If $\underline{\tw}_\mu$ does not contain $0$, then $\Phi_\mu$ is also unique. In what follows, we put $\mu^\tc\coloneqq\mu\circ\tc$.

\begin{definition}\label{de:conjugate2}
Let $\mu$ be a conjugate self-dual automorphic character.
\begin{enumerate}
  \item We call $\underline{\tw}_\mu$ the \emph{weight} of $\mu$. If $\underline{\tw}_\mu$ is a constant $m$, then we say that $\mu$ is of weight $m$.

  \item If $\underline{\tw}_\mu$ does not contain zero, then we call $\Phi_\mu$ the \emph{CM type} of $\mu$. Furthermore, we denote by $M'_\mu\subseteq\dC$ the reflex field of $(E,\Phi_\mu)$, with the induced CM type $\Psi_\mu$.
\end{enumerate}
\end{definition}

Now let $\mu$ be a conjugate symplectic automorphic character, which is not algebraic. We put
\[
\mu^\alg\coloneqq\mu\cdot|\;|_E^{-1/2},
\]
which is then algebraic. Denote by $M_\mu\subseteq\dC$ the subfield generated by values $\mu^\alg(x)$ for $x\in(\bA_E^\infty)^\times$, which is a number field containing $M'_\mu$.

\begin{remark}\label{re:mu}
It is clear that $\mu^\tc$ is conjugate symplectic of the same weight as $\mu$. Moreover, we have $M_{\mu^\tc}=M_\mu$, $M'_{\mu^\tc}=M'_\mu$, and that $\Psi_{\mu^\tc}$ is the opposite CM type of $\Psi_\mu$.
\end{remark}

\begin{definition}\label{de:cm_data}
Let $\mu$ be a conjugate symplectic automorphic character of weight one.
\begin{enumerate}
  \item We denote by $\eta'_\mu\colon\Res_{M'_\mu/\dQ}\dG_\rm\to\Res_{E/\dQ}\dG_\rm$ the reciprocity map, and put
     \[
     \eta_\mu\coloneqq\eta'_\mu\circ\Nm_{M_\mu/M'_\mu}\colon\Res_{M_\mu/\dQ}\dG_\rm\to\Res_{E/\dQ}\dG_\rm.
     \]

  \item We define a \emph{CM data for $\mu$} to be a quadruple $D_\mu=(A_\mu,i_\mu,\lambda_\mu,r_\mu)$, in which
     \begin{itemize}
       \item $A_\mu$ is an abelian variety over $E$,

       \item $i_\mu\colon M_\mu\to\End_E(A_\mu)_\dQ$ is a CM structure such that
          \begin{itemize}
            \item for every $x\in M_\mu$, the determinant of the action of $i_\mu(x)$ on the $E$-vector space $\Lie_E(A_\mu)$ equals $\eta_\mu(x)$,

            \item the associated CM character of $A_\mu$ with respect to the inclusion $M_\mu\hookrightarrow\dC$ coincides with $\mu^\alg$,
          \end{itemize}

       \item $\lambda_\mu\colon A_\mu\to A^\vee_\mu$ is a polarization satisfying $\lambda\circ i_\mu(x)=i_\mu(\ol{x})^\vee\circ\lambda$ for every $x\in M_\mu$,

       \item $r_\mu\colon M_\mu\otimes_\dQ E\to\rH_1^\dr(A_\mu/E)$ is an isomorphism of $M_\mu\otimes_\dQ E$-modules satisfying that there exist an element $\beta\in M_\mu$ and an isomorphism $c\colon\rH_{2\dim A_\mu}^\dr(A_\mu/E)\to E$ of $E$-modules, such that for every $x,y\in M_\mu\otimes_\dQ E$, we have $c(\langle r_\mu(x),r_\mu(y)\rangle_\lambda)=\Tr_{M_\mu\otimes_\dQ E/E}(x\beta\ol{y})$, where $\langle\;,\;\rangle_\lambda\colon\rH_1^\dr(A_\mu/E)\times\rH_1^\dr(A_\mu/E)\to\rH_{2\dim A_\mu}^\dr(A_\mu/E)$ denotes the pairing induced by $\lambda$.
     \end{itemize}

  \item We denote by $\cA(\mu)$ the \emph{category of CM data for $\mu$}, whose objects are CM data $D_\mu$, and morphisms from $D_\mu=(A_\mu,i_\mu,\lambda_\mu,r_\mu)$ to $D'_\mu=(A'_\mu,i'_\mu,\lambda'_\mu,r'_\mu)$ are isogenies $\varphi\colon A_\mu\to A'_\mu$ satisfying $\varphi\circ i_\mu(x)=i'_\mu(x)\circ\varphi$ for every $x\in M_\mu$, $\varphi^\vee\circ\lambda'_\mu\circ\varphi=c\lambda_\mu$ for some element $c\in\dQ^\times$, and $r'_\mu=\varphi_*\circ r_\mu$.

  \item From a CM data $D_\mu=(A_\mu,i_\mu,\lambda_\mu,r_\mu)$ for $\mu$, we define another quadruple $D^\vee_\mu=(A^\vee_\mu,i^\vee_\mu,\lambda^\vee_\mu,r^\vee_\mu)$ in which $(A^\vee_\mu,\lambda^\vee_\mu)$ is simply the dual of $(A_\mu,\lambda_\mu)$, $i^\vee_\mu$ is defined by the formula $i_\mu^\vee(x)=i_\mu(x)^\vee$ for $x\in M_\mu$, and $r^\vee_\mu\coloneqq(\lambda_\mu)_*\circ r_\mu$.
\end{enumerate}
\end{definition}

\begin{proposition}\label{pr:cm_data}
Let $\mu$ be as in Definition \ref{de:cm_data}.
\begin{enumerate}
  \item The category $\cA(\mu)$ is a nonempty and connected partially ordered set.

  \item The assignment sending $D_\mu$ to $D^\vee_\mu$ induces an equivalence $\cA(\mu)^{op}\xrightarrow{\sim}\cA(\mu^\tc)$ of categories.
\end{enumerate}
\end{proposition}

\begin{proof}
For (1), we first show that $\cA(\mu)$ is nonempty. Take a finite abelian extension $E'/E$ such that the character $\mu^{\alg\prime}\coloneqq\mu^\alg\circ\Nm_{E'/E}$ satisfies \cite{Shi71}*{(1.12) \& (1.13)} for $(K',\Phi')=(E,\Phi_\mu)$, $k=E'$, $(K,\Phi)=(M'_\mu,\Psi_\mu)$ with $_{\infty1}$ the archimedean place of $M'_\mu$ induced by the inclusion $M'_\mu\hookrightarrow\dC$, and $\fa=O_{M'_\mu}$. For example, we may take an open compact subgroup $U$ of $\bA_E^\infty$ on which $\mu$ (hence $\mu^\alg$) is trivial and take $E'$ to be the abelian extension corresponding to $U$ via the global class field theory. By Casselman's theorem \cite{Shi71}*{Theorem~6}, we have a pair $(A',i')$ where $A'$ is an abelian variety over $E'$ and $i'\colon M'_\mu\to\End_{E'}(A')_\dQ$ is a CM structure such that
\begin{itemize}
  \item the determinant of the action of $i'(x')$ on the $E'$-vector space $\Lie_{E'}(A')$ is $\eta'_\mu(x')$ for every $x'\in M'_\mu$,

  \item the associated CM character of $A'$ with respect to the inclusion $M'_\mu\hookrightarrow\dC$ coincides with $\mu^{\alg\prime}$.
\end{itemize}
By \cite{Shi71}*{Lemma~1 \& Lemma~2}, $A'$ is simple, hence $i'$ is an isomorphism. The same argument in \cite{Den89}*{(2.1)} implies that there is an isogeny factor $A_\mu$ of the abelian variety $\Res_{E'/E}A'$ over $E$ together with a CM structure $i_\mu\colon M_\mu\to\End_E(A_\mu)_\dQ$ satisfying the conditions in the proposition. In other words, we have obtained the part $(A_\mu,i_\mu)$ for a CM data for $\mu$. By \cite{Shi71}*{Theorem~5}, we know the existence of $\lambda_\mu$. The existence of $r_\mu$ is obvious. Thus, we obtain an object $D_\mu=(A_\mu,i_\mu,\lambda,r_\mu)$ of $\cA(\mu)$

The connectedness of $\cA(\mu)$ also follows from \cite{Shi71}*{Theorem~5}. Finally, the compatibility condition $r'_\mu=\varphi_*\circ r_\mu$ ensures that $\cA(\mu)$ is a partially ordered set.

Part (2) is clear from the definition.
\end{proof}

\begin{remark}\label{re:motive}
Let $\mu$ be as in Definition \ref{de:cm_data}. Although we will not use in the main body of the article, we propose the definition of the \emph{motive for $\mu$}, denoted by $\sfL_\mu$, as a Grothendieck motive. For a CM data $D_\mu=(A_\mu,i_\mu,\lambda_\mu,r_\mu)\in\cA(\mu)$, let $\sfh^1(A_\mu,M_\mu)$ be the Picard motive of $A_\mu$ with coefficients in $M_\mu$, and $\sfh^1(D_\mu)$ the direct summand of $\sfh^1(A_\mu,M_\mu)$ on which the induced action of $M_\mu$ via $i_\mu$ coincides with the underlying linear action of $M_\mu$. The assignment $D_\mu\mapsto\sfh^1(D_\mu)$ is a functor from $\cA(\mu)$ to the category of Grothendieck motives over $E$. We put $\sfL_\mu\coloneqq\varinjlim_{D_\mu\in\cA(\mu)}\sfh^1(D_\mu)$, which is of rank $1$ with coefficients in $M_\mu$. It follows from Proposition \ref{pr:cm_data}(1) that the canonical map $\sfL_\mu\to\sfh^1(D_\mu)$ is an isomorphism for every $D_\mu\in\cA(\mu)$.
\end{remark}

To end this subsection, we construct a certain canonical projector on $A_\mu$, which will be used in Subsection \ref{ss:construction_cycles}. Let $\mu$ be as in Definition \ref{de:cm_data}. Denote by $I_\mu$ the set of all complex embeddings of $M_\mu$. We take a CM data $D_\mu=(A_\mu,i_\mu,\lambda_\mu,r_\mu)\in\cA(\mu)$ for $\mu$. For every element $x\in M_\mu$ such that $i_\mu(x)\in\End_E(A_\mu)$, we denote by $\langle x\rangle$ the correspondence
\[
A_\mu\xleftarrow{i_\mu(x)}A_\mu\xrightarrow{\id}A_\mu
\]
of $A_\mu$. In particular, $\langle x\rangle^*=i_\mu(x)_*$. For every $\tau'\colon E\hookrightarrow\dC$ and every integer $0\leq i\leq [M_\mu:\dQ]$, we have a canonical decomposition
\[
\rH^{[M_\mu:\dQ]-i}_{\rB,\tau'}(A_\mu,\dC)=\bigoplus_{I\subseteq I_\mu,|I|=i}\rH^{[M_\mu:\dQ]-i}_{\rB,\tau'}(A_\mu,\dC)_I,
\]
where $\rH^{[M_\mu:\dQ]-i}_{\rB,\tau'}(A_\mu,\dC)_I$ denotes the subspace on which $\cl^*_{\rB,\tau'}(\langle x\rangle)$ acts by $\prod_{\iota\in I}\iota(x)$ for every $x\in M_\mu$ satisfying $i_\mu(x)\in\End_E(A_\mu)$. Let $\tT_\mu^{\rB,\tau'}$ be the endomorphism of $\bigoplus_i\rH^i_{\rB,\tau'}(A_\mu,\dC)$ such that
\begin{itemize}
  \item the restriction $\tT_\mu^{\rB,\tau'}\res\rH^i_{\rB,\tau'}(A_\mu,\dC)$ is zero if $i\neq[M_\mu:\dQ]-1$,

  \item the restriction $\tT_\mu^{\rB,\tau'}\res\rH^{[M_\mu:\dQ]-1}_{\rB,\tau'}(A_\mu,\dC)$ is the canonical projection to the direct summand $\rH^{[M_\mu:\dQ]-1}_{\rB,\tau'}(A_\mu,\dC)_{I_1}$, where $I_1\subseteq I_\mu$ is the subset consisting only of the inclusion $M_\mu\hookrightarrow\dC$.
\end{itemize}

\begin{definition}\label{de:generator}
Let $I\subseteq I_\mu$ be a subset.
\begin{enumerate}
  \item We say that $x\in M_\mu$ is an \emph{$I$-generator} if $x$ generates the field $M_\mu$ with $i_\mu(x)\in\End_E(A_\mu)$ such that
      \[
      \prod_{\iota\in I}\iota(x)\neq\prod_{\iota\in J}\iota(x)
      \]
      for every $J\subseteq I_\mu$ other than $I$.

  \item For an $I$-generator $x$, we put
      \[
      \tT^x_\mu\coloneqq\prod_{J\neq I}\frac{\langle x\rangle-\prod_{\iota\in J}\iota(x)}{\prod_{\iota\in I}\iota(x)-\prod_{\iota\in J}\iota(x)}\in
      \CH^{[M_\mu:\dQ]/2}(A_\mu\times A_\mu)_\dC,
      \]
\end{enumerate}
\end{definition}

We now choose an $I_1$-generator $x$. It is easy to see that $\tT_\mu^x$ lies in $\CH^{[M_\mu:\dQ]/2}(A_\mu\times A_\mu)_{M_\mu}$, and moreover $\cl^*_{\rB,\tau'}(\tT^x_\mu)=\tT_\mu^{\rB,\tau'}$. In particular, the numerical equivalence class of $\tT_\mu^x$ is independent of the choice of $x$, which we denote by $\tT_\mu^{\r{num}}$. Applying the main theorem of \cite{OS11}\footnote{The author states the theorem with coefficients in $\dQ$. However, by \cite{OS11}*{Corollary~6.2.6}, one may replace $\dQ$ by any field of characteristic zero, for example, $M_\mu$.} to $A_\mu\times A_\mu$, we know that $\tT_\mu^{\r{num}}$ has a canonical lift in $\CH^{[M_\mu:\dQ]/2}(A_\mu\times A_\mu)_{M_\mu}$, which we denote by $\tT_\mu^{\r{can}}$.

\begin{definition}\label{de:mu_canonical}
We call $\tT_\mu^{\r{can}}\in\CH^{[M_\mu:\dQ]/2}(A_\mu\times A_\mu)_{M_\mu}$ the \emph{canonical projector} of $A_\mu$.
\end{definition}

\begin{lem}\label{le:mu_canonical}
For every $\tau'\colon E\hookrightarrow\dC$, we have $\cl^*_{\rB,\tau'}(\tT_\mu^{\r{can}})=\tT_\mu^{\rB,\tau'}$.
\end{lem}

\begin{proof}
By part (iii) of the main theorem of \cite{OS11}, we know that $\tT_\mu^{\r{can}}$ hence $\tT_\mu^{\r{can}}-\tT_\mu^x$ commute with $\langle y\rangle$ for all $y\in M_\mu$ such that $i_\mu(y)\in\End_E(A_\mu)$. Now we show that $\tT_\mu^{\r{can}}-\tT_\mu^x$ is homologically trivial. If not, then there exist some $0\leq i\leq [M_\mu:\dQ]$ and a set $I\subseteq I_\mu$ with $|I|=[M_\mu:\dQ]-i$ such that the restriction $\cl^*_{\rB,\tau'}(\tT_\mu^{\r{can}}-\tT_\mu^x)\res\rH^i_{\rB,\tau'}(A,\dC)_I$ is the canonical embedding.

Now we take an $I^c$-generator $y\in M_\mu$, where $I^c\coloneqq I_\mu\setminus I$. Then the cycle $\tT^y_\mu$ (Definition \ref{de:generator}) has nonzero intersection number with $\tT_\mu^{\r{can}}-\tT_\mu^x$. This contradicts with the fact that $\tT_\mu^{\r{can}}-\tT_\mu^x$ is numerically trivial. Thus, the lemma follows.
\end{proof}

\subsection{Albanese of unitary Shimura varieties}
\label{ss:albanese_unitary}

Let $n\geq 2$ be an integer. Let $\bV$ be a totally definite incoherent hermitian space over $\bA_E$ of rank $n$ (Definition \ref{de:incoherent_hermitian}). We distinguish between two cases:
\begin{description}
  \item[Noncompact Case] $d=1$, and either $n\geq3$ or $n=2$ and the hermitian space $\bV\otimes_\bA\dQ_p$ is isotropic for every rational prime $p$.

  \item[Compact Case] if it is not in the Noncompact Case.
\end{description}

Let $\bG\coloneqq\UG(\bV)$ be the unitary group of $\bV$, which is a reductive group over $\bA_F$. Let $\{\Sh(\bV)_K\}_K$ be the projective system of Shimura varieties for $\bV$ indexed by sufficiently small open compact subgroups $K$ of $\bG(\bA_F^\infty)$ (Definition \ref{de:shimura_incoherent}). Every scheme $\Sh(\bV)_K$ is smooth, quasi-projective, and of dimension $n-1$ over $E$; it is projective if and only if we are in the Compact Case. In all cases, we have the compactified Shimura variety $\widetilde\Sh(\bV)_K$ (Definition \ref{de:shimura_incoherent_toroidal}). Put
\[
X_K\coloneqq\widetilde\Sh(\bV)_K
\]
for short. Then $\{X_K\}_K$ is a projective system of smooth projective schemes in $\Sch_{/E}$ of dimension $n-1$. For $K'\subseteq K$, we denote the transition morphism by $u^{K'}_K\colon X_{K'}\to X_K$, which is a generically finite dominant morphism. Put $X_\infty\coloneqq\varprojlim_K X_K$.

We denote by $A_K$ the Albanese variety $\Alb_{X_K}$ of $X_K$ (Definition \ref{de:albanese}) for short, and by
\begin{align}\label{eq:albanese_shimura}
\alpha_K\coloneqq\alpha_{X_K}\colon\nabla X_K\to A_K
\end{align}
the Albanese morphism (see Definition \ref{de:split} for the meaning of $\nabla$). By functoriality, we obtain a projective system $\{A_K\}_K$. Put
\[
A_\infty\coloneqq\varprojlim_K A_K,
\]
which is an abelian group pro-object in $\Sch_{/E}$. Then the Hecke correspondences provide a homomorphism $\bG(\bA_F^\infty)\to\Aut_E(A_\infty)$.

To study isogeny factors of $A_K$, it suffices to study the $L$-function of $\rH^1_{\et}((A_K)_{E^\ac},\dQ_\ell^\ac)$ by Faltings' isogeny theorem. We start from describing its Betti cohomology $\rH^1_{\rB,\tau'}(A_K,\dC)$. For every embedding $\tau'\colon E\to\dC$, put
\[
\rH^1_{\rB,\tau'}(A_\infty,\dC)\coloneqq\varinjlim_K\rH^1_{\rB,\tau'}(A_K,\dC),
\]
which is an admissible representation of $\bG(\bA_F^\infty)$. To study this representation, we need to recall the oscillator representations of unitary groups.

\begin{definition}\label{de:oscillator_triple}
An \emph{ad\`{e}lic oscillator triple} is a triple $(\mu,\varepsilon,\chi)$ consisting of
\begin{itemize}
  \item a conjugate symplectic automorphic character (Definition \ref{de:conjugate}) $\mu=\otimes\mu_v\colon E^\times\backslash\bA_E^\times\to\dC^\times$ (whose value is necessarily in $\dC^1$),

  \item a collection $\varepsilon=(\varepsilon_v\in E_v^{-\times}/\Nm_{E_v/F_v}E_v^\times)_v$ for every nonarchimedean place $v$ of $F$ such that $\varepsilon_v\in O_{E_v}^\times\Nm_{E_v/F_v}E_v^\times$ for all but finitely many $v$, and

  \item an automorphic character $\chi=\otimes\chi_v\colon E^1\backslash(\bA_E^\infty)^1\to\dC^\times$ (whose value is necessarily in $\dC^1$).
\end{itemize}
For an ad\`{e}lic oscillator triple $(\mu,\varepsilon,\chi)$, the local oscillator representation $\omega(\mu_v,\varepsilon_v,\chi_v)$ of $\bG(F_v)$ introduced in Subsection \ref{ss:local_theta} is unramified for all but finitely many $v$. Thus, it makes sense to define the \emph{ad\`{e}lic oscillator representation} attached to $(\mu,\varepsilon,\chi)$
\begin{align*}
\omega(\mu,\varepsilon,\chi)\coloneqq\bigotimes_v{\!}^\prime\omega(\mu_v,\varepsilon_v,\chi_v),
\end{align*}
which is an irreducible admissible representation of $\bG(\bA_F^\infty)$.
\end{definition}

\begin{definition}\label{de:admissible_collection}
In an ad\`{e}lic oscillator triple $(\mu,\varepsilon,\chi)$, we say that $\varepsilon$ is \emph{$\mu$-admissible} if there exists some $e\in E^{\times-}$ such that
\begin{itemize}
  \item $\varepsilon_v=e\Nm_{E_v/F_v}E^\times_v$ for every nonarchimedean place $v$ of $F$, and

  \item $\tau'(e)$ has negative imaginary part for every $\tau'\in\Phi_\mu$.
\end{itemize}
It is clear by Remark \ref{re:mu} that $\varepsilon$ is $\mu$-admissible if and only if $-\varepsilon$ is $\mu^\tc$-admissible.
\end{definition}

\begin{proposition}\label{pr:endoscopy_general}
Suppose that $n\geq 3$. Then for every embedding $\tau'\colon E\to\dC$, there is an isomorphism
\[
\rH^1_{\rB,\tau'}(A_\infty,\dC)\simeq\bigoplus_{(\mu,\varepsilon,\chi)}\omega(\mu,\varepsilon,\chi)
\]
of $\dC[\bG(\bA_F^\infty)]$-modules, where the direct sum is taken over all ad\`{e}lic oscillator triples in which $\mu$ is of weight one and $\varepsilon$ is $\mu$-admissible.
\end{proposition}

\begin{proof}
By Lemma \ref{le:albanese}, we have a canonical isomorphism
\[
\rH^1_{\rB,\tau'}(A_\infty,\dC)\simeq\rH^1_{\rB,\tau'}(X_\infty,\dC)\coloneqq\varinjlim_K\rH^1_{\rB,\tau'}(X_K,\dC)
\]
of $\dC[\bG(\bA_F^\infty)]$-modules. We regard $E$ as a subfield of $\dC$ via a fixed embedding $\tau'\in\Phi_E$ and put $\tau\coloneqq\tau'\res{F}$. We choose a CM type $\Phi$ that contains $\tau'$. Take a hermitian space $\rV$ that is $\tau$-nearby to $\bV$ (Definition \ref{de:nearby}). Put $\rG\coloneqq\Res_{F/\dQ}\rU(\rV)$ and $\rh\coloneqq\rh^\flat_{\rV,\Phi}$ for short. Then by Propositions \ref{pr:incoherent_shimura} and \eqref{eq:pink}, we have an isomorphism
\[
\rH^1_{\rB,\tau'}(X_\infty,\dC)\simeq\varinjlim_K\rH^1_\rB(\widetilde\Sh(\rG,\rh)_K,\dC)
\]
of $\dC[\bG(\bA_F^\infty)]$-modules. By \cite{MR92}*{Lemma~1}, for every $K$, there is a canonical isomorphism
\[
\rH^1_\rB(\widetilde\Sh(\rG,\rh)_K,\dC)\simeq\IH^1(\ol\Sh(\rG,\rh)_K,\dC),
\]
where the right-hand side is the complex analytic intersection cohomology of the Baily--Borel compactification $\ol\Sh(\rG,\rh)_K$ of $\Sh(\rG,\rh)_K$. Combining \eqref{eq:zucker} and \eqref{eq:matsushima}, we have an isomorphism
\[
\rH^1_{\rB,\tau'}(A_\infty,\dC)\simeq\bigoplus_{\pi}m_\disc(\pi)\cdot\rH^1(\fg,\rK_\rG;\pi_\infty)\otimes\pi^\infty
\]
of $\dC[\bG(\bA_F^\infty)]$-modules. We say that an irreducible admissible representation $\pi$ of $\rG(\bA)$ contributes to the Albanese if $m_\disc(\pi)>0$ and $\rH^1(\fg,\rK_\rG;\pi_\infty)\neq\{0\}$. We determine all such $\pi$ together with the value $m_\disc(\pi)$. By the proof of \cite{BMM}*{Proposition~13.4} (with $m=n$, $p=n-1$, $q=1$, $a+b=1$), we know that there exists a strictly unitary automorphic character (Definition \ref{de:strictly_unitary}) $\mu\colon E^\times\backslash\bA_E^\times\to\dC^\times$ such that the partial $L$-function $L^S(s,\pi\times\mu)$ has a simple pole at $s_0$ with $s_0\geq\frac{n}{2}$ where $S$ is a finite set of places of $F$ containing all archimedean ones and such that for $v\not\in S$, both $\pi_v$ and $\mu_v$ are unramified.\footnote{Note that $\pi$ is assumed to be cuspidal in the statement of \cite{BMM}*{Proposition~13.4}. However, this step works for $\pi$ discrete.} We separate the discussion into two cases.

\textit{Case 1.} Suppose that $\pi$ contributes to the Albanese and $m_\cusp(\pi)>0$. Let $V_\pi$ be a cuspidal realization of $\pi$ (Definition \ref{de:realization}). By Corollary \ref{co:pole1}, $s_0$ is either $\frac{n+1}{2}$ or $\frac{n}{2}$, not both. If $s_0=\frac{n+1}{2}$, then Theorem \ref{th:pole} implies that $\Theta_{(\mu,\nu),\rV}^\rW(V_\pi)\neq\{0\}$, where $\rW$ is the zero skew-hermitian space. By Corollary \ref{co:pole2}(1), $V_\pi$ is a character, hence $\rH^1(\fg,\rK_\rG;\pi_\infty)=\{0\}$, which is a contradiction. Thus, we must have $s_0=\frac{n}{2}$. By Theorem \ref{th:pole}, we have a one-dimensional skew-hermitian space $\rW$ such that $\Theta_{(\mu,\nu),\rV}^\rW(V_\pi)$ and is cuspidal. By Corollary \ref{co:pole2}(1), we have $V_\pi=\Theta_{(\mu^{-1},\nu^{-1})}^{-\rW}(\pi_\rW)$. Note that the central character $\chi$ of $\pi$ satisfies $\chi_\infty=1$. In other words, there is a unique element $e\in E^{-\times}/\Nm_{E/F}E^\times$ determined by $\rW$ such that $\pi^\infty\simeq\omega(\mu,\varepsilon_e,\chi)$, where $\varepsilon_e$ is the collection given by $e$.\footnote{Here, we have replaced $\mu$ by its inverse to match notation in the statement of the proposition.} To determine $\pi_\infty$, we suppose that $\Phi_F=\{\tau_1=\tau,\tau_2,\dots,\tau_d\}$ and $\Phi=\{\tau_1^+,\dots,\tau_d^+\}$ with $\pi(\tau_i^+)=\tau_i$. Using $\Phi$, we obtain an isomorphism $\rG_\dR\simeq\rU(n-1,1)_\dR\times\rU(n,0)_\dR\times\cdots\times\rU(n,0)_\dR$, and accordingly a decomposition $\pi_\infty=\otimes_{i=1}^d\pi_{\infty i}$. Under the notation from Subsection \ref{ss:local_theta}, we have $\pi_{\infty1}\simeq\omega_{n-1,1}^{m_1,\pm,1}$ and $\pi_{\infty i}\simeq\omega_{n,0}^{m_i,\pm,1}$ for $i\geq 2$, where $(m_1,\dots,m_d)$ is the weight of $\mu$ and the sign in the parameter is the sign of $i^{-1}\tau_i^+(e)$. By Lemma \ref{le:weil_arch}, we know that $\mu$ is of weight one and $\varepsilon_e$ is $\mu$-admissible. Moreover, $\rH^1(\fg,\rK_\rG;\pi_\infty)$ is of dimension $1$. By Corollary \ref{co:pole2}(3), we have $m_\cusp(\pi)=1$.

\textit{Case 2.} Suppose that $\pi$ contributes to the Albanese and $m_\disc(\pi)-m_\cusp(\pi)>0$. This might happen only when $F=\dQ$. In this case, $\rV$ has Witt index $1$. Write $\rV=\rV_0\oplus\rD$ where $\rV_0$ is anisotropic and $\rD$ is a hyperbolic hermitian plane. Let $V_\pi$ be a discrete realization (Definition \ref{de:realization}) of $\pi$ that is perpendicular to $\rL^2_\cusp(\rG)$. By Langlands theory of Eisenstein series \cite{MW95}, there exist a strictly unitary automorphic character $\mu\colon E^\times\backslash\bA_E^\times\to\dC^\times$, an irreducible subrepresentation $V_{\pi_0}\subseteq\rL^2_\cusp(\rU(\rV_0))$ of $\rU(\rV_0)(\bA_F)$ with the underlying representation $\pi_0$, and a real number $s_1>0$, such that $V_\pi$ is contained in $\sR_{s_1}(V_{\pi_0}\boxtimes\mu')$: the space generated by residues of $\{\sE_\rQ(g;f_s)\res f\in\rI(V_{\pi_0}\boxtimes\mu')\}$ at $s=s_1$. Here, we adopt the notation in Subsection \ref{ss:pole}.\footnote{The pair $(\rV,\rV_0)$ correspond to the pair $(\rV_1,\rV)$ in Subsection \ref{ss:pole}.} Since $L^S(s,\pi)=L^S(s-s_1,\mu')\cdot L^S(s,\pi_0)$, and $L^S(s,\pi_0\times\mu)$ can not have poles at $s_0\geq\frac{n}{2}$ by Corollary \ref{co:pole1}, we must have $s_1=s_0-1$ and $\mu'=\mu^{-1}$. Again by Corollary \ref{co:pole1}, $\mu'$ is conjugate self-dual, so is $\mu$, and $\mu'=\mu^\tc$. The appearance of the residue implies that $\{\sE_\rQ(g;f_s)\res f\in\rI(V_{\pi_0}\boxtimes\mu^\tc)\}$ has a pole at $s_0-1$. If $s_0=\frac{n+1}{2}$, then by the similar argument in Case 1, we conclude that $\pi_0$ is a character, so is $\pi_1$. This contradicts with $\rH^1(\fg,\rK_\rG;\pi_\infty)\neq\{0\}$. Thus, $s_0=\frac{n}{2}$ and $s_1=\frac{n-2}{2}$. By Corollary \ref{co:pole2}(2) and the similar argument in Case 1, we conclude that $V_\pi=\Theta_{(\mu^{-1},\nu^{-1})}^{-\rW}(\pi_\rW)$ for a unique one-dimensional skew-hermitian space $\rW$ and a unique character $\pi_\rW$; and $\pi^\infty\simeq\omega(\mu,\varepsilon_e,\chi)$ in which $\mu$ is of weight one and $\varepsilon_e$ is $\mu$-admissible. Moreover, we have $m_\cusp(\pi)=1$ and $m_\disc(\pi)=1$ by Corollary \ref{co:pole2}(3).

To summarize, we have shown that if an irreducible admissible representation $\pi$ of $\rG(\bA)$ contributes to the Albanese, then $\pi^\infty\simeq\omega(\mu,\varepsilon,\chi)$ for a unique ad\`{e}lic oscillator triple in which $\mu$ is of weight one and $\varepsilon$ is $\mu$-admissible, and $m_\disc(\pi)=1$. Conversely, for every such ad\`{e}lic oscillator triple $(\mu,\varepsilon,\chi)$, there exists a pair $(\rW,\pi_\rW)$, unique up to isomorphism, such that if we denote by $\pi$ an irreducible subrepresentation of $\Theta^\rV_{(\mu,\nu),\rW}(\pi_\rW)$, then $\omega(\mu,\varepsilon,\chi)$ is isomorphic to $\pi^\infty$ and $\rH^1(\fg,\rK_\rG;\pi_\infty)\neq\{0\}$. Moreover, by the Rallis inner product formula,\footnote{Note that the global theta lifting $\Theta^\rV_{(\mu,\nu),\rW}(\pi_\rW)$ is always in Weil's convergent range.} $\Theta^\rV_{(\mu,\nu),\rW}(\pi_\rW)$ is contained in $\rL^2_\disc(\rG)$. Thus, we may apply the above discussions to the representation $\pi$ to conclude that the dimension of $\rH^1_{\rB,\tau'}(A_\infty,\dC)[\omega(\mu,\varepsilon,\chi)]$ is $1$. The proposition follows.
\end{proof}

\begin{remark}
When $n=3$, Proposition \ref{pr:endoscopy_general} can be deduced from \cites{GR91,Rog92}.
\end{remark}

Now we study the $\ell$-adic cohomology of $A_\infty$. Take an embedding $\tau'\colon E\to\dC$, a rational prime $\ell$, and an isomorphism $\iota_\ell\colon\dC\xrightarrow{\sim}\dQ_\ell^\ac$. We have a canonical isomorphism
\[
\rH^1_{\et}(A_K\otimes_{E,\tau'}\dC,\dQ_\ell^\ac)\simeq\rH^1_{\rB,\tau'}(A_K,\dC)\otimes_{\dC,\iota_\ell}\dQ_\ell^\ac
\]
by the comparison theorem. Put
\[
\rH^1_{\et}(A_\infty\otimes_{E,\tau'}\dC,\dQ_\ell^\ac)\coloneqq\varinjlim_K\rH^1_{\et}(A_K\otimes_{E,\tau'}\dC,\dQ_\ell^\ac),
\]
which is a $\dQ_\ell^\ac[\Gal(\dC/\tau'(E))\times\bG(\bA_F^\infty)]$-module.

Suppose that $n\geq 3$ and consider an ad\`{e}lic oscillator triple $(\mu,\varepsilon,\chi)$ in which $\mu$ is of weight one and $\varepsilon$ is $\mu$-admissible. Then
\[
\Hom_{\dQ_\ell^\ac[\bG(\bA_F^\infty)]}\(\iota_\ell\circ\omega(\mu,\varepsilon,\chi),\rH^1_{\et}(A_\infty\otimes_{E,\tau'}\dC,\dQ_\ell^\ac)\)
\]
is a representation of $\Gal(\dC/\tau'(E))$ over $\dQ_\ell^\ac$. By Proposition \ref{pr:endoscopy_general}, such representation is an $\ell$-adic character, denoted by
\[
\rho_{\tau',\iota_\ell}(\mu,\varepsilon,\chi)\colon\Gal(\dC/\tau'(E))\to(\dQ_\ell^\ac)^\times.
\]
It induces, via the isomorphism $\iota_\ell$, an automorphic character
\[
\rho_{\tau',\ell}(\mu,\varepsilon,\chi)\colon\tau'(E)^\times\backslash\bA_{\tau'(E)}^\times\to\dC^\times.
\]
It is easy to see that the character $\rho_{\tau',\ell}(\mu,\varepsilon,\chi)$ does not depend on the isomorphism of $\iota_\ell$, which justifies its notation.

\begin{theorem}\label{th:cm_albanese_pre}
Suppose that $n\geq 3$ and let $(\mu,\varepsilon,\chi)$ be an ad\`{e}lic oscillator triple in which $\mu$ is of weight one and $\varepsilon$ is $\mu$-admissible. Then we have
\[
\rho_{\tau',\ell}(\mu,\varepsilon,\chi)\circ\tau'=\mu^\alg
\]
for every $\tau'\in\Phi_\mu$ and every rational prime $\ell$.
\end{theorem}

\begin{proof}
We fix a rational prime $\ell$ and an isomorphism $\iota_\ell\colon\dC\xrightarrow{\sim}\dQ_\ell^\ac$. We also fix an element $\tau'\in\Phi_\mu$, and identify $E$ as a subfield of $\dC$ via $\tau'$. Let $\rV$ be the hermitian space that is $\tau$-nearby to $\bV$ as in the proof of Proposition \ref{pr:endoscopy_general}, and $\Sh(\rG,\rh)$ the corresponding Shimura variety from Subsection \ref{ss:appendix_isometry} with $\rG\coloneqq\Res_{F/\dQ}\rU(\rV)$ and $\rh\coloneqq\rh^\flat_{\rV,\Phi}$. Then in view of the discussion in Subsection \ref{ss:setup}, we have canonical isomorphisms
\begin{align*}
\rH^1_{\et}(A_\infty\otimes_{E,\tau'}\dC,\dQ_\ell^\ac)\otimes_{\dQ_\ell^\ac,\iota_\ell^{-1}}\dC
\simeq\rH^1_{(2)}(\Sh(\rG,\rh),\dC)
\simeq\bigoplus_\chi\rH^1(\fg,\rK_\rG;\rL^2(\rG(\dQ)\backslash\rG(\bA),\chi)).
\end{align*}

For every orthogonal decomposition $\rV=\rV_\star\oplus\rV_\star^\perp$ of hermitian spaces such that $\rV_\star^\perp$ is totally positive definite, we have similarly the Shimura variety $\Sh(\rG_\star,\rh_\star)$ together with the morphism $\Sh(\rG_\star,\rh_\star)\to\Sh(\rG,\rh)$ over $E$. For an element $e\in E^{\times-}$, we choose a maximal isotropic $F$-subspace $\rV^e$ of the symplectic space $(\Res_{E/F}\rV,\Tr_{E/F}e(\;,\;)_\rV)$. Then $\rV^e_\star\coloneqq\rV^e\cap\Res_{E/F}\rV_\star$ is a maximal isotropic $F$-subspace of $(\Res_{E/F}\rV_\star,\Tr_{E/F}e(\;,\;)_{\rV_\star})$. Denote by $V(\mu,e)\subseteq\sC^\infty(\rG(\dQ)\backslash\rG(\bA),\dC)$ the subspace of theta functions
\[
\theta_\mu^\phi(g)\coloneqq\sum_{v\in\rV^e}(\omega_{\mu,\varepsilon}(g)\phi)(v)
\]
on $\rG(\bA)$, where $\phi$ is in the Schwartz space $\sS(\rV^e(\bA_F))$ in which we use the Fock model at archimedean places. Similarly, we define the subspace $V_\star(\mu,e)\subseteq\sC^\infty(\rG_\star(\dQ)\backslash\rG_\star(\bA))$.

We claim that the map $\sC^\infty(\rG(\dQ)\backslash\rG(\bA),\dC)\to\sC^\infty(\rG_\star(\dQ)\backslash\rG_\star(\bA),\dC)$ induced by the inclusion $\rG_\star\hookrightarrow\rG$ sends $V(\mu,e)$ to $V_\star(\mu,e)$. In fact, we can find finitely many pairs $(\phi_{\star,i},\phi_{\star,i}^\perp)$ with $\phi_{\star,i}\in\sS(\rV^e_\star(\bA_F))$ and $\phi_{\star,i}^\perp\in\sS(\rV^{\perp e}_\star(\bA_F))$, where $\rV^{\perp e}_\star\coloneqq\rV^e\cap\Res_{E/F}\rV_\star^\perp$, such that
\[
\phi(v_\star,v_\star^\perp)=\sum_i\phi_{\star,i}(v_\star)\cdot\phi_{\star,i}^\perp(v_\star^\perp)
\]
for every $v_\star\in\rV^{\perp e}_\star(\bA_F)$ and $v_\star^\perp\in\rV^{\perp e}_\star(\bA_F)$. Then for $g_\star\in\rG_\star(\bA)\subseteq\rG(\bA)$, we have
\begin{align*}
\theta_\mu^\phi(g_\star)
=\sum_i\(\sum_{v_\star\in\rV^e_\star}(\omega_{\mu,\varepsilon}(g_\star)\phi_{\star,i})(v_\star)\)
\(\sum_{v_\star^\perp\in\rV^{\perp e}_\star}\phi_{\star,i}^\perp(v_\star^\perp)\)
=\sum_i\(\sum_{v_\star^\perp\in\rV^{\perp e}_\star}\phi_{\star,i}^\perp(v_\star^\perp)\)\theta_\mu^{\phi_{\star,i}}(g_\star).
\end{align*}
Thus, the claim follows.

To prove the theorem, note that for every class $c\in\rH^1_{(2)}(\Sh(\rG,\rh),\dC)$, using the same proof of \cite{MR92}*{Proposition~6}, one can find a decomposition $\rV=\rV_\star\oplus\rV_\star^\perp$ as above with $\dim\rV_\star=2$ such that the image of $c$ under the restriction map $\rH^1_{(2)}(\Sh(\rG,\rh),\dC)$ in $\rH^1_\rB(\Sh(\rG_\star,\rh_\star),\dC)$ is nonzero.\footnote{Although \cite{MR92}*{Proposition~6} only implies the existence of such $\rV_\star$ with $\dim\rV_\star=3$, its proof actually shows the existence of such $\rV_\star$ with $\dim\rV_\star=2$ by only changing the term $n-2$ to $n-1$ in the proof of Lemma B. The authors of \cite{MR92} presented their argument for $\dim\rV_\star=3$ simply because the property they aimed to reduce does not hold when $\dim\rV_\star=2$.} We denote such image of $c_\star$, and note that $c_\star$ actually belongs to (the image of) $\rH^1_\rB(\ol{\Sh}(\rG_\star,\rh_\star),\dC)$. Then the theorem follows from the above claim, Remark \ref{re:galois_curve}, and Theorem \ref{th:galois_curve}(1).
\end{proof}

\begin{definition}\label{de:cm_space}
Let $\mu\colon E^\times\backslash\bA_E^\times\to\dC^\times$ be a conjugate symplectic character of weight one. For every object $D_\mu=(A_\mu,i_\mu,\lambda_\mu,r_\mu)\in\cA(\mu)$ (Definition \ref{de:cm_data}), the $\dQ$-vector space $\Hom_E(A_\infty,A_\mu)_\dQ$ is an $M_\mu[\bG(\bA_F^\infty)]$-module, where $M_\mu$ acts via $i_\mu$ and $\bG(\bA_F^\infty)$ acts $M_\mu$-linearly via its action on $A_\infty$. Put
\[
\Omega(\mu)\coloneqq\varinjlim_{D_\mu\in\cA(\mu)}\Hom_E(A_\infty,A_\mu)_\dQ
\]
in the category of $M_\mu[\bG(\bA_F^\infty)]$-modules.
\end{definition}

\begin{remark}\label{re:cm_space}
It follows from Proposition \ref{pr:cm_data}(1) that for every object $D_\mu=(A_\mu,i_\mu,\lambda_\mu,r_\mu)\in\cA(\mu)$, the canonical map $\Hom_E(A_\infty,A_\mu)_\dQ\to\Omega(\mu)$ is an isomorphism.
\end{remark}

\begin{theorem}\label{th:cm_albanese}
There is an isomorphism
\[
\Omega(\mu)\otimes_{M_\mu}\dC\simeq\bigoplus_{\varepsilon}\bigoplus_{\chi}\omega(\mu,\varepsilon,\chi)
\]
of $\dC[\bG(\bA_F^\infty)]$-modules, where the direct sum is taken over all $\varepsilon,\chi$ such that $\varepsilon$ is $\mu$-admissible. Moreover,
\begin{enumerate}
  \item For every object $D_\mu=(A_\mu,i_\mu,\lambda_\mu,r_\mu)\in\cA(\mu)$, we have a canonical isomorphism $\Omega(\mu)^K\simeq\Hom_E(A_K,A_\mu)_\dQ$ for every sufficiently small open compact subgroup $K\subseteq\bG(\bA_F^\infty)$.

  \item The $\dC[\bG(\bA_F^\infty)]$-modules in the direct sum in Theorem \ref{th:cm_albanese} are mutually non-isomorphic.

  \item For every given $\varepsilon$ that is $\mu$-admissible, the subspace $\bigoplus_{\chi}\omega(\mu,\varepsilon,\chi)$ is stable under the action of $\Gal(\dC/M_\mu)$.
\end{enumerate}
\end{theorem}

\begin{proof}
Take an arbitrary object $D_\mu=(A_\mu,i_\mu,\lambda_\mu,r_\mu)\in\cA(\mu)$ and identify $\Omega(\mu)$ with $\Hom_E(A_\infty,A_\mu)_\dQ$ by Remark \ref{re:cm_space}.

Take an embedding $\tau'\colon E\to\dC$ in $\Phi_\mu$. It is clear that the maximal subspace of the complex vector space $\rH^1_{\rB,\tau'}(A_\mu,\dC)$ over which $M_\mu$ acts via the inclusion $M_\mu\hookrightarrow\dC$ has dimension $1$. We choose a basis $\alpha$ of this subspace. Then we obtain a map $\Omega(\mu)\to\rH^1_{\rB,\tau'}(A_\infty,\dC)$ by pulling back $\alpha$, which is $\dC[\bG(\bA_F^\infty)]$-linear. It canonically extends to a map
\begin{align}\label{eq:cm_albanese}
\Omega(\mu)\otimes_{M_\mu}\dC\to\rH^1_{\rB,\tau'}(A_\infty,\dC).
\end{align}
To compute this map, we choose a rational prime $\ell$ and an isomorphism $\iota_\ell\colon\dC\xrightarrow{\sim}\dQ_\ell^\ac$. By the comparison theorem, \eqref{eq:cm_albanese} induces the following map
\[
\Omega(\mu)\otimes_{M_\mu,\iota_\ell}\dQ_\ell^\ac\to\rH^1_{\et}(A_\infty\otimes_{E,\tau'}\dC,\dQ_\ell^\ac)
\]
by pulling back $\alpha$, as a class in $\rH^1_{\et}(A_\mu\otimes_{E,\tau'}\dC,\dQ_\ell^\ac)$. By Faltings' isogeny theorem \cite{Fal83}, we have a canonical isomorphism
\[
\Omega(\mu)\otimes_{M_\mu,\iota_\ell}\dQ_\ell^\ac
\simeq\Hom_{\dQ_\ell^\ac[\Gal(\dC/\tau'(E))]}\(\dQ_\ell^\ac\cdot\alpha,\rH^1_{\et}(A_\mu\otimes_{E,\tau'}\dC,\dQ_\ell^\ac)\).
\]
However, by Definition \ref{de:cm_data}(2), the action of $\Gal(\dC/\tau'(E))$ on the line $\dQ_\ell^\ac\cdot\alpha$ spanned by $\alpha$ is given by the automorphic character $\iota_\ell\circ\mu^\alg\circ(\tau')^{-1}\colon\tau'(E)^\times\backslash\bA_{\tau'(E)}^\times\to(\dQ_\ell^\ac)^\times$. When $n\geq 3$ (resp.\ $n=2$), by Proposition \ref{pr:endoscopy_general} (resp.\ Proposition \ref{pr:endoscopic_curve}(1) with Remark \ref{re:galois_curve}) and Theorem \ref{th:cm_albanese_pre} (resp.\ Theorem \ref{th:galois_curve}(1)), we have an isomorphism
\[
\Hom_{\dQ_\ell^\ac[\Gal(\dC/\tau'(E))]}\(\dQ_\ell^\ac\cdot\alpha,\rH^1_{\et}(A_\mu\otimes_{E,\tau'}\dC,\dQ_\ell^\ac)\)
\simeq\bigoplus_{\varepsilon}\bigoplus_{\chi}\omega(\mu,\varepsilon,\chi)\otimes_{\dC,\iota_\ell}\dQ_\ell^\ac
\]
of $\dQ_\ell^\ac[\bG(\bA_F^\infty)]$-modules induced by pulling back $\alpha$, where the direct sum is taken over all $\varepsilon,\chi$ such that $\varepsilon$ is $\mu$-admissible. Thus, we obtain an isomorphism as in the theorem, which depends only on $\alpha$, not on $\ell$, $\iota_\ell$, and $\tau'$.

The additional statement (1) follows from the above discussion as well. Statement (2) follows from Lemma \ref{le:weil_nonarch}.

Now we consider statement (3). Since $\Gal(\dC/M_\mu)$ stabilizes $\mu$ and by (2), it suffices to show that for every rational prime $p$, the image of $\Gal(\dC/M_\mu)$ under the $p$-adic cyclotomic character $\chi_p\colon\Gal(\dC/\dQ)\to\dZ_p^\times$ is contained in $\dZ_p^\times\cap\Nm_{E_\fp/F_\fp}E_\fp^\times$ for every prime $\fp$ of $F$ above $p$. This only becomes a problem if $\fp$ is ramified in $E$. To ease notation, we suppress the subscript $\fp$. So we have a ramified quadratic extension $E/F$, where $F/\dQ_p$ is a finite extension. Put $U_{E/F}\coloneqq\dZ_p^\times\cap\Nm_{E/F}E^\times$, which we may assume a subgroup of $\dZ_p^\times$ of index $2$. Denote by $M_{E/F}\subseteq\dC$ the subfield corresponding to the kernel of the composite homomorphism $\Gal(\dC/\dQ)\xrightarrow{\chi_p}\dZ_p^\times\to\dZ_p^\times/U_{E/F}$, which is a quadratic field. Thus, our goal is to show that $M_{E/F}$ is contained in $M_\mu$.

We first assume $p$ odd. Then the residue extension degree $f$ of $F/\dQ_p$ must be odd. Write $E=F(\sqrt{u})$ for a uniformizer $u$ of $F$. Then $\mu(\sqrt{u})^2=\mu(\sqrt{u}^2)=\mu(-\Nm_{E/F}u)=\mu(-1)$.
\begin{itemize}
  \item If $\mu(-1)=1$, then $-1$ is a quadratic residue modulo $p$, hence $M_{E/F}=\dQ(\sqrt{p})$. On the other hand, since $\mu(\sqrt{u})=\pm 1$, we have $\mu^\alg(\sqrt{u})=\pm p^{f/2}$. Thus, $M_\mu$ contains $\sqrt{p}$ as $f$ is odd.

  \item If $\mu(-1)=-1$, then $-1$ is not a quadratic residue modulo $p$, hence $M_{E/F}=\dQ(\sqrt{-p})$. On the other hand, since $\mu(\sqrt{u})=\pm \sqrt{-1}$, we have $\mu^\alg(\sqrt{u})=\pm\sqrt{-1} p^{f/2}$. Thus, $M_\mu$ contains $\sqrt{-p}$ as $f$ is odd.
\end{itemize}

We now assume $p=2$. Write $v_F\colon F\to\dZ\cup\{\infty\}$ for the valuation function on $F$. We choose an Eisenstein polynomial $X^2+aX+b$ for $E/F$ with $v_F(a)\geq 1$ and $v_F(b)=1$. Put $d\coloneqq\min\{2v_F(a)-1,v_F(4)\}$, which is an invariant of $E/F$. There are three cases.
\begin{itemize}
  \item Suppose that $M_{E/F}=\dQ(\sqrt{-1})$. Then $U_{E/F}=1+4\dZ_2$. If $d=v_F(4)$, then by \cite{BH06}*{Proposition~41.2(2)}, $3$ is contained in $\Nm_{E/F}E^\times$, which is a contradiction. Thus, we have $d<v_F(4)$. Then we can find $u\in O_F^\times$ such that $E=F(\sqrt{u})$. It follows that $\mu(\sqrt{u})^2=\mu(-1)=-1$ since $-1\not\in U_{E/F}$. Thus, $\mu^\alg(\sqrt{u})=\pm\sqrt{-1}$ is contained in $M_\mu$.

  \item Suppose that $M_{E/F}=\dQ(\sqrt{2})$. Then $U_{E/F}=\pm1+8\dZ_2$. In particular, $U_{E/F}$ does not contain $5$, hence the residue extension degree $f$ of $F/\dQ_2$ must be odd. Moreover, by \cite{BH06}*{Proposition~41.2(2)} again, we must have $d=v_F(4)$, hence $v_F(a)\geq v_F(2)+1$. Then we can find a uniformizer $u$ of $F$ such that $E=F(\sqrt{u})$. We have $\mu(\sqrt{u})^2=\mu(-1)=1$, and $\mu^\alg(\sqrt{u})=\pm 2^{f/2}$. In particular, $\sqrt{2}$ is contained in $M_\mu$.

  \item Suppose that $M_{E/F}=\dQ(\sqrt{-2})$. Then $U_{E/F}=\pm1 +2+8\dZ_2$. In particular, $U_{E/F}$ does not contain $5$ or $-1$. The remaining discussion is same as the above case, which we omit.
\end{itemize}
Statement (3) is proved.
\end{proof}

Theorem \ref{th:cm_albanese}(2,3) allows us to make the following definition.

\begin{definition}\label{de:cm_albanese}
For every collection $\varepsilon$ that is $\mu$-admissible, we denote by $\Omega(\mu,\varepsilon)$ the unique $M_\mu[\bG(\bA_F^\infty)]$-submodule of $\Omega(\mu)$, such that $\Omega(\mu,\varepsilon)\otimes_{M_\mu}\dC$ is isomorphic to $\bigoplus_{\chi}\omega(\mu,\varepsilon,\chi)$ as a $\dC[\bG(\bA_F^\infty)]$-module.
\end{definition}

\begin{corollary}\label{co:cm_albanese}
Take an arbitrary object $D_\mu=(A_\mu,i_\mu,\lambda_\mu,r_\mu)\in\cA(\mu)$. For every sufficiently small open compact subgroup $K$ of $\bG(\bA^\infty_F)$, there is an isogeny decomposition
\[
A_K\sim\prod_\mu A_\mu^{d(\mu,K)},\qquad
\text{resp. }A_K^{\r{end}}\sim\prod_\mu A_\mu^{d(\mu,K)}
\]
of abelian varieties over $E$ when $n\geq 3$ (resp.\ $n=2$), where the product is taken over representatives of $\Gal(\dC/\dQ)$-orbits of all conjugate symplectic automorphic characters of $\bA_E^\times$ of weight one. Here, $A_K^{\r{end}}$ is the endoscopic part of $A_K$ when $n=2$, defined in \eqref{eq:endoscopic_albanese}, and
\[
d(\mu,K)\coloneqq\sum_{\varepsilon}\sum_{\chi}\dim_\dC\omega(\mu,\varepsilon,\chi)^K,
\]
where the sum is taken over all $\varepsilon,\chi$ such that $\varepsilon$ is $\mu$-admissible.
\end{corollary}

It is clear that the integer $d(\mu,K)$ depends only on the $\Gal(\dC/\dQ)$-orbit of $\mu$.

\begin{proof}
This is a direct consequence of Theorem \ref{th:cm_albanese}.
\end{proof}

\begin{remark}
Corollary \ref{co:cm_albanese} has a very interesting implication. Namely, if $n\geq 3$ and $X_K$ has exotic smooth reduction, that is, $X_K$ has proper smooth reduction at some nonarchimedean place of $E$ that is ramified over $F$, then $\rH^1_\dr(X_K/E)=\{0\}$ since $A_\mu$ cannot have good reduction at such a place.
\end{remark}

At the end of this subsection, we will construct a canonical pairing
\begin{align}\label{eq:albanese_pair}
(\;,\;)_\mu\colon\Omega(\mu)\times\Omega(\mu^\tc)\to M_\mu
\end{align}
that is $M_\mu$-bilinear, non-degenerate, and $\bG(\bA_F^\infty)$-invariant.

\begin{definition}\label{de:hodge_divisor}
We define
\begin{enumerate}
  \item the \emph{Hodge divisor} $D_K$ on $X_K$, as an element in $\CH^1(X_K)_\dQ$, to be
    \begin{itemize}
      \item the usual Hodge divisor on the Shimura variety $\Sh(\bV)_K$ if $\Sh(\bV)_K$ is proper (Compact Case),

      \item the canonical extension of the usual Hodge divisor on $\Sh(\bV)_K$ to $X_K$ if $\Sh(\bV)_K$ is not proper (Noncompact Case).
    \end{itemize}

  \item the \emph{canonical volume} of $K$ to be
    \[
    \vol(K)\coloneqq\frac{1}{\deg D_K^{n-1}\cdot|\pi_0((X_K)_{E^\ac})|},
    \]
    in which $\deg D_K^{n-1}$ is regarded as a constant positive integer by Lemma \ref{le:almost_ample}(4) below.
\end{enumerate}
\end{definition}

\begin{lem}\label{le:almost_ample}
We have
\begin{enumerate}
  \item The Hodge divisor $D_K$ is almost ample (Definition \ref{de:almost_ample}).

  \item For every transition morphism $u^{K'}_K\colon X_{K'}\to X_K$, $(u^{K'}_K)^*D_K$ is rationally equivalent to $D_{K'}$.

  \item For every $g\in\bG(\bA_F^\infty)$, $\tT_g^*D_K$ is rationally equivalent to $D_{gKg^{-1}}$, where $\tT_g\colon X_{gKg^{-1}}\to X_K$ is the Hecke translation.

  \item The degree function $\deg D_K^{n-1}$ is a constant positive integer on $\pi_0(X_K)$.
\end{enumerate}
\end{lem}

\begin{proof}
Consider (1) first. If $n=2$, then (for sufficiently small $K$) $X_K$ has genus at least $2$. Since $D_K$ has positive degree on every connected component, it is ample, hence almost ample. Now suppose that $n\geq 3$. If we are in the Compact Case, then the usual Hodge divisor is already ample. If we are in the Noncompact Case, then $D_K$ is the pullback of the Hodge divisor on the Baily--Borel compactification of $\Sh(\bV)_K$. Since the latter is ample, $D_K$ is almost ample (and in fact, \emph{not} ample).

For (2,3), since $D_K$ is the (canonical extension of the) usual Hodge divisor of $\Sh(\bV)_K$, it is functorial under pullbacks and Hecke translations. For (4), the positivity follows from (1) and Remark \ref{re:biproduct}; the constancy is a consequence of (2,3).
\end{proof}

Thus, by Proposition \ref{pr:almost_ample}, we obtain a polarization
\[
\theta_K\coloneqq\theta_{X_K,D_K}\colon A_K^\vee\to A_K.
\]
Now we define the pairing \eqref{eq:albanese_pair}. We choose an object $D_\mu=(A_\mu,i_\mu,\lambda_\mu,r_\mu)\in\cA(\mu)$, which induces the object $D^\vee_\mu=(A^\vee_\mu,i^\vee_\mu,\lambda^\vee_\mu,r^\vee_\mu)\in\cA(\mu^\tc)$. Then we have $\Omega(\mu)=\Hom_E(A_\infty,A_\mu)_\dQ$ and $\Omega(\mu^\tc)=\Hom_E(A_\infty,A^\vee_\mu)_\dQ$. It suffices to consider elements $\phi\in\Hom_E(A_\infty,A_\mu)$ and $\phi_\tc\in\Hom_E(A_\infty,A^\vee_\mu)$. Since both $A_\mu$ and $A_{\mu^\tc}$ are of finite type, we may choose some $K$ such that both $\phi$ and $\phi_\tc$ factor through $A_K$. The composite map
\[
A_\mu\simeq A_\mu^{\vee\vee}=(A_{\mu^\tc})^\vee\xrightarrow{\phi_\tc^\vee}A_K^\vee\xrightarrow{\theta_K}A_K\xrightarrow{\phi}A_\mu
\]
belongs to $\End_E(A_\mu)_\dQ=i_\mu(M_\mu)$. Now we define
\[
(\phi,\phi_\tc)^K_\mu\coloneqq\vol(K)\cdot i_\mu^{-1}(\phi\circ\theta_K\circ\phi_\tc^\vee)\in M_\mu.
\]
For sufficiently small $K$ and $K'\subseteq K$, the degree of the transition morphism $u^{K'}_K$ equals $\vol(K)\cdot\vol(K')^{-1}$ by Lemma \ref{le:almost_ample}(2). Thus, by Lemma \ref{le:almost_ample}(2) and Proposition \ref{pr:albanese_functoriality}, we know that $(\phi,\phi_\tc)^K_\mu$ does not depend on the choice of $K$, which we define as $(\phi,\phi_\tc)_\mu$. It is clear from the construction that \eqref{eq:albanese_pair} is bilinear, independent of the choice of $D_\mu$, non-degenerate since $\theta_K$ is a polarization for every $K$, and $\bG(\bA_F^\infty)$-invariant since $\{D_K\}_K$ is functorial under Hecke translations.


\subsection{Construction of Fourier--Jacobi cycles}
\label{ss:construction_cycles}

Let $\bV$ be a totally definite incoherent hermitian space over $\bA_E$ of rank $n\geq 2$, with $\bG\coloneqq\rU(\bV)$. From now on to the end of Section \ref{ss:4}, we
\begin{itemize}
  \item fix a conjugate symplectic automorphic character $\mu\colon E^\times\backslash\bA_E^\times\to\dC^\times$ of weight one, and

  \item will only consider sufficiently small open compact subgroups $K\subseteq\bG(\bA_F^\infty)$ that are decomposable, that is, $K$ can be written as $\prod_vK_v$ when $v$ runs over all nonarchimedean places of $F$; we call such $K$ a \emph{level subgroup}.
\end{itemize}

Let $R$ be a ring containing $\dQ$. Let
\[
\sH_R\coloneqq\sC^\infty_c(\bG(\bA_F^\infty),R)
\]
be the full Hecke algebra with coefficients in $R$, whose multiplication is given by the convolution with respect to the canonical volume (Definition \ref{de:hodge_divisor}(3)). It is known that $\sH_R$ is an $R[\bG(\bA_F^\infty)\times\bG(\bA_F^\infty)]$-module via left and right translations. For $g\in\bG(\bA_F^\infty)$, we denote by $\rL_g$ and $\rR_g$ the left and right translations on $\sH_R$, respectively.

For a level subgroup $K\subseteq\bG(\bA_F^\infty)$, we have the Hecke (sub)algebra $\sH_{K,R}\coloneqq\sC^\infty_c(K\backslash\bG(\bA_F^\infty)/K,R)$, which admits an $R$-linear map
\[
\tT_K\colon\sH_{K,R}\to\rZ^{n-1}(X_K\times X_K)_R
\]
sending $f$ to the Hecke correspondence $\tT_K^f$, normalized by $\vol(K)$. For example, if $f=\CF_K$, then $\tT_K^f=\vol(K)\cdot\Delta X_K\in\rZ^{n-1}(X_K\times X_K)_R$. The induced map (with the same notation)
\[
\tT_K\colon\sH_{K,R}\to\CH^{n-1}(X_K\times X_K)_R
\]
is a homomorphism of $R$-algebras. It is clear that $\sH_R=\varinjlim_K\sH_{K,R}$.

\begin{definition}
Let $\Pi$ be a relevant representation of $\GL_n(\bA_E)$ (Definition \ref{de:relevant}). We define $\Phi_\Pi$ to be the set of isomorphism classes of pairs $(\bV,\pi^\infty)$, where
\begin{itemize}
  \item $\bV$ is a totally positive definition incoherent hermitian space over $\bA_E$ of rank $n$,

  \item $\pi^\infty$ is an irreducible admissible representation of $\bG(\bA_F^\infty)$ such that
    \begin{itemize}
      \item for a nonarchimedean place $v$ of $F$ either split in $E$ or at which $\pi^\infty_v$ is unramified, we have $\r{BC}(\pi^\infty_v)\simeq\Pi_v$,

      \item $\pi^\infty$ appears in $\rH^i_{\rB,\tau'}(\widetilde\Sh(\bV),\dC)$ as a subquotient representation of $\bG(\bA_F^\infty)$ for some $i\in\dZ$ and some place $\tau'\colon E\to\dC$.
    \end{itemize}
\end{itemize}
\end{definition}

\begin{proposition}\label{pr:arthur}
Let $\Pi$ be a relevant representation of $\GL_n(\bA_E)$. For $(\bV,\pi^\infty)\in\Phi_\Pi$, we have for every $\tau'\colon E\to\dC$ that
\begin{enumerate}
  \item $\pi^\infty$ appears in $\rH^i_{\rB,\tau'}(\widetilde\Sh(\bV),\dC)$ semisimply for every $i$,

  \item $\pi^\infty$ does not appear in $\rH^i_{\rB,\tau'}(\widetilde\Sh(\bV),\dC)$ if $i\neq n-1$,

  \item $\rH^i_{\rB,\tau'}(\widetilde\Sh(\bV),\dC)[\pi^\infty]=\IH^i_{\rB,\tau'}(\widetilde\Sh(\bV),\dC)[\pi^\infty]$ for every $i$.
\end{enumerate}
\end{proposition}

\begin{proof}
Put $\tau\coloneqq\tau'\res F$, and fix a hermitian space $\rV$ that is $\tau$-nearby to $\bV$ (Definition \ref{de:nearby}). Put $\rG\coloneqq\Res_{F/\dQ}\rU(\rV)$, and identify $\widetilde\Sh(\bV)\otimes_{E,\tau'}\tau'(E)$ with the (compactified) Shimura variety $\widetilde\Sh(\rG,\rh_{\rV,\tau'})$ under the notation in Subsection \ref{ss:appendix_isometry}.

We first note that $\pi^\infty$ is not a constituent of the quotient representation p$\rH^i_{\rB}(\widetilde\Sh(\rG,\r h_{\rV,\tau'}),\dC)/\break\IH^i_{\rB}(\widetilde\Sh(\rG,\rh_{\rV,\tau'}),\dC)$, since otherwise $\Pi$ will have two isomorphic cuspidal factors under Definition \ref{de:relevant}(1), which can not happen by Definition \ref{de:relevant}(2). Then (1) and (3) follow by the discussion in Subsection \ref{ss:setup}.

If $\pi^\infty$ appears in $\IH^i_{\rB}(\widetilde\Sh(\rG,\rh_{\rV,\tau'}),\dC)$, then there is an automorphic representation $\pi_\infty\otimes\pi^\infty$ of $\rG(\bA)$ with $m_\disc(\pi_\infty\otimes\pi^\infty)\geq 1$ such that $\rH^i(\fg,\rK_\rG;\pi_\infty)\neq\{0\}$. By \cite{Car12}*{Theorem~1.2}, we know that $\Pi$ is everywhere tempered. By Arthur's endoscopic classification \cite{Art13}, which has been worked out in \cite{Mok15} and \cite{KMSW} for tempered representations for unitary groups, we know that the local base change of $\pi_\infty$ must be $\Pi_\infty$, which implies that $\pi_\infty$ is a discrete series representation. In particular, $i$ has to be the middle degree $n-1$. Thus, (2) follows.

\end{proof}

\begin{definition}\label{de:test_function}
Let $\Pi$ and $(\bV,\pi^\infty)$ be as in Proposition \ref{pr:arthur}. Let $K\subseteq\bG(\bA_F^\infty)$ be a level subgroup. We say that a function $f\in\sH_{K,\dL}$, where $\dL$ is some subfield of $\dC$, is a \emph{test function for $\pi^\infty$}, if the element
\[
\cl_{\rB,\tau'}(\tT_K^f)\in\rH^{2n-2}_{\rB,\tau'}(\widetilde\Sh(\bV)_K\times\widetilde\Sh(\bV)_K,\dC)
\]
belongs to the subspace $\rH^{n-1}_{\rB,\tau'}(\widetilde\Sh(\bV)_K,\dC)[(\pi^\infty)^K]\otimes_\dC
\rH^{n-1}_{\rB,\tau'}(\widetilde\Sh(\bV)_K,\dC)[((\pi^\infty)^\vee)^K]$
under the K\"{u}nneth decomposition for every $\tau'\colon E\to\dC$.
\end{definition}

Now we start to construct the Fourier--Jacobi cycles. We fix two relevant representations $\Pi_1$ and $\Pi_2$ of $\GL_n(\bA_E)$, and consider pairs $(\bV,\pi_i^\infty)\in\Phi_{\Pi_i}$ for $i=1,2$ with the same $\bV$. Let $\dL\subseteq\dC$ be a subfield containing $M_\mu$ over which $\Pi_1^\infty$ and $\Pi_2^\infty$ (hence $\pi_1^\infty$ and $\pi_2^\infty$) are both defined. In what follows, we will regard $\pi_1^\infty$ and $\pi_2^\infty$ as irreducible $\dL[\bG(\bA_F^\infty)]$-modules. Take a CM data $D_\mu=(A_\mu,i_\mu,\lambda_\mu,r_\mu)\in\cA(\mu)$.

Let $K\subseteq\bG(\bA_F^\infty)$ be a level subgroup; and we now write $X_K$ for $\widetilde\Sh(\bV)_K$ as in Subsection \ref{ss:albanese_unitary}.

\begin{description}
  \item[Step 1] We start from the cycle
      \[
      \Delta^3X_K\times D_K^{n-1}\in\CH^{3(n-1)}(X_K\times X_K\times X_K\times X_K)_\dQ,
      \]
      where we recall that $D_K$ is the Hodge divisor on $X_K$ (Definition \ref{de:hodge_divisor}(1)). Put
      \[
      (\Delta^3X_K\times D_K^{n-1})^\nabla\coloneqq\Delta^3X_K\times D_K^{n-1}\cap X_K\times X_K\times\nabla X_K
      \]
      as an element in $\CH^{3(n-1)}(X_K\times X_K\times\nabla X_K)_\dQ$ (see Definition \ref{de:split} for the meaning of $\nabla$).

  \item[Step 2] Choose an element $\phi\in\Hom_E(A_K,A_\mu)$. We push the above cycle along the morphism
      \[
      \id_{X_K\times X_K}\times(\phi\circ\alpha_K)\colon X_K\times X_K\times\nabla X_K\to X_K\times X_K\times A_\mu
      \]
      to obtain a cycle
      \[
      (\id_{X_K\times X_K}\times(\phi\circ\alpha_K))_*(\Delta^3X_K\times D_K^{n-1})^\nabla\in\CH^{n-1+[M_\mu:\dQ]/2}(X_K\times X_K\times A_\mu)_\dQ,
      \]
      where we recall that $\alpha_K$ is the Albanese morphism \eqref{eq:albanese_shimura}.

  \item[Step 3] To proceed, we need to homologically trivialize the cycle in Step 2. Moreover, heuristically, the Chow group $\CH^{n-1+[M_\mu:\dQ]/2}(X_K\times X_K\times A_\mu)_\dQ^0$ should be encoded in the cohomology/motive $\rH^{2(n-1)+[M_\mu:\dQ]-1}(X_K\times X_K\times A_\mu)$. The motive we study comes from the product $\Pi_1\times\Pi_2\otimes\mu$, which appears in the cohomology $\rH^{n-1}(X_K)\otimes\rH^{n-1}(X_K)\otimes\rH^{[M_\mu:\dQ]-1}(A_\mu)$ as a direct summand of the previous cohomology by a suitable K\"{u}nneth decomposition. To make sense of it, we need to introduce certain correspondences serving as projectors to the correct piece of cohomology. For the factor $A_\mu$, we use the canonical projector $\tT_\mu^{\r{can}}$ in Definition \ref{de:mu_canonical}. For Shimura varieties, we choose test functions $f_1$ and $f_2$ in $\sH_{K,\dL}$ for $\pi_1^\infty$ and $\pi_2^\infty$ (Definition \ref{de:test_function}), respectively. Put
      \[
      \FJ(f_1,f_2;\phi)_K\coloneqq|\pi_0((X_K)_{E^\ac})|\cdot(\tT_K^{f_1}\otimes\tT_K^{f_2}\otimes\tT_\mu^\can)^*(\id_{X_K\times X_K}\times(\phi\circ\alpha_K))_*(\Delta^3X_K\times D_K^{n-1})^\nabla
      \]
      as an element in $\CH^{n-1+[M_\mu:\dQ]/2}(X_K\times X_K\times A_\mu)_\dL$.
\end{description}

For $i\in\dZ$, we denote by $\CH^i(X_K\times X_K\times A_\mu)_\dL^\natural[i_\mu]$ the subspace of $\CH^i(X_K\times X_K\times A_\mu)_\dL^\natural$ on which $i_\mu(M_\mu)$ acts via the inclusion $M_\mu\hookrightarrow\dL$.

\begin{proposition}\label{pr:independence}
Let the notation be as above.
\begin{enumerate}
  \item The cycle $\FJ(f_1,f_2;\phi)_K$ belongs to $\CH^{n-1+[M_\mu:\dQ]/2}(X_K\times X_K\times A_\mu)_\dL^0$.

  \item The image of $\FJ(f_1,f_2;\phi)_K$ in $\CH^{n-1+[M_\mu:\dQ]/2}(X_K\times X_K\times A_\mu)_\dL^\natural$ belongs to the subspace $\CH^{n-1+[M_\mu:\dQ]/2}(X_K\times X_K\times A_\mu)_\dL^\natural[i_\mu]$ and depends only on the homological equivalence class of $\tT_K^{f_1}\otimes\tT_K^{f_2}$.
\end{enumerate}
\end{proposition}

\begin{proof}
Take an embedding $\tau'\colon E\to\dC$.

For (1), we realize that the image of $\cl_{\rB,\tau'}^*(\tT_K^{f_i})$ for $i=1,2$ is contained in $\rH^{n-1}_{\rB,\tau'}(X_K,\dC)$, while, by Lemma \ref{le:mu_canonical}, the image of $\cl_{\rB,\tau'}^*(\tT_\mu^\can)$ is contained in $\bigoplus_{i\leq[M_\mu:\dQ]-1}\rH^i_{\rB,\tau'}(A_\mu,\dC)$. Thus, $\FJ(f_1,f_2;\phi)_K$ is homologically trivial.

For (2), by construction, it is clear that the image of $\FJ(f_1,f_2;\phi)_K$ belongs to $\CH^{n-1+[M_\mu:\dQ]/2}(X_K\times X_K\times A_\mu)_\dL^\natural[i_\mu]$. For the other part, we pick another pair of test functions $(f'_1,f'_2)$ such that $\tT_K^{f'_1}\otimes\tT_K^{f'_2}$ is homologically equivalent to $\tT_K^{f_1}\otimes\tT_K^{f_2}$. By (1), it suffices to show that for every rational prime $\ell$ and every isomorphism $\iota_\ell\colon\dC\xrightarrow{\sim}\dQ_\ell^\ac$, the pullbacks $(\tT_K^{f_1}\otimes\tT_K^{f_2}\otimes\tT_\mu^\can)^*$ and $(\tT_K^{f'_1}\otimes\tT_K^{f'_2}\otimes\tT_\mu^\can)^*$ induce the same map from $\CH^{n-1+[M_\mu:\dQ]/2}(X_K\times X_K\times A_\mu)_\dL$ to
\begin{align}\label{eq:independence}
\rH^1(E,\rH^{2(n-1)+[M_\mu:\dQ]-1}_{\et}((X_K\times X_K\times A_\mu)_{E^\ac},\dQ_\ell^\ac(n-1+[M_\mu:\dQ]/2)))\otimes_{\dQ_\ell^\ac,\iota_\ell^{-1}}\dC.
\end{align}
We denote the difference by $\zeta_\ell$. Again, since $\tT_K^{f_1}\otimes\tT_K^{f_2}\otimes\tT_\mu^\can$ and $\tT_K^{f'_1}\otimes\tT_K^{f'_2}\otimes\tT_\mu^\can$ are homologically equivalent, the kernel of $\zeta_\ell$ contains $\CH^{n-1+[M_\mu:\dQ]/2}(X_K\times X_K\times A_\mu)_\dL^0$. Thus, $\zeta_\ell$ induces a complex linear map from
\begin{align}\label{eq:independence1}
\CH^{n-1+[M_\mu:\dQ]/2}(X_K\times X_K\times A_\mu)_\dC/\CH^{n-1+[M_\mu:\dQ]/2}(X_K\times X_K\times A_\mu)_\dC^0
\end{align}
to \eqref{eq:independence}. We now explain that such map must be zero.

In fact, let $\Sigma$ be a finite set of places of $F$ such that for $v\not\in\Sigma$, $K_v$ is hyperspecial maximal. Let $\sH_{K,\dC}^\Sigma$ be the partial Hecke algebra away from $\Sigma$. Then $\sH_{K,\dC}^\Sigma\otimes_\dC\sH_{K,\dC}^\Sigma$ acts on both \eqref{eq:independence} and \eqref{eq:independence1} via the factor $X_K\times X_K$, under which $\zeta_\ell$ is equivariant. In other words, $\zeta_\ell$ is a map of $\sH_{K,\dC}^\Sigma\otimes_\dC\sH_{K,\dC}^\Sigma$-modules. Since $f_1$ and $f_2$ are test functions for $\pi_1^\infty$ and $\pi_2^\infty$, respectively, the image of $\zeta_\ell$ is isomorphic to a finite copy of $(\pi^{\infty,\Sigma}_1)^{K^\Sigma}\otimes_\dC(\pi^{\infty,\Sigma}_2)^{K^\Sigma}$ as an $\sH_{K,\dC}^\Sigma\otimes_\dC\sH_{K,\dC}^\Sigma$-module. Therefore, by Proposition \ref{pr:arthur}(2), $\zeta_\ell$ must factor through the image of the cycle class map from \eqref{eq:independence1} to
\[
\(\rH^{2(n-1)}_{\et}((X_K\times X_K)_{E^\ac},\dQ^\ac_\ell(n-1))
\otimes\rH^{[M_\mu:\dQ]}_{\et}((A_\mu)_{E^\ac},\dQ^\ac_\ell([M_\mu:\dQ]/2))\)\otimes_{\dQ_\ell^\ac,\iota_\ell^{-1}}\dC.
\]
However, $\zeta_\ell$ also commutes with the action of $M_\mu$ through the factor $A_\mu$ by the functoriality of $\tT_\mu^\can$. As the actions of $\dQ^\ac_\ell[M_\mu]$ on $\rH^{[M_\mu:\dQ]}_{\et}((A_\mu)_{E^\ac},\dQ^\ac_\ell)$ and $\rH^{[M_\mu:\dQ]-1}_{\et}((A_\mu)_{E^\ac},\dQ^\ac_\ell)$ have disjoint support, we conclude that $\zeta_\ell$ must be zero.
\end{proof}

\begin{definition}[Fourier--Jacobi cycles]\label{de:fjcycle}
We call $\FJ(f_1,f_2;\phi)_K$ a \emph{Fourier--Jacobi cycle} for $\Pi_1\times\Pi_2\otimes\mu$. We call the image of $\FJ(f_1,f_2;\phi)_K$ in $\CH^{n-1+[M_\mu:\dQ]/2}(X_K\times X_K\times A_\mu)_\dL^\natural[i_\mu]$, denoted by $\FJ(f_1,f_2;\phi)_K^\natural$, a \emph{natural Fourier--Jacobi cycle} for $\Pi_1\times\Pi_2\otimes\mu$.
\end{definition}

The following lemma states that Fourier--Jacobi cycles are compatible with changing level subgroups.

\begin{lem}\label{le:level_changing}
We have
\begin{enumerate}
  \item Let $K'\subseteq K$ be a smaller level subgroup. Then we have
      \[
      (u^{K'}_K\times u^{K'}_K\times\id_{A_\mu})^*\FJ(f_1,f_2;\phi)_K=\FJ(f_1,f_2;\phi)_{K'}.
      \]

  \item For $g\in\bG(\bA_F^\infty)$, we have
      \[
      (\tT_g\times\tT_g\times\id_{A_\mu})^*\FJ(f_1,f_2;\phi)_K=\FJ(\rR_g\rL_g f_1,\rR_g\rL_g f_2;g\phi)_{gKg^{-1}},
      \]
      where $\tT_g\colon X_{gKg^{-1}}\to X_K$ is the Hecke translation.
\end{enumerate}
\end{lem}

\begin{proof}
For (1), put $u\coloneqq u^{K'}_K\colon X_{K'}\to X_K$ for short. Note that by definition, we have
\[
(u\times u)^*\tT_K^{f_i}=\frac{\vol(K)}{\vol(K')}\cdot\tT_{K'}^{f_i}
\]
for $i=1,2$. Thus, for every $\alpha\in\CH^{n-1+[M_\mu:\dQ]/2}(X_{K'}\times X_{K'}\times A_\mu)_\dQ$, we have
\[
(\tT_{K'}^{f_1}\otimes\tT_{K'}^{f_2}\otimes\tT_\mu^\can)^*\alpha=
\(\frac{\vol(K')}{\vol(K)}\)^2\cdot(u\times u\times\id_{A_\mu})^*(\tT_K^{f_1}\otimes\tT_K^{f_2}\otimes\tT_\mu^\can)^*(u\times u\times\id_{A_\mu})_*\alpha
\]
by a standard computation of correspondences. Therefore, it suffices to show that
\begin{align*}
&(u\times u\times\id_{A_\mu})_*(\id_{X_{K'}\times X_{K'}}\times(\phi\circ\alpha_{K'}))_*(\Delta^3X_{K'}\times D_{K'}^{n-1})^\nabla\\
&=\frac{\vol(K)}{\vol(K')}\cdot\frac{\deg D_{K'}^{n-1}}{\deg D_K^{n-1}}\cdot(\id_{X_K\times X_K}\times(\phi\circ\alpha_K))_*(\Delta^3X_K\times D_K^{n-1})^\nabla.
\end{align*}
This is an easy consequence of the equality $\alpha_K\circ\nabla u=\alpha_{K'}$. Thus, (1) follows.

Part (2) follows from the same argument for (1), together with the relations $\tT_g^*f_i=\rR_g\rL_gf_i$ for $i=1,2$, $\phi\circ\Alb_{\tT_g}=g\phi$, and $|\pi_0((X_K)_{E^\ac})|=|\pi_0((X_{gKg^{-1}})_{E^\ac})|$.
\end{proof}

For $i\in\dZ$, put
\[
\CH^i(X_\infty\times X_\infty\times A_\mu)_\dL^?\coloneqq\varinjlim_K\CH^i(X_K\times X_K\times A_\mu)_\dL^?
\]
for $?=0,\natural$. The above lemma implies that we have well-defined elements
\begin{align*}
\FJ(f_1,f_2;\phi)&\in\CH^{n-1+[M_\mu:\dQ]/2}(X_\infty\times X_\infty\times A_\mu)_\dL^0,\\
\FJ(f_1,f_2;\phi)^\natural&\in\CH^{n-1+[M_\mu:\dQ]/2}(X_\infty\times X_\infty\times A_\mu)_\dL^\natural[i_\mu].
\end{align*}

\begin{lem}\label{le:fj_function}
For every elements $g,g_1,g_2\in\bG(\bA_F^\infty)$, we have
\begin{align*}
\FJ(\rR_{g_1} f_1,\rR_{g_2} f_2;\phi)&=(\tT_{g_1}\times\tT_{g_2}\times\id_{A_\mu})^*\FJ(f_1,f_2;\phi),\\
\FJ(\rL_g f_1,\rL_g f_2;g\phi)&=\FJ(f_1,f_2;\phi),
\end{align*}
where $\tT_g\colon X_\infty\to X_\infty$ denotes the Hecke translation by $g$.
\end{lem}

\begin{proof}
The first equality is obvious. For the second one, we have
\[
\FJ(\rR_g\rL_g f_1,\rR_g\rL_g f_2;g\phi)=(\tT_g\times\tT_g\times\id_{A_\mu})^*\FJ(\rL_g f_1,\rL_g f_2;g\phi)
\]
from the first one. Thus,
\begin{align*}
\FJ(\rL_g f_1,\rL_g f_2;g\phi)&=(\tT_{g^{-1}}\times\tT_{g^{-1}}\times\id_{A_\mu})^*\FJ(\rR_g\rL_g f_1,\rR_g\rL_g f_2;g\phi)\\
&=(\tT_{g^{-1}}\times\tT_{g^{-1}}\times\id_{A_\mu})^*(\tT_g\times\tT_g\times\id_{A_\mu})^*\FJ(f_1,f_2;\phi)\\
&=\FJ(f_1,f_2;\phi),
\end{align*}
in which the second equality is due to Lemma \ref{le:level_changing}(2).
\end{proof}

\subsection{Arithmetic Gan--Gross--Prasad conjecture}
\label{ss:aggp}

We first summarize the construction of the natural Fourier--Jacobi cycles in a more functorial way. Let $\Pi_1$, $\Pi_2$, $(\bV,\pi_1^\infty)$, $(\bV,\pi_2^\infty)$, and $\dL$ be as in the previous subsection.

Similar to Definition \ref{de:cm_space}, for $i\in\dZ$, we put
\begin{align}\label{eq:cm_chow}
\CH^i_\mu(X_\infty\times X_\infty)^\natural_\dL\coloneqq\varinjlim_{D_\mu=(A_\mu,i_\mu,\lambda_\mu,r_\mu)\in\cA(\mu)}
\CH^i(X_\infty\times X_\infty\times A_\mu)_\dL^\natural[i_\mu]
\end{align}
in the category of $\dL[\bG(\bA_F^\infty)\times\bG(\bA_F^\infty)]$-modules. It follows from Proposition \ref{pr:cm_data}(1) that the canonical map $\CH^i(X_\infty\times X_\infty\times A_\mu)_\dL^\natural[i_\mu]\to\CH^i_\mu(X_\infty\times X_\infty)^\natural_\dL$ from \eqref{eq:cm_chow} is an isomorphism for every object $D_\mu\in\cA(\mu)$, similar to Remark \ref{re:cm_space}.

Then it is clear that the assignment $(f_1,f_2,\phi)\mapsto\FJ(f_1,f_2,\phi)^\natural$ defines an $\dL$-linear map
\begin{align*}
\FJ^\natural\colon\sH_\dL\otimes_\dL\sH_\dL\otimes_{M_\mu}\Omega(\mu)\to\CH^{n-1+[M_\mu:\dQ]/2}_\mu(X_\infty\times X_\infty)_\dL^\natural,
\end{align*}
which is independent of the choice of $D_\mu\in\cA(\mu)$.

The Hecke actions induce canonical surjective maps
\begin{align}\label{eq:matrix}
\sH_\dL\to\pi_i^\infty\otimes_\dL(\pi_i^\infty)^\vee
\end{align}
of $\dL[\bG(\bA_F^\infty)\times\bG(\bA_F^\infty)]$-modules for $i=1,2$. Proposition \ref{pr:independence}(2) implies that $\FJ^\natural$ factors through the quotient
\[
\(\pi_1^\infty\otimes_\dL(\pi_1^\infty)^\vee\)\otimes_\dL\(\pi_2^\infty\otimes_\dL(\pi_2^\infty)^\vee\)\otimes_{M_\mu}\Omega(\mu).
\]
Together with Lemma \ref{le:fj_function}, we conclude that $\FJ^\natural$ is actually an $\dL$-linear map
\begin{align*}
\FJ^\natural&\colon\pi_1^\infty\otimes_\dL\pi_2^\infty\otimes_{M_\mu}\Omega(\mu)
\to\Hom_{\dL[\bG(\bA_F^\infty)\times\bG(\bA_F^\infty)]}
\((\pi_1^\infty)^\vee\otimes_\dL(\pi_2^\infty)^\vee,\CH^{n-1+[M_\mu:\dQ]/2}_\mu(X_\infty\times X_\infty)_\dL^\natural\),
\end{align*}
which is invariant under the diagonal action of $\bG(\bA_F^\infty)$ on the left-hand side. For every $\mu$-admissible collection $\varepsilon$ (Definition \ref{de:admissible_collection}), we denote by
\begin{align*}
\FJ^\natural_\varepsilon&\colon\pi_1^\infty\otimes_\dL\pi_2^\infty\otimes_{M_\mu}\Omega(\mu,\varepsilon)
\to\Hom_{\dL[\bG(\bA_F^\infty)\times\bG(\bA_F^\infty)]}
\((\pi_1^\infty)^\vee\otimes_\dL(\pi_2^\infty)^\vee,\CH^{n-1+[M_\mu:\dQ]/2}_\mu(X_\infty\times X_\infty)_\dL^\natural\)
\end{align*}
the restriction of $\FJ^\natural$ to $\pi_1^\infty\otimes_\dL\pi_2^\infty\otimes_{M_\mu}\Omega(\mu,\varepsilon)$ (Definition \ref{de:cm_albanese}).

\begin{conjecture}[Unrefined arithmetic Gan--Gross--Prasad conjecture for $\rU(n)\times\rU(n)$]\label{co:ggp_unrefined}
Let $\Pi_1$ and $\Pi_2$ be two relevant representations of $\GL_n(\bA_E)$ (Definition \ref{de:relevant}). Let $\mu\colon E^\times\backslash\bA_E^\times\to\dC^\times$ be a conjugate symplectic automorphic character of weight one, and $\varepsilon$ a $\mu$-admissible collection. Let $\dL\subseteq\dC$ be a subfield containing $M_\mu$ over which both $\Pi_1^\infty$ and $\Pi_2^\infty$ are defined. For pairs $(\bV,\pi_1^\infty)\in\Phi_{\Pi_1}$ and $(\bV,\pi_2^\infty)\in\Phi_{\Pi_2}$, the following three statements are equivalent:
\begin{enumerate}[label=(\alph*)]
  \item We have $\FJ^\natural_\varepsilon\neq 0$.

  \item We have $\FJ^\natural_\varepsilon\neq 0$, and
      \[
      \dim_\dL\Hom_{\dL[\bG(\bA_F^\infty)\times\bG(\bA_F^\infty)]}
      \((\pi_1^\infty)^\vee\otimes_\dL(\pi_2^\infty)^\vee,\CH^{n-1+[M_\mu:\dQ]/2}_\mu(X_\infty\times X_\infty)_\dL^\natural\)=1.
      \]

  \item We have $L'(\tfrac{1}{2},\Pi_1\times\Pi_2\otimes\mu)\neq 0$, and
      \[
      \Hom_{\dL[\bG(\bA_F^\infty)]}(\pi_1^\infty\otimes_\dL\pi_2^\infty\otimes_{M_\mu}\Omega(\mu,\varepsilon),\dL)\neq\{0\}.
      \]
\end{enumerate}
\end{conjecture}

\begin{remark}\label{re:ggp_unrefined}
We have the following remarks concerning Conjecture \ref{co:ggp_unrefined}.
\begin{enumerate}

  \item The equivalence between (a) and (b) can be regarded as a generalization of Kolyvagin's theorem for Heegner points.

  \item The assertion $\FJ^\natural_\varepsilon\neq 0$ immediately implies $\Hom_{\dL[\bG(\bA_F^\infty)]}(\pi_1^\infty\otimes_\dL\pi_2^\infty\otimes_{M_\mu}\Omega(\mu,\varepsilon),\dL)\neq\{0\}$.

  \item By the multiplicity one part of the local Gan--Gross--Prasad conjecture, which is proved in \cite{Sun12} for our particular Fourier--Jacobi model, we know that
      \[
      \dim_\dL\Hom_{\dL[\bG(\bA_F^\infty)]}(\pi_1^\infty\otimes_\dL\pi_2^\infty\otimes_{M_\mu}\Omega(\mu,\varepsilon),\dL)\leq 1.
      \]

  \item By the (refined) local Gan--Gross--Prasad conjecture, which is proved in \cite{GI16} for our particular Fourier--Jacobi model, we know that if
      \begin{align}\label{eq:local_functional}
      \dim_\dL\Hom_{\dL[\bG(\bA_F^\infty)]}(\pi_1^\infty\otimes_\dL\pi_2^\infty\otimes_{M_\mu}\Omega(\mu,\varepsilon),\dL)=1
      \end{align}
      from some $\mu$-admissible collection $\varepsilon$, then the global root number of $\Pi_1\times\Pi_2\otimes\mu$ is $-1$, that is, $L(s,\Pi_1\times\Pi_2\otimes\mu)$ has odd vanishing order at the center $s=\tfrac{1}{2}$. Moreover, we have
      \begin{itemize}
        \item If $n$ is even, then the triple $(\bV,\pi_1^\infty,\pi_2^\infty)$ is uniquely determined; but $\varepsilon$ could be an arbitrary $\mu$-admissible collection.

        \item If $n$ is odd, then $\bV$ could be arbitrary; but once $\bV$ is chosen, $\pi_1^\infty$, $\pi_2^\infty$, and $\varepsilon$ are uniquely determined.
      \end{itemize}
      In other words, in both cases, once $\varepsilon$ is given, the triple $(\bV,\pi_1^\infty,\pi_2^\infty)$ is uniquely determined.
\end{enumerate}
\end{remark}

\if false

We choose a prime $\ell$ such that both $X_K$ and $A_\mu$ have proper smooth reduction at places of $E$ over $\ell$, and an isomorphism $\iota_\ell\colon\dC\xrightarrow{\sim}\dQ_\ell^\ac$. The refined version is a conjectural equality relating the Beilinson--Bloch--Poincar\'{e} height pairing $\langle\;,\;\rangle_{X_K\times X_K,A_\mu}^{\r{BBP},\iota_\ell}$ introduced in Section \ref{ss:bbp} and the central derivative $L'(\tfrac{1}{2},\Pi_1\times\Pi_2\otimes\mu)$.

\begin{lem}\label{le:purity}
The Fourier--Jacobi cycle $\FJ(f_1,f_2;\phi)_K$ in Definition \ref{de:fjcycle} belongs to the subspace $\CH^{n-1+[M_\mu:\dQ]/2}(X_K\times X_K\times A_\mu)_\dL^{0\ell}$.
\end{lem}

\fi

Now we state a refined version of the arithmetic Gan--Gross--Prasad conjecture for $\rU(n)\times\rU(n)$. We assume that all height pairings are defined. Take a level subgroup $K\subseteq\bG(\bA_F^\infty)$. For every object $D_\mu=(A_\mu,i_\mu,\lambda_\mu,r_\mu)\in\cA(\mu)$ and every $\mu$-admissible collection $\varepsilon$, put
\begin{align*}
\Hom_E(A_K,A_\mu,\varepsilon)&\coloneqq\Hom_E(A_K,A_\mu)\cap\Omega(\mu,\varepsilon);\\
\Hom_E(A_K,A_\mu^\vee,-\varepsilon)&\coloneqq\Hom_E(A_K,A^\vee_\mu)\cap\Omega(\mu^\tc,-\varepsilon).
\end{align*}

\begin{conjecture}[Refined arithmetic Gan--Gross--Prasad conjecture for $\rU(n)\times\rU(n)$]\label{co:ggp_refined}
Let the setup be as in Conjecture \ref{co:ggp_unrefined}. Moreover, let $K\subseteq\bG(\bA_F^\infty)$ be a level subgroup, and $D_\mu=(A_\mu,i_\mu,\lambda_\mu,r_\mu)\in\cA(\mu)$ a CM data for $\mu$ (Definition \ref{de:cm_data}). For every test functions $f_1,f_1^\vee,f_2,f_2^\vee\in\sH_{K,\dL}$ for $\pi_1^\infty$, $(\pi_1^\infty)^\vee$, $\pi_2^\infty$, $(\pi_2^\infty)^\vee$, respectively, and every elements $\phi\in\Hom_E(A_K,A_\mu,\varepsilon)$ and $\phi_\tc\in\Hom_E(A_K,A_\mu^\vee,-\varepsilon)$, the equality
\begin{align}\label{eq:ggp_refined}
&\vol(K)^2\cdot\langle\FJ(f_1,f_2;\phi)_K,\FJ(f_1^\vee,f_2^\vee;\phi_\tc)_K\rangle_{X_K\times X_K,A_\mu}^{\r{BBP}}\\
&=\frac{\prod_{i=1}^nL(i,\mu_{E/F}^i)}{2^{s(\Pi_1)+s(\Pi_2)}}\cdot\frac{L'(\tfrac{1}{2},\Pi_1\times\Pi_2\otimes\mu)}
{L(1,\Pi_1,\As^{(-1)^n})\cdot L(1,\Pi_2,\As^{(-1)^n})}\cdot\beta(f_1,f_1^\vee,f_2,f_2^\vee,\phi,\phi_\tc) \notag
\end{align}
holds. Here, $s(\Pi_i)$ has appeared in Definition \ref{de:relevant}; $\r{As}^\pm$ stand for the two Asai representations (see, for example, \cite{GGP}*{Section~7}); and $\beta$ is a certain normalized matrix coefficient integral defined immediately below.
\end{conjecture}

For $i=1,2$, we have $\dL$-linear maps
\[
\sH_\dL\to\pi_i^\infty\otimes_\dL(\pi_i^\infty)^\vee\to\dL(\subseteq\dC),
\]
in which the first is \eqref{eq:matrix} and the second is the evaluation map. For every $f\in\sH_\dL$ and $g\in\bG(\bA_F^\infty)$, we denote by $\r{ev}(\pi_i^\infty(g),f)$ the image of $\rL_gf$ under the above composite map. In particular, the assignment $g\mapsto\r{ev}(\pi_i^\infty(g),f)$ is a matrix coefficient of $\pi_i^\infty$.

Consider a finite set $\Sigma$ of nonarchimedean places of $F$ such that $K_v$ is hyperspecial maximal for $v\not\in\Sigma$. Let $\rd_\Sigma$ be the unique Haar measure on $\bG(F_\Sigma)$ under which the volume of $K_\Sigma$ equals $2\vol(K)$. For $f_1,f_1^\vee,f_2,f_2^\vee\in\sH_\dL$, $\phi\in\Omega(\mu,\varepsilon)$ and $\phi_\tc\in\Omega(\mu^\tc,-\varepsilon)$, we define
\begin{align*}
\beta_\Sigma(f_1,f_1^\vee,f_2,f_2^\vee,\phi,\phi_\tc)&\coloneqq
\(\prod_{v\in\Sigma\cup\Phi_F}\frac{\prod_{i=1}^nL(i,\mu_{E/F,v}^i)\cdot L(\tfrac{1}{2},\Pi_{1,v}\times\Pi_{2,v}\otimes\mu_v)}
{L(1,\Pi_{1,v},\As^{(-1)^n})\cdot L(1,\Pi_{2,v},\As^{(-1)^n})}\)^{-1}\\
&\qquad\int_{\bG(F_\Sigma)}\r{ev}(\pi_1^\infty(g),f_1^\rt\ast f_1^\vee)\cdot\r{ev}(\pi_2^\infty(g),f_2^\rt\ast f_2^\vee)\cdot
(g\phi,\phi_\tc)_\mu\cdot\rd_\Sigma g,
\end{align*}
in which
\begin{itemize}
  \item $f_i^\rt$ is the transpose of $f_i$, that is, $f_i^\rt(g)=f_i(g^{-1})$.

  \item $f_i^\rt\ast f_i^\vee$ denotes the convolution product in $\sH_{K,\dL}$,

  \item $(\;,\;)_\mu$ is the pairing \eqref{eq:albanese_pair}.
\end{itemize}
By \cite{Xue16}*{Proposition~1.1.1(1,3)}, the value of $\beta_\Sigma(f_1,f_1^\vee,f_2,f_2^\vee,\phi,\phi_\tc)$ is finite and stabilizes when $\Sigma$ is large enough; and we denote the stable value by $\beta(f_1,f_1^\vee,f_2,f_2^\vee,\phi,\phi_\tc)$.

\begin{remark}\label{re:ggp_refined}
We have the following remarks concerning Conjecture \ref{co:ggp_unrefined}.
\begin{enumerate}
  \item The left-hand side of \eqref{eq:ggp_refined} is independent of $K$. More precisely, if we take a smaller level subgroup $K'$ contained in $K$, then the left-hand side of \eqref{eq:ggp_refined} is equal to
      \[
      \vol(K')^2\cdot\langle\FJ(f_1,f_2;\phi)_{K'},\FJ(f_1^\vee,f_2^\vee;\phi_\tc)_{K'}\rangle_{X_{K'}\times X_{K'},A_\mu}^{\r{BBP}}
      \]
      by the projection formula.

  \item The refined Gan--Gross--Prasad conjecture for the central value formula in this case is formulated by Hang~Xue \cite{Xue16}*{Conjecture~1.1.2}.

  \item It is known by \cite{Xue16}*{Proposition~1.1.1(2)} that \eqref{eq:local_functional} holds if and only if $\beta$ is nonvanishing as a functional.
\end{enumerate}
\end{remark}

At the end of this subsection, we state a variant of Conjecture \ref{co:ggp_refined}. The following definition (with slightly different terminology) is taken from \cite{RSZ}.

\begin{definition}\label{de:projector_hodge}
We say that a collection of correspondences $z=(z_K\in\CH^{n-1}(X_K\times X_K)_\dQ)_K$ is a \emph{Hecke system of projectors} if
\begin{enumerate}
  \item $z_K$ is an odd projector (Definition \ref{de:projector}) for every $K$,

  \item we have $(\r{id}_{X_{K'}}\times u^{K'}_K)_*z_{K'}=(u^{K'}_K\times\r{id}_{X_K})^*z_K\in\CH^{n-1}(X_{K'}\times X_K)_\dQ$ for every transition morphism $u^{K'}_K\colon X_{K'}\to X_K$,

  \item for every $g\in\bG(\bA_F^\infty)$, we have $\tT_g^*z_K=z_{gKg^{-1}}$ where $\tT_g\colon X_{gKg^{-1}}\to X_K$ is the Hecke translation.
\end{enumerate}
\end{definition}

\begin{remark}
We have the following remarks concerning the existence of Hecke system of projectors.
\begin{enumerate}
  \item If $n=2$, then $z=(z_{X_K,D_K})_K$ constructed in Lemma \ref{le:easy_projector}(1) is a Hecke system of projectors by Lemma \ref{le:almost_ample}.

  \item If $n=3$, then $z=(z_{X_K,D_K})_K$ constructed in Lemma \ref{le:easy_projector}(2) is a Hecke system of projectors by Lemma \ref{le:almost_ample} and Proposition \ref{pr:picard_functoriality}(2).

  \item If $n\geq 4$ and $F\neq\dQ$, then odd projectors exist by \cite{MS}*{Theorem~1.3}. Note that since we consider trivial coefficients, there is no need to require the Shimura data to be of PEL type in that theorem; see \cite{MS}*{Remark~2.7}.

  \item If $n\geq 4$ and $F=\dQ$, then one probably needs to use projectors for intersection cohomology; see \cite{MS}*{Theorem~1.4}.
\end{enumerate}
\end{remark}

Now take a Hecke system of projectors $z=(z_K)_K$. We will use $z$ to modify Step 1 in the construction of $\FJ(f_1,f_2;\phi)_K$. Namely, we consider
\[
\Delta^3_zX_K\coloneqq\pr^{[3]}_{z_K}\Delta^3X_K\in\CH^{2(n-1)}(X_K\times X_K\times X_K)_\dQ^0,
\]
where $\pr^{[3]}_{z_K}$ is defined in Definition \ref{de:triple_product}. Then we replace $\Delta^3X_K$ by $\Delta^3_zX_K$ in every later step, and denote the final outcome by
\[
\FJ(f_1,f_2;\phi)_K^z\in\CH^{n-1+[M_\mu:\dQ]/2}(X_K\times X_K\times A_\mu)_\dL^0.
\]

\begin{conjecture}[Refined arithmetic Gan--Gross--Prasad conjecture for $\rU(n)\times\rU(n)$, variant]\label{co:ggp_refined_bis}
Let the setup be as in Conjecture \ref{co:ggp_refined}. Take a Hecke system of projectors $z=(z_K)_K$. Then the equality
\begin{align}\label{eq:ggp_refined_bis}
&\vol(K)^2\cdot\langle\FJ(f_1,f_2;\phi)_K^z,\FJ(f_1^\vee,f_2^\vee;\phi_\tc)_K^z\rangle_{X_K\times X_K,A_\mu}^{\r{BBP}}\\
&=\frac{\prod_{i=1}^nL(i,\mu_{E/F}^i)}{2^{s(\Pi_1)+s(\Pi_2)}}\cdot\frac{L'(\tfrac{1}{2},\Pi_1\times\Pi_2\otimes\mu)}
{L(1,\Pi_1,\As^{(-1)^n})\cdot L(1,\Pi_2,\As^{(-1)^n})}\cdot\beta(f_1,f_1^\vee,f_2,f_2^\vee,\phi,\phi_\tc) \notag
\end{align}
holds.
\end{conjecture}

\begin{remark}\label{re:ggp_refined_bis}
We have the following remarks concerning Conjecture \ref{co:ggp_refined_bis}.
\begin{enumerate}
  \item We have a similar statement for $\FJ(f_1,f_2;\phi)_K^z$ as in Lemma \ref{le:level_changing}. In particular, the left-hand side of \eqref{eq:ggp_refined_bis} is independent of $K$.

  \item By a similar argument for Proposition \ref{pr:independence}(2), one can show that the image of $\FJ(f_1,f_2;\phi)_K^z$ in $\CH^{n-1+[M_\mu:\dQ]/2}(X_K\times X_K\times A_\mu)_\dL^\natural$ equals $\FJ(f_1,f_2;\phi)_K^\natural$. This is why we expect the variant conjecture to hold as well, in view of Remark \ref{re:beilinson}.

  \item One of the advantages of introducing the auxiliary projector $z$ is that one can show that the left-hand side of \eqref{eq:ggp_refined_bis}, regarded as a functional in $(\phi,\phi_\tc)$, factors through the map
      \[
      \Hom_E(A_K,A_\mu,\varepsilon)\times\Hom_E(A_K,A_\mu^\vee,-\varepsilon)\to
      \Omega(\mu,\varepsilon)\otimes_{M_\mu}\Omega(\mu^\tc,-\varepsilon)
      \]
      and becomes $M_\mu$-linear. See Remark \ref{re:bilinear}.
\end{enumerate}
\end{remark}

\section{Arithmetic relative trace formula}
\label{ss:4}

In this section, we discuss a relative trace formula approach toward the arithmetic GGP conjecture for $\rU(n)\times\rU(n)$. In Subsection \ref{ss:doubling}, we prove the doubling formula for CM data. In Subsection \ref{ss:invariant}, we introduce the global arithmetic invariant functional and its local version at good inert primes for which we perform some preliminary computation. In Subsection \ref{ss:orbital}, we prove the formula for the orbital decomposition of the local arithmetic invariant functional.

We keep the notation from Section \ref{ss:3}. We fix a conjugate symplectic automorphic character $\mu\colon E^\times\backslash\bA_E^\times\to\dC^\times$ of weight one, and a $\mu$-admissible collection $\varepsilon$ (Definition \ref{de:admissible_collection}).

From now on, we will restrict ourselves to the Compact Case. We will identify $E$ as a \emph{subfield} of $\dC$ via a fixed complex embedding $\tau'\in\Phi_\mu$. Put $\tau\coloneqq\tau'\res F$, and fix a hermitian space $\rV$ that is $\tau$-nearby to $\bV$ (Definition \ref{de:nearby}). In particular, $\rV$ is anisotropic. Put $\rG\coloneqq\Res_{F/\dQ}\rU(\rV)$, and identify $X_K$ with the (proper) Shimura variety $\Sh(\rG,\rh_{\rV,\tau'})_K$ under the notation in Remark \ref{re:picard}.

\subsection{A doubling formula for CM data}
\label{ss:doubling}

We start by performing some preliminary computation of the Beilinson--Bloch--Poincar\'{e} height pairing
\begin{align}\label{eq:bbp}
\vol(K)^2\cdot\langle\FJ(f_1,f_2;\phi)_K^z,\FJ(f_1^\vee,f_2^\vee;\phi_\tc)_K^z\rangle_{X_K\times X_K,A_\mu}^{\r{BBP}}
\end{align}
for a level subgroup $K\subseteq\rG(\bA^\infty)$ and a CM data $D_\mu=(A_\mu,i_\mu,\lambda_\mu,r_\mu)\in\cA(\mu)$, as in Conjecture \ref{co:ggp_refined_bis}.

Consider an intermediate number field $E\subseteq E'\subseteq \dC$ such that $E'$ splits $X_K$ (Definition \ref{de:split}). Put $X'_K\coloneqq (X_K)_{E'}$, $A'_K\coloneqq(A_K)_{E'}$, $A'_\mu\coloneqq(A_\mu)_{E'}$, $\phi'\coloneqq\phi_{E'}$, and $\phi'_\tc\coloneqq(\phi_\tc)_{E'}$. We will suppress $E'$ in the fiber product $X\times_{E'}Y$ of schemes if $X$ and $Y$ are obviously over $E'$. For every element $P\in X_K(\pi_0(X'_K))$, we have the induced morphism
\[
\alpha_P\coloneqq(\alpha_K)_P\colon X'_K\to A'_K
\]
from Definition \ref{de:albanese} and Definition \ref{de:split}. We put
\begin{align*}
\Delta^{\phi,P}_zX_K&\coloneqq(\id_{X'_K\times X'_K}\times(\phi'\circ\alpha_P))_*\Delta^3_zX_K
\in\CH^{n-1+[M_\mu:\dQ]/2}(X'_K\times X'_K\times A'_\mu)_\dQ^0,\\
\Delta^{\phi_\tc,P}_zX_K&\coloneqq(\id_{X'_K\times X'_K}\times(\phi'_\tc\circ\alpha_P))_*\Delta^3_zX_K
\in\CH^{n-1+[M_\mu:\dQ]/2}(X'_K\times X'_K\times A'^\vee_\mu)_\dQ^0.
\end{align*}

\begin{lem}\label{le:doubling0}
Suppose that $E'$ is sufficiently large such that $D_K^{n-1}$ can be represented by a finite sum $\sum_i c_i P_i$ with $c_i\in\dQ$ and $P_i\in X_K(\pi_0(X'_K))$. Then we have
\begin{align*}
\eqref{eq:bbp}&=\frac{1}{[E':E](\deg D_K^{n-1})^2} \\
&\quad\sum_{i,j}c_ic_j\cdot\langle(\tT_K^{f_1}\otimes\tT_K^{f_2}\otimes\tT_\mu^\can)^*\Delta^{\phi,P_i}_zX_K,
(\tT_K^{f^\vee_1}\otimes\tT_K^{f^\vee_2}\otimes\tT_{\mu^\tc}^\can)^*\Delta^{\phi_\tc,P_j}_zX_K
\rangle_{X'_K\times X'_K,A'_\mu}^{\r{BBP}}.
\end{align*}
\end{lem}

\begin{proof}
This follows immediately from the definition of $\FJ(f_1^\vee,f_2^\vee;\phi_\tc)_K^z$ and Definition \ref{de:hodge_divisor}(2).
\end{proof}

Let $\cP_\mu\in\CH^1(A_\mu\times A^\vee_\mu)$ be the Poincar\'{e} class on $A_\mu\times A^\vee_\mu$. Put
\[
\cQ_\mu\coloneqq (\tT_\mu^{\can,\rt}\otimes\tT_{\mu^\tc}^\can)^*\cP_\mu\in\CH^1(A_\mu\times A^\vee_\mu)_{M_\mu},
\]
and for $P,Q\in X_K(\pi_0(X'_K))$, put
\[
\cQ_{\mu,K}^{\phi,\phi_\tc,P,Q}\coloneqq(\phi'\circ\alpha_P\times\phi'_\tc\circ\alpha_Q)^*\cQ_\mu\in\CH^1(X'_K\times X'_K)_{M_\mu}.
\]

\begin{lem}\label{le:doubling1}
For $P,Q\in X_K(\pi_0(X'_K))$, we have
\begin{align*}
&\langle(\tT_K^{f_1}\otimes\tT_K^{f_2}\otimes\tT_\mu^\can)^*\Delta^{\phi,P}_zX_K,
(\tT_K^{f^\vee_1}\otimes\tT_K^{f^\vee_2}\otimes\tT_{\mu^\tc}^\can)^*\Delta^{\phi_\tc,Q}_zX_K
\rangle_{X'_K\times X'_K,A'_\mu}^{\r{BBP}}\\
&=\langle \tp_{123}^*\Delta^3_zX_K,
(X'_K\times X'_K\times\cQ_{\mu,K}^{\phi,\phi_\tc,P,Q}).\tp_{124}^*\Delta^3_{z,f_1^\rt\ast f_1^\vee,f_2^\rt\ast f_2^\vee}X_K\rangle_{X'_K\times X'_K\times X'_K\times X'_K}^{\r{BB}}
\end{align*}
where $\Delta^3_{z,\bbf_1,\bbf_2}X_K\coloneqq(\tT_K^{\bbf_1}\otimes\tT_K^{\bbf_2}\otimes\id_{X_K})^*\Delta^3_zX_K\in\CH^{2(n-1)}(X_K\times X_K\times X_K)^0_\dL$ for $\bbf_1,\bbf_2\in\sH_{K,\dL}$.
\end{lem}

\begin{proof}
Consider the following commutative diagram of in the category $\Sch_{/E'}$
\begin{align*}
\resizebox{17cm}{!}{\xymatrix{
&& X'_K\times X'_K\times X'_K\times X'_K \ar[ld]_-{\gamma_2}\ar[rd]^-{\gamma_1}\ar[dd]^-\delta \\
& X'_K\times X'_K\times X'_K\times A'^\vee_\mu \ar[ld]_-{r_1}\ar[rd]^-{\beta_1}
&& X'_K\times X'_K\times A'_\mu\times X'_K \ar[ld]_-{\beta_2}\ar[rd]^-{r_2} \\
X'_K\times X'_K\times X'_K \ar[rd]^-{\alpha_1} && X'_K\times X'_K\times A'_\mu\times A'^\vee_\mu \ar[ld]_-{q_1}\ar[rd]^-{q_2}
&& X'_K\times X'_K\times X'_K \ar[ld]_-{\alpha_2} \\
& X'_K\times X'_K\times A'_\mu && X'_K\times X'_K\times A'^\vee_\mu
}}
\end{align*}
in which all diamonds are Cartesian, and $\alpha_1\coloneqq\id_{X'_K\times X'_K}\times(\phi'\circ\alpha_P)$, $\alpha_2\coloneqq\id_{X'_K\times X'_K}\times(\phi'_\tc\circ\alpha_Q)$, $q_1\coloneqq \tp_{123}$, $q_2\coloneqq \tp_{124}$.

Put $\cP'_\mu\coloneqq(X'_K\times X'_K)\times\cP_\mu$ and $\cQ'_\mu\coloneqq(X'_K\times X'_K)\times\cQ_\mu$. By the definition of the Beilinson--Bloch--Poincar\'{e} height pairing, we have
\begin{align}\label{eq:doubling1}
&\langle(\tT_K^{f_1}\otimes\tT_K^{f_2}\otimes\tT_\mu^\can)^*\Delta^{\phi,P}_zX_K,
(\tT_K^{f^\vee_1}\otimes\tT_K^{f^\vee_2}\otimes\tT_{\mu^\tc}^\can)^*\Delta^{\phi_\tc,Q}_zX_K
\rangle_{X'_K\times X'_K,A'_\mu}^{\r{BBP}}\notag\\
&=\langle(\tT_K^{f_1}\otimes\tT_K^{f_2}\otimes\tT_\mu^\can)^*\Delta^{\phi,P}_zX_K,
q_{1*}(\cP'_\mu.q_2^*(\tT_K^{f^\vee_1}\otimes\tT_K^{f^\vee_2}\otimes\tT_{\mu^\tc}^\can)^*\Delta^{\phi_\tc,Q}_zX_K)
\rangle_{X'_K\times X'_K\times A'_\mu}^{\r{BB}}\notag\\
&=\langle q_1^*(\tT_K^{f_1}\otimes\tT_K^{f_2}\otimes\tT_\mu^\can)^*\Delta^{\phi,P}_zX_K,
\cP'_\mu.q_2^*(\tT_K^{f^\vee_1}\otimes\tT_K^{f^\vee_2}\otimes\tT_{\mu^\tc}^\can)^*\Delta^{\phi_\tc,Q}_zX_K
\rangle_{X'_K\times X'_K\times A'_\mu\times A'^\vee_\mu}^{\r{BB}}
\end{align}
where we have used \cite{Bei87}*{4.0.3} for the last equality. Note that we have
\begin{align*}
q_1^*(\tT_K^{f_1}\otimes\tT_K^{f_2}\otimes\tT_\mu^\can)^*\Delta^{\phi,P}_zX_K
&=(\tT_K^{f_1}\otimes\tT_K^{f_2}\otimes\tT_\mu^\can\otimes\id_{A'^\vee_\mu})^*q_1^*\Delta^{\phi,P}_zX_K \\
&=(\tT_K^{f_1}\otimes\tT_K^{f_2}\otimes\tT_\mu^\can\otimes\id_{A'^\vee_\mu})^*q_1^*\alpha_{1*}\Delta^3_zX_K \\
&=(\tT_K^{f_1}\otimes\tT_K^{f_2}\otimes\tT_\mu^\can\otimes\id_{A'^\vee_\mu})^*\beta_{1*}r_1^*\Delta^3_zX_K,
\end{align*}
and similarly
\[
q_2^*(\tT_K^{f^\vee_1}\otimes\tT_K^{f^\vee_2}\otimes\tT_{\mu^\tc}^\can)^*\Delta^{\phi_\tc,Q}_zX_K
=(\tT_K^{f^\vee_1}\otimes\tT_K^{f^\vee_2}\otimes\id_{A'_\mu}\otimes\tT_{\mu^\tc}^\can)^*\beta_{2*}r_2^*\Delta^3_zX_K.
\]
Then it follows that
\begin{align}\label{eq:doubling2}
\eqref{eq:doubling1}&=\langle(\tT_\mu^\can)^*\beta_{1*}r_1^*\Delta^3_{z,f_1,f_2}X_K,
\cP'_\mu.(\tT_{\mu^\tc}^\can)^*\beta_{2*}r_2^*\Delta^3_{z,f_1^\vee,f_2^\vee}X_K\rangle_{X'_K\times X'_K\times A'_\mu\times A'^\vee_\mu}^{\r{BB}}
\end{align}
where we have suppressed the expression $\id_?$ in the notation of correspondences as it is clear which factor the correspondence acts on. Using \cite{Bei87}*{4.0.3} again, we further have
\begin{align*}
\eqref{eq:doubling2}&=\langle\beta_{1*}r_1^*\Delta^3_{z,f_1,f_2}X_K,
\cQ'_\mu.\beta_{2*}r_2^*\Delta^3_{z,f_1^\vee,f_2^\vee}X_K\rangle_{X'_K\times X'_K\times A'_\mu\times A'^\vee_\mu}^{\r{BB}} \\
&=\langle\gamma_2^*r_1^*\Delta^3_{z,f_1,f_2}X_K,
\delta^*\cQ'_\mu.\gamma_1^*r_2^*\Delta^3_{z,f_1^\vee,f_2^\vee}X_K\rangle_{X'_K\times X'_K\times X'_K\times X'_K}^{\r{BB}} \\
&=\langle\gamma_2^*r_1^*\Delta^3_zX_K,
\delta^*\cQ'_\mu.\gamma_1^*r_2^*\Delta^3_{z,f_1^\rt\ast f_1^\vee,f_2^\rt\ast f_2^\vee}X_K\rangle_{X'_K\times X'_K\times X'_K\times X'_K}^{\r{BB}}\\
&=\langle \tp_{123}^*\Delta^3_zX_K,
\delta^*\cQ'_\mu.\tp_{124}^*\Delta^3_{z,f_1^\rt\ast f_1^\vee,f_2^\rt\ast f_2^\vee}X_K\rangle_{X'_K\times X'_K\times X'_K\times X'_K}^{\r{BB}}.
\end{align*}
The lemma follows by noting that $\delta=\beta_2\circ\gamma_1=\beta_1\circ\gamma_2=\id_{X'_K\times X'_K}\times(\phi'\circ\alpha_P)\times(\phi'_\tc\circ\alpha_Q)$.
\end{proof}

Lemma \ref{le:doubling1} suggests us to compute the class $\cQ_{\mu,K}^{\phi,\phi_\tc,P,Q}\in\CH^1(X'_K\times X'_K)_{M_\mu}$, or rather its homological equivalence class by Lemma \ref{le:divisor_coh}. To do this, we first review some doubling construction and Kudla's generating series of special divisors which was introduced by Kudla \cite{Kud97} in the context of orthogonal Shimura varieties.

\begin{definition}\label{de:generating_series}
Let $\rV(E)^+\subseteq\rV(E)$ be the subset consisting of $x$ such that $(x,x)_\rV$ is totally positive. Take $x\in\rV(E)^+$, and denote its orthogonal complement in $\rV$ by $\rV^x$. For $g\in\rG(\bA^\infty)$, we have the composite morphism
\begin{align*}
\ts_{x,g}\colon\Sh(\rG^x,\rh_{\rV^x,\tau'})_{gKg^{-1}\cap\rG^x(\bA^\infty)}\to\Sh(\rG,\rh_{\rV,\tau'})_{gKg^{-1}}=X_{gKg^{-1}}
\xrightarrow{\tT_g}X_K,
\end{align*}
where $\rG^x\coloneqq\Res_{F/\dQ}\rU(\rV^x)$, and the first arrow is induced by the inclusion $\rV^x\subseteq\rV$ of hermitian subspaces. The morphism $\ts_{x,g}$ is finite and unramified. We define
\[
Z(x,g)_K\coloneqq(\ts_{x,g})_*\Sh(\rG^x,\rh_{\rV^x,\tau'})_{gKg^{-1}\cap\rG^x(\bA^\infty)}
\]
as an element in $\rZ^1(X_K)$.

We denote by $\sS(\rV(\bA_E^\infty))$ the space of complex valued Schwartz functions on $\rV(\bA_E^\infty)$, which admits an action by $\rG(\bA^\infty)$ via the variable. For every $\bphi\in\sS(\rV(\bA_E^\infty))$, we define \emph{the generating series of special divisors} attached to $\bphi$ (of level $K$) to be
\[
Z(\bphi)_K\coloneqq-\bphi(0)D_K+\sum_{x\in\rU(\rV)(F)\backslash\rV(E)^+}\re^{-2\pi\cdot\Tr_{F/\dQ}(x,x)_\rV}
\sum_{g\in\rG^x(\bA^\infty)\backslash\rG(\bA^\infty)/K}\bphi(g^{-1}x)Z(x,g)_K
\]
as a formal series in $\rZ^1(X_K)_\dC$, where $D_K$ is (some representative of) the Hodge divisor (Definition \ref{de:hodge_divisor}).
\end{definition}

\begin{lem}
The generating series of special divisors $Z(\bphi)_K$ is Chow convergent, that is, an element in $\CZ^1(X_K)$ (Definition \ref{de:chow_convergent}).
\end{lem}

\begin{proof}
This is \cite{Liu11}*{Theorem~3.5(2)} (with $g=1$), together with the fact that $\CH^1(X_K)_\dC$ is of finite dimension.
\end{proof}

We study the relation between generating series of special divisors and the spaces $\Omega(\mu,\varepsilon)$ and $\Omega(\mu^\tc,-\varepsilon)$. Choose a nonzero element $\alpha$ (resp.\ $\alpha_\tc$) in $\rH^0(A_\mu(\dC),\Omega^1)$ (resp.\ $\rH^0(A_\mu^\vee(\dC),\Omega^1)$) on which $M_\mu$ acts via the inclusion $M_\mu\hookrightarrow\dC$, such that under the canonical pairing $\rH^1_\rB(A_\mu,\dC)\times\rH^1_\rB(A_\mu^\vee,\dC)\to\dC$, $\alpha$ and $\ol{\alpha_\tc}$ pair to one. It is clear that for $\phi\in\Omega(\mu,\varepsilon)$ and $\phi_\tc\in\Omega(\mu^\tc,-\varepsilon)$, the $(1,1)$-form
\[
\phi\diamond\phi_\tc\coloneqq\phi^*\alpha\wedge\phi_\tc^*\ol{\alpha_\tc}
\]
on $X_K(\dC)$ does not depend on the choice of the pair $(\alpha,\alpha_\tc)$, which is moreover in $\rH^2_\rB(X_K,M_\mu(1))$. By \cite{BMM}*{Proposition~5.19} and \cite{Liu14}*{Lemma~5.3}, $\phi\diamond\phi_\tc$ is a Kudla--Milson form which, in the notation of \cite{BMM}*{(8.8)}, equals $\theta_{\psi_F,\mu,\tilde\bphi}(-,1)$, where
\[
\tilde\bphi=\varphi_{1,1}\otimes\(\bigotimes_{\Phi_F\setminus\{\tau\}}\varphi_0\)\otimes\bphi
\]
for a unique $\bphi\in\sS(\rV(\bA_E^\infty))$ as in \cite{BMM}*{(8.9)}. The assignment $(\phi,\phi_\rc)\mapsto\bphi$ gives rise to a map \begin{align}\label{eq:doubling0}
\fd\colon\Omega(\mu,\varepsilon)\otimes_{M_\mu}\Omega(\mu^\tc,-\varepsilon)\otimes_{M_\mu}\dC
\to\sS(\rV(\bA_E^\infty)).
\end{align}

\begin{lem}\label{le:doubling2}
The map $\fd$ \eqref{eq:doubling0} is an isomorphism of $\dC[\rG(\bA^\infty)]$-modules.
\end{lem}

\begin{proof}
For $g\in\rG(\bA^\infty)$, we have $g\phi\diamond g\phi_\tc=\tT_g^*(\phi\diamond\phi_\tc)=\tT_g^*\theta_{\psi_F,\mu,\tilde\bphi}(-,1)=
\theta_{\psi_F,\mu,\tT_g^*\tilde\bphi}(-,1)$. On the other hand, we have $\tT_g^*\tilde\bphi=\varphi_{1,1}\otimes\(\bigotimes_{\Phi_F\setminus\{\tau\}}\varphi_0\)\otimes(g.\bphi)$. Thus, $\fd$ is $\rG(\bA^\infty)$-equivariant. The map $\fd$ is apparently injective, so is surjective by \cite{Liu14}*{Lemma~5.3}. The lemma follows.
\end{proof}

\begin{lem}\label{le:doubling3}
We have
\begin{enumerate}
  \item The cohomology class $\cl_\rB(\cQ_{\mu,K}^{\phi,\phi_\tc,P,Q})\in\rH^2_\rB(X_K\times X_K,\dC)$ depends only on $\fd(\phi\otimes\phi_\tc)$.

  \item There is a unique $\dC$-linear map
    \begin{align*}
    \rc\cQ_{\mu,K}\colon\sS(\rV(\bA_E^\infty))^K&\to\rH^2_\rB(X_K\times X_K,\dC)\\
    \bphi&\mapsto\rc\cQ_{\mu,K}^{\bphi}
    \end{align*}
    such that
    \begin{enumerate}
      \item $\rc\cQ_{\mu,K}^{\fd(\phi\otimes\phi_\tc)}=\cl_\rB(\cQ_{\mu,K}^{\phi,\phi_\tc,P,Q})$ for every pair $(\phi,\phi_\tc)\in\Hom_E(A_K,A_\mu,\varepsilon)\times\Hom_E(A_K,A_\mu^\vee,-\varepsilon)$ and every $P,Q\in X_K(\pi_0(X'_K))$;

      \item $\Delta^*\rc\cQ_{\mu,K}^{\bphi}=\cl_\rB(Z(\bphi)_K)\in\rH^2_\rB(X_K,\dC)$ for every $\bphi\in\sS(\rV(\bA_E^\infty))^K$.
    \end{enumerate}
\end{enumerate}
\end{lem}

\begin{proof}
The class $\cl_\rB(\cQ_{\mu,K}^{\phi,\phi_\tc,P,Q})\in\rH^2_\rB(X_K\times X_K,\dC)$ is given by the $(1,1)$-form
$\tp_1^*\phi^*\alpha\wedge\tp_2^*\phi_\tc^*\ol{\alpha_\tc}$ on $X_K(\dC)\times X_K(\dC)$. Part (1) follows immediately.

For (2), by (1) and Lemma \ref{le:doubling2}, there is a unique $\dC$-linear map $\rc\cQ_{\mu,K}$ satisfying (a). However, it also satisfies (b) due to \cite{BMM}*{Proposition~8.3}.
\end{proof}

\begin{remark}
The maps $\{\rc\cQ_{\mu,K}\}_K$ in Lemma \ref{le:doubling3} are clearly compatible under pullbacks, hence induce a $\dC$-linear map
$\rc\cQ_\mu\colon\sS(\rV(\bA_E^\infty))\to\rH^2_\rB(X_\infty\times X_\infty,\dC)\coloneqq\varinjlim_K\rH^2_\rB(X_K\times X_K,\dC)$.
\end{remark}

\begin{definition}\label{de:doubling}
We say that an element $Z_K\in\rZ^1(X_K\times X_K)_\dC$ is a \emph{doubling divisor (of level $K$)} for an element $\bphi\in\sS(\rV(\bA_E^\infty))^K$ if $\cl_\rB(Z_K)=\rc\cQ_{\mu,K}^{\bphi}$, and $Z_K$ has proper intersection with $\Delta X_K$.
\end{definition}

\begin{lem}\label{le:doubling4}
For every element $\bphi\in\sS(\rV(\bA_E^\infty))^K$, there exists a doubling divisor of level $K$.
\end{lem}

\begin{proof}
By linearity, it suffices to consider the case $\bphi=\fd(\phi\otimes\phi_\tc)$ for $(\phi,\phi_\tc)\in\Hom_E(A_K,A_\mu,\varepsilon)\times\Hom_E(A_K,A_\mu^\vee,-\varepsilon)$. Take an intermediate number field $E\subseteq E'\subseteq \dC$ such that $E'$ is Galois over $E$, splits $X_K$, and satisfies $X_K(\pi_0(X'_E))\neq\emptyset$. We choose an element $P\in X_K(\pi_0(X'_E))$. Then
\[
Z_K\coloneqq\frac{1}{[E':E]}\sum_{\sigma\in\Gal(E'/E)}\cQ_{\mu,K}^{\phi,\phi_\tc,\sigma P,\sigma P}
\]
is an element in $\rZ^1(X_K\times X_K)_{M_\mu}$ such that $\cl_\rB(Z_K)=\rc\cQ_{\mu,K}^{\bphi}$ by Lemma \ref{le:doubling3}(2). By Chow's moving lemma, we may replace $Z_K$ by another rationally equivalent cycle that has proper intersection with $\Delta X_K$. The lemma follows.
\end{proof}

Now we can state and prove our doubling formula for CM data.

\begin{proposition}\label{pr:doubling}
Put $\bbf_i\coloneqq f_i^\rt\ast f_i^\vee$ for $i=1,2$. If we write $\bbf_1=\sum_s d_s\CF_{g_s^{-1}K\cap Kg_s^{-1}}$ as a finite sum with $d_s\in\dL$ and $g_s\in\rU(\rV)(\bA_F^\infty)$,\footnote{It is elementary to see that every element in $\sH_{K,\dL}$ can be written in this way.} then
\[
\eqref{eq:bbp}=\sum_sd_s\cdot\langle\tp_{135}^*\Delta^3_zX_{K_s},
(\Delta X_{K_s}\times\tT^{\rL_{g_s}\bbf_2}_{K_s}\times Z_{K_s}^s).\tp_{246}^*\Delta^3_zX_{K_s}\rangle_{X_{K_s}^6}^{\r{BB}}
\]
holds, where $K_s\coloneqq K\cap g_sKg_s^{-1}$, and $Z_{K_s}^s\in\rZ^1(X_K\times X_K)_\dC$ is an arbitrary doubling divisor for $\fd(\phi\otimes g_s\phi_\tc)$ (which exists by Lemma \ref{le:doubling4}).
\end{proposition}

\begin{proof}
To shorten notation, we put $\delta_K\coloneqq(\deg D_K^{n-1})^{-1}$.

By Lemma \ref{le:doubling0}, Lemma \ref{le:doubling1}, we have
\begin{align}\label{eq:bbp0}
\eqref{eq:bbp}&=\frac{\delta_K^2}{[E':E]}\sum_{i,j}c_ic_j\langle \tp_{123}^*\Delta^3_zX_K,
(X'_K\times X'_K\times\cQ_{\mu,K}^{\phi,\phi_\tc,P_i,P_j}).\tp_{124}^*\Delta^3_{z,\bbf_1,\bbf_2}X_K\rangle_{(X'_K)^4}^{\r{BB}}.
\end{align}
By Remark \ref{re:ggp_refined_bis}(1), we may replace $K$ by $K'\coloneqq\bigcap_s(g_s^{-1}Kg_s\cap K_s)$ and possibly enlarge $E'$ to obtain
\begin{align}\label{eq:bbp1}
\eqref{eq:bbp0}&=\frac{\delta_{K'}^2}{[E':E]}\sum_{i,j}c_ic_j\langle \tp_{123}^*\Delta^3_zX_{K'},
(X'_{K'}\times X'_{K'}\times\cQ_{\mu,{K'}}^{\phi,\phi_\tc,P_i,P_j}).\tp_{124}^*\Delta^3_{z,\bbf_1,\bbf_2}X_{K'}\rangle_{(X'_{K'})^4}^{\r{BB}} \notag\\
&=\frac{d_{K'}^2}{[E':E]}\sum_{i,j,s}c_ic_jd_s\langle \tp_{123}^*\Delta^3_zX_{K'},
(X'_{K'}\times X'_{K'}\times\cQ_{\mu,{K'}}^{\phi,\phi_\tc,P_i,P_j}).\tp_{124}^*\Delta^3_{z,\CF_{g_s^{-1}K\cap Kg_s^{-1}},\bbf_2}X_{K'}\rangle_{(X'_{K'})^4}^{\r{BB}}.
\end{align}
Since $\rL_{g_s}\CF_{g_s^{-1}K\cap Kg_s^{-1}}=\CF_{K\cap g_sKg_s^{-1}}=\CF_{K_s}$, by Lemma \ref{le:fj_function}, we have
\begin{align}\label{eq:bbp2}
\eqref{eq:bbp1}&=\frac{\delta_{K'}^2}{[E':E]}\sum_{i,j,s}c_i c_j d_s\langle\tp_{123}^*\Delta^3_zX_{K'},
(X'_{K'}\times X'_{K'}\times\cQ_{\mu,K'}^{\phi,g_s\phi_\tc,P_i^s,P_j^s}).
\tp_{124}^*\Delta^3_{z,\CF_{K_s},\rL_{g_s}\bbf_2}X_{K'}\rangle_{(X'_{K'})^4}^{\r{BB}} \notag\\
&=\sum_s\frac{d_s\cdot\delta_{K'}^2}{[E':E]}\sum_{i,j}c_i c_j\langle\tp_{123}^*\Delta^3_zX_{K'},
(X'_{K'}\times X'_{K'}\times\cQ_{\mu,K'}^{\phi,g_s\phi_\tc,P_i^s,P_j^s}).
\tp_{124}^*\Delta^3_{z,\CF_{K_s},\rL_{g_s}\bbf_2}X_{K'}\rangle_{(X'_{K'})^4}^{\r{BB}}.
\end{align}
For each individual $s$, we may descend the corresponding term down to $X'_{K_s}$ again by Remark \ref{re:ggp_refined_bis}(1). Choose a representative $\sum_ic_i^sP_i^s$ of $D_{K_s}^{n-1}$ with $c_i^s\in\dQ$ and $P_i^s\in X_{K_s}(\pi_0(X'_{K_s}))$. Moreover, by Lemma \ref{le:divisor_coh} and Lemma \ref{le:doubling3}, we may replace $\cQ_{\mu,K_s}^{\phi,g_s\phi_\tc,P_i^s,P_j^s}$ by $Z_{K_s}^s\otimes_EE'$. Then we have
\begin{align*}
\eqref{eq:bbp2}&=\sum_sd_s\cdot\delta_{K_s}^2\cdot\sum_{i,j}c_i^s c_j^s\langle \tp_{123}^*\Delta^3_zX_{K_s},
(X_{K_s}\times X_{K_s}\times Z_{K_s}^s).\tp_{124}^*\Delta^3_{z,\CF_{K_s},\rL_{g_s}\bbf_2}X_{K_s}\rangle_{(X_{K_s})^4}^{\r{BB}}\\
&=\sum_sd_s\langle \tp_{123}^*\Delta^3_zX_{K_s},
(X_{K_s}\times X_{K_s}\times Z_{K_s}^s).\tp_{124}^*\Delta^3_{z,\CF_{K_s},\rL_{g_s}\bbf_2}X_{K_s}\rangle_{(X_{K_s})^4}^{\r{BB}},
\end{align*}
where in the second equality, we use the fact that $\sum_ic_i^s=\deg D_{K_s}^{n-1}=\delta_{K_s}^{-1}$. The proposition then follows by \cite{Bei87}*{4.0.3}.
\end{proof}

\begin{remark}\label{re:bilinear}
Proposition \ref{pr:doubling} implies that, for given data $f_1,f_2,f_1^\vee,f_2^\vee,z$, the assignment
\begin{align*}
\Hom_E(A_K,A_\mu,\varepsilon)\times\Hom_E(A_K,A_\mu^\vee,-\varepsilon)&\to\dC\\
(\phi,\phi_\tc)&\mapsto\vol(K)^2\cdot\langle\FJ(f_1,f_2;\phi)_K^z,\FJ(f_1^\vee,f_2^\vee;\phi_\tc)_K^z\rangle_{X_K\times X_K,A_\mu}^{\r{BBP}}
\end{align*}
factors through $\Omega(\mu,\varepsilon)\otimes_{M_\mu}\Omega(\mu^\tc,-\varepsilon)$ and extends uniquely to an $M_\mu$-linear map
\[
\Omega(\mu,\varepsilon)\otimes_{M_\mu}\Omega(\mu^\tc,-\varepsilon)\to\dC
\]
by considering all level subgroups $K$.
\end{remark}

\subsection{Arithmetic invariant functionals}
\label{ss:invariant}

In view of Proposition \ref{pr:doubling}, we need to study global arithmetic invariant functionals defined as follows.

\begin{definition}[Global arithmetic invariant functional]\label{de:global_arithmetic}
Let $K\subseteq\rG(\bA^\infty)$ be a level subgroup. For test functions $\bbf\in\sH_{K,\dC}$ and $\bphi\in\sS(\rV(\bA_E^\infty))^K$, we define the \emph{global arithmetic invariant functional} to be
\[
\cI^z_K(\bbf,\bphi)\coloneqq\langle\tp_{135}^*\Delta^3_zX_K,
(\Delta X_K\times\tT^{\bbf}_K\times Z_K).\tp_{246}^*\Delta^3_zX_K\rangle_{X_K^6}^{\r{BB}},
\]
where $Z_K\in\rZ^1(X_K\times X_K)_\dC$ is an arbitrary doubling divisor for $\bphi$ (Definition \ref{de:doubling}, linearly extended to coefficients in $\dC$).
\end{definition}

We introduce two important conventions, which will be adopted from now on.
\begin{enumerate}
  \item We will regard $\tT^{\bbf}_K$ as an algebraic cycle, rather than a Chow cycle, on $X_K\times X_K$.

  \item Whenever we have two cycles $A$ and $B$ in a regular scheme $X$ that have proper intersection, $A.B$ will be regarded as the cycle $\sum_C m_C(A,B)\cdot C$ (rather than the associated Chow cycle), where the sum is taken over all irreducible components $C$ in $A\cap B$ with $m_C(A,B)$ the intersection multiplicity.
\end{enumerate}

\begin{definition}
Let $K\subseteq\rG(\bA^\infty)$ be a level subgroup. For a doubling divisor $Z_K\in\rZ^1(X_K\times X_K)_\dC$ for $\bphi\in\sS(\rV(\bA_E^\infty))^K$, we put
\[
Z_K^\heartsuit\coloneqq Z_K-\tp_2^*(\Delta X_K.Z_K-Z(\bphi)_K),
\]
where we regard $\Delta X_K.Z_K$ as in $\rZ^1(X_K)_\dC$ and recall that $\tp_2\colon X_K\times X_K\to X_K$ is the projection to the second factor.
\end{definition}

It is clear that $Z_K^\heartsuit\in\CZ^1(X_K\times X_K)$ (Definition \ref{de:chow_convergent}), $\Delta X_K.Z_K^\heartsuit=Z(\bphi)_K$, $\cl_\rB(Z_K^\heartsuit)=\cl_\rB(Z_K)$ by Lemma \ref{le:doubling3}(2), and
\begin{align}\label{eq:calibration1}
\cI^z_K(\bbf,\bphi)=\langle\tp_{135}^*\Delta^3_zX_K,
(\Delta X_K\times\tT^{\bbf}_K\times Z_K^\heartsuit).\tp_{246}^*\Delta^3_zX_K\rangle_{X_K^6}^{\r{BB}}
\end{align}
by Lemma \ref{le:divisor_coh}.

To proceed, we introduce the notation of (relative) regular semisimple elements.

\begin{definition}\label{de:regular_unitary}
Consider a field extension $F'/F$ and put $E'\coloneqq E\otimes_FF'$.
\begin{enumerate}
  \item We say that a pair of elements $(\xi,x)\in\rU(\rV)(F')\times\rV(E')$ is \emph{regular semisimple} if
  the vectors $\{\xi^i x\res i=0,\dots,n-1\}$ span the $E'$-module $\rV(E')$.

  \item The group $\rU(\rV)(F')$ acts on $\rU(\rV)(F')\times\rV(E')$ via the formula $(\xi,x)g=(g^{-1}\xi g,g^{-1}x)$, which preserves regular semisimple pairs. Denote by $[\rU(\rV)(F')\times\rV(E')]$ the orbits of $\rU(\rV)(F')\times\rV(E')$ under the above action, and by $[\rU(\rV)(F')\times\rV(E')]_{\r{rs}}$ the subset of regular semisimple orbits.

  \item We say that a function on $\rU(\rV)(F')\times\rV(E')$ is \emph{regularly supported} if its support consists of only regular semisimple pairs.

  \item We say that a function $\boldsymbol{F}$ on $\rU(\rV)(\bA_F^\infty)\times\rV(\bA_E^\infty)$ is \emph{regularly supported} at some nonarchimedean place $v$ of $F$ if we can write $\boldsymbol{F}=\boldsymbol{F}^v\otimes\boldsymbol{F}_v$ in which $\boldsymbol{F}_v$, as a function on $\rU(\rV)(F_v)\times\rV(E_v)$, is regularly supported in the sense of (3).
\end{enumerate}
\end{definition}

\begin{proposition}\label{pr:intersection}
Let $K,\bbf,\bphi,Z_K$ be as in Definition \ref{de:global_arithmetic}.
\begin{enumerate}
  \item The cycles $\Delta X_K\times\tT^{\bbf}_K\times Z_K^\heartsuit$ and $\tp_{246}^*\Delta^3X_K$ intersect properly in $X_K^6$.

  \item If $\bbf\otimes\bphi$ is regularly supported at some nonarchimedean place $v$ of $F$, then $\tp_{135}^*\Delta^3X_K$ and $(\Delta X_K\times\tT^{\bbf}_K\times Z_K^\heartsuit).\tp_{246}^*\Delta^3X_K$ have empty intersection on $X_K^6$.
\end{enumerate}
\end{proposition}

\begin{proof}
For (1), we have to show that every irreducible component $C$ of the intersection of $\Delta X_K\times\tT^{\bbf}_K\times Z_K^\heartsuit$ and $\tp_{246}^*\Delta^3X_K$ has dimension $2n-3$. However, it is easy to see that $C$ is a closed subscheme of the fiber product
\[
\Delta^3 X_K\times_{(X_K\times X_K\times X_K)}(X_K\times Y\times Z)\simeq Y\times_{X_K}Z,
\]
where $Y$ (resp. $Z$) is an irreducible component in the support of $\tT^{\bbf}_K$ (resp.\ $Z_K^\heartsuit$). But then the morphism $Y\to X_K$ is finite \'{e}tale, and $Z$ has dimension $2n-3$. Thus, $C$ has dimension at most $2n-3$. On the other hand, since $\tp_{246}^*\Delta^3X_K$ is a regular subscheme, the dimension of $C$ is at least $2n-3$.

For (2), it is clear that the statement is equivalent to that $\Delta X_K\cap\tT_K^{\bbf}\cap Z_K^\heartsuit$ is empty in $X_K\times X_K$. As $\Delta X_K\cap Z_K^\heartsuit=Z(\bphi)_K$, we have to show that $\tT_K^{\bbf}\cap Z(\bphi)_K=\emptyset$, which can be checked on $X_K(\dC)\times X_K(\dC)$. By complex uniformization, we have
\[
X_K(\dC)=\rU(\rV)(F)\backslash\(\cD\times\rU(\rV)(\bA_F^\infty)/K\),
\]
where $\cD$ is the corresponding hermitian domain of dimension $n-1$.

If $\bbf\equiv 0$, then there is nothing to prove. Otherwise, we have $\bphi(0)=0$. Thus, we need to show that for every $x\in\rV(E)^+$ and $g,h\in\rU(\rV)(\bA_F^\infty)$, if $\bbf(h)\bphi(g^{-1}x)\neq 0$, then $\tT_{KhK}\cap Z(x,g)=\emptyset$ in $X_K(\dC)\times X_K(\dC)$. We prove by contradiction. Let $\cD^x\subseteq \cD$ be the subdomain that is perpendicular to $x$. If $\tT_{KhK}\cap Z(x,g)\neq\emptyset$, then we may find $z_1\in\cD$, $g_1\in\rU(\rV)(\bA_F^\infty)$, $h\in KhK$, $z_x\in\cD^x$, $g_x\in\rU(\rV^x)(\bA_F^\infty)$, and $\xi\in\rU(\rV)(F)$, such that $(z_x,g_xg)=(z_1,g_1)$ and $(z_x,g_xg)=\xi(z_1,g_1h)$. These relations imply that $z_x=\xi z_x$ and $g_xg=\xi g_xgh$. The second equality implies that $h^ig^{-1}x=g^{-1}g_x^{-1}\xi^{-i}x$ for $i\geq 0$. Now since $z_x=\xi z_x$, the vectors $\{x,\xi^{-1}x,\dots,\xi^{-(n-1)}x\}\subseteq\rV(E)$ are linearly dependent. In particular, the pair $(h_v,g_v^{-1}x)\in\rU(\rV)(F_v)\times\rV(E_v)$ is not regular semisimple, which is a contradiction. Thus, (2) follows.
\end{proof}

We would also like to know whether one can choose a cycle representative of $z_K$ such that p$\tp_{135}^*\Delta^3_zX_K$ and $(\Delta X_K\times\tT^{\bbf}_K\times Z_K^\heartsuit).\tp_{246}^*\Delta^3_zX_K$ have empty intersection as well. At this moment, we do not find a uniform answer to this question. On the other hand, the contribution of the difference $\Delta^3X_K-\Delta^3_zX_K$ in the height pairing should be negligible in the comparison of relative trace formulae. In what follows, we will only consider $\Delta^3X_K$ in the decomposition into local heights, suggested by Proposition \ref{pr:intersection}. Moreover, in this article, we only consider the local heights at good inert primes, which we now explain.

\begin{definition}\label{de:good_inert}
We say that a prime $\fp$ of $F$ is a \emph{good inert prime} (with respect to $K,\bbf,\bphi$) if
\begin{itemize}
  \item $\fp$ is inert in $E$,

  \item the underlying rational prime $p$ is odd and unramified in $E$,

  \item if we denote by ${\underline\fp}$ the set of all primes of $F$ above $p$ that are inert in $E$, then there exists a self-dual lattice $\Lambda_\fq\subseteq\rV(F_\fq)$ for every $\fq\in\ul\fp$ such that
      \begin{itemize}
        \item $K=K^{\underline\fp}\times\prod_{\fq\in{\underline\fp}}K_\fq$ in which $K_\fq$ is the stabilizer of $\Lambda_\fq$ for every $\fq\in{\underline\fp}$,

        \item $\bbf=\bbf^{\underline\fp}\otimes\bigotimes_{\fq\in{\underline\fp}}\bbf_\fq$ in which $\bbf_\fq=\CF_{K_\fq}$,

        \item $\bphi=\bphi^{\underline\fp}\otimes\bigotimes_{\fq\in{\underline\fp}}\bphi_\fq$ in which $\bphi_\fq=\CF_{\Lambda_\fq}$.
      \end{itemize}
\end{itemize}
\end{definition}

We fix a good inert prime $\fp$. From now on, we work in the category $\Sch_{/O_{E_\fp}}$.

Let $\cX_K$ be the canonical integral model of $X_K$ over $O_{E_\fp}$ (Definition \ref{de:integral_canonical}), which is a proper smooth scheme in $\Sch_{/O_{E_\fp}}$ of relative dimension $n-1$. Then the Zariski closure of $\tT_K^{\bbf}$ in $\cX_K\times\cX_K$ is an \'{e}tale correspondence, which will be denoted by the same notation. Let $\cZ_K$ (resp.\ $\cZ(\bphi)_K$) be the Zariski closure of $Z_K$ (resp.\ $Z(\bphi)_K$) in $\cX_K\times\cX_K$ (resp.\ $\cX_K$). Similar to $Z_K^\heartsuit$, we put
\begin{align}\label{eq:calibration}
\cZ_K^\heartsuit\coloneqq\cZ_K-\tp_2^*(\Delta\cX_K.\cZ_K-\cZ(\bphi)_K),
\end{align}
which is a formal series of divisors on $\cX_K$, whose generic fiber is $Z_K^\heartsuit$.

From now on, we work in the category $\Sch_{/O_{E_\fp}}$.

\begin{definition}[Local arithmetic invariant functional]\label{de:local_arithmetic}
Let $K,\bbf,\bphi,Z_K$ be as in Definition \ref{de:global_arithmetic} such that $\bbf\otimes\bphi$ is regularly supported at some nonarchimedean place $v$ of $F$,\footnote{It is clear that $v$ can not be in $\underline\fp$.} we define the \emph{local arithmetic invariant functional} at (a good inert prime) $\fp$ to be
\[
\cI_K(\bbf,\bphi)_\fp\coloneqq 2\log|O_F/\fp|\cdot\chi\(\cO(\tp_{135}^*\Delta^3\cX_K)\otimes^\dL_{\cO_{\cX_K^6}}
\cO((\Delta\cX_K\times\tT^{\bbf}_K\times\cZ_K^\heartsuit).\tp_{246}^*\Delta^3\cX_K)\),
\]
where $\chi$ denotes the Euler--Poincar\'{e} characteristic (see Remark \ref{re:local_arithmetic} below), and for a formal series $\sum_jc_jZ_j$ of cycles on $\cX_K^6$, we put $\cO(\sum_jc_jZ_j)\coloneqq\sum_jc_j\cO_{Z_j}$ as a formal series of $\cO_{\cX_K^6}$-modules.\footnote{The reason we add the factor $2$ in front of $\log|O_F/\fp|$ is the following: $\cI_K(\bbf,\bphi)_\fp$ is supposed to ``approximate'' the local term of $\cI^z_K(\bbf,\bphi)$ at the unique place $u$ of $E$ above $\fp$, hence the factor $c(u)$ in \eqref{eq:bb_decompose} is $\log|O_E/\fp O_E|=2\log|O_F/\fp|$.}
\end{definition}

\begin{remark}\label{re:local_arithmetic}
For a Noetherian scheme $X$, we denote by $\rD^\rb_{\r{coh}}(X)$ the bounded derived category of $\cO_X$-modules with coherent cohomology. By Proposition \ref{pr:intersection}(2),
\[
\cO(\tp_{135}^*\Delta^3\cX_K)\otimes^\dL_{\cO_{\cX_K^6}}
\cO((\Delta\cX_K\times\tT^{\bbf}_K\times\cZ_K^\heartsuit).\tp_{246}^*\Delta^3\cX_K)
\]
is a formal series in $\rD^\rb_{\r{coh}}(\cX_K^6\otimes_\dZ\dF_p)$, which implies that its Euler--Poincar\'{e} characteristic is a formal series in $\dC$.
\end{remark}

\begin{proposition}\label{pr:local_invariant}
In the situation of Definition \ref{de:local_arithmetic}, we have
\[
\cI_K(\bbf,\bphi)_\fp=2\log|O_F/\fp|\cdot\chi\(\cO(\tT^{\bbf}_K)\otimes^\dL_{\cO_{\cX_K^2}}\cO(\Delta\cZ(\bphi)_K)\).
\]
\end{proposition}

\begin{proof}
First, by the same argument for Proposition \ref{pr:intersection}(1), we know that $\Delta\cX_K\times\tT^{\bbf}_K\times\cZ_K^\heartsuit$ and $\tp_{246}^*\Delta^3\cX_K$ have proper intersection on $\cX_K^6$. Since $\Delta\cX_K$, every component of $\tT^{\bbf}_K$, and $\cZ_K^\heartsuit$ are all Cohen--Macaulay schemes, we have
\begin{align*}
\cO((\Delta\cX_K\times\tT^{\bbf}_K\times\cZ_K^\heartsuit).\tp_{246}^*\Delta^3\cX_K)&=
\cO(\Delta\cX_K\times\tT^{\bbf}_K\times\cZ_K^\heartsuit)\otimes_{\cO_{\cX_K^6}}\cO(\tp_{246}^*\Delta^3\cX_K) \\
&=\cO(\Delta\cX_K\times\tT^{\bbf}_K\times\cZ_K^\heartsuit)\otimes^\dL_{\cO_{\cX_K^6}}\cO(\tp_{246}^*\Delta^3\cX_K).
\end{align*}
Thus, we have
\begin{align*}
&\cO(\tp_{135}^*\Delta^3\cX_K)\otimes^\dL_{\cO_{\cX_K^6}}
\cO((\Delta\cX_K\times\tT^{\bbf}_K\times\cZ_K^\heartsuit).\tp_{246}^*\Delta^3\cX_K)\\
&=\cO(\tp_{135}^*\Delta^3\cX_K)\otimes^\dL_{\cO_{\cX_K^6}}\cO(\Delta\cX_K\times\tT^{\bbf}_K\times\cZ_K^\heartsuit)
\otimes^\dL_{\cO_{\cX_K^6}}\cO(\tp_{246}^*\Delta^3\cX_K) \\
&=\(\cO(\tp_{135}^*\Delta^3\cX_K)\otimes^\dL_{\cO_{\cX_K^6}}\cO(\tp_{246}^*\Delta^3\cX_K)\)
\otimes^\dL_{\cO_{\cX_K^6}}\cO(\Delta\cX_K\times\tT^{\bbf}_K\times\cZ_K^\heartsuit)\\
&=\cO(\Delta^3(\cX_K\times\cX_K))\otimes^\dL_{\cO_{(\cX_K\times\cX_K)^3}}\cO(\Delta\cX_K\times\tT^{\bbf}_K\times\cZ_K^\heartsuit).
\end{align*}
Restricting to $\cX_K^2$, we have
\begin{align*}
\cI_K(\bbf,\bphi)_\fp=2\log|O_F/\fp|\cdot
\chi\(\cO(\Delta\cX_K)\otimes^\dL_{\cO_{\cX_K^2}}\cO(\tT^{\bbf}_K)\otimes^\dL_{\cO_{\cX_K^2}}\cO(\cZ_K^\heartsuit)\).
\end{align*}
By \eqref{eq:calibration}, $\Delta\cX_K$ and $\cZ_K^\heartsuit$ have proper intersection. Since both have Cohen--Macaulay components, we have
\[
\cO(\Delta\cX_K)\otimes^\dL_{\cO_{\cX_K^2}}\cO(\cZ_K^\heartsuit)\simeq
\cO(\Delta\cX_K)\otimes_{\cO_{\cX_K^2}}\cO(\cZ_K^\heartsuit)=\cO(\Delta\cX_K\cap\cZ_K^\heartsuit)=
\cO(\Delta\cZ(\bphi)_K).
\]
The proposition then follows.
\end{proof}

\subsection{Orbital decomposition of local arithmetic invariant functionals}
\label{ss:orbital}

To further study the intersection number in Proposition \ref{pr:local_invariant}, we need a certain moduli interpretation of the integral model $\cX_K$ and $\cZ(\bphi)_K$. We will follow the discussion and notation in Subsection \ref{ss:integral_models}. In particular, we denote by $\r{Spl}_p$ the set of primes of $F$ above $p$ that are split in $E$.

\begin{definition}\label{de:frame}
A \emph{frame} for the (good inert) prime $\fp$ (with the underlying rational prime $p$) contains the following
\begin{itemize}
  \item an isomorphism between the two $E$-extensions $\dC$ and $E_\fp^\ac$,

  \item a CM type $\Phi$ of $E$ containing the fixed embedding $\tau'$, such that elements in $\Phi$ inducing the same prime in $\r{Spl}_p$ induce the same prime of $E$,

  \item a rational skew-hermitian space $\bW^\infty_0$ over $\bA_E^\infty$ of rank $1$ such that $\cW(\bW^\infty_0,\Phi^\tc)$ is nonempty and that $\bW_0^\infty\otimes_{\bA^\infty}\dQ_p$ admits a self-dual lattice,

  \item a sufficiently small open compact subgroup $L_0=L_0^p\times(L_0)_p$ of $\bH_0^\infty(\bA^\infty)$ in which $(L_0)_p$ is the stabilizer of a self-dual lattice in $\bW_0^\infty\otimes_{\bA^\infty}\dQ_p$, where $\bH_0^\infty$ is the group of similitude of $\bW_0^\infty$,

  \item a point $\bbP\colon\Spec O_{E_\fp^\ur}\to\cM(\rV,\bW_0^\infty,\Phi)_{K_{\underline\fp},L_0}^\ur$ as in \eqref{eq:frame}, whose reduction is in the supersingular locus, where $E_\fp^\ur$ is the maximal unramified extension of $E_\fp$ contained in $E_\fp^\ac$.
\end{itemize}
\end{definition}

Now we take a frame. Put $k\coloneqq O_{E_\fp^\ur}\otimes_\dZ\dF_p$ and $\cX_K^\ur\coloneqq\cX_K\otimes_{O_{E_\fp}}O_{E_\fp^\ur}$. By Remark \ref{re:integral_canonical}, the point $\bbP$ provides us with a Cartesian diagram
\begin{align}\label{eq:frame1}
\xymatrix{
\cX_K^\ur \ar[r]\ar[d] & \Spec O_{E_\fp^\ur} \ar[d]^-{\bq_0\circ\bbP}\\
\cM(\rV,\bW_0^\infty,\Phi)_{K,L_0}^\ur \ar[r]^-{\bq_0}& \cM(\bW_0^\infty,\Phi^\tc)_{L_0}\otimes_{O_{E_\Phi,(p)}}O_{E_\fp^\ur}
}
\end{align}
of schemes over $O_{E_\fp^\ur}$. In particular, for every locally Noetherian scheme $S$ over $O_{E_\fp^\ur}$, the set $\cX_K^\ur(S)$ consists of equivalence classes of nonuples $(A_0,i_0,\lambda_0,\eta_0^p;A,i,\lambda,\eta^p,\eta_p^\spl)$ in which $(A_0,i_0,\lambda_0,\eta_0^p)$ is the base change of $(\bbA_0,\bbi_0,\bblambda_0,\bbeta_0^p)$ to $S$.

We introduce the moduli interpretation of integral special divisors.

\begin{definition}\label{de:special_cycle}
For $x\in\rV(E)^+$ and $g^{\underline\fp}=(g^p,g_\fq\res\fq\in\r{Spl}_p)\in\rU(\rV)({\bA_F^\infty}^{,\underline\fp})$, we define a relative functor
\[
\bs_{x,g^{\underline\fp}}\colon\cZ(x,g^{\underline\fp})_K^\ur\to\cX_K^\ur
\]
in the way that the fiber over a point $(A_0,i_0,\lambda_0,\eta_0^p;A,i,\lambda,\eta^p,\eta_p^\spl)\in\cX_K^\ur(S)$ consists of
\[
\rho\in\Hom_S((A_0,i_0),(A,i_A))\otimes_{O_F}O_{F,(\underline\fp)}
\]
such that for every geometric point $s$ of $S$,
\begin{itemize}
  \item the element $\rho_*\in\Hom_{E\otimes_\dQ\bA^{\infty,p}}(\rH^{\et}_1(A_{0s},\bA^{\infty,p}),\rH^{\et}_1(A_s,\bA^{\infty,p}))$ belongs to $\eta^p((g^p)^{-1}x)$,

  \item the element $\rho_*\in\prod_{\fq\in\r{Spl}_p}\Hom_{O_{E_{\fq^-}}}\(A_{0s}[(\fq^-)^\infty],A_s[(\fq^-)^\infty]\)\otimes_{O_{E_{\fq^-}}}E_{\fq^-}$ belongs to $\eta_p^\spl((g_\fq^{-1}x)_{\fq\in\r{Spl}_p})$.
\end{itemize}
\end{definition}

\begin{proposition}\label{pr:integral_special}
For $x\in\rV(E)^+$ and $g^{\underline\fp}\in\rU(\rV)({\bA_F^\infty}^{,\underline\fp})$, we have
\begin{enumerate}
  \item The relative morphism $\bs_{x,g^{\underline\fp}}$ is representable, finite, and unramified.

  \item There is an isomorphism
      \[
      \bs_{x,g^{\underline\fp}}\otimes_{O_{E_\fp^\ur}}E_\fp^\ur\simeq
      \coprod_{(g_\fq\res\fq\in\underline\fp),g_\fq\in\rU(\rV^x)(F_\fq)\backslash\rU(\rV)(F_\fq)/K_\fq,g_\fq^{-1}x\in\Lambda_\fq}
      \ts_{x,(g^{\underline\fp},g_\fq\res\fq\in\underline\fp)}\otimes_EE_\fp^\ur
      \]
      of relative functors over $X_K\otimes_EE_\fp^\ur$, where $\ts_{x,(g^{\underline\fp},g_\fq\res\fq\in\underline\fp)}$ is defined in Definition \ref{de:generating_series}.

  \item For every point $z\in\cZ(x,g)_K^\ur(k)$, the induced ring homomorphism $R_y\to R_z$ is surjective whose kernel is a principal ideal that is not contained in $pR_y$. Here $R_z$ (resp. $R_y$) denotes completed local ring of $\cX_K^\ur$ (resp.\ $\cZ(x,g)_K^\ur$) at $z$ (resp.\ $y\coloneqq\bs_{x,g^{\underline\fp}}(z)$).
\end{enumerate}
\end{proposition}

\begin{proof}
Part (1) follows from the same argument in the proof of \cite{KR14}*{Proposition~2.9}.

For (2), put $X_K^\ur\coloneqq X_K\otimes_{E_\fp}E_\fp^\ur$. For every point $P=(A_0,i_0,\lambda_0,\eta_0^p;A,i,\lambda,\eta^p,\eta_p^\spl)\in X_K^\ur(S)$, we will construct a functorial bijection $\bs_{x,g^{\underline\fp}}^{-1}P\xrightarrow{\sim}\coprod_g\ts_{x,g}^{-1}P$ between the fibers.

For the forward direction, take an element $\rho$ as in Definition \ref{de:special_cycle}. Let $(A_\rho,i_\rho)$ be the quotient abelian scheme $(A/\rho(A_0),i)$, which is naturally an $(E,\sig_{\rV,\Phi}-\Phi^\tc)$-abelian scheme (Definition \ref{de:abelian_data}). Denote by $\varrho\colon A\to A_\rho$ the quotient homomorphism, and define a homomorphism $\rho_0\coloneqq c\lambda_0^{-1}\circ\rho^\vee\circ\lambda\colon A\to A_0$ for some $c\in\dZ_{(p)}^\times$. Then we obtain a prime-to-$\underline\fp$ isogeny $(\varrho,\rho_0)\colon A\to A_\rho\times A_0$. Let $\lambda_\rho$ be the induced $\underline\fp$-principal polarization of $(A_\rho,i_\rho)$. Choose a representative $\eta^p$ in its $K^p$-class such that $\varrho_*\circ\eta^p((g^p)^{-1}x)$ is the zero map. We define $\eta_\rho^p$ to be the composition
\begin{align*}
\rV^x(\bA_F^{\infty,p})\hookrightarrow\rV(\bA_F^{\infty,p})&\xrightarrow{\eta^p}
\Hom_{E\otimes_\dQ\bA^{\infty,p}}(\rH^{\et}_1(A_{0s},\bA^{\infty,p}),\rH^{\et}_1(A_s,\bA^{\infty,p}))\\
&\xrightarrow{\varrho_*\circ}\Hom_{E\otimes_\dQ\bA^{\infty,p}}(\rH^{\et}_1(A_{0s},\bA^{\infty,p}),\rH^{\et}_1(A_{\rho s},\bA^{\infty,p})).
\end{align*}
Let $\eta_x^p\colon\bA_E^{\infty,p}x\to\Hom_{E\otimes_\dQ\bA^{\infty,p}}(\rH^{\et}_1(A_{0s},\bA^{\infty,p}),\rH^{\et}_1(A_{0s},\bA^{\infty,p}))$ be the homomorphism sending $x$ to $c\cdot(x,x)_\rV$. Then we have $(\eta_\rho^p\oplus \eta_x^p)\circ(g^p)^{-1}=(\varrho,\rho_0)_*\circ\eta^p$. We have a similar construction for $\eta_{p\rho}^\spl$, whose details we omit. Finally, we obtain $(A_0,i_0,\lambda_0,\eta_0^p;A_\rho,i_\rho,\lambda_\rho,\eta_\rho^p,\eta_{p\rho}^\spl)$ together with the $O_E$-linear prime-to-$\underline\fp$ isogeny $(\varrho,\rho_0)\colon A\to A_\rho\times A_0$, which provides an element in the fiber $\ts_{x,g}^{-1}P$ where $g=(g^{\underline\fp},g_\fq\res\fq\in\underline\fp)$ is (a representative in) the unique double coset in the disjoint union satisfying $g_\fq^{-1}x=\rho_*$ under any isomorphism $\Lambda_\fq\simeq\Hom_{O_{E_\fq}}(\rT_\fq A_{0s},\rT_\fq A_s)$ of hermitian lattices over $O_{E_\fq}$, where $\rT_\fq$ denotes the $\fq$-adic Tate module.

For the backward direction, take an element in the fiber $\ts_{x,g}^{-1}P$ for $g$ in the disjoint union, given by data
\[
(A_0,i_0,\lambda_0,\eta_0^p;A',i',\lambda',\eta^{p\prime},\eta_p^{\spl\prime})
\in\Sh(\rG^x,\rh_{\rV^x,\tau'})_{gKg^{-1}\cap\rG^x(\bA^\infty)}(S)
\]
together with an $O_E$-linear prime-to-$\underline\fp$ isogeny $\rho'\colon A\to A'\times A_0$ satisfying relevant properties. We just take $\rho$ as the composite homomorphism
\[
A_0\xrightarrow{\lambda_0}A_0^\vee\xrightarrow{\rho'^\vee\res A_0^\vee}A^\vee\xrightarrow{c\lambda^{-1}}A
\]
for some $c\in O_{F,(\underline\fp)}^\times$. It is straightforward to check that the above constructions are inverse to each other, hence (2) is proved.

For (3), let $I$ be the kernel of $R_y\to R_z$. To show that $I$ is principal, we follow the strategy in the proof of \cite{How15}*{Proposition~3.2.3} for the case $F=\dQ$.\footnote{Note that \cite{How15} considers all residue characteristics; while we only consider $p$ that is unramified in $E$.} Let $(A_0,i_0,\lambda_0,\eta_0^p;A,i,\lambda,\eta^p,\eta_p^\spl)$ be the universal object over $R_y$, which is equipped with the universal $O_E$-linear homomorphism $\rho\colon(A_0)_{R_z}\to A_{R_z}$. It suffices to study the obstruction to lifting $\rho$ to a homomorphism $A_{0S}\to A_S$ where $S\coloneqq R_y/\fm I$ with $\fm$ the maximal ideal of $R_y$.
Note that the Hodge exact sequence
\[
0 \to \Fil\rH^\dr_1(A_0) \to\rH^\dr_1(A_0) \to \Lie(A_0) \to 0
\]
splits into a direct sum of
\[
0 \to \Fil\rH^\dr_1(A_0)_\fq \to\rH^\dr_1(A_0)_\fq \to \Lie(A_0)_\fq \to 0
\]
indexed by primes $\fq$ of $F$ above $p$, in which $\rH^\dr_1(A_0)_\fq$ is the direct summand of $\rH^\dr_1(A_0)$ on which $O_{F,(p)}$ acts via the prime $\fq$. We have a similar splitting for $A$. Moreover $\rH^\dr_1(A_0)_\fq$ (resp.\ $\rH^\dr_1(A)_\fq$) is a free $O_{E_\fq}\otimes_{\dZ_p}R_y$-module of rank $1$ (resp.\ $n$). By the signature condition, the obstruction to lifting $\rho$ coincides with the obstruction for the canonical lifting $\tilde\rho_*\colon\rH_1^\dr(A_{0S})_\fp\to\rH_1^\dr(A_S)_\fp$ to respect Hodge filtration. The remaining argument is then same as \cite{How15}*{p.668} by taking $\bj\coloneqq \pi\otimes 1_S-1\otimes\pi_S$ for some $\pi\in O_{E_\fp}^\times\cap E_\fp^-$. Note that, $\bj.\Lie(A_k)$ is always nonzero in our case.

Finally, we show that $I$ is not contained in $pR_y$. If it is, then by (1) the image of $\bs_{x,g^{\underline\fp}}$ contains the entire connected component of $(\cX_K^\ur)_k$ at $y$. Thus, for every $k$ point $(A_0,i_0,\lambda_0,\eta_0^p;A,i,\lambda,\eta^p,\eta_p^\spl)$ in this connected component, there exists a nonzero homomorphism from $(A_0,i_0)$ to $(A,i)$. In particular, $(A,i)$ is not $\mu$-ordinary, which contradicts to the main theorem of \cite{Wed99} saying that $\mu$-ordinary locus is dense. Here, we apply \cite{Wed99} to the PEL type moduli scheme in Remark \ref{re:integral_similitude} parameterizing $(A,i,\lambda,\tilde\eta^p)$ where $\tilde\eta^p$ is an away-from-$p$ level structure induced from $\lambda_0^p$ and $\eta^p$. Thus, (3) is proved.
\end{proof}

By the above proposition, $(\bs_{x,g^{\underline\fp}})_*\cZ(x,g^{\underline\fp})_K^\ur$ is a relative divisor on $\cX_K^\ur$. In what follows, by abuse of notation, we denote the cycle $(\bs_{x,g^{\underline\fp}})_*\cZ(x,g^{\underline\fp})_K^\ur$ again by $\cZ(x,g^{\underline\fp})_K^\ur$. The following corollary is immediate.

\begin{corollary}\label{co:generating_integral}
Let $\fp$ be a good inert prime. If $\bphi(0)=0$, then we have
\[
\cZ(\bphi)_K\otimes_{O_{E_\fp}}O_{E_\fp^\ur}=\sum_{x\in\rU(\rV)(F)\backslash\rV(E)^+}\re^{-2\pi\cdot\Tr_{F/\dQ}(x,x)_\rV}
\sum_{g^{\underline\fp}\in\rU(\rV^x)({\bA_F^\infty}^{,\underline\fp})\backslash\rU(\rV)({\bA_F^\infty}^{,\underline\fp})/K^{\underline\fp}}
\bphi^{\underline\fp}((g^{\underline\fp})^{-1}x)\cZ(x,g^{\underline\fp})_K^\ur
\]
as a formal series in $\rZ^1(\cX_K^\ur)_\dC$.
\end{corollary}

\begin{proof}
By Proposition \ref{pr:integral_special}, the relative divisor $\cZ(x,g^{\underline\fp})_K^\ur$ is the Zariski closure of
\[
\sum_{(g_\fq\res\fq\in\underline\fp),g_\fq\in\rU(\rV^x)(F_\fq)\backslash\rU(\rV)(F_\fq)/K_\fq}
\prod_{\fq\in\underline\fp}\CF_{\Lambda_\fq}(g_\fq^{-1}x)\cdot Z(x,(g^{\underline\fp},g_\fq\res\fq\in\underline\fp))_K
\]
in $\cX_K^\ur$. The corollary follows since $\bphi=\bphi^{\underline\fp}\otimes\bigotimes_{\fq\in{\underline\fp}}\bphi_\fq$ in which $\bphi_\fq=\CF_{\Lambda_\fq}$.
\end{proof}

\begin{lem}\label{le:supersingular_intersection}
Let $K,\bbf,\bphi$ be as in Definition \ref{de:global_arithmetic} such that $\bbf\otimes\bphi$ is regularly supported at some nonarchimedean place $v$ of $F$. For a point $y\in\cX_K^\ur(k)$, if $(y,y)$ belongs to both $\tT^{\bbf}_K$ and the support of $\Delta\cZ(\bphi)_K$, then $y$ is supersingular (Definition \ref{de:supersingular}).
\end{lem}

\begin{proof}
By Corollary \ref{co:generating_integral}, it suffices to consider $\tT_{KhK}\cap\cZ(x,g)_K^\ur$ for some $g,h\in\rU(\rV)({\bA_F^\infty}^{,\underline\fp})$ and $x\in\rV(E)^+$ such that $(h_v,g_v^{-1}x)$ is regular semisimple for some nonarchimedean place $v\not\in\underline\fp$. We only consider the case where $v$ is not above $p$, and leave the similar case where $v\in\r{Spl}_p$ to the reader.

Let $(y,y)$ be a $k$-point in $\tT_{KhK}\cap\cZ(x,g)_K^\ur$, with $y$ as in the lemma represented by the object $(A_0,i_0,\lambda_0,\eta_0^p;A,i,\lambda,\eta^p,\eta_p^\spl)\in\cX_K^\ur(k)$. By the moduli interpretation, there is a coprime-to-$\underline\fp$ isogeny $\xi\colon A\to A$ such that $\xi_*\eta^p=\eta^p\circ h^p$, and an element $\rho\in\Hom_k((A_0,i_0),(A,i))\otimes_{O_F}O_{F,(\underline\fp)}$ such that $\rho_*\in\Hom_{E\otimes_\dQ\bA^{\infty,p}}(\rH^{\et}_1(A_0,\bA^{\infty,p}),\rH^{\et}_1(A,\bA^{\infty,p}))$ belongs to $\eta^p((g^p)^{-1}x)$. Consider the situation at $v$. We may choose a representation $\kappa_v$ in the $K_v$-class of $\eta^p_v$ such that $\rho_{*v}=\eta^p_v(g_v^{-1}x)$. Possibly replacing $h_v$ by some element in $K_vh_vK_v$, we have $(\xi^i\circ\rho)_{*v}=\eta^p_v(h_v^ig_v^{-1}x)$ for every integer $i\geq 0$. Since $(h_v,g_v^{-1}x)$ is regular semisimple, $\{(\xi^i\circ\rho)_{*v},i\geq 0\}$ generates $\Hom_{O_{E_v}}(\rT_vA_0,\rT_vA)_\dQ$ as an $E_v$-module where $\rT_v$ denotes the $v$-adic Tate module. In particular, $\Hom_k((A_0,i_0),(A,i))_\dQ$ has dimension $n$ over $E$. Thus, $A[\fp^\infty]$ is isogenous to $A_0[\fp^\infty]^{\oplus n}$, hence is supersingular.
\end{proof}

In Subsection \ref{ss:integral_models}, we define the supersingular locus $\cM(\rV,\bW_0^\infty,\Phi)_{K_{\underline\fp},L_0}^\ssl$. For $K$ as above, we define $\cM(\rV,\bW_0^\infty,\Phi)_{K,L_0}^\ssl$ to be the image of $\cM(\rV,\bW_0^\infty,\Phi)_{K_{\underline\fp},L_0}^\ssl$ under the natural quotient morphism $\cM(\rV,\bW_0^\infty,\Phi)_{K_{\underline\fp},L_0}\to\cM(\rV,\bW_0^\infty,\Phi)_{K,L_0}$. Define $\cX_K^\ssl$ to be the preimage of $\cM(\rV,\bW_0^\infty,\Phi)_{K,L_0}^\ssl$ under the left vertical morphism in the diagram \eqref{eq:frame1}, which is a Zariski closed subset of $\cX_K^\ur\otimes_{O_{E_\fp^\ur}}k$. Finally, let $\cX_K^{\ssl,\wedge}$ be the completion of $\cX_K^\ur$ along $\cX_K^\ssl$. Proposition \ref{pr:uniformization} provides us with the following uniformization isomorphism
\begin{align}\label{eq:uniformization0}
\cX_K^{\ssl,\wedge}\simeq\rU(\bar\rV)(F)\backslash\(\cN\times\rU(\bar\rV)({\bA_F^\infty}^{,\fp})/\bar{K}^\fp\)
\end{align}
depending on the frame we chose,\footnote{However, one can show that the supersingular locus $\cX_K^\ssl$ itself is intrinsic, which does not depend on the choice of the frame.} in particular, the point $\bbP$. Here, $\bar{K}^\fp=\bar{K}^{\underline\fp}\times\prod_{\fq\in\underline\fp\setminus\{\fp\}}\bar{K}_\fq$, where $\bar{K}^{\underline\fp}=K^{\underline\fp}$ under the isomorphism $\iota_{\bbP}$p and $\bar{K}_\fq$ is the stabilizer of $\bar\Lambda_\fq$ in Lemma \ref{le:definite_nearby}(6). The uniformization isomorphism is functorial in $K^{\underline\fp}$ and under Hecke translations. We recall the new hermitian space
\[
\bar\rV\coloneqq\Hom_k((\bbA_{0k},\bbi_{0k}),(\bbA_k,\bbi_k))_\dQ
\]
equipped with the hermitian form \eqref{eq:nearby_form}, satisfying Lemma \ref{le:definite_nearby}, which is ``$\fp$-nearby to $\bV$''. In particular, we have an isomorphism
\[
\bar\rV\otimes_FF_\fp\simeq\Hom_k((\bbX_{0k},\bbi_{0k}),(\bbX_k,\bbi_k))_\dQ.
\]
Applying the constructions from Subsection \ref{ss:afl},\footnote{Comparing the notations with those in Subsection \ref{ss:afl}, we have $(\bbX_k,\bbi_k,\bblambda_k)=(\bbX_n,\bbi_n,\bblambda_n)$, $\cN=\cN_n$, and $\bar\rV\otimes_FF_\fp=\rV^-_n$.} we have for every nonzero $\bar{x}\in\bar\rV(E_\fp)$, a sub-formal scheme $\cZ(\bar{x})$ of $\cN$; and for every $\bar{g}\in\rU(\bar\rV)(F_\fp)$, an isomorphism $\bar{g}\colon\cN\to\cN$ with its graph $\Gamma_{\bar{g}}\subseteq\cN^2\coloneqq\cN\times_{O_{E_\fp^\ur}^\wedge}\cN$. Now we arrive at the theorem on the orbital decomposition.

\begin{theorem}\label{th:orbital}
Let $K,\bbf,\bphi$ be as in Definition \ref{de:global_arithmetic} such that $\bbf\otimes\bphi$ is regularly supported at some nonarchimedean place $v$ of $F$. For a good inert prime $\fp$, we have
\[
\cI_K(\bbf,\bphi)_\fp=2\log|O_F/\fp|\cdot\sum_{(\bar\xi,\bar{x})\in[\rU(\bar\rV)(F)\times\bar\rV(E))]_{\r{rs}}}
\re^{-2\pi\cdot\Tr_{F\dQ}(\bar{x},\bar{x})_{\bar\rV}}
\Orb(\bar\bbf^\fp,\bar\bphi^\fp;\bar\xi,\bar{x})\cdot\chi\(\cO_{\Gamma_{\bar\xi}}\otimes^\dL_{\cO_{\cN^2}}\cO_{\Delta\cZ(\bar{x})}\)
\]
after choosing a frame (Definition \ref{de:frame}). Here, we define the orbital integral as
\[
\Orb(\bar\bbf^\fp,\bar\bphi^\fp;\bar\xi,\bar{x})\coloneqq\int_{\rU(\bar\rV)({\bA_F^\infty}^{,\fp})}
\bar\bbf^\fp(\bar{g}^{-1}\bar\xi\bar{g})\bar\bphi^\fp(\bar{g}^{-1}\bar{x})\;\rd\bar{g},
\]
where
\begin{itemize}
  \item $\bar\bbf^\fp=\bar\bbf^{\underline\fp}\otimes\bigotimes_{\fq\in\underline\fp\setminus\{\fp\}}\bar\bbf_\fq$ in which $\bar\bbf^{\underline\fp}=\bbf^{\underline\fp}$ under the isomorphism $\iota_{\bbP}$ \eqref{eq:nearby_definite}, and $\bar\bbf_\fq=\CF_{\bar{K}_\fq}$,

  \item $\bar\bphi^\fp=\bar\bphi^{\underline\fp}\otimes\bigotimes_{\fq\in\underline\fp\setminus\{\fp\}}\bar\bphi_\fq$ in which $\bar\bphi^{\underline\fp}=\bphi^{\underline\fp}$ under the isomorphism $\iota_{\bbP}$, and $\bar\bphi_\fq=\CF_{\bar\Lambda_\fq}$,

  \item $\rd\bar{g}$ is the Haar measure on $\rU(\bar\rV)({\bA_F^\infty}^{,\fp})$ such that $\bar{K}^\fp$ has volume $\vol(K)$.
\end{itemize}
In particular, the Euler--Poincar\'{e} characteristic appearing in the formula is finite for every regular semisimple pair $(\bar\xi,\bar{x})$.
\end{theorem}

\begin{proof}
We choose a representative $x\in\rV(E)^+$ in the coset $\rU(\rV)(F)\backslash\rV(E)^+$. We first compute
\begin{align}\label{eq:orbital1}
\sum_{g^{\underline\fp}=(g^p,g_\fq\res\fq\in\r{Spl}_p)
\in\rU(\rV^x)({\bA_F^\infty}^{,\underline\fp})\backslash\rU(\rV)({\bA_F^\infty}^{,\underline\fp})/K^{\underline\fp}}
\bphi^{\underline\fp}((g^{\underline\fp})^{-1}x)\cZ(x,g^{\underline\fp})_K^{\ssl,\wedge},
\end{align}
where $\cZ(x,g^{\underline\fp})_K^{\ssl,\wedge}$ is the formal completion of $\cZ(x,g^{\underline\fp})_K^\ur$ along the supersingular locus. Let $S$ be a connected scheme in $\Sch'_{/O_{E_\fp}^\ur}$ on which $p$ is locally nilpotent, and take a point $(A_0,i_0,\lambda_0,\eta_0^p;A,i,\lambda,\eta^p,\eta_p^\spl)\in \cZ(x,g^{\underline\fp})_K^{\ssl,\wedge}(S)$. Then we can choose an element $\bar{x}\in\bar\rV$ and an $E$-linear quasi-isogeny $\rho\colon A\to\bbA_k\times_k S$ over $S$ such that
\begin{itemize}
  \item the image of $\rho^{-1}_*\circ\bar{x}_*$ in $\Hom_{E\otimes_\dQ\bA^{\infty,p}}(\rH^{\et}_1(A_{0s},\bA^{\infty,p}),\rH^{\et}_1(A_s,\bA^{\infty,p}))$ belongs to $\eta^p((g^p)^{-1}x)$,

  \item the image of $\rho^{-1}_*\circ\bar{x}_*$ in $\prod_{\fq\in\r{Spl}_p}\Hom_{O_{E_{\fq^-}}}\(A_{0s}[(\fq^-)^\infty],A_s[(\fq^-)^\infty]\)\otimes_{O_{E_{\fq^-}}}E_{\fq^-}$ belongs to $\eta_p^\spl((g_\fq^{-1}x)_{\fq\in\r{Spl}_p})$,

  \item $\rho^{-1}\circ\bar{x}$ lifts to an $O_E$-linear homomorphism $A_0[\fq^\infty]\to A[\fq^\infty]$ for every $\fq\in\underline\fp$.
\end{itemize}
Here, we note that $(A_0,i_0,\lambda_0,\eta_0^p)$ is identified with the base change of $(\bbA_0,\bbi_0,\bblambda_0,\bbeta_0^p)$ to $S$. By Proposition \ref{pr:uniformization}, $\rho$ is given by an element $\bar{g}^\fp\in\rU(\bar\rV)({\bA_F^\infty}^{,\fp})$ on $S$. In particular, we have $(\bar{x},\bar{x})_{\bar\rV}=(x,x)_\rV$. Choose a representative $\bar{x}$ in the coset $\rU(\bar\rV)(F)\backslash\bar\rV(E)$ of this norm. Then under the isomorphism \eqref{eq:uniformization0}, we have
\begin{align*}
\eqref{eq:orbital1}=\sum_{\bar{g}^\fp\in\rU(\bar\rV^{\bar{x}})({\bA_F^\infty}^{,\fp})\backslash\rU(\bar\rV)({\bA_F^\infty}^{,\fp})/\bar{K}^\fp}
\bar\bphi^\fp((\bar{g}^\fp)^{-1}\bar{x})\cdot [\cZ(\bar{x}),\bar{g}^\fp],
\end{align*}
where $[\cZ(\bar{x}),\bar{g}^\fp]$ denotes the corresponding double coset in the right-hand side of \eqref{eq:uniformization0}.

By linearity, we may assume $\bbf=\CF_{KhK}$ for some $h\in\rU(\rV)(\bA_F^\infty)$ with $h_{\underline\fp}=1$. In particular, $\tT^{\bbf}_K=\vol(K)\tT_{KhK}$. By Proposition \ref{pr:uniformization}, the formal completion of $\tT_{KhK}$ in $(\cX_K^{\ssl,\wedge})^2$ is simply the set-theoretical Hecke correspondence $\tT_{\bar{K}^\fp\bar{h}^\fp\bar{K}^\fp}$ under the isomorphism \eqref{eq:uniformization0} by Proposition \ref{pr:uniformization}, where $\bar{h}^\fp_\fq=1$ for $\fq\in\underline\fp\setminus\{\fp\}$. We first analyze the intersection $\tT_{\bar{K}^\fp\bar{h}^\fp\bar{K}^\fp}\cap\Delta[\cZ(\bar{x}),\bar{g}^\fp]$. If the intersection is nonempty, then $[\cZ(\bar{x}),\bar{g}^\fp\bar{h}^\fp]$ and $[\cZ(\bar{x}),\bar{g}^\fp]$ are in the same connected component.
By \eqref{eq:uniformization0}, there exists $\bar\xi\in\rU(\bar\rV)(F)$ such that $\bar\xi\bar{g}^\fp\bar{K}^\fp=\bar{g}^\fp\bar{h}^\fp\bar{K}^\fp$, that is, $\CF_{KhK}((\bar{g}^\fp)^{-1}\bar\xi\bar{g}^\fp)=1$. Moreover, if we fix a set of representatives of the orbits of $\rU(\bar\rV)(F)$ under conjugation, then one can always choose $\bar\xi$ to be one of the representatives. Now we think conversely, for any such representative $\bar\xi$, the cosets $\bar{g}^\fp\bar{K}^\fp$ satisfying $\bar\xi\bar{g}^\fp\bar{K}^\fp=\bar{g}^\fp\bar{h}^\fp\bar{K}^\fp$ are those satisfying $\CF_{KhK}((\bar{g}^\fp)^{-1}\bar\xi\bar{g}^\fp)=1$. In this case, the intersection $\tT_{\bar{K}^\fp\bar{h}^\fp\bar{K}^\fp}\cap\Delta[\cZ(\bar{x}),\bar{g}^\fp]$ is isomorphic to the image of $\Gamma_{\bar\xi}\cap\Delta\cZ(\bar{x})$ under the quotient morphism $\cN^2\to(C\backslash\cN)^2$ for some subgroup $C\subseteq\rU(\bar\rV)(F_\fp)$ acting on $\cN$ discretely.

Now we claim that $\Gamma_{\bar\xi}\cap\Delta\cZ(\bar{x})$ is a proper scheme in $\Sch_{/k}$. By definition, we have $\bar\xi\cZ(\bar{x})=\cZ(\bar\xi\bar{x})$. It follows that $\Gamma_{\bar\xi}\cap\Delta\cZ(\bar{x})$ is isomorphic to a closed sub-formal scheme of $\bigcap_{i=0}^{n-1}\cZ(\bar\xi^i\bar{x})$, whose underlying reduced scheme is a proper scheme in $\Sch_{/k}$ by \cite{KR11}*{Theorem~4.12} for $F_\fp=\dQ_p$ and \cite{Cho} in general. Thus, the underlying reduced scheme of $\Gamma_{\bar\xi}\cap\Delta\cZ(\bar{x})$ is of finite type over $k$. By the previous discussion, it suffices to show that $\tT_{\bar{K}^\fp\bar{h}^\fp\bar{K}^\fp}\cap\Delta[\cZ(\bar{x}),\bar{g}^\fp]$ is a scheme of finite type over $k$. However, this follows from Lemma \ref{le:supersingular_intersection}. As a consequence, $\chi\(\cO_{\Gamma_{\bar\xi}}\otimes^\dL_{\cO_{\cN^2}}\cO_{\Delta\cZ(\bar{x})}\)$ is finite. Moreover, it is equal to $\chi\(\cO_{\tT_{\bar{K}^\fp\bar{h}^\fp\bar{K}^\fp}}\otimes^\dL_{\cO_{\cN^2}}\cO_{\Delta[\cZ(\bar{x}),\bar{g}^\fp]}\)$. Therefore, the theorem follows from \eqref{eq:orbital1} and Lemma \ref{le:supersingular_intersection}.
\end{proof}

\begin{remark}
We believe that a more general notion of \emph{good inert prime}, for which a result similar to Theorem \ref{th:orbital} holds, should just be a prime $\fp$ of $F$ that is inert in $E$, and such that there is a self-dual lattice $\Lambda_\fp\subseteq\rV(F_\fp)$ satisfying
\begin{itemize}
  \item $K=K^\fp\times K_\fp$ in which $K_\fp$ is the stabilizer of $\Lambda_\fp$,

  \item $\bbf=\bbf^\fp\otimes\bbf_\fp$ in which $\bbf_\fp=\CF_{K_\fp}$,

  \item $\bphi=\bphi^\fp\otimes\bphi_\fp$ in which $\bphi_\fp=\CF_{\Lambda_\fp}$.
\end{itemize}
\end{remark}

In the formula for $\cI_K(\bbf,\bphi)_\fp$ in Theorem \ref{th:orbital}, the orbital integral has the decomposition
\[
\Orb(\bar\bbf^\fp,\bar\bphi^\fp;\bar\xi,\bar{x})=
\Orb(\bar\bbf^{\underline\fp},\bar\bphi^{\underline\fp};\bar\xi,\bar{x})\cdot
\prod_{\fq\in\underline\fp\setminus\{\fp\}}\Orb(\CF_{\bar{K}_\fq},\CF_{\bar\Lambda_\fq};\bar\xi,\bar{x}),
\]
in which we decompose the Haar measure on $\rU(\bar\rV)({\bA_F^\infty}^{,\fp})$ such that $\bar{K}_\fq$ has volume $1$ for every $\fq\in\underline\fp\setminus\{\fp\}$.

We now compare the term
\[
2\log|O_F/\fp|\cdot\prod_{\fq\in\underline\fp\setminus\{\fp\}}\Orb(\CF_{\bar{K}_\fq},\CF_{\bar\Lambda_\fq};\bar\xi,\bar{x})
\cdot\chi\(\cO_{\Gamma_{\bar\xi}}\otimes^\dL_{\cO_{\cN^2}}\cO_{\Delta\cZ(\bar{x})}\)
\]
with the orbital integrals on the general linear side. Recall the notations $\Mat_{r,s}$ and $\rM_n$ from Subsection \ref{ss:notation}, and denote by $\rS_n$ the $O_F$-subscheme of $\Res_{O_E/O_F}\Mat_{n,n}$ consisting of matrices $g$ satisfying $g\cdot g^\tc=\rI_n$.

\begin{definition}[\cite{Liu14}*{Section~5.3}\footnote{Note that we have changed the roles of rows and columns from \cite{Liu14}, in order to match the convention of generating series.}]\label{de:regular_general}
Consider a field extension $F'/F$ and put $E'\coloneqq E\otimes_FF'$.
\begin{enumerate}
  \item We say that a pair of elements $(\zeta,y)\in\rS_n(F')\times\rM_n(F')$ is \emph{regular semisimple} if the matrix $(y_2\zeta^{i+j-2}y_1)_{i,j=1}^n$ is invertible in $E'$, where we write $y=(y_1,y_2)$ for $y_1\in\Mat_{n,1}(F')$ and $y_2\in\Mat_{1,n}(F')$.

  \item The group $\GL_n(F')$ acts on $\rS_n(F')\times\rM_n(F')$ via the formula $(\zeta,y_1,y_2).g=(g^{-1}\zeta g,g^{-1}y_1,y_2g)$, which preserves regular semisimple pairs. Denote by $[\rS_n(F')\times\rM_n(F')]$ the orbits of $\rS_n(F')\times\rM_n(F')$ under the above action, and by $[\rS_n(F')\times\rM_n(F')]_{\r{rs}}$ the subset of regular semisimple orbits.

  \item Suppose that $F'=F_v$ for some place $v$ of $F$. For a regular semisimple pair $(\zeta,y)\in\rS_n(F')\times\rM_n(F')$, we define its \emph{local transfer factor} to be $\omega_v(\zeta,y)\coloneqq\mu_{E/F}(\det(y_1,\zeta y_1,\dots,\zeta^{n-1}y_1))$. We denote by $[\rS_n(F')\times\rM_n(F')]_{\r{rs}}^\pm$ the subset of $[\rS_n(F')\times\rM_n(F')]_{\r{rs}}$ of orbits $(\zeta,y)$ such that $\mu_{E/F}(\det(y_2\zeta^{i+j-2}y_1)_{i,j=1}^n)=\pm 1$.

  \item We say that two regular semisimple orbits $(\zeta,y)\in[\rS_n(F')\times\rM_n(F')]_{\r{rs}}$ and $(\bar\xi,\bar{x})\in[\rU(\bar\rV)(F')\times\bar\rV(E')]_{\r{rs}}$ (Definition \ref{de:regular_unitary}) \emph{match} if
      \begin{itemize}
        \item $\zeta$ and $\bar\xi$ have the same characteristic polynomial as elements in $\Mat_{n,n}(E')$,

        \item $y_2\zeta^iy_1=(\bar\xi^i\bar{x},\bar{x})_{\bar\rV}$ for $0\leq i\leq n-1$.
      \end{itemize}
\end{enumerate}
\end{definition}

\begin{corollary}\label{co:orbital}
In the situation of Theorem \ref{th:orbital}, suppose that for every orbit $(\bar\xi,\bar{x})\in[\rU(\bar\rV)(F)\times\bar\rV(E))]_{\r{rs}}$, Conjecture \ref{co:rfl_general}(2) for $E_\fq/F_\fq$ for every $\fq\in\underline\fp\setminus\{\fp\}$ and Conjecture \ref{co:afl_general} for $E_\fp/F_\fp$ hold. Then we have
\begin{align*}
\cI_K(\bbf,\bphi)_\fp&=-\sum_{(\bar\xi,\bar{x})\in[\rU(\bar\rV)(F)\times\bar\rV(E))]_{\r{rs}}}
\re^{-2\pi\cdot\Tr_{F/\dQ}(\bar{x},\bar{x})_{\bar\rV}}\\
&\Orb(\bar\bbf^\fp,\bar\bphi^\fp;\bar\xi,\bar{x})
\cdot\left.\frac{\rd}{\rd s}\right|_{s=0}\(\prod_{\fq\in\underline\fp}\omega_\fq(\zeta,y)\Orb(s;\CF_{\rS_n(O_{F_\fq})},\CF_{\rM_n(O_{F_\fq})};\zeta,y)\),
\end{align*}
where $(\zeta,y)\in[\rS_n(F)\times\rM_n(F)]_{\r{rs}}$ is the unique orbit that matches $(\bar\xi,\bar{x})$.
\end{corollary}

\begin{proof}
It suffices to note that $\Orb(0;\CF_{\rS_n(O_{F_\fp})},\CF_{\rM_n(O_{F_\fp})};\zeta,y)=0$, which is Conjecture \ref{co:rfl_general}(1) and is known (see Remark \ref{re:rfl}).
\end{proof}

\begin{remark}
To obtain a global result, we would like to find test functions $\tilde\bbf^{\underline\fp}$, $\tilde\bphi^{\underline\fp}$ on the general linear side, in order to obtain some matching relation with the local intersection number $\cI_K(\bbf,\bphi)_v$ at every place $v$ of $F$. If $v$ is split in $E$, then it is expected that $\cI_K(\bbf,\bphi)_v$ vanishes, and the matching test functions $\tilde\bbf_v$, $\tilde\bphi_v$ are obtained from $\bbf_v$, $\bphi_v$ by an elementary way as in \cite{Liu14}*{Proposition~5.11}. If $v$ is neither split nor a good inert prime, then we do not know what to do at this moment.
\end{remark}

\if false

\section{Proof of arithmetic fundamental lemma for $n\leq 2$}
\label{ss:5}

In this section, we prove the arithmetic fundamental lemma for $\rU(1)\times\rU(1)$ and $\rU(2)\times\rU(2)$ when the residue characteristic is odd.

Suppose that we are in the situation of Subsection \ref{ss:afl} with $q$ odd. Choose a uniformizer $\varpi$ of $F$. We let $\val\colon E\to\dZ\cup\{\infty\}$ be the valuation function normalized such that $\val(\pi)=1$.

\subsection{Computation on the analytic side}
\label{ss:analytic}

For a regular semisimple pair $(\zeta,y)\in\rS_2(F)\times\rM_2(F)$ such that $y_2y_1\neq 0$, we introduce two invariants:
\[
v(\zeta)\coloneqq\val\((\tr\zeta)^2-4\det\zeta\),\qquad
u(\zeta,y)\coloneqq\val\(1-\frac{y_2\zeta y_1\cdot y_2\zeta^\tc y_1}{(y_2y_1)^2}\)
\]
both of which depend only on the $\GL_2(F)$-orbit.

\begin{definition}
We say that an element $\(\begin{smallmatrix}a&b\\c&d\end{smallmatrix}\)\in\rS_2(F)\subseteq\GL_2(E)$ is \emph{reduced} if $c\in O_E^\times$. We say that a pair $(\zeta,y)\in\rS_2(F)\times\rM_2(F)$ is \emph{standard} if $\zeta$ is reduced, $y=(y_1,y_2)$ with $y_1=(0,1)^\rt$ and $y_2=(0,y_0)$ for some $y_0\in F^\times$.
\end{definition}

\begin{lem}\label{le:analytic1}
Every regular semisimple orbit in $[\rS_2(F)\times\rM_2(F)]_{\r{rs}}^-$ contains a standard pair. Moreover, if it contains two standard pairs $\(\(\begin{smallmatrix}a&b\\c&d\end{smallmatrix}\),(0,1)^\rt,(0,y_0)\)$ and $\(\(\begin{smallmatrix}a'&b'\\c'&d'\end{smallmatrix}\),(0,1)^\rt,(0,y'_0)\)$, then $a'=a$, $d'=d$, $y'_0=y_0$, and there exists a unique $u\in O_F^\times$ such that $c'=uc$, $d'=u^{-1}d$.
\end{lem}

\begin{proof}
Let $(\zeta,y)$ be a regular semisimple pair matching a regular semisimple pair $(\xi,x)\in\rU(\rV^-_2)(F)\times\rV^-_2(E)$. Then we have $y_2y_1=(x,x)\neq 0$, since $\rV^-_2$ is anisotropic. Thus, in the orbit of $(\zeta,y)$, we may find one with $y_1=(0,1)^\rt$ and $y_2=(0,y_0)$ for some $y_0\in F^\times$. Write $\zeta=\(\begin{smallmatrix}a&b\\c&d\end{smallmatrix}\)$. Then we have $\det((y_2\zeta^{i+j-2}y_1)_{i,j=1}^2)=cdy_0^2$; hence $cd\neq0$ and $\val(cd)$ is odd. We may conjugate $\zeta$ by an element of the form $\(\begin{smallmatrix}\alpha &0\\0&1\end{smallmatrix}\)\in\GL_2(F)$ with $\val(\alpha)=-\val(c)$ to obtain a standard pair in the orbit. The last statement is obvious.
\end{proof}

For a reduced element $\gamma=\(\begin{smallmatrix}a&b\\c&d\end{smallmatrix}\)\in\rS_2(F)$ and $s\in\dZ$, we put
\[
\Xi_s(\gamma)\coloneqq\CF_{O_E}(a)\CF_{O_E}(d)\CF_{\varpi^s O_E}(b).
\]

\begin{lem}\label{le:analytic2}
Let $(\zeta,y)$ be a standard regular semisimple pair with $y=((0,1)^\rt,(0,y_0))$. Then
\[
\Orb(s;\CF_{\rS_n(O_F)},\CF_{\rM_n(O_F)};\zeta,y)=\sum_{j=0}^{\val(y_0)}\sum_{i\leq j}
(-q^{2s})^{i+j}\int_{F}q^{j-i}
\Xi_{j-i}\(\(\begin{smallmatrix}1&-\beta\\0&1\end{smallmatrix}\)\zeta\(\begin{smallmatrix}1&\beta\\0&1\end{smallmatrix}\)\)
\CF_{\varpi^{-i}O_F}(\beta)\;\rd\beta
\]
where $\rd\beta$ is the Haar measure on $F$ under which $O_F$ has volume $1$.
\end{lem}

\begin{proof}
In the integration \eqref{eq:orbital_general}, we use Iwasawa decomposition for $g$. Write $g=\(\begin{smallmatrix}1&\beta\\0&1\end{smallmatrix}\)\(\begin{smallmatrix}\varpi^{-i}&0\\0&\varpi^{-j}\end{smallmatrix}\)k$ for some $k\in\GL_2(O_F)$. Then we have
\begin{itemize}
  \item $g^{-1}(0,1)^\rt=(-\beta\varpi^i,\varpi^j)^\rt$; $(0,y_0)g=(0,\varpi^{-j}y_0)$;

  \item $\CF_{\GL_2(O_F)}(g^{-1}\zeta g)=\Xi_{j-i}\(\(\begin{smallmatrix}1&-\beta\\0&1\end{smallmatrix}\)\zeta\(\begin{smallmatrix}1&\beta\\0&1\end{smallmatrix}\)\)$;

  \item $\rd g=q^{j-i}\rd\beta\rd k$ where $\rd k$ is the Haar measure on $\GL_2(O_F)$ with total volume $1$.
\end{itemize}
The lemma follows immediately.
\end{proof}

In what follows, we fix a regular semisimple orbit in $(\zeta,y)\in[\rS_2(F)\times\rM_2(F)]_{\r{rs}}^-$ with a standard representative $\(\(\begin{smallmatrix}a&b\\c&d\end{smallmatrix}\),(0,1)^\rt,(0,y_0)\)$. Then $y_0=y_2y_1$.

\begin{lem}\label{le:analytic3}
We have
\begin{enumerate}
  \item $a,d\in O_E^\times$, $aa^\tc=dd^\tc$;

  \item $\val(b)$ and is odd and positive;

  \item $aa^\tc+bc^\tc=1$, $ab^\tc+bd^\tc=a^\tc c+c^\tc d=0$.
\end{enumerate}
\end{lem}

\begin{proof}
We know already in the proof of Lemma \ref{le:analytic1} that $\val(bc)$ is odd. Since $\val(c)=0$, $\val(b)$ is odd. Moreover, as $\(\begin{smallmatrix}a&b\\c&d\end{smallmatrix}\)\(\begin{smallmatrix}a&b\\c&d\end{smallmatrix}\)^\tc=1$, we have (3) immediately and $aa^\tc=dd^\tc$. However, since $\val(aa^\tc)$ is even and $\val(bc^\tc)$ is odd, we must have $\val(a)=0$ and $\val(b)>0$, which imply (1) and (2).
\end{proof}

\begin{lem}\label{le:analytic4}
For $\beta\in F$, we have
\[
\Xi_s\(\(\begin{smallmatrix}1&-\beta\\0&1\end{smallmatrix}\)\zeta\(\begin{smallmatrix}1&\beta\\0&1\end{smallmatrix}\)\)
=\CF_{\varpi^sO_E}(b+\beta(a-d)-\beta^2c)\CF_{O_E}(\beta).
\]
\end{lem}

\begin{proof}
Using the formula
\[
\left(
  \begin{array}{cc}
    1 & -\beta \\
    0 & 1 \\
  \end{array}
\right)
\left(
  \begin{array}{cc}
    a & b \\
    c & d \\
  \end{array}
\right)
\left(
  \begin{array}{cc}
    1 & \beta \\
    0 & 1 \\
  \end{array}
\right)=
\left(
  \begin{array}{cc}
    a-\beta c & b+\beta(a-d)-\beta^2c \\
    c & \beta c +d \\
  \end{array}
\right),
\]
we have $\Xi_s\(\(\begin{smallmatrix}1&-\beta\\0&1\end{smallmatrix}\)\zeta\(\begin{smallmatrix}1&\beta\\0&1\end{smallmatrix}\)\)
=\CF_{O_E}(a-\beta c)\CF_{O_E}(\beta c +d)\CF_{\varpi^sO_E}(b+\beta(a-d)-\beta^2c)$. By Lemma \ref{le:analytic3}(1) and the fact that $c\in O_E^\times$, we have $\CF_{O_E}(a-\beta c)\CF_{O_E}(\beta c +d)=\CF_{O_E}(\beta)$. The lemma follows.
\end{proof}

\begin{proposition}\label{pr:analytic1}
Let $(\zeta,y)\in[\rS_2(F)\times\rM_2(F)]_{\r{rs}}^-$ be a regular semisimple orbit such that $v(\zeta)$ is odd. Then we have
\[
-\omega(\zeta,y)\left.\frac{\rd}{\rd s}\right|_{s=0}\Orb(s;\CF_{\rS_n(O_F)},\CF_{\rM_n(O_F)};\zeta,y)=2\log q\sum_{j=0}^{\val(y_2y_1)}\sigma_j(\zeta,y)
\]
where we put
\[
\sigma_j(\zeta,y)\coloneqq
\begin{dcases}
\sum_{i=0}^{\frac{v(\zeta)-1}{2}}q^i, &\text{if }j>\frac{v(\zeta)-1}{2};\\
\sum_{i=0}^jq^i+\(\frac{v(\zeta)-1}{2}-j\)q^j, &\text{if }j\leq\frac{v(\zeta)-1}{2}.
\end{dcases}
\]
\end{proposition}

\begin{proof}
We may suppose that $(\zeta,y)=\(\(\begin{smallmatrix}a&b\\c&d\end{smallmatrix}\),(0,1)^\rt,(0,y_0)\)$ is standard by Lemma \ref{le:analytic1}. Then $\omega(\zeta,y)=\mu_{E/F}(-b)=-1$ by Lemma \ref{le:analytic3}(2).

Since $\tr(\zeta)^2-4\det\zeta=(d-a)^2+4bc$, the oddness of $v(\zeta)$ implies $2\val(d-a)>\val(bc)=\val(b)$ and $v(\zeta)=\val(b)$. By Lemma \ref{le:analytic2} and Lemma \ref{le:analytic4}, we need to compute for $i\leq j$ the area of
\begin{align}\label{eq:analytic}
\Omega_{j,i}\coloneqq\{\beta\in F\res\val(b+\beta(a-d)-\beta^2c)\geq j-i\text{ and }\val(\beta)\geq \max\{0,-i\}\}.
\end{align}
We have the formula
\[
b+\beta(a-d)-\beta^2c=-c\(\(\beta+\frac{d-a}{2c}\)^2-\frac{(d-a)^2+4bc}{4c^2}\).
\]
As $\val(c)=0$ and $2\val(d-a)>\val(b)$, we have $\val\(\frac{(d-a)^2+4bc}{4c^2}\)=\val(b)$ which is odd. Thus, if $j-i>\val(b)$, then $\Omega_{j,i}=\emptyset$. If $j-i\leq\val(b)$, then
\[
\Omega_{j,i}=\left\{\beta\in F\left|\val\(\beta+\frac{d-a}{2}\)\geq\left\lceil\frac{j-i}{2}\right\rceil\text{ and }\val(\beta)\geq \max\{0,-i\}\right.\right\}.
\]
Since $\val\(\frac{d-a}{2}\)=\val(d-a)\geq\left\lceil\frac{\val(b)}{2}\right\rceil\geq\left\lceil\frac{j-i}{2}\right\rceil\geq 0$, we have
\[
\Omega_{j,i}=\left\{\beta\in F\left|\val(\beta)\geq\max\left\{\left\lceil\frac{j-i}{2}\right\rceil,-i\right\}\right.\right\}.
\]
Thus, by Lemma \ref{le:analytic2}, we have
\[
\left.\frac{\rd}{\rd s}\right|_{s=0}\Orb(s;\CF_{\rS_n(O_F)},\CF_{\rM_n(O_F)};\zeta,y)=2\log q\sum_{j=0}^{\val(y_0)}
\sum_{i=j-\val(b)}^j(-1)^{i+j}(i+j)q^{j-i-\max\left\{\left\lceil\frac{j-i}{2}\right\rceil,-i\right\}}.
\]
The proposition follows as it is elementary to see that
\[
\sigma_j(\zeta,y)=\sum_{i=j-\val(b)}^j(-1)^{i+j}(i+j)q^{j-i-\max\left\{\left\lceil\frac{j-i}{2}\right\rceil,-i\right\}}.
\]
\end{proof}

\begin{proposition}\label{pr:analytic2}
Let $(\zeta,y)\in[\rS_2(F)\times\rM_2(F)]_{\r{rs}}^-$ be a regular semisimple orbit such that $v(\zeta)$ is even. Then we have
\[
-\omega(\zeta,y)\left.\frac{\rd}{\rd s}\right|_{s=0}\Orb(s;\CF_{\rS_n(O_F)},\CF_{\rM_n(O_F)};\zeta,y)=2\log q\sum_{j=0}^{\val(y_2y_1)}\sigma_j(\zeta,y)
\]
where we put
\[
\sigma_j(\zeta,y)\coloneqq
\begin{dcases}
\sum_{i=0}^{\frac{v(\zeta)}{2}}q^i+(-1)^{j-\frac{v(\zeta)}{2}}\(\frac{u(\zeta,y)}{2}-\frac{v(\zeta)}{2}\)q^{\frac{v(\zeta)}{2}}
-\frac{1}{2}q^{\frac{v(\zeta)}{2}}, &\text{if }j>\frac{v(\zeta)}{2};\\
\sum_{i=0}^jq^i+\(\frac{u(\zeta,y)-1}{2}-j\)q^j, &\text{if }j\leq\frac{v(\zeta)}{2}.
\end{dcases}
\]
\end{proposition}

\begin{proof}
We may suppose that $(\zeta,y)=\(\(\begin{smallmatrix}a&b\\c&d\end{smallmatrix}\),(0,1)^\rt,(0,y_0)\)$ is standard by Lemma \ref{le:analytic1}. Again, we have $\omega(\zeta,y)=-1$.

Since $v(\zeta)$ is even, we have $2\val(d-a)<\val(bc)=\val(b)$ and $v(\zeta)=2\val(d-a)$. We compute the area of $\Omega_{j,i}$ defined in \eqref{eq:analytic} in the proof of Proposition \ref{pr:analytic1}, again via the formula
\[
b+\beta(a-d)-\beta^2c=-c\(\(\beta+\frac{d-a}{2c}\)^2-\frac{(d-a)^2+4bc}{4c^2}\).
\]
By Lemma \ref{le:analytic3}, we have $\frac{d-a}{2c}\in F$ and $\frac{(d-a)^2+4bc}{4c^2}=\gamma^2$ for some $\gamma\in F$. It is easy to see that
\[
\left\{\val\(\frac{d-a}{2c}+\gamma\),\val\(\frac{d-a}{2c}-\gamma\)\right\}=\left\{\frac{v(\zeta)}{2},\val(b)-\frac{v(\zeta)}{2}\right\}.
\]
Without lost of generality, we assume $\val\(\frac{d-a}{2c}+\gamma\)=\frac{v(\zeta)}{2}$. We have the following cases.
\begin{enumerate}
  \item If $v(\zeta)\geq j-i$, then
    \[
    \Omega_{j,i}=\left\{\beta\in F\left|\val(\beta)\geq\max\left\{\left\lceil\frac{j-i}{2}\right\rceil,-i\right\}\right.\right\}.
    \]

  \item If $\val(b)\geq j-i>v(\zeta)$, then we have two subcases:
    \begin{enumerate}
      \item If $-i\leq\frac{v(\zeta)}{2}$, then $\Omega_{j,i}=\Omega_{j,i}^0\coprod\Omega_{j,i}^1$ where
        \begin{align*}
        \Omega_{j,i}^0&\coloneqq\left\{\beta\in F\left|\val(\beta)\geq\max\left\{j-i-\frac{v(\zeta)}{2},-i\right\}\right.\right\};\\
        \Omega_{j,i}^1&\coloneqq\left\{\beta\in F\left|\val\(\beta-\(\frac{d-a}{2c}+\gamma\)\)\geq j-i-\frac{v(\zeta)}{2}\right.\right\}.
        \end{align*}

      \item If $-i>\frac{v(\zeta)}{2}$, then $\Omega_{j,i}=\Omega_{j,i}^0$.
    \end{enumerate}

  \item If $j-i>\val(b)$, then we have three subcases:
    \begin{enumerate}
      \item If $-i\leq\frac{v(\zeta)}{2}$, then $\Omega_{j,i}=\Omega_{j,i}^1\coprod\Omega_{j,i}^2$ where
        \[
        \Omega_{j,i}^2\coloneqq\left\{\beta\in F\left|\val\(\beta-\(\frac{d-a}{2c}-\gamma\)\)\geq j-i-\frac{v(\zeta)}{2}\right.\right\}.
        \]

      \item If $\frac{v(\zeta)}{2}<-i\leq\val(b)-\frac{v(\zeta)}{2}$, then $\Omega_{j,i}=\Omega_{j,i}^2$.

      \item If $-i>\val(b)-\frac{v(\zeta)}{2}$, then $\Omega_{j,i}=\emptyset$.
    \end{enumerate}
\end{enumerate}

Suppose that $j\leq\frac{v(\zeta)}{2}$. Then only Cases (1) and (2b) may contribute; and we have
\begin{align*}
&\left.\frac{\rd}{\rd s}\right|_{s=0}\sum_{i\leq j}
(-q^{2s})^{i+j}\int_{F}q^{j-i}
\Xi_{j-i}\(\(\begin{smallmatrix}1&-\beta\\0&1\end{smallmatrix}\)\zeta\(\begin{smallmatrix}1&\beta\\0&1\end{smallmatrix}\)\)
\CF_{\varpi^{-i}O_F}(\beta)\;\rd\beta\\
&=\sum_{i=j-\val(b)}^j(-1)^{i+j}(i+j)q^{j-i-\max\left\{\left\lceil\frac{j-i}{2}\right\rceil,-i\right\}}=\sigma_j(\zeta,y).
\end{align*}

Suppose that $j>\frac{v(\zeta)}{2}$. Then all six cases may contribute; but in Cases (2) and (3), the volume of $\Omega_{j,i}^k$ is always $q^{-\(j-i-\frac{v(\zeta)}{2}\)}$ for $k=0,1,2$. It follows that
\begin{align*}
&\left.\frac{\rd}{\rd s}\right|_{s=0}\sum_{i\leq j}
(-q^{2s})^{i+j}\int_{F}q^{j-i}
\Xi_{j-i}\(\(\begin{smallmatrix}1&-\beta\\0&1\end{smallmatrix}\)\zeta\(\begin{smallmatrix}1&\beta\\0&1\end{smallmatrix}\)\)
\CF_{\varpi^{-i}O_F}(\beta)\;\rd\beta\\
&=\sum_{i=j-v(\zeta)}^j(-1)^{i+j}(i+j)q^{j-i-\left\lceil\frac{j-i}{2}\right\rceil}
+\sum_{i=-\frac{v(\zeta)}{2}}^{j-v(\zeta)-1}(-1)^{i+j}(i+j)q^{\frac{v(\zeta)}{2}}
+\sum_{i=\frac{v(\zeta)}{2}-\val(b)}^{j-v(\zeta)-1}(-1)^{i+j}(i+j)q^{\frac{v(\zeta)}{2}}.
\end{align*}
Since it is easy to see that $\val(b)=u(\zeta,y)$, the above summation adds up to $\sigma_j(\zeta,y)$. The proposition is proved.
\end{proof}

\subsection{Computation on the arithmetic side}
\label{ss:arithmetic}

The computation is very similar to the one in \cite{Zha12}*{Section~5}. We will adopt the strategy there but with necessary modification to our case.

For a regular semisimple pair $(\xi,x)\in\rU(\rV_2^-)(F)\times\rV_2^-(E)$, we introduce two invariants:
\[
v(\xi)\coloneqq\val(\tr(\xi)^2-4\det\xi),\qquad
u(\xi,x)\coloneqq\val\(1-\frac{(x,\xi x)\cdot(x,\xi x)^\tc}{(x,x)^2}\)
\]
both of which depend only on the $\rU(\rV_2^-)(F)$-orbit.

\begin{lem}\label{le:arithmetic1}
Let $(\xi,x)\in\rU(\rV_2^-)(F)\times\rV_2^-(E)$ be a regular semisimple pair. Then
\begin{align}\label{eq:arithmetic1}
\chi\(\cO_{\Gamma_\xi}\otimes^\dL_{\cO_{\cN_2^2}}\cO_{\Delta\cZ_2(x)}\)
=\length_{O_{\breve{E}}}(\Gamma_\xi\cap\Delta\cZ_2(x)).
\end{align}
Moreover, $\Gamma_\xi\cap\Delta\cZ_2(x)$ is isomorphic to the maximal closed subscheme of $\cZ_2(x)$ where the quasi-endomorphism $\xi\in\End(\bbX_2)_\dQ$ extends.
\end{lem}

\begin{proof}
By \cite{KR11}*{Proposition~8.1} (see \cite{Liu12}*{Proposition~4.11} for general $F$), both $\cZ_2(x)$ and $\cZ_2(\xi x)$ are finite schemes in $\Sch_{/O_{\breve{E}}}$. Since $\Gamma_\xi\cap\Delta\cZ_2(x)$ is isomorphic to a closed subscheme of $\cZ_2(x)\cap\cZ_2(\xi x)$ which has empty generic fiber, it is an Artinian scheme in $\Sch_{/O_{\breve{E}}}$. Then since both $\cO_{\Gamma_\xi}$ and $\cO_{\Delta\cZ_2(x)}$ are Cohen--Macaulay rings, we have
\[
\cO_{\Gamma_\xi}\otimes^\dL_{\cO_{\cN_2^2}}\cO_{\Delta\cZ_2(x)}\simeq\cO_{\Gamma_\xi}\otimes_{\cO_{\cN_2^2}}\cO_{\Delta\cZ_2(x)}
\simeq\cO_{\Gamma_\xi\cap\Delta\cZ_2(x)}.
\]
Thus, \eqref{eq:arithmetic1} follows. The next assertion is obvious from the moduli interpretation.
\end{proof}

We denote by $(\bar\bbX_0,\bar\bbi_0,\bar\bblambda_0)$ the unitary $O_F$-module obtained from $(\bbX_0,\bbi_0,\bblambda_0)$ by simply changing the $O_E$-action by conjugation, which is then of signature $(0,1)$. Let $O_D$ be the endomorphism algebra of the triple $(\bbX_0,\bbi_0\res O_F,\bblambda_0)$. Then $O_D$ is the maximal order in $D\coloneqq(O_D)_\dQ$ which is a division quaternion algebra over $F$ containing $E$ via $\bbi_0$. Choose a normalizer $\fj$ of $O_E$ in $O_D$ such that $\fj^2=-\varpi$. To compute \eqref{eq:arithmetic1}, we may choose our base supersingular unitary $O_F$-module of signature $(1,1)$ to be
\begin{align}\label{eq:triple}
(\bbX_2,\bbi_2,\bblambda_2)\coloneqq(\bar\bbX_{0k},\bar\bbi_{0k},\bar\bblambda_{0k})\times(\bbX_{0k},\bbi_{0k},\bblambda_{0k}).
\end{align}
Then we may identify $\rU(\rV_2^-)(F)$ with matrices $\xi\in\GL_2(E)$ satisfying $\xi\(\begin{smallmatrix}-\varpi&0\\0&1\end{smallmatrix}\)(\xi^\tc)^\rt=\(\begin{smallmatrix}-\varpi&0\\0&1\end{smallmatrix}\)$. For every integer $j\geq 0$, let $\cZ_j$ be the closed subscheme $\cZ_j$ of $\cN_2$ giving by quasi-canonical lifting of level $j$ as in \cite{Zha12}*{Section~5.2}. We denote by $\cZ_j(\xi)$ the maximal closed subscheme of $\cZ_j$ on which $\xi$ extends.

\begin{proposition}\label{pr:arithmetic}
Suppose that $\xi=\(\begin{smallmatrix}a&\varpi b\\c&d\end{smallmatrix}\)$.
\begin{enumerate}
  \item If $v(\xi)$ is odd, then
     \[
     \length_{O_{\breve{E}}}\cZ_j(\xi)=
     \begin{dcases}
     2\sum_{i=0}^{\frac{v(\xi)-1}{2}}q^i, &\text{if }j>\frac{v(\xi)-1}{2};\\
     2\sum_{i=0}^{j-1}q^i+\(\frac{v(\xi)+1}{2}-j\)(q^j+q^{j-1}), &\text{if }1\leq j\leq\frac{v(\xi)-1}{2};\\
     \frac{v(\xi)+1}{2}, &\text{if }j=0.
     \end{dcases}
     \]

  \item If $v(\xi)$ is even, then
     \[
     \length_{O_{\breve{E}}}\cZ_j(\xi)=
     \begin{dcases}
     2\sum_{i=0}^{\frac{v(\xi)}{2}}q^i-q^{\frac{v(\xi)}{2}}, &\text{if }j>\frac{v(\xi)}{2};\\
     2\sum_{i=0}^{j-1}q^i+(\val(b)+1-j)(q^j+q^{j-1}), &\text{if }1\leq j\leq\frac{v(\xi)}{2};\\
     \val(b)+1, &\text{if }j=0.
     \end{dcases}
     \]
\end{enumerate}
\end{proposition}

\begin{proof}
We have $aa^\tc-\varpi bb^\tc=1$, $-\varpi c c^\tc+dd^\tc=1$, and $ac^\tc=bd^\tc$, which imply
that $a,d\in O_E^\times$, $\val(b)=\val(c)\geq 0$, and $v(\xi)=\min\{2\val(a-d),2\val(b)+1\}$. Choose a totally imaginary element $\sqrt\epsilon\in O_E^\times$. By the same argument (and notation) in the proof of \cite{Zha12}*{Proposition~5.4}, $\cZ_j(\xi)$ coincides with the maximal closed subscheme where all of
\[
a^\tc+d,\qquad (a^\tc-d)\sqrt\epsilon,\qquad (c-b^\tc)\fj,\qquad (c+b^\tc)\fj\sqrt\epsilon
\]
extend. Thus,
\[
\length_{O_{\breve{E}}}\cZ_j(\xi)=\min\{\ell_j(a^\tc+d),\ell_j((a^\tc-d)\sqrt\epsilon),
\ell_j((c-b^\tc)\fj),\ell_j((c+b^\tc)\fj\sqrt\epsilon)\}.
\]
We have
\begin{align*}
\min\{\ell_j((c-b^\tc)\fj),\ell_j((c+b^\tc)\fj\sqrt\epsilon)\}&=2\val(b)+1\\
\min\{\ell_j(a^\tc+d),\ell_j((a^\tc-d)\sqrt\epsilon)\}&=
\begin{dcases}
2\val(a-d), &\text{if }j>\val(a-d);\\
\infty, &\text{if }j\leq\val(a-d).
\end{dcases}
\end{align*}
We have two cases.

If $v(\xi)$ is odd, then $2\val(b)+1<2\val(a-d)$. Thus, $\length_{O_{\breve{E}}}\cZ_j(\xi)=2\val(b)+1$; and (1) follows from a result of Keating as stated in \cite{Zha12}*{Proposition~5.3}.

If $v(\xi)$ is even, then $2\val(b)+1>2\val(a-d)=v(\xi)$. Thus,
\[
\length_{O_{\breve{E}}}\cZ_j(\xi)=
\begin{dcases}
2\val(b)+1, &\text{if }j\leq\frac{v(\xi)}{2};\\
v(\xi), &\text{if }j>\frac{v(\xi)}{2}.
\end{dcases}
\]
Then (2) follows from \cite{Zha12}*{Proposition~5.3}. The proposition is proved.
\end{proof}

\begin{proof}[Proof of Theorem \ref{th:afl}]
We first consider the apparently more difficult case where $n=2$. We keep the choice \eqref{eq:triple}. Let $(\xi,x)\in\rU(\rV_2^-)(F)\times\rV_2^-(E)$ be a regular semisimple pair. We may assume that $x=(x_0,0)$ (resp.\ $x=(0,x_0)$) for some $x_0\in E$ if $\val((x,x))$ is odd (resp.\ even). In this case, if we write $\xi=\(\begin{smallmatrix}a&\varpi b\\c&d\end{smallmatrix}\)$, then $u(\xi,x)=2\val(b)+1$. Also in this case, by \cite{KR11}*{Proposition~8.1} (see \cite{Liu12}*{Proposition~4.11} for general $F$), we have
\[
\cZ_2(x)=\sum_{j=0}^{\left\lfloor\frac{\val((x,x))}{2}\right\rfloor}\cZ_{\val((x,x))-2j}.
\]
Note that if $(\zeta,y)\in[\rS_2(F)\times\rM_2(F)]^-_{\r{rs}}$ matches $(\xi,x)$, then we have $v(\zeta)=v(\xi)$ and $u(\zeta,y)=u(\xi,x)$. Therefore, Conjecture \ref{co:afl_general} for $n=2$ follows from Lemma \ref{le:arithmetic1}, Proposition \ref{pr:analytic1}, Proposition \ref{pr:analytic2}, and Proposition \ref{pr:arithmetic}.

Now we consider the recreational case where $n=1$. For every orbit $(\zeta,y)\in[\rS_1(F)\times\rM_1(F)]^-_{\r{rs}}$, we may assume $(\zeta,y)=(\zeta,1,y_0)$ for $y_0\in F^\times$ such that $\val(y_0)$ is odd. Then
\begin{align*}
&-\omega(\zeta,y)\left.\frac{\rd}{\rd s}\right|_{s=0}\Orb(s;\CF_{\rS_1(O_F)},\CF_{\rM_1(O_F)};\zeta,y)\\
&=-\left.\frac{\rd}{\rd s}\right|_{s=0}\int_{F^\times}\CF_{O_F}(g^{-1})\CF_{O_F}(y_0g)\mu_{E/F}(g)|g|_E^s\;\rd g\\
&=-2\log q\sum_{i=0}^{\val(y_0)}(-1)^ii=2\log q\cdot\max\left\{\frac{\val(y_0)+1}{2},0\right\}.
\end{align*}
For every orbit $(\xi,x)\in[\rU(\rV_1^-)(F)\times\rV_1^-(E)]_{\r{rs}}$, we have
\[
\chi\(\cO_{\Gamma_\xi}\otimes^\dL_{\cO_{\cN_1^2}}\cO_{\Delta\cZ_1(x)}\)
=\length_{O_{\breve{E}}}(\cZ_1(x))=\max\left\{\frac{\val((x,x))+1}{2},0\right\}
\]
where the second equality is a special case of \cite{Zha12}*{Proposition~5.3} for $j=0$ and $\ell=\val((x,x))$. Thus, Conjecture \ref{co:afl_general} for $n=1$ follows as $y_0=(x,x)$ if $(\zeta,1,y_0)$ matches $(\xi,x)$.
\end{proof}

\fi

\appendix

\section{Proof of the arithmetic fundamental lemma in the minuscule case (by Chao~Li and Yihang~Zhu)}
\label{ss:e}

The purpose of this appendix is to prove the arithmetic fundamental lemma for $\rU(n)\times\rU(n)$, namely, Conjecture \ref{co:afl_general}, in the minuscule case. We follow the setup and notation in Subsection \ref{ss:afl}.

\subsection{Derivatives of orbital integrals via lattice counting}

We take a regular semisimple orbit $(\zeta,y)\in[\rS_n(F)\times\rM_n(F)]_{\r{rs}}^-$, where $y=(y_1,y_2)\in\Mat_{n,1}(F)\times \Mat_{1,n}(F)$. Let $(\xi,x)\in [\rU(\rV_n^-)(F)\times \rV_n^-(E)]_{\r{rs}}$ be the unique orbit that matches $(\zeta, y)$. By definition, $\zeta$ and $\xi$ have the same characteristic polynomial; and we have
\begin{align}\label{eq:matching}
y_2\zeta^iy_1=(\xi^ix,x),\quad i=0,\dots, n-1.
\end{align}
Recall that we denote $v(\zeta,y)\coloneqq\val(\det(y_1,\zeta y_1,\dots, \zeta^{n-1}y_1))$, and define the transfer factor to be $\omega(\zeta,y)\coloneqq(-1)^{v(\zeta,y)}$. We also put $\Delta(\zeta,y)\coloneqq\det(y_2\zeta^{i+j-2}y_1)_{i,j=1}^{n}$ and $\delta(\zeta,y)\coloneqq\val(\Delta(\zeta,y))$. As $(\zeta,y)\in[\rS_n(F)\times \rM_n(F)]^-_{\r{rs}}$, we know that $\delta(\zeta,y)$ is odd.

Define two $O_E$-lattices
\begin{align*}
L_1&=L_{\zeta,y_1}\coloneqq O_E y_1 \oplus O_E \zeta y_1 \oplus \cdots \oplus O_E\zeta^{n-1}y_1\subseteq \Mat_{n,1}(E),\\
L_2&=L_{\zeta,y_2}\coloneqq O_E y_2 \oplus O_E y_2\zeta \oplus \cdots \oplus O_E y_2\zeta^{n-1}\subseteq \Mat_{1,n}(E).
\end{align*}
For every integer $i\geq 0$, we define the set
\[
M_i(\zeta,y)\coloneqq\{O_E\text{-lattice } \Lambda\subseteq \Mat_{n,1}(E) \res L_1\subseteq\Lambda, L_2\subseteq \Lambda^\vee,\Lambda^\tc=\Lambda, \zeta\Lambda=\Lambda,  \length_{O_E}(\Lambda/L_1)=i\},
\]
where $\vee$ denotes dual lattice under the standard sesquilinear form
\begin{align}\label{eq:sesqui}
\Mat_{n,1}(E)\times \Mat_{1,n}(E)\to E,\quad (x_1,x_2)\mapsto x_2^\tc\cdot x_1.
\end{align}

\begin{lem}\label{lem:derorb}
We have
\[
\left.\frac{\rd}{\rd s}\right|_{s=0}\Orb(s;\CF_{\rS_n(O_F)},\CF_{\rM_n(O_F)};\zeta,y)=-2\log q\cdot \omega(\zeta,y) \sum_{i\geq 0}(-1)^i(v(\zeta,y)-i)\cdot\#M_i(\zeta,y).
\]
\end{lem}

\begin{proof}
By definition, we have
\[
\Orb(s;\CF_{\rS_n(O_F)},\CF_{\rM_n}(O_F);\zeta,y)=\int_{\GL_n(F)}\CF_{\rS_n(O_F)}(g^{-1}\zeta g)\CF_{\rM_n(O_F)}(g^{-1}y_1, y_2g)\mu_{E/F}(\det g)|\det g|_E^s\rd g.
\]
Notice that $(g^{-1}y_1,y_2g)$ belongs to $\rM_n(O_F)$ if and only if $y_1\in g \Mat_{n,1}(O_F)$ and $y_2\in \Mat_{1,n}(O_F)g^{-1}$ hold. We also notice that $g^{-1}\zeta g$ belongs to $\rS_n(O_F)$ if and only if $\zeta g \Mat_{n,1}(O_E)=g \Mat_{n,1}(O_E)$ and  $\Mat_{1,n}(O_E)g^{-1}\zeta=\Mat_{1,n}(O_E)g^{-1}$ hold. Moreover, the $O_E$-lattice $g\Mat_{n,1}(O_E)$ is invariant under the involution $\tc$, and is dual to $\Mat_{1,n}(O_E)g^{-1}$ under the pairing \eqref{eq:matching}. It follows that the assignment
\[
g\mapsto\Lambda=\Lambda(g)\coloneqq g\Mat_{n,1}(O_E)
\]
induces a bijection between the set
\[
\{g\in \GL_n(F)/\GL_n(O_F) \res (g^{-1}y_1, y_2g)\in\rM_n(O_F), g^{-1}\zeta g\in \rS_n(O_F)\},
\]
and the set
\[
\{O_E\text{-lattice }\Lambda\subseteq \Mat_{n,1}(E) \res y_1\in \Lambda, y_2\in \Lambda^\vee, \Lambda^\tc=\Lambda, \zeta \Lambda=\Lambda\},
\]
in which the latter is equal to
\[
\{O_E\text{-lattice }\Lambda\subseteq \Mat_{n,1}(E) \res L_1\subseteq \Lambda, L_2\subseteq \Lambda^\vee, \Lambda^\tc=\Lambda, \zeta \Lambda=\Lambda\}.
\]
Clearly it further induces a bijection between such elements $g$ with $\val(\det g)=i$ and such $O_E$-lattices $\Lambda$ with $\length_{O_E}(\Lambda/L_1)=i$, namely, the set $M_i(\zeta,y)$. Now notice that we have
\[
\val(\det g)=\length_{O_E}(\Mat_{n,1}(O_E)/L_1)- \length_{O_E}(\Lambda/L_1)
\]
and
\[
\length_{O_E}(\Mat_{n,1}(O_E)/L_1)=\val(\det(y_1, \zeta y_1,\dots, \zeta^{n-1}y_1))=v(\zeta,y).
\]
It follows that if $\length_{O_E}(\Lambda/L_1)=i$, then we have
\[
\mu_{E/F}(\det g)=(-1)^{v(\zeta,y)-i}=(-1)^i\cdot\omega(\zeta,y)
\]
and
\[
\left.\frac{\rd}{\rd s}\right|_{s=0}|\det(g)|_E^s=-2\log q\cdot (v(\zeta,y)-i).
\]
The lemma is proved.
\end{proof}

Define two isomorphisms of $E$-vector spaces
\[
\phi_1\colon \Mat_{n,1}(E)\to \rV^-_n(E),\quad \zeta^i y_1\mapsto \xi^ix,\quad i=0,\ldots, n-1,
\]
and
\[
\phi_2\colon \Mat_{1,n}(E)\to \rV^-_n(E),\quad y_2\zeta^i \mapsto \xi^ix,\quad i=0,\ldots, n-1.
\]
By \eqref{eq:matching}, the standard sesquilinear form \eqref{eq:sesqui} transfers to the hermitian form on $\rV_n^-(E)$ under $\phi_1\times\phi_2$. It is clear that under $\phi_1$, the unique $F$-linear involution $\Mat_{n,1}(E)\to \Mat_{n,1}(E)$ sending $a\cdot\zeta^iy_1$ to $a^\tc\cdot (\zeta^i)^\tc y_1=a^\tc\cdot\zeta^{-i}y_1$ for every $a\in E$ and $i=0,\dots,n-1$ transfers to the unique $F$-linear involution $\tau: \rV_n^-(E)\to\rV_n^-(E)$ satisfying $\tau(a\cdot\xi^ix)=a^\tc\cdot\xi^{-i}x$ for every $a\in E$ and $i=0,\dots,n-1$.

Define the $O_E$-lattice
\[
L=L_{\xi,x}\coloneqq O_Ex \oplus O_E\xi x \oplus \cdots \oplus O_E \xi^{n-1}x\subseteq \rV_n^-(E).
\]
Then we have $\phi_i(L(\zeta,y_i))=L$ for $i=1,2$. For every integer $i\ge0$, we define the set
\[
N_i(\zeta,y)\coloneqq\{O_E\text{-lattice } \Lambda\subseteq \rV_n^-(E) \res L\subseteq \Lambda\subseteq L^*, \xi\Lambda=\Lambda, \Lambda^\tau=\Lambda, \length_{O_E}(\Lambda/L)=i\},
\]
where $*$ denotes dual lattice under the hermitian form on $\rV_n^-(E)$.

\begin{proposition}\label{pro:derorb}
We have
\[
\left.\frac{\rd}{\rd s}\right|_{s=0}\Orb(s;\CF_{\rS_n(O_F)},\CF_{\rM_n(O_F)}; \zeta,y)
=-2\log q\cdot \omega(\zeta,y)\sum_{i=0}^{\delta(\zeta,y)}(-1)^i(-i)\cdot\#N_i(\zeta,y).
\]
\end{proposition}

\begin{proof}
Notice that the isomorphisms $\phi_1$ and $\phi_2$ induce a bijection between the sets $M_i(\zeta,y)$ and $N_i(\zeta,y)$ for every $i$. As $\length_{O_E}(L^*/L)=\delta(\zeta,y)$, we know that $N_i$ is empty unless $0\le i\le \delta(\zeta,y)$. Moreover, the assignment $\Lambda\mapsto \Lambda^*$ induces an isomorphism between $N_i(\zeta,y)$ and $N_{\delta(\zeta,y)-i}(\zeta,y)$. Since $\delta(\zeta,y)$ is odd, we have
\[
\sum_{i=0}^{\delta(\zeta,y)}(-1)^iv(\zeta,y)=0.
\]
Thus, we have
\[
\sum_{i=0}^{\delta(\zeta,y)}(-1)^i(v(\zeta,y)-i)\cdot\#M_i(\zeta,y)=\sum_{i=0}^{\delta(\zeta,y)}(-1)^i(-i)\cdot\#N_i(\zeta,y).
\]
The proposition then follows from Lemma \ref{lem:derorb}.
\end{proof}

\begin{remark}
There seems to be a sign error in \cite{RTZ}*{Corollary~7.3(2)}, which is corrected in the more general Proposition~\ref{pro:derorb}.
\end{remark}

\subsection{The minuscule case}

Choose a uniformizer $\varpi$ of $F$. From now on we assume that $(\xi,x)$ is \emph{minuscule}, namely, we assume
\[
\varpi L^*\subseteq L\subseteq L^*,
\]
where $L=L_{\xi,x}$ as we recall. In this case, $L^*/L$ is a vector space over the residue field $\kappa_E\cong\mathbb{F}_{q^2}$ of $E$, which is equipped with an hermitian form induced from $\rV^-_n$. Since $\xi\in\rU(\rV_n^-)(F)$ stabilizes $L$ and $L^*$, we know that $\xi$ induces an action $\bar\xi$ on $L^*/L$, which is an element in $\rU(L^*/L)$. We denote by $P(T)$ the characteristic polynomial of $\bar\xi$ on $L^*/L$. Since $\bar\xi$ belongs to $\rU(L^*/L)$, we know that $P(T)$ is \emph{self-reciprocal}. Here we recall that for a polynomial
\[
R(T)=a_k T^k+\cdots +a_1 T+a_0\in \kappa_E[T]
\]
with $a_0a_k\neq 0$, we define its \emph{reciprocal polynomial} as
\[
R^*(T)\coloneqq(a_0^\tc)^{-1}\cdot T^k\cdot R(1/T)^\tc;
\]
and we say that $R(T)$ is \emph{self-reciprocal} if $R(T)=R^*(T)$.

Now for any irreducible factor $R(T)$ of $P(T)$, for $P(T)$ defined above, we denote the multiplicity of $R(T)$ in $P(T)$ by $m(R(T))$. Since $P(T)$ is self-reciprocal, if $R(T)$ is an irreducible factor of $P(T)$, then $R^*(T)$ is also an irreducible factor of $P(T)$. Thus, taking reciprocal $R(T)\mapsto R^*(T)$ induces an involution on the set of irreducible factors of $P(T)$. We denote by $\mathsf{NSR}$ the set of all orbits of \emph{non-self-reciprocal} monic irreducible factors of $P(T)$ under this involution.

\begin{lem}\label{le:minanalytic}
If $P(T)$ has a unique self-reciprocal monic irreducible factor $Q(T)$ such that $m(Q(T))$ is odd, then
\[
\sum_{i=0}^{\delta(\zeta,y)}(-1)^i(-i)\cdot\#N_i(\zeta,y)=\deg Q(T)\cdot\frac{m(Q(T))+1}{2}\cdot \prod_{\{R(T),R^*(T)\}\in\mathsf{NSR}}(1+m(R(T))).
\]
Otherwise, we have
\[
\sum_{i=0}^{\delta(\zeta,y)}(-1)^i(-i)\cdot\#N_i(\zeta,y)=0.
\]
\end{lem}

\begin{proof}
This follows from the same proof as \cite{RTZ}*{Proposition~8.2}.
\end{proof}

Put $\Lambda\coloneqq L^*$. Since $(\xi,x)$ is minuscule, we know that $\Lambda$ is a vertex lattice, namely, it satisfies $\varpi\Lambda\subseteq\Lambda^*\subseteq\Lambda$. Let $\cV(\Lambda)$ be the Deligne--Lusztig variety associated to the vertex lattice $\Lambda$ as in \cite{LZ17}*{Section~2.5} and \cite{RTZ}*{Section~3}, which is a smooth projective variety over $k$, where $k$ is the residue field of $\breve{E}$ as in Subsection \ref{ss:afl}.

\begin{lem}\label{le:fixedpoints}
We have a canonical isomorphism
\[
\Gamma_\xi\cap\Delta\cZ_n(x)\cong \cV(\Lambda)^{\bar\xi}
\]
of $k$-schemes.
\end{lem}

\begin{proof}
Notice that we have a canonical isomorphism $\Gamma_\xi\cap\Delta\cZ_n(x)\cong \cZ_n(x)\cap \cN_n^\xi$. Let $\cN_\Lambda\subseteq \cN_n$ be the closed Bruhat--Tits stratum associated to the vertex lattice $\Lambda$ as in \cite{LZ17}*{Section~2.6}. Then by definition, we have
\[
\cZ_n(x)\cap\cN_n^\xi\cong \cN_\Lambda^\xi.
\]
By \cite{LZ17}*{Corollary~3.2.3 \& Section~2.6}, we have $\cN_\Lambda\cong \cV(\Lambda)$. The lemma then follows.
\end{proof}

\begin{lem}\label{le:minarithmetic}
We have that $\cV(\Lambda)^{\bar\xi}$ is empty unless $P(T)$ has a unique self-reciprocal monic irreducible factor $Q(T)$ such that $m(Q(T))$ is odd. Assume that $\cV(\Lambda)^{\bar\xi}$ is non-empty. Then $\cV(\Lambda)^{\bar\xi}$ is an Artinian $k$-scheme, and
\[
\chi(\cO_{\Gamma_\xi}\otimes_{\cO_{\cN_n^2}}^\dL\cO_{\Delta\cZ_n(x)})=\length_k\cV(\Lambda)^{\bar\xi}=\deg Q(T)\cdot\frac{m(Q(T))+1}{2}\cdot \prod_{\{R(T),R^*(T)\}\in\mathsf{NSR}}(1+m(R(T))).
\]
\end{lem}

\begin{proof}
The result follows directly from Lemma \ref{le:fixedpoints}, \cite{LZ17}*{Corollary~3.2.3}, \cite{RTZ}*{Proposition~8.1}, and \cite{HLZ19}*{Lemma~5.1.1 \& Theorem~4.6.3}. Strictly speaking, these references assume $F=\dQ_p$, but the same proof works for general $F$ as long as one replaces results related to the Bruhat--Tits stratification and special cycles by more general ones in \cite{Cho}.
\end{proof}

\begin{theorem}
Conjecture \ref{co:afl_general} holds when $(\xi,x)$ is minuscule.
\end{theorem}

\begin{proof}
This follows immediately from Proposition \ref{pro:derorb}, Lemma \ref{le:minanalytic}, and Lemma \ref{le:minarithmetic}.
\end{proof}

\section{Poles of Eisenstein series and theta lifting for unitary groups}
\label{ss:b}

In this appendix, we prove some results about global theta lifting for unitary groups, namely, Theorem \ref{th:pole} and its two corollaries. These results are only used in the proof of Proposition \ref{pr:endoscopy_general}. Thus, if the readers are willing to admit these results from the theory of automorphic forms, they are welcome to skip the entire section except the very short Subsection \ref{ss:discrete} where we introduce some notation for the discrete automorphic spectrum.

\subsection{Discrete automorphic spectrum}
\label{ss:discrete}

We recall some setup about the discrete automorphic spectrum. Let $\rG$ be a reductive group over a number field $F$. Let $Z_\rG$ be the center of $\rG$. For an automorphic character $\chi\colon Z_\rG(F)\backslash Z_\rG(\bA_F)\to\dC^\times$, we denote by $\rL^2(\rG(F)\backslash\rG(\bA_F),\chi)$ the space of measurable complex valued functions $f$ on $\rG(F)\backslash\rG(\bA_F)$ satisfying $f(gz)=\chi(z)f(g)$ for $z\in Z_\rG(\bA_F)$ such that $|f(g)\chi'(g)|^2$ is integrable on $\rG(F)\backslash\rG(\bA_F)/Z_\rG(\bA_F)$ for some (hence every) character $\chi'\colon\rG(F)\backslash\rG(\bA_F)\to\dC^\times$ such that $\chi\cdot(\chi'\res{Z_\rG(\bA_F)})$ is unitary. The group $\rG(\bA_F)$ acts on $\rL^2(\rG(F)\backslash\rG(\bA_F),\chi)$ by the right translation. Denote by $\rL^2_\disc(\rG(F)\backslash\rG(\bA_F),\chi)$ the maximal closed subspace of $\rL^2(\rG(F)\backslash\rG(\bA_F),\chi)$ that is a direct sum of irreducible (closed) subrepresentations of $\rG(\bA_F)$. We put
\[
\rL^2_\disc(\rG)\coloneqq\bigoplus_\chi\rL^2_\disc(\rG(F)\backslash\rG(\bA_F),\chi)
\]
where $\chi$ runs through all automorphic characters of $Z_\rG(\bA_F)$. Finally, denote by $\rL^2_\cusp(\rG)$ the subspace of $\rL^2_\disc(\rG)$ consisting of cuspidal functions. Both $\rL^2_\disc(\rG)$ and $\rL^2_\cusp(\rG)$ are representations of $\rG(\bA_F)$ via the right translation.

\begin{definition}\label{de:realization}
Let $\pi$ be an irreducible admissible representation of $\rG(\bA_F)$.
\begin{enumerate}
  \item We define the \emph{discrete (resp.\ cuspidal) multiplicity} $m_\disc(\pi)$ (resp.\ $m_\cusp(\pi)$) of $\pi$ to be the dimension of $\Hom_{\rG(\bA_F)}(\pi,\rL^2_\disc(\rG))$ (resp.\ $\Hom_{\rG(\bA_F)}(\pi,\rL^2_\cusp(\rG))$).

  \item We define a \emph{discrete (resp.\ cuspidal) realization} of $\pi$ to be an irreducible subrepresentation $V_\pi$ contained in $\rL^2_\disc(\rG)$ (resp.\ $\rL^2_\cusp(\rG)$) that is isomorphic to $\pi$.
\end{enumerate}
\end{definition}

It is known that $0\leq m_\cusp(\pi)\leq m_\disc(\pi)<\infty$.

\subsection{Main theorem and consequences}
\label{ss:pole}

Now we let $F$ be a totally real number field, and $E/F$ a totally imaginary quadratic extension. Denote by $\tc$ the nontrivial involution of $E$ over $F$.

\begin{definition}\label{de:strictly_unitary}
We say that an automorphic character $\mu\colon E^\times\backslash\bA_E^\times\to\dC^\times$ is \emph{strictly unitary} if $\mu_\infty$ takes value $1$ on the diagonal $\Delta^{[F:\dQ]}\dR^\times_{>0}\subseteq(\dR^\times_{>0})^{[F:\dQ]}$ as a subgroup of $E_\infty^\times\subseteq\bA_E^\infty$.
\end{definition}

\begin{remark}
It is clear that a strictly unitary automorphic character is unitary. For every automorphic character $\mu$ of $\bA_E^\times$, there exists a unique complex number $s$ such that $\mu|\;|^s_E$ is strictly unitary.
\end{remark}

Let $\rV,(\;,\;)_\rV$ be a (non-degenerate) hermitian space over $E$ (with respect to $\tc$) of rank $n$ and let $\rW,\langle\;,\;\rangle_\rW$ be a (non-degenerate) skew-hermitian space over $E$ (with respect to $\tc$) of rank $m$. Let $\rG\coloneqq\rU(\rV)$ and $\rH\coloneqq\rU(\rW)$ be the unitary groups of $\rV$ and $\rW$, respectively. We form the symplectic space $\Res_{E/F}\rV\otimes_E\rW$, and let $\Mp(\Res_{E/F}\rV\otimes_E\rW)$ be the metaplectic cover of $\Sp(\Res_{E/F}\rV\otimes_E\rW)(\bA_F)$ with center $\dC^1$. Then we have the oscillator representation $\omega$ of $\Mp(\Res_{E/F}\rV\otimes_E\rW)$ with respect to the standard additive character $\psi_F$.\footnote{In this article, we will always use $\psi_F$ to form oscillator representations. Thus, in the sequel, we will no longer mention the dependence of $\psi_F$ when discussing oscillator representations.} Let $\underline\mu=(\mu_\rV,\mu_\rW)$ be a pair of splitting characters for $(\rV,\rW)$, that is, $(\mu_\rV,\mu_\rW)$ is a pair of automorphic characters of $\bA_E^\times$ satisfying $\mu_\rV\res\bA_F^\times=\mu_{E/F}^m$ and $\mu_\rW\res\bA_F^\times=\mu_{E/F}^n$. Then it induces an embedding $\iota_{\underline\mu}\colon\rG(\bA_F)\times\rH(\bA_F)\hookrightarrow\Mp(\Res_{E/F}\rV\otimes_E\rW)$. By restriction, we obtain the Weil representation
\begin{align}\label{eq:weil}
\omega_{\underline\mu}^{\rV,\rW}\coloneqq\omega\circ\iota_{\underline\mu}
\end{align}
of $\rG(\bA_F)\times\rH(\bA_F)$. It induces the global theta lifting map $\Theta_{\underline\mu,\rV}^\rW$: For an irreducible smooth subrepresentation $V\subseteq\rL^2_\cusp(\rG)$ of $\rG(\bA_F)$, we obtain a subrepresentation $\Theta_{\underline\mu,\rV}^\rW(V)\subseteq\sC^\infty(\rH(F)\backslash\rH(\bA_F),\dC)$ of $\rH(\bA_F)$. More precisely, there is a space of theta functions $\theta_{\underline\mu}(g,h)$ on $\rG(F)\backslash\rG(\bA_F)\times\rH(F)\backslash\rH(\bA_F)$, which is an automorphic realization of $\omega_{\underline\mu}^{\rV,\rW}$. Then $\Theta_{\underline\mu,\rV}^\rW(V)$ is spanned by functions
\[
h\mapsto \int_{\rG(F)\backslash\rG(\bA_F)}\theta_{\underline\mu}(g,h)f(g) dg
\]
on $\rH(F)\backslash\rH(\bA_F)$ for $f\in V$. Similarly, we have the reverse global theta lifting map $\Theta_{\underline\mu,\rW}^\rV$.

We consider an automorphic representation $\pi$ of $\rG(\bA_F)$ and a strictly unitary automorphic character $\mu\colon E^\times\backslash\bA_E^\times\to\dC^\times$. We study three objects associated to $\pi$ and $\mu$ as follows.
\begin{itemize}
  \item Let $S$ be a finite set of places of $F$ containing all archimedean ones and such that for $v\not\in S$, both $\pi_v$ and $\mu_v$ are unramified. We then have the partial standard $L$-function $L^S(s,\pi\times\mu)$.

  \item Let $\rG_1$ be the unitary group of the hermitian space $\rV_1\coloneqq\rV\oplus\rD$, where $\rD$ is the hyperbolic hermitian plane, let $\rQ$ be a parabolic subgroup of $\rG_1$ stabilizing an isotropic line in $\rD$, and let $\bK\subseteq\rG_1(\bA_F)$ be a maximal compact subgroup such that the Cartan decomposition $\rG_1(\bA_F)=\rQ(\bA_F)\bK$ holds. Let $V_\pi$ be a cuspidal realization of $\pi$ (Definition \ref{de:realization}). Let $\rI(V_\pi\boxtimes\mu^\tc)$ be the space of functions $f$ on $\rG_1(\bA_F)$ such that for every $k\in\bK$ the function $p\mapsto f(pk)$ is a $\bK\cap\rQ(\bA_F)$-finite vector in $V_\pi\boxtimes(\mu^\tc\cdot|\;|_E^{(n+1)/2})$.\footnote{Here, we regard vectors in $V_\pi\boxtimes(\mu^\tc\cdot|\;|_E^{(n+1)/2})$ as functions on $\rQ(\bA_F)$ via the Levi quotient map.} For every $f\in\rI(V_\pi\boxtimes\mu^\tc)$, we can form an Eisenstein series $\sE_\rQ(g;f_s)$ normalized such that $\RE(s)=0$ is the unitary line (see \cite{Sha88}*{Section~2} for details). By Langlands theory of Eisenstein series \cites{Lan71,MW95}, $\sE_\rQ(g;f_s)$ is absolutely convergent for $\RE(s)>\frac{n+1}{2}$ and has a meromorphic continuation to the entire complex plane.

  \item Let $V_\pi$ be a cuspidal realization of $\pi$ (Definition \ref{de:realization}). Then we have the global theta lifting $\Theta_{(\mu,\nu),\rV}^\rW(V_\pi)$. We will adopt the convention that if $\mu\res\bA_F^\times\neq\mu_{E/F}^m$ with $m\coloneqq\dim_E\rW$, then $\Theta_{(\mu,\nu),\rV}^\rW(V_\pi)=0$.
\end{itemize}

We have the following theorem, which is the unitary version of a weaker form of \cite{GJS}*{Theorem~1.1}.

\begin{theorem}\label{th:pole}
Let $\pi$ be an irreducible admissible representation of $\rG(\bA_F)$, let $V_\pi$ be a cuspidal realization of $\pi$ (Definition \ref{de:realization}), and let $\mu\colon E^\times\backslash\bA_E^\times\to\dC^\times$ be a strictly unitary automorphic character. We have
\begin{enumerate}
  \item For $s_0\in\dC$ with $\RE(s_0)>0$, consider the following statements:
     \begin{enumerate}[label=(\alph*)]
       \item $L^S(s,\pi\times\mu)\cdot L^S(2s,\mu,\r{As}^{(-1)^n})$ has a pole at $s_0$, where $\r{As}^\pm$ stand for the two Asai representations (see, for example, \cite{GGP}*{Section~7}).

       \item $\{\sE_\rQ(g;f_s)\res f\in\rI(V_\pi\boxtimes\mu^\tc)\}$ has a pole at $s_0+j$ for some integer $j\geq 0$.

       \item $\Theta_{(\mu,\nu),\rV}^\rW(V_\pi)\neq 0$ for some skew-hermitian space $\rW$ of dimension $n+1-2s_0$ and some $\nu$ with $\nu\res\bA_F^\times=\mu_{E/F}^n$.\footnote{This property is independent of the choice of such $\nu$ since changing $\nu$ results in a twist of $\Theta_{(\mu,\nu),\rV}^\rW(V_\pi)$ by a character.}
     \end{enumerate}
     Then (a) $\Rightarrow$ (b) $\Rightarrow$ (c).



  \item The skew-hermitian space $\rW$ in (1c) is unique up to isomorphism.
\end{enumerate}
\end{theorem}

We will prove the theorem in Subsection \ref{ss:proof_pole}.

\begin{corollary}\label{co:pole1}
Denote the set of poles of $L^S(s,\pi\times\mu)\cdot L^S(2s,\mu,\r{As}^{(-1)^n})$ in the region $\RE(s)>0$ by $\r{Pol}_{\pi,\mu}^S$. Then
\begin{enumerate}
  \item If $\mu$ is not conjugate self-dual, then $\r{Pol}_{\pi,\mu}^S$ is empty.

  \item If $\mu$ is conjugate orthogonal (Definition \ref{de:conjugate}), then $\r{Pol}_{\pi,\mu}^S$ is contained in the set $\{\frac{n+1}{2},\frac{n-1}{2},\dots,\frac{n+1}{2}-\lfloor\frac{n}{2}\rfloor\}$.

  \item If $\mu$ is conjugate symplectic (Definition \ref{de:conjugate}), then $\r{Pol}_{\pi,\mu}^S$ is contained in the set $\{\frac{n}{2},\frac{n-2}{2},\dots,\frac{n}{2}-\lfloor\frac{n-1}{2}\rfloor\}$.
\end{enumerate}
\end{corollary}

\begin{proof}
This is a direct consequence of Theorem \ref{th:pole}.
\end{proof}

Let $V_\pi$ be a cuspidal realization of $\pi$, and suppose that $\{\sE_\rQ(g;f_s)\res f\in\rI(V_\pi\boxtimes\mu^\tc)\}$ has the largest pole at $s_{\r{max}}$. By Theorem \ref{th:pole}, there is a skew-hermitian space $\rW$ of dimension $n+1-2s_{\r{max}}$, unique up to isomorphism, such that $\Theta_{(\mu,\nu),\rV}^\rW(V_\pi)$ is nonzero.

\begin{corollary}\label{co:pole2}
Let the notation be as above. Suppose that $\Theta_{(\mu,\nu),\rV}^\rW(V_\pi)$ is cuspidal. Then
\begin{enumerate}
  \item The space $\Theta_{(\mu,\nu),\rV}^\rW(V_\pi)$ is an irreducible representation of $\rU(\rW)(\bA_F)$; and
      \begin{align}\label{eq:pole}
      V_\pi=\Theta_{(\mu^{-1},\nu^{-1}),-\rW}^\rV(\Theta_{(\mu,\nu),\rV}^\rW(V_\pi)),
      \end{align}
      where $-\rW,\langle\;,\;\rangle_{-\rW}$ denotes the skew-hermitian space $\rW,-\langle\;,\;\rangle_\rW$, and we naturally identify $\rU(\rW)$ with $\rU(-\rW)$.

  \item The space $\sR_{s_{\r{max}}}(V_\pi\boxtimes\mu^\tc)$ generated by residues of $\{\sE_\rQ(g;f_s)\res f\in\rI(V_\pi\boxtimes\mu^\tc)\}$ at $s=s_{\r{max}}$ is an irreducible representation of $\rG_1(\bA_F)$; and
      \begin{align}\label{eq:pole1}
      \sR_{s_{\r{max}}}(V_\pi\boxtimes\mu^\tc)=\Theta^{\rV_1}_{(\mu^{-1},\nu^{-1}),-\rW}(\Theta^\rW_{(\mu,\nu),\rV}(V_\pi)).
      \end{align}

  \item In the situation of (1) (resp.\ (2)), let $\pi_\rW$ (resp.\ $\pi_1$) be the underlying (irreducible) representation of $\Theta_{(\mu,\nu),\rV}^\rW(V_\pi)$ (resp.\ $\sR_{s_{\r{max}}}(V_\pi\boxtimes\mu^\tc)$). If $\pi_\rW$ has a unique realization as a subquotient in the space of automorphic forms on $\rU(\rW)$, then $m_\cusp(\pi)=1$ (resp.\ $m_\cusp(\pi_1)=0$).
\end{enumerate}
\end{corollary}

\begin{proof}
Put $m\coloneqq n+1-2s_{\r{max}}$ for simplicity.

For (1), the irreducibility follows from \cite{Wu13}*{Theorem~5.3}, and \eqref{eq:pole} follows from \cite{Wu13}*{Theorem~5.1}.

For (2), since the space generated by the constant terms of forms in $\sR_{s_{\r{max}}}(V_\pi\boxtimes\mu^\tc)$ is an irreducible representation of $\rM_\rQ(\bA_F)$, where $\rM_\rQ$ is the Levi quotient of $\rQ$, the space $\sR_{s_{\r{max}}}(V_\pi\boxtimes\mu^\tc)$ is an irreducible representation of $\rG_1(\bA)$. Then \eqref{eq:pole1} follows from \cite{Wu13}*{Proposition~5.9}.

For (3), we first study $m_\cusp(\pi)$. Let $V'_\pi$ be an arbitrary cuspidal realization of $\pi$. By Theorem \ref{th:pole}, there exists a skew-hermitian space $\rW'$ of the same dimension $m$ such that $\Theta_{(\mu,\nu),\rV}^{\rW'}(V'_\pi)\neq 0$. As $m\leq n$, by the local theta dichotomy \cite{SZ15}*{Theorem~1.10},\footnote{At a nonarchimedean place, the local theta dichotomy is also proved in \cite{GG11}.} we have $\rW'\simeq\rW$. By the Howe duality \cite{GT16}*{Theorem~1.2}, the underlying representation of $\Theta_{(\mu,\nu),\rV}^{\rW'}(V'_\pi)$ must be isomorphic to a finite sum of $\pi_\rW$. Then the assumption of $\pi_\rW$ implies $\Theta_{(\mu,\nu),\rV}^\rW(V_\pi)=\Theta_{(\mu,\nu),\rV}^{\rW'}(V'_\pi)$. Thus, $V_\pi=V'_\pi$ by \cite{Wu13}*{Theorem~5.1}. In particular, $m_\cusp(\pi)=1$.

Then we study $m_\cusp(\pi_1)$. Let $V_{\pi_1}$ be a cuspidal realization of $\pi_1$. By the identity $L^S(s,\pi_1)=L^S(s,\pi)\cdot L^S(s-s_{\r{max}},\mu^\tc)$, we know that $L^S(s,\pi_1\times\mu)$ has a pole at $s_{\r{max}}+1$. By Theorem \ref{th:pole}, there exists a skew-hermitian space $\rW_1$ of dimension $(n+2)+1-2(s_{\r{max}}+1)=m$ such that $\Theta_{(\mu,\nu),\rV_1}^{\rW_1}(V_{\pi_1})\neq 0$. Again by the local theta dichotomy and the Howe correspondence, we have $\rW_1\simeq\rW$ and that the underlying representation of $\Theta_{(\mu,\nu),\rV}^{\rW_1}(V_{\pi_1})$ must be isomorphic to a finite sum of $\pi_\rW$. Then the assumption of $\pi_\rW$ implies  $\Theta_{(\mu,\nu),\rV_1}^{\rW_1}(V_{\pi_1})=\Theta^\rW_{(\mu,\nu),\rV}(V_\pi)$. Thus, $V_{\pi_1}=\Theta^{\rV_1}_{(\mu^{-1},\nu^{-1}),-\rW}(\Theta^\rW_{(\mu,\nu),\rV}(V_\pi))$ by \cite{Wu13}*{Theorem~5.1}, which is simply $\sR_{s_{\r{max}}}(V_\pi\boxtimes\mu)$ by \eqref{eq:pole1}. This is a contradiction. Therefore, $m_\cusp(\pi_1)=0$.
\end{proof}

\begin{remark}
In fact, in Corollary \ref{co:pole2}, the space $\Theta_{(\mu,\nu),\rV}^\rW(V_\pi)$ is always cuspidal, which follows from an analogous statement of \cite{GJS}*{Theorem~5.1}, whose proof can be adopted to the unitary case as well. Since we do not need this fact, we will leave the details to interested readers as an exercise.
\end{remark}

\subsection{Proof of Theorem \ref{th:pole}}
\label{ss:proof_pole}

We follow the strategy in \cite{GJS}. We first prove the following proposition, which is a part of Theorem \ref{th:pole}.

\begin{proposition}\label{pr:pole}
Suppose that $\mu^\tc=\mu^{-1}$. Let $s_1$ be the maximal positive real pole of $\{\sE_\rQ(g;f_s)\res f\in\rI(V_\pi\boxtimes\mu^\tc)\}$. Then
\begin{enumerate}
  \item There is some skew-hermitian space $\rW$ of dimension $n+1-2s_1$ such that $\Theta_{(\mu,\nu),\rV}^\rW(V_\pi)\neq 0$.

  \item All other positive real poles of $\sE_\rQ(g;f_s)$ have the form $s_1-j$ for some integer $j\geq 0$.
\end{enumerate}
\end{proposition}

Part (1) of this proposition is the unitary version of \cite{GJS}*{Theorem~3.1}. The proof is very similar to the argument in \cite{Moe97}*{Section~2.1} and \cite{GJS}*{Section~3}, which are for orthogonal groups. We will only sketch the proof with necessary modification for the unitary case.

We first introduce some notation. Fix a polarization $\rD=\delta^+\oplus\delta^-$ of the hyperbolic hermitian plane $\rD$. For an integer $a\geq 0$, put $\delta^\pm_a\coloneqq(\delta^\pm)^{\oplus a}$ and $\rV_a\coloneqq\rV\oplus(\delta^+_a\oplus\delta^-_a)$. Put $\rG_a\coloneqq\rU(\rV_\ra)$ and let $\rQ_a\subseteq\rG_a$ be the parabolic subgroup stabilizing the subspace $\delta^+_a$. In particular, we may identify $\rQ_1$ with $\rQ$. Note that the Levi quotient of $\rQ_a$ is isomorphic to $\rG\times\Res_{E/F}\GL_a$. In particular, we have the space of functions $V_\pi\boxtimes(\mu^\tc\cdot|\;|_E^{(n+a)/2})\circ\det_a$ on $\rQ_a(\bA_F)$, where $\det_a\colon\GL_a\to\dG_\rm$ is the determinant map. Similar to $\rI(V_\pi\boxtimes\mu^\tc)$, we have the space $\rI_a(V_\pi\boxtimes\mu^\tc)$ of functions on $\rG_a(\bA_F)$; and for $f_a\in\rI_a(V_\pi\boxtimes\mu^\tc)$, one can form the Eisenstein series $\sE_{\rQ_a}(\quad;f_{a,s})$ on $\rG_a(\bA_F)$, which is absolutely convergent for $\RE(s)>\frac{n+a}{2}$. In particular, $\rI_1(V_\pi\boxtimes\mu^\tc)=\rI(V_\pi\boxtimes\mu^\tc)$. Let $\r{Pol}_a(V_\pi\boxtimes\mu^\tc)$ be the set of positive real poles of $\sE_{\rQ_a}(\quad;f_{a,s})$. Then $s_1$ is the largest number in $\r{Pol}_1(V_\pi\boxtimes\mu^\tc)$ by our assumption.

\begin{lem}\label{le:pole1}
Let $s_0$ be an element in $\r{Pol}_1(V_\pi\boxtimes\mu^\tc)$ such that $s_0+j\not\in\r{Pol}_1(V_\pi\boxtimes\mu^\tc)$ for every integer $j>0$. Then $s_0+\frac{a-1}{2}$ lies in $\r{Pol}_a(V_\pi\boxtimes\mu^\tc)$.
\end{lem}

\begin{proof}
This is the unitary analogue of \cite{Moe97}*{Remarque~1.1} and \cite{GJS}*{Proposition~1.1}. The argument for \cite{Moe97}*{Remarque~1.1} works in the unitary case as well. However, we would like to remark that in \cite{Moe97}*{Remarque~1.1}, the author assumes that $s_0$ is the maximal element of $\r{Pol}_1(V_\pi\boxtimes\mu^\tc)$. This is unnecessary since the argument only uses the fact that $s_0+j\not\in\r{Pol}_1(V_\pi\boxtimes\mu^\tc)$ for every integer $j>0$.
\end{proof}

Now we recall the generalized doubling method for unitary groups. Again let $-\rV$ be the hermitian space with the negative hermitian form on $\rV$. Let $\rV^\diamond$ be the doubling space $\rV\oplus(-\rV)$. For an integer $a\geq 0$, put
\[
\rV^\diamond_a\coloneqq\rV^\diamond\oplus(\delta^+_a\oplus\delta^-_a)=\rV_a\oplus(-\rV).
\]
Via this decomposition, we have a canonical embedding
\[
\iota\colon\rG_a\times\rG\to\rU(\rV^\diamond_a),
\]
where we have identified $\rG$ with $\rU(-\rV)$. Put $\rV^\pm\coloneqq\{(v,\pm v)\in\rV^\diamond\res v\in\rV\}$ and $\rV^\pm_a\coloneqq\rV^\pm\oplus\delta_a^\pm$. Let $\rP_a$ be the parabolic subgroup of $\rU(\rV^\diamond_a)$ stabilizing the maximal totally isotropic subspace $\rV^+_a$ of $\rV^\diamond_a$. Then the Levi quotient of $\rP_a$ is isomorphic to $\Res_{E/F}\GL_{n+a}$. We have the space of degenerate series $\rJ_a(s,\mu^\tc)$ as the normalized induced representation $\Ind_{\rP_a(\bA_F)}^{\rU(\rV^\diamond_a)(\bA_F)}(\mu^\tc\cdot|\;|^s)\circ\det_{n+a}$. Let $f^\diamond_{a,s}$ be a standard section in $\rJ_a(s,\mu^\tc)$. Then we can form the Siegel--hermitian Eisenstein series $\sE_{\rP_a}(\quad;f^\diamond_{a,s})$ on $\rU(\rV^\diamond_a)(\bA_F)$, which is absolutely convergent for $\RE(s)>\frac{n+a}{2}$. See \cite{Tan99}*{Section~1} for more details.

Now for a standard section $f^\diamond_{a,s}\in\rJ_a(s,\mu^\tc)$ and a cusp form $\phi\in V_\pi$, we have the function
\begin{align}\label{eq:pole0}
f^{\diamond,\phi}_{a,s}(g')\coloneqq\int_{\rG(\bA_F)}f^\diamond_{a,s}(\iota(g^{-1}g',1))\phi(g)\rd g
\end{align}
on $\rG_a(\bA_F)$. The following lemma is analogous to \cite{GJS}*{Proposition~3.2}.

\begin{lem}\label{le:pole2}
Suppose that $\mu\res\bA_F^\times=\mu_{E/F}^i$ for $i\in\{0,1\}$. We have
\begin{enumerate}
  \item The poles of the Siegel--hermitian Eisenstein series $\sE_{\rP_a}(\quad;f^\diamond_{a,s})$ in the region $\RE(s)>0$ are all simple, and are contained in the set $\{\frac{n+a-i}{2},\frac{n+a-i}{2}-1,\dots\}$.

  \item The integral \eqref{eq:pole0} is absolutely convergent for $\RE(s)>\frac{n+a}{2}$.

  \item The function $f^{\diamond,\phi}_{a,s}$ has a meromorphic continuation to the entire complex plane, whose possible poles in the region $\RE(s)>0$ are contained in the set $\{\frac{n-i}{2},\frac{n-i}{2}-1,\dots\}$.

  \item If $s$ is not a pole of $f^{\diamond,\phi}_{a,s}$, then $f^{\diamond,\phi}_{a,s}$ is a section in the normalized induced representation $\Ind^{\rG_a(\bA_F)}_{\rQ_a(\bA_F)}V_\pi\boxtimes(\mu^\tc\cdot|\;|_E^s)\circ\det_a$.
\end{enumerate}
\end{lem}

\begin{proof}
Part (1) follows from Main Theorem of \cite{Tan99}. The proof of (2--4) is same as in \cite{Moe97}*{Section~2.1}. In particular, the poles of $f^{\diamond,\phi}_{a,s}$ are contained in the set of poles of the Eisenstein series $\sE_{\rP_0}(g;f_s\res\rG(\bA_F))$. Thus, (3) follows from Main Theorem of \cite{Tan99}.
\end{proof}

The following lemma is analogous to \cite{Moe97}*{Proposition~2.1} and \cite{GJS}*{Proposition~3.3}.

\begin{lem}\label{le:pole3}
For a standard section $f^\diamond_{a,s}\in\rJ_a(s,\mu^\tc)$ and a cusp form $\phi\in V_\pi$, we have the identity
\[
\int_{\rG(F)\backslash\rG(\bA_F)}\sE_{\rP_a}(\iota(g',g);f^\diamond_{a,s})\phi(g)\mu(\det g)\rd g=\sE_{\rQ_a}(g';f^{\diamond,\phi}_{a,s})
\]
for $g'\in\rG_a(\bA_F)$, as meromorphic functions in $s$ away from the poles of $f^{\diamond,\phi}_{a,s}$.
\end{lem}

\begin{proof}
The proof is almost same to the argument on \cite{Moe97}*{p.214--215}. We will sketch the process. To ease notation, we identify $\rG_a\times\rG$ as a subgroup of $\rU(\rV_a^\diamond)$ via $\iota$. We consider the double coset
\begin{align}\label{eq:pole4}
\rP_a(F)\backslash\rU(\rV^\diamond_a)(F)/\rG_a(F)\times\rG(F).
\end{align}
We identify $\rP_a(F)\backslash\rU(\rV^\diamond_a)(F)$ with the set of maximal isotropic subspaces of $\rV^\diamond_a$. Let $L$ be such a subspace. Put $d_L\coloneqq\dim_E(L\cap(-\rV))$. Then $L$ and $L'$ are in the same double coset of \eqref{eq:pole4} if and only if $d_L=d_{L'}$. In other words, we have a canonical bijection between \eqref{eq:pole4} and $\{0,1,\dots,r\}$ where $r$ is the Witt index of $\rV$. Moreover, the identity double coset corresponds to $0$. For every $d=0,1,\dots,r$, we fix a representative $\gamma_d$ of the corresponding double coset (we take $\gamma_0$ to be the identity matrix). Then for $g'\in\rG_a(\bA_F)$, we have
\begin{align*}
&\int_{\rG(F)\backslash\rG(\bA_F)}\sE_{\rP_a}(\iota(g',g);f^\diamond_{a,s})\phi(g)\mu(\det g)\rd g\\
&=\sum_{d=0}^r\int_{\rG(F)\backslash\rG(\bA_F)}
\sum_{(\gamma',\gamma)\in\gamma_d^{-1}\rP_a(F)\gamma_d\cap(\rG_a\times\rG)(F)\backslash(\rG_a\times\rG)(F)}
f^\diamond_{a,s}(\gamma_d(\gamma'g',\gamma g))\phi(g)\mu(\det g)\rd g \\
&=\sum_{\gamma'\in\gamma_d^{-1}\rP_a(F)\gamma_d\rG(F)\cap\rG_a(F)\backslash\rG_a(F)}
\int_{\rG(F)\cap\gamma_d^{-1}\rP_a(F)\gamma_d\backslash\rG(\bA_F)}f^\diamond_{a,s}(\gamma_d(\gamma'g',g))\phi(g)\mu(\det g)\rd g.
\end{align*}
It is easy to see that, since $\phi$ is cuspidal, the integration vanishes unless $d=0$. Thus, we have
\begin{align*}
&\int_{\rG(F)\backslash\rG(\bA_F)}\sE_{\rP_a}(\iota(g',g);f^\diamond_{a,s})\phi(g)\mu(\det g)\rd g\\
&=\sum_{\gamma'\in\rP_a(F)\rG(F)\cap\rG_a(F)\backslash\rG_a(F)}\int_{\rG(\bA_F)}
f^\diamond_{a,s}(\gamma_d(\gamma'g',g))\phi(g)\mu(\det g)\rd g\\
&=\sum_{\gamma'\in\rP_a(F)\rG(F)\cap\rG_a(F)\backslash\rG_a(F)}\int_{\rG(\bA_F)}
f^\diamond_{a,s}((g^{-1}\gamma'g',1))\phi(g)\rd g\\
&=\sum_{\gamma'\in\rP_a(F)\rG(F)\cap\rG_a(F)\backslash\rG_a(F)}\int_{\rG(\bA_F)}f^{\diamond,\phi}_{a,s}(\gamma'g')\\
&=\sE_{\rQ_a}(g';f^{\diamond,\phi}_{a,s}).
\end{align*}
Here, the last equality is due to the fact that $\rP_a(F)\rG(F)\cap\rG_a(F)=\rQ_a(F)$. The lemma follows.
\end{proof}

The following lemma suggests that sections of the form $f^{\diamond,\phi}_{a,s}$ detect poles of $\sE_{\rQ_a}$ when $a$ is sufficiently large.

\begin{lem}\label{le:pole4}
There exists an integer $a_0$ depending only on $V_\pi$ and $\mu$ such that for every integer $a\geq a_0$, if $s$ is not a pole of $\{f^{\diamond,\phi}_{a,s}\}$, then the functions $\{f^{\diamond,\phi}_{a,s}\}$ for all standard sections $f^\diamond_{a,s}\in\rJ_a(s,\mu^\tc)$ and $\phi\in V_\pi$ span the whole space $\Ind^{\rG_a(\bA_F)}_{\rQ_a(\bA_F)}V_\pi\boxtimes(\mu^\tc\cdot|\;|_E^s)\circ\det_a$.
\end{lem}

\begin{proof}
This follows from the same discussion after \cite{GJS}*{Proposition~3.3}.
\end{proof}

\begin{proof}[Proof of Proposition \ref{pr:pole}]
Let $s_0$ be an element in $\r{Pol}_1(V_\pi\boxtimes\mu^\tc)$ such that $s_0+j\not\in\r{Pol}_1(V_\pi\boxtimes\mu^\tc)$ for every integer $j>0$. Put $s_a\coloneqq s_0+\frac{a-1}{2}$. Let $a$ be an integer such that $s_a>\frac{n}{2}$ and $a\geq a_0$, where $a_0$ is as in Lemma \ref{le:pole4}. By Lemma \ref{le:pole2}, $f^{\diamond,\phi}_{a,s}$ is holomorphic at $s=s_a$. By Lemma \ref{le:pole1} and Lemma \ref{le:pole4}, we may find some standard section $f^\diamond_{a,s}\in\rJ_a(s,\mu^\tc)$ and $\phi\in V_\pi$ such that $\sE_{\rQ_a}(\quad;f^{\diamond,\phi}_{a,s})$ has a pole at $s=s_a$. By Lemma \ref{le:pole3}, we know that $\sE_{\rP_a}(\quad;f^\diamond_{a,s})$ has a pole at $s=s_a$ for such $f^\diamond_{a,s}$. Therefore, $s_0$ has to be the maximal element in $\r{Pol}_1(V_\pi\boxtimes\mu^\tc)$, that is, $s_0=s_1$. In particular, (2) follows.

We continue for (1). By Lemma \ref{le:pole2}(1), the pole must be simple, that is, $\r{Res}_{s=s_a}\sE_{\rP_a}(\quad;f^\diamond_{a,s})\neq 0$. Put $m\coloneqq n+1-2s_1$ and $m_a\coloneqq 2(n+a)-m$. Let $\rW^a$ be a skew-hermitian space over $E$ of rank $m_a$. We have a Weil representation of $\rU(\rV^\diamond_a)(\bA_F)\times\rU(\rW^a)(\bA_F)$ on the Schwartz space $\sS((\rV^+_a\otimes_E\rW^a)(\bA_F))$, and a $\rU(\rV^\diamond_a)(\bA_F)$-equivariant map
\[
f^{(s_a)}\colon\sS((\rV^+_a\otimes_E\rW^a)(\bA_F))
\to\Ind_{\rP_a(\bA_F)}^{\rU(\rV^\diamond_a)(\bA_F)}(\mu^\tc\cdot|\;|^{s_a})\circ\det\nolimits_{n+a}
\]
sending $\Phi$ to $f^{(s_a)}_\Phi$, which is known as taking Siegel--Weil sections. For more details, see, for example, \cite{Ich04}. Since $m_a\geq n+a$, by \cite{KS97}*{Theorem~1.2 \& Theorem~1.3} and \cite{Lee94}*{Theorem~6.10}, the map
\[
f^{(s_a)}\colon\bigoplus_{\rW^a}\sS((\rV^+_a\otimes_E\rW^a)(\bA_F))
\to\Ind_{\rP_a(\bA_F)}^{\rU(\rV^\diamond_a)(\bA_F)}(\mu^\tc\cdot|\;|^{s_a})\circ\det\nolimits_{n+a},
\]
by considering all possible skew-hermitian spaces $\rW^a$ of rank $m_a$ up to isomorphism, is surjective. Thus, there exist some $\rW^a$ in
the above direct sum and an element $\Phi\in\sS((\rV^+_a\otimes_E\rW^a)(\bA_F))$ such that $f^{(s_a)}_\Phi=f^\diamond_{a,s}$, hence $\r{Res}_{s=s_a}\sE_{\rP_a}(\quad;f^{(s_a)}_\Phi)\neq 0$. In particular, the Witt index of $\rW^a$ is at least $m_a-(n+a)$. Now by the main theorem on \cite{Ich04}*{p.243}, we have the identity
\[
\r{Res}_{s=s_a}\sE_{\rP_a}(\quad;f^{(s_a)}_\Phi)=c\cdot\int_{\rU(\rW)(F)\backslash\rU(\rW)(\bA_F)}\theta_{(\mu^\tc,\b{1})}(\quad,h)\rd h
\]
as functions on $\rU(\rV^\diamond_a)(\bA_F)$. Here, $c$ is a nonzero constant; $\rW$ is a certain skew-hermitian space of rank $2(n+a)-m_a=m$ determined by $\rW^a$; and $\theta_{(\mu^\tc,\b{1})}$ is a certain theta series on $\rU(\rV^\diamond_a)(\bA_F)\times\rU(\rW)(\bA_F)$ with respect to the pair of splitting characters $(\mu^\tc,\b{1})$ in which $\b{1}$ denotes the trivial character. By Lemma \ref{le:pole3} and our choices of $f^\diamond_{a,s}$ and $\phi$, the integral
\begin{align}\label{eq:pole2}
\int_{\rG(F)\backslash\rG(\bA_F)}\int_{\rU(\rW)(F)\backslash\rU(\rW)(\bA_F)}
\theta_{(\mu^\tc,\b{1})}(\iota(g',g),h)\phi(g)\mu(\det g)\rd h\rd g
\end{align}
is nonzero for some $g'\in\rG_a(\bA_F)$. Now we need to separate the variables $g'$ and $g$ in the above theta series. Choose an arbitrary automorphic character $\nu$ of $\bA_E^\times$ such that $\nu\res\bA_F^\times=\mu_{E/F}^n$. We have two embeddings
\begin{align*}
\iota'\coloneqq\iota\times\r{id}_{\rU(\rW)}
&\colon \rG_a\times\rG\times\rU(\rW)\hookrightarrow\rU(\rV^\diamond_a)\times\rU(\rW),\\
\iota''&\colon \rG_a\times\rG\times\rU(\rW)\hookrightarrow(\rG_a\times\rU(\rW))\times(\rG\times\rU(\rW)),
\end{align*}
in which the second one is induced by the diagonal embedding of $\rU(\rW)$. It follows from \cite{HKS}*{Lemma~1.1} that
\[
\omega_{(\mu^\tc,\b{1})}^{\rV^\diamond_a,\rW}\circ\iota'\simeq
\(\omega_{(\mu^\tc,\nu^\tc)}^{\rV_a,\rW}\widehat\otimes\omega_{(\mu^\tc,\nu)}^{\rV,\rW}\)\circ\iota''
\]
for the restriction of Weil representations \eqref{eq:weil}. Therefore, without lost of generality, we may assume that there exist finitely many pairs $(\theta^{(i)}_{(\mu^\tc,\nu^\tc)},\theta^{[i]}_{(\mu^\tc,\nu)})$ in which $\theta^{(i)}_{(\mu^\tc,\nu^\tc)}$ (resp.\ $\theta^{[i]}_{(\mu^\tc,\nu)}$) is a theta series on $\rG_a(\bA_F)\times\rU(\rW)(\bA_F)$ (resp.\ $\rG(\bA_F)\times\rU(\rW)(\bA_F)$) with respect to $(\mu^\tc,\nu^\tc)$ (resp.\ $(\mu^\tc,\nu)$) such that
\[
\theta_{(\mu^\tc,\b{1})}(\iota(g',g),h)=\sum_{i}\theta^{(i)}_{(\mu^\tc,\nu^\tc)}(g',h)\theta^{[i]}_{(\mu^\tc,\nu)}(g,h),
\]
and that \eqref{eq:pole2} is nonzero for some $g'\in\rG_a(\bA_F)$. Then we have
\begin{align*}
\eqref{eq:pole2}
&=\int_{\rG(F)\backslash\rG(\bA_F)}\int_{\rU(\rW)(F)\backslash\rU(\rW)(\bA_F)}
\theta_{\underline\mu}(\iota(g',g),h)\phi(g)\mu(\det g)\rd h\rd g \\
&=\int_{\rG(F)\backslash\rG(\bA_F)}\int_{\rU(\rW)(F)\backslash\rU(\rW)(\bA_F)}
\sum_{i}\theta^{(i)}_{(\mu^\tc,\nu^\tc)}(g',h)\theta^{[i]}_{(\mu^\tc,\nu)}(g,h)\phi(g)\mu(\det g)\rd h\rd g \\
&=\sum_{i}\int_{\rU(\rW)(F)\backslash\rU(\rW)(\bA_F)}\theta^{(i)}_{(\mu^\tc,\nu^\tc)}(g',h)
\(\int_{\rG(F)\backslash\rG(\bA_F)}\theta^{[i]}_{(\mu^\tc,\nu)}(g,h)\phi(g)\mu(\det g)\rd g\)\rd h\\
&=\sum_{i}\int_{\rU(\rW)(F)\backslash\rU(\rW)(\bA_F)}\theta^{(i)}_{(\mu^\tc,\nu^\tc)}(g',h)
\(\int_{\rG(F)\backslash\rG(\bA_F)}\theta^{[i]}_{(\mu,\nu)}(g,h)\phi(g)\rd g\)\rd h.
\end{align*}
In particular, there exists some $i$ such that
\[
\int_{\rG(F)\backslash\rG(\bA_F)}\theta^{[i]}_{(\mu,\nu)}(g,h)\phi(g)\rd g\not\equiv 0.
\]
In other words, $\Theta_{(\mu,\nu),\rV}^\rW(V_\pi)\neq 0$, and (1) follows.
\end{proof}

\begin{proof}[Proof of Theorem \ref{th:pole}]
By the Langlands--Shahidi theory, the poles of the Eisenstein series $\sE_\rQ(\quad;f_s)$ are controlled by its constant term, which in term are control by the intertwining operator attached to the longest Weyl element in $\rQ\backslash\rG_1/\rQ$. By the Gindikin--Karpelevich formula, we know that the poles of the $L$-function
\begin{align}\label{eq:pole3}
\frac{L^S(s,\pi\times\mu)\cdot L^S(2s,\mu,\r{As}^{(-1)^n})}{L^S(s+1,\pi\times\mu)\cdot L^S(2s+1,\mu,\r{As}^{(-1)^n})}
\end{align}
in the region $\RE(s)>0$ are contained in the set $\r{Pol}_1(V_\pi\boxtimes\mu^\tc)$. See the proof of \cite{GJS}*{Proposition~2.2} for a similar discussion in the orthogonal case.

We first consider the case where $\mu^\tc\neq\mu^{-1}$. Then $L^S(s,\mu,\r{As}^{(-1)^n})$ has no pole for $\RE(s)>0$. On the other hand, by \cite{Kim99}*{Corollary~2.2}, the set $\r{Pol}_1(V_\pi\boxtimes\mu^\tc)$ is empty. Thus, it follows easily that $L^S(s,\pi\boxtimes\mu)$ has no pole for $\RE(s)>0$ as well. Theorem \ref{th:pole} is proved in this case.

Now we assume that $\mu^\tc=\mu^{-1}$. In other words, $\mu\res\bA_F^\times=\mu_{E/F}^i$ for a unique $i\in\{0,1\}$. Part (2) is a consequence of the local theta dichotomy \cite{SZ15}*{Theorem~1.10}. It remains to consider (1). Let $s_0$ be a pole of $L^S(s,\pi\times\mu)\cdot L^S(2s,\mu,\r{As}^{(-1)^n})$ as in (a). Let $j\geq 0$ be the largest nonnegative integer such that $s_0+j$ is a pole of $L^S(s,\pi\times\mu)\cdot L^S(2s,\mu,\r{As}^{(-1)^n})$. Then the $L$-function \eqref{eq:pole3} has a pole at $s_0+j$. Thus, we have $s_0+j\in\r{Pol}_1(V_\pi\boxtimes\mu^\tc)$, and (b) holds. For the implication (b) $\Rightarrow$ (c), by Rallis' tower property for the global theta lifting, we may assume that $j=0$ in (b) and $s_0+j\not\in\r{Pol}_1(V_\pi\boxtimes\mu^\tc)$ for every integer $j>0$. Then by Proposition \ref{pr:pole}(2), $s_0=s_1$. Then (c) follows from Proposition \ref{pr:pole}(1).
\end{proof}

\section{Shimura varieties for hermitian spaces}
\label{ss:c}

In this appendix, we summarize different versions of unitary Shimura varieties. In Subsection \ref{ss:appendix_isometry}, we recall Shimura varieties associated to isometry groups of hermitian spaces, which are of abelian type; we also introduce the Shimura varieties associated to incoherent hermitian spaces. In Subsection \ref{ss:appendix_similitude}, we recall the well-known PEL type Shimura varieties associated to groups of rational similitude of skew-hermitian spaces, and their integral models at good primes, after Kottwitz. These Shimura varieties are only for the preparation of the next subsection, which are not logically needed in the main part of the article. In Subsection \ref{ss:appendix_connection}, we summarize the connection of these two kinds of unitary Shimura varieties via the third one which possesses a moduli interpretation but is not of PEL type in the sense of Kottwitz, after \cites{BHKRY,RSZ}. In Subsection \ref{ss:integral_models}, we discuss integral models of the third unitary Shimura varieties at good inert primes and their uniformization along the basic locus.

Let $F$ be a totally real number field of degree $d\geq 1$, and $E/F$ a totally imaginary quadratic extension. Denote by $\tc$ the nontrivial involution of $E$ over $F$. Denote by $\Phi_F$ the set of real embeddings of $F$ and by $\Phi_E$ the set of complex embeddings of $E$. Let $\dN[\Phi_E]$ be the commutative monoid freely generated by $\Phi_E$. The Galois group $\Gal(\dC/\dQ)$ acts on $\Phi_E$, hence on $\dN[\Phi_E]$. We have the projection map $\pi\colon\Phi_E\to\Phi_F$ given by restriction. Recall that a CM type (of $E$) is a subset $\Phi$ of $\Phi_E$ such that $\pi$ induces a bijection from $\Phi$ to $\Phi_F$. For a CM type $\Phi$, put $\Phi^\tc\coloneqq\Phi_E\setminus\Phi$, which is again a CM type.

\subsection{Case of isometry}
\label{ss:appendix_isometry}

Let $\rV$ be a (non-degenerate) hermitian space over $E$ (with respect to $\tc$) of rank $n\geq 1$, with the hermitian form $(\;,\;)_\rV\colon\rV\times\rV\to E$ that is $E$-linear in the first variable. For every $\tau\in\Phi_F$, let $(p_\tau,q_\tau)$ be the signature of $\rV\otimes_{F,\tau}\dR$. We take a CM type $\Phi\subseteq\Phi_E$. Then we have two elements
\begin{align}\label{eq:signature}
\sig_{\rV,\Phi}\coloneqq\sum_{\tau\in\Phi_F}p_\tau\tau^+ + \sum_{\tau\in\Phi_F}q_\tau\tau^-,\qquad
\sig^\flat_{\rV,\Phi}\coloneqq\sum_{\tau\in\Phi_F}q_\tau\tau^-
\end{align}
in $\dN[\Phi_E]$. Here, $\tau^-$ (resp.\ $\tau^+$) is the unique element in $\Phi$ (resp.\ $\Phi^\tc$) whose image under $\pi$ is $\tau$.

\begin{definition}
We define the \emph{reflex field} (resp.\ \emph{reduced reflex field}) of the pair $(\rV,\Phi)$ to be the fixed field of the stabilizer in $\Gal(\dC/\dQ)$ of the element $\sig_{\rV,\Phi}$ (resp.\ $\sig^\flat_{\rV,\Phi}$), denoted by $E_{\rV,\Phi}$ (resp.\ $E^\flat_{\rV,\Phi}$).
\end{definition}

Let $\rU(\rV)$ be the unitary group (of isometry) of $\rV$, that is, the reductive group over $F$ such that for every $F$-algebra $R$, we have
\[
\rU(\rV)(R)=\{g\in\GL_R(\rV\otimes_F R)\res (gx,gy)_\rV=(x,y)_\rV\text{ for all }x,y\in\rV\otimes_F R\}.
\]
For every $\tau\in\Phi_F$, we may identify $\rV\otimes_{E,\tau^-}\dC$ with $\dC^{\oplus n}$, hence $\rU(\rV)\otimes_{F,\tau}\dR$ is identified with the subgroup of $\Res_{\dC/\dR}\GL_n$ of elements preserving the hermitian form given by the matrix $\(\begin{smallmatrix}\rI_{p_\tau}&\\&-\rI_{q_\tau}\end{smallmatrix}\)$.

Put $\rG\coloneqq\Res_{F/\dQ}\rU(\rV)$. We define the Hodge map
\[
\rh^\flat_{\rV,\Phi}\colon\Res_{\dC/\dR}\dG_\rm\to\rG_\dR
\]
to be the one sending $z\in\dC^\times=(\Res_{\dC/\dR}\dG_\rm)(\dR)$ to
\[
\(\left(
    \begin{array}{cc}
      \rI_{p_{\tau_1}} &  \\
       &  (z/\ol{z})\rI_{q_{\tau_1}} \\
    \end{array}
  \right),\cdots,
  \left(
    \begin{array}{cc}
      \rI_{p_{\tau_d}} &  \\
       &  (z/\ol{z})\rI_{q_{\tau_d}} \\
    \end{array}
  \right)
\)\in\rG_\dR(\dR),
\]
where we identify $\rG_\dR(\dR)$ as a subgroup of $\GL_n(\dC)^{d}$ via $\{\tau_1^-,\dots,\tau_d^-\}$. Then we obtain a Shimura data $(\rG,\rh^\flat_{\rV,\Phi})$. It is of abelian type but not Hodge type; and its reflex field coincides with $E^\flat_{\rV,\Phi}$. The theory of Shimura varieties provides us with a projective system of schemes $\{\Sh(\rG,\rh^\flat_{\rV,\Phi})_K\}_K$, quasi-projective and smooth over $E^\flat_{\rV,\Phi}$ of dimension $\sum_{\tau\in\Phi_F}p_\tau q_\tau$, indexed by neat open compact subgroups $K$ of $\rG(\bA^\infty)=\rU(\rV)(\bA^\infty_F)$.

\begin{remark}\label{re:picard}
Suppose that there is an element $\tau\in\Phi_F$ such that $\rV$ has signature $(n-1,1)$ at $\tau$ and $(n,0)$ at other places. Then the Hodge map $\rh^\flat_{\rV,\Phi}$ hence the Shimura variety $\Sh(\rG,\rh^\flat_{\rV,\Phi})_K$ depend only on $\Phi\cap\pi^{-1}\tau$, that is, the unique element contained in $\Phi$ above $\tau$. Thus, for an element $\tau'\in\Phi_E$ above $\tau$, we may write $\rh_{\rV,\tau'}$ and $\Sh(\rG,\rh_{\rV,\tau'})_K$ for those $\Phi$ containing $\tau'$. In particular, the reflex field of $\rh_{\rV,\tau'}$ is $\tau'(E)$. The Galois group $\Gal(\dC/\tau'(E))$ acts on the set of connected components of $\Sh(\rG,\rh_{\rV,\tau'})_K\otimes_{\iota'(E)}\dC$ via the composite homomorphism
\[
\Gal(\dC/\tau'(E))\xrightarrow{\r{rec}}\tau'(E)^\times\backslash(\bA_{\tau'(E)}^\infty)^\times
\xrightarrow{(\tau')^{-1}}E^\times\backslash(\bA_E^\infty)^\times\xrightarrow{e\mapsto e/e^\tc}E^1\backslash(\bA_E^\infty)^1,
\]
where $\r{rec}$ is the global reciprocity map for the number field $\tau'(E)$.
\end{remark}

Now we would like to attach Shimura varieties to an incoherent hermitian space, a concept originated from \cite{KR94} in the orthogonal case and explored in \cite{Zha19}. This observation generalizes the case of Shimura curves in \cite{YZZ}, and has already appeared in some old work \cites{Liu11,Liu12}, with more details explained by Gross \cite{Gro} recently.

\begin{definition}\label{de:incoherent_hermitian}
An \emph{incoherent hermitian space} over $\bA_E$ is a free $\bA_E$-module $\bV$ of some rank $n\geq 1$, equipped with a non-degenerate hermitian form $(\;,\;)_\bV\colon\bV\times\bV\to\bA_E$ with respect to the (induced) involution $\tc$ on $\bA_E$ such that its determinant belongs to $\bA_F^\times\setminus F^\times\Nm_{\bA_E/\bA_F}\bA_E^\times$. We say that $\bV$ is totally positive definite if for every $\tau\in\Phi_F$, $\bV\otimes_{\bA_F,\tau}\dR$ is positive definite.
\end{definition}

Let $\bV$ be a totally positive definite incoherent hermitian space over $\bA_E$ of rank $n\geq 1$, and let $\bG\coloneqq\rU(\bV)$ be its group of isometry, which is a reductive group over $\bA_F$.

\begin{definition}\label{de:nearby}
For $\tau\in\Phi_F$, we say that a hermitian space $\rV$ over $E$ is \emph{$\tau$-nearby to $\bV$} if  $\rV\otimes_F\bA_F^\tau\simeq\bV\otimes_{\bA_F}\bA_F^\tau$, and $\rV\otimes_{F,\tau}\dR$ has signature $(n-1,1)$.
\end{definition}

It is clear that for every $\tau\in\Phi_F$, there exists a hermitian space that is $\tau$-nearby to $\bV$, unique up to isomorphism. We fix such a space $\rV(\tau)$. Put $\rG(\tau)\coloneqq\Res_{F/\dQ}\rU(\rV(\tau))$. We fix an isomorphism $\bV\otimes_{\bA_F}\bA_F^\infty\simeq\rV(\tau)\otimes_F\bA_F^\infty$, hence an isomorphism $\bG(\bA_F^\infty)\simeq\rG(\tau)(\bA^\infty)$.

\begin{proposition}\label{pr:incoherent_shimura}
There is a projective system of schemes $\{\Sh(\bV)_K\}_K$ over $E$ indexed by sufficiently small open compact subgroups $K$ of $\bG(\bA_F^\infty)$, such that for every $\tau\in\Phi_F$ and every $\tau'\in\Phi_E$ above it, we have an isomorphism
\[
\{\Sh(\bV)_K\otimes_{E,\tau'}\tau'(E)\}_K\simeq\{\Sh(\rG(\tau),\rh_{\rV(\tau),\tau'})_K\}_K
\]
of projective systems of schemes over $\tau'(E)$. Here, we use the fixed isomorphism $\bG(\bA_F^\infty)\simeq\rG(\tau)(\bA^\infty)$ to regard $K$ as a subgroup of $\rG(\tau)(\bA^\infty)$.
\end{proposition}

\begin{proof}
See \cite{Gro}*{Section~10}.
\end{proof}

\begin{definition}\label{de:shimura_incoherent}
We call the projective system of schemes $\{\Sh(\bV)_K\}_K$ over $E$ in Proposition \ref{pr:incoherent_shimura} the \emph{Shimura varieties associated to $\bV$}.
\end{definition}

\begin{remark}
One can also interpret Proposition \ref{pr:incoherent_shimura} in the following way: The scheme
\[
\prod_{\tau\in\Phi_F}\prod_{\tau'\in\pi^{-1}\tau}\Sh(\rG(\tau),\rh_{\rV(\tau),\tau'})_K
\]
over
\[
\prod_{\tau\in\Phi_F}\prod_{\tau'\in\pi^{-1}\tau}\Spec\tau'(E)=\prod_{\tau'\in\Phi_E}\Spec\tau'(E)
\]
descends to a scheme $\Sh(\bV)_K$ over $\Spec E$, where the above fiber products are taken over $\Spec\dQ$.
\end{remark}

The scheme $\Sh(\bV)_K$ (for $K$ sufficiently small) is quasi-projective and smooth over $E$ of dimension $n-1$. It is projective if $d>1$ or $n=1$. In all cases, we denote by $\ol\Sh(\bV)_K$ the Baily--Borel compactification of $\Sh(\bV)_K$ over $E$. Then $\ol\Sh(\bV)_K\setminus\Sh(\bV)_K$ is either empty or consists of isolated singular points. Let $\widetilde\Sh(\bV)_K$ be the blow-up of $\ol\Sh(\bV)_K$ along $\ol\Sh(\bV)_K\setminus\Sh(\bV)_K$. If $\Sh(\bV)_K$ is proper, then $\widetilde\Sh(\bV)_K=\Sh(\bV)_K$. Otherwise, we must have $d=1$, that is, $F=\dQ$. In this case, there is only one choice for $\tau\in\Phi_F$, for which we will suppress from various notation like $\rV(\tau)$, $\rG(\tau)$, etc. However, there are still two choices of $\Phi$, say, $\{\tau^+\}$ and $\{\tau^-\}$. We have isomorphisms
\begin{align}\label{eq:pink}
\widetilde\Sh(\bV)_K\otimes_{E,\tau^\pm}\tau^\pm(E)\simeq\widetilde\Sh(\rG,\rh_{\rV,\tau^\pm})_K
\end{align}
extending those in Proposition \ref{pr:incoherent_shimura}. Here, $\widetilde\Sh(\rG,\rh_{\rV,\tau^\pm})_K$ is the unique toroidal compactification of $\Sh(\rG,\rh_{\rV,\tau^\pm})_K$ over $E$ \cites{AMRT,Pin90}.

\begin{definition}\label{de:shimura_incoherent_toroidal}
We call the projective system of schemes $\{\widetilde\Sh(\bV)_K\}_K$ over $E$ the \emph{compactified Shimura varieties associated to $\bV$} (even when $\Sh(\bV)_K$ is already proper).
\end{definition}

\begin{remark}\label{re:shimura_incoherent_toroidal}
The boundary $\widetilde\Sh(\bV)_K\setminus\Sh(\bV)_K$ is a smooth divisor.
\end{remark}

\subsection{Case of similitude}
\label{ss:appendix_similitude}

In this subsection, we recall the notion of Shimura varieties attached to the group of similitude of a hermitian space, which are of PEL type. They will not be used in the main part of the article, but it is instructional to introduce them for the later discussion.

Let
\[
\Psi=\sum_{\tau\in\Phi_F}p_\tau\tau^+ + \sum_{\tau\in\Phi_F}q_\tau\tau^-
\]
be an element of $\dN[\Phi_E]$ such that $p_\tau+q_\tau=n$ for every $\tau\in\Phi_F$. Let $E_\Psi$ be the fixed field of the stabilizer of $\Psi$ in $\Gal(\dC/\dQ)$.

\begin{definition}\label{de:abelian_data}
Let $S$ be an $E_\Psi$-scheme.
\begin{enumerate}
  \item An \emph{$(E,\Psi)$-abelian scheme} over $S$ is a pair $(A,i)$, where $A$ is an abelian scheme over $S$, and $i\colon E\to\End_S(A)_\dQ$ is a homomorphism of $\dQ$-algebras such that for every $e\in E$, the characteristic polynomial of $i(e)$ on the locally free sheaf $\Lie_S(A)$ on $S$ is equal to
      \[
      \prod_{\tau\in\Phi_F}(T-\tau^+(e))^{p_\tau}(T-\tau^-(e))^{q_\tau}\in\cO_S[T].
      \]

  \item A polarization of an $(E,\Psi)$-abelian scheme $(A,i)$ is a polarization $\lambda\colon A\to A^\vee$ satisfying $\lambda\circ i(e)=i(e^\tc)^\vee\circ\lambda$ for every $e\in E$.
\end{enumerate}
\end{definition}

\begin{definition}
For a ring $R$ containing $\dQ$, a \emph{rational skew-hermitian space} over $E\otimes_\dQ R$ of rank $n$ is a free $E\otimes_\dQ R$-module $\rW$ of rank $n$ together with a $R$-bilinear skew-symmetric non-degenerate pairing
\[
\langle\;,\;\rangle_\rW\colon\rW\times\rW\to R
\]
satisfying $\langle ex,y\rangle_\rW=\langle x,e^\tc y\rangle_\rW$ for every $e\in E$ and $x,y\in\rW$. We say that two rational skew-hermitian spaces $\rW$ and $\rW'$ over $E\otimes_\dQ R$ is \emph{similar} if there exists an isomorphism $f\colon\rW\to\rW'$ of $E\otimes_\dQ R$-modules such that there exists some $\nu(f)\in R^\times$ satisfying $\langle f(x),f(y)\rangle_{\rW'}=\nu(f)\langle x,y\rangle_\rW$ for every $x,y\in\rW$.
\end{definition}

We take a rational skew-hermitian space $\bW^\infty$ over $\bA_E^\infty=E\otimes_\dQ\bA^\infty$ of rank $n$. Let $\bH^\infty$ be the group of similitude of $\bW^\infty$, which is a reductive group over $\bA^\infty$. We denote by $\cW(\bW^\infty,\Psi)$ the set of similarity classes of rational skew-hermitian spaces $\rW$ over $E$ of rank $n$ such that
\begin{itemize}
  \item $\rW\otimes_E\bA^\infty_E$ is similar to $\bW^\infty$ as a rational skew-hermitian space over $\bA_E^\infty=E\otimes_\dQ\bA^\infty$ (and we fix a similarity isomorphism),

  \item the signature of the hermitian form $\langle\;,i\cdot\;\rangle_\rW$ on the $\dC$-vector space $\rW\otimes_{E,\tau^-}\dC$ is $(p_\tau,q_\tau)$.
\end{itemize}
It is a finite set; and its cardinality is at most one if $n$ is even.

For every $\rW\in\cW(\bW^\infty,\Psi)$, let $\rH$ be its group of similitude, that is, the reductive group over $\dQ$ such that for every ring $R$ containing $\dQ$, we have
\[
\rH(R)=\{h\in\GL_{E\otimes_\dQ R}(\rW\otimes_\dQ R)\res\langle hx,hy\rangle_\rW=\nu(h)\langle x,y\rangle_\rW\text{ for some $\nu(h)\in R^\times$}\}.
\]
We define the Hodge map
\[
\rh_{\rW,\Psi}\colon\Res_{\dC/\dR}\dG_\rm\to\rH_\dR
\]
to be the one sending $z\in\dC^\times=(\Res_{\dC/\dR}\dG_\rm)(\dR)$ to
\[
\(\left(
    \begin{array}{cc}
      \ol{z}\rI_{p_{\tau_1}} &  \\
       &  z\rI_{q_{\tau_1}} \\
    \end{array}
  \right),\cdots,
  \left(
    \begin{array}{cc}
      \ol{z}\rI_{p_{\tau_d}} &  \\
       &  z\rI_{q_{\tau_d}} \\
    \end{array}
  \right);z\ol{z}
\)\in\rH_\dR(\dR),
\]
where we identify $\rH_\dR(\dR)$ as a subgroup of $\GL_n(\dC)^{d}\times\dC^\times$ via $\{\tau_1^-,\dots,\tau_d^-\}$. Then we have a Shimura data $(\rH,\rh_{\rW,\Psi})$ with the reflex field $E_\Psi$. We obtain a projective system of schemes $\{\Sh(\rH,\rh_{\rW,\Psi})_L\}_L$, quasi-projective and smooth over $E_\Psi$ of dimension $\sum_{\tau\in\Phi_F}p_\tau q_\tau$, indexed by neat open compact subgroups $L$ of $\bH^\infty(\bA^\infty)\simeq\rH(\bA^\infty)$.

The Shimura data $(\rH,\rh_{\rW,\Psi})$ is of PEL type. In particular, it has a moduli interpretation which we roughly recall in the following definition.

\begin{definition}[\cite{Kot92}]\label{de:moduli_similitude}
For an open compact subgroup $L\subseteq\bH^\infty(\bA^\infty)$, we define a presheaf $\rM(\bW^\infty,\Psi)_L$ on $\Sch'_{/E_\Psi}$ as follows: For every object $S\in\Sch'_{/E_\Psi}$, we let $\rM(\bW^\infty,\Psi)_L(S)$ be the set of equivalence classes of quadruples $(A,i,\lambda,\eta)$, where
\begin{itemize}
  \item $(A,i)$ is an $(E,\Psi)$-abelian scheme over $S$ (Definition \ref{de:abelian_data}),

  \item $\lambda$ is a polarization of $(A,i)$ (Definition \ref{de:abelian_data}),

  \item $\eta$ is an $L$-level structure (see \cite{Kot92}*{Section~5} for more details).
\end{itemize}
Two quadruples $(A,i,\lambda,\eta)$ and $(A',i',\lambda',\eta'))$ are equivalent if there is an isogeny $\varphi\colon A\to A'$ taking $i,\lambda,\eta$ to $i',c\lambda',\eta'$ for some $c\in\dQ^\times$.
\end{definition}

From \cite{Kot92}, it is known that $\rM(\bW^\infty,\Psi)_L$ is a scheme if $L$ is sufficiently small, and we have a canonical isomorphism
\begin{align*}
\rM(\bW^\infty,\Psi)_L\simeq\coprod_{\rW\in\cW(\bW^\infty,\Psi)}\Sh(\rH,\rh_{\rW,\Psi})_L
\end{align*}
functorial in $L$.

\begin{remark}\label{re:integral_similitude}
Let $p$ be a rational prime unramified in $E$ such that we may write $L=L^p\times L_p$ in which $L_p$ is the stabilizer of a self-dual lattice in $\bW^\infty\otimes_{\bA^\infty}\dQ_p$. Then the presheaf $\rM(\bW^\infty,\Psi)_L$ admits an extension $\cM(\bW^\infty,\Psi)_L$ to a presheaf on $\Sch'_{/O_{E_\Psi,(p)}}$ as follows: For every object $S\in\Sch'_{/O_{E_\Psi,(p)}}$, we let $\cM(\bW^\infty,\Psi)_L(S)$ be the set of equivalence classes of quadruples $(A,i,\lambda,\eta^p)$, where
\begin{itemize}
  \item $(A,i)$ is an $(E,\Psi)$-abelian scheme over $S$ in the sense similar to Definition \ref{de:abelian_data} but with $i\colon O_{E,(p)}\to\End_S(A)\otimes_\dZ\dZ_{(p)}$ being a homomorphism of $\dZ_{(p)}$-algebras,

  \item $\lambda$ is a $p$-principal polarization of $(A,i)$,

  \item $\eta^p$ is an $L^p$-level structure.
\end{itemize}
The equivalence relation is defined in a similar way as in Definition \ref{de:moduli_similitude} except that we require the isogenies to be coprime to $p$ and $c\in\dZ_{(p)}^\times$. The functor $\cM(\bW^\infty,\Psi)_L$ is a smooth separated scheme in $\Sch_{/O_{E_\Psi,(p)}}$ if $L$ is sufficiently small; and is functorial in $L$.
\end{remark}

\subsection{Their connection}
\label{ss:appendix_connection}

In this subsection, we study the connection between Shimura varieties in the case of isometry and those in the case of similitude. Consider
\begin{itemize}
  \item a hermitian space $\rV,(\;,\;)_\rV$ over $E$ of rank $n$,

  \item a rational skew-hermitian space $\bW^\infty_0,\langle\;,\;\rangle_0$ over $\bA_E^\infty=E\otimes_\dQ\bA^\infty$ of rank $1$ with the group of similitude $\bH^\infty_0$,

  \item a CM type $\Phi$ of $E$ such that $\cW(\bW^\infty_0,\Phi^\tc)$ is nonempty.
\end{itemize}
We now equip $\bW^\infty\coloneqq\rV\otimes_E\bW^\infty_0$ with a rational skew-hermitian form over $\bA_E^\infty=E\otimes_\dQ\bA^\infty$. For $x,y\in\bW^\infty_0$, let $\langle x,y\rangle_0^\dag\in\bA_E^\infty$ be the unique element such that $\Tr_{E/\dQ}(e\cdot\langle x,y\rangle_0^\dag)=\langle ex,y\rangle_0$ for every $e\in\bA_E^\infty$. Thus, we obtain a non-degenerate pairing $\langle\;,\;\rangle_0^\dag\colon\bW^\infty_0\times\bW^\infty_0\to\bA_E^\infty$ that is $\bA_E^\infty$-linear in the first variable. We equip $\bW^\infty$ with the pairing $\Tr_{E/\dQ}(\;,\;)_\rV\otimes_E\langle\;,\;\rangle_0^\dag$, which becomes a rational skew-hermitian space over $E\otimes_\dQ\bA^\infty$. By a similar construction, we obtain a map $\cW(\bW^\infty_0,\Phi^\tc)\to\cW(\bW^\infty,\Psi)$ sending $\rW_0$ to $\rW$, where $\Psi=\sig_{\rV,\Phi}$ \eqref{eq:signature}. Take an element $\rW_0\in\cW(\bW^\infty_0,\Phi^\tc)$ with $\rH_0$ its group of similitude. We obtain three Shimura data: $(\rG,\rh^\flat_{\rV,\Phi})$, $(\rH_0,\rh_{\rW_0,\Phi^\tc})$, and $(\rH,\rh_{\rW,\Psi})$ with reflex fields $E^\flat_{\rV,\Phi}$, $E_\Phi$, and $E_\Psi=E_{\rV,\Phi}$, respectively.

\begin{lem}\label{le:reflex}
Let $E^\sharp_{\rV,\Phi}$ be the subfield of $\dC$ generated by $E^\flat_{\rV,\Phi}$ and $E_\Phi$. Then $E^\sharp_{\rV,\Phi}$ contains $E_{\rV,\Phi}$.
\end{lem}

\begin{proof}
By definition, the subgroup of $\Gal(\dC/\dQ)$ fixing $E^\sharp_{\rV,\Phi}$ stabilizes both $\sig^\flat_{\rV,\Phi}$ and $\Phi$. Thus, it stabilizes $\sig_{\rV,\Phi}$. The lemma follows.
\end{proof}

\begin{remark}
In the main part of the article, the hermitian space $\rV$ we encounter will have signature $(n-1,1)$ at one place $\tau\in\Phi_F$ and $(n,0)$ elsewhere for some $n\geq 2$. Then for whatever $\Phi$, we have $E^\flat_{\rV,\Phi}=\tau'(E)$, where $\tau'\in\Phi_E$ is either place above $\tau$. However, it is possible that $\bigcap_\Phi E^\sharp_{\rV,\Phi}$ strictly contains $\tau'(E)$, where $\Phi$ runs over all CM types of $E$.
\end{remark}

Now we consider the reductive group $\rG^\sharp\coloneqq\rG\times\rH_0$ over $\dQ$. Put $\rh^\sharp_\Phi\coloneqq(\rh^\flat_{\rV,\Phi},\rh_{\rW_0,\Phi^\tc})$. Then we have a product Shimura data $(\rG^\sharp,\rh^\sharp_\Phi)$, whose reflex field is $E^\sharp_{\rV,\Phi}$. On the other hand, there is a homomorphism $\tq_\rW\colon\rG^\sharp=\rG\times\rH_0\to\rH$ induced by taking tensor product. It is clear that $\tq_\rW\circ\rh^\sharp_\Phi=\rh_{\rW,\Psi}$. To summarize, we have the following diagram of Shimura data
\begin{align}\label{eq:shimura_data}
\xymatrix{
(\rG,\rh^\flat_{\rV,\Phi}) \\
& (\rG^\sharp,\rh^\sharp_\Phi) \ar[lu]_-{\tq_\rV}\ar[ld]^-{\tq_{\rW_0}}\ar[r]^-{\tq_\rW} & (\rH,\rh_{\rW,\Psi}). \\
(\rH_0,\rh_{\rW_0,\Phi^\tc})
}
\end{align}
For neat open compact subgroups $K\subseteq\rG(\bA^\infty)$, $L_0\subseteq\rH_0(\bA^\infty)$, and $L\subseteq\rH(\bA^\infty)$ satisfying $\tq_\rW(K\times L_0)\subseteq L$, we have the following diagram of Shimura varieties induced from \eqref{eq:shimura_data}
\begin{align}\label{eq:shimura_varieties}
\xymatrix{
\Sh(\rG,\rh^\flat_{\rV,\Phi})_K\otimes_{E^\flat_{\rV,\Phi}}E^\sharp_{\rV,\Phi} \\
& \Sh(\rG^\sharp,\rh^\sharp_\Phi)_{K\times L_0} \ar[lu]_-{\tq_\rV}\ar[ld]^-{\tq_{\rW_0}}\ar[r]^-{\tq_\rW} & \Sh(\rH,\rh_{\rW,\Psi})_L\otimes_{E_{\rV,\Phi}}E^\sharp_{\rV,\Phi} \\
\Sh(\rH_0,\rh_{\rW_0,\Phi^\tc})_{L_0}\otimes_{E_\Phi}E^\sharp_{\rV,\Phi}
}
\end{align}
in view of Lemma \ref{le:reflex}, in which $(\tq_\rV,\tq_{\rW_0})$ induces an isomorphism
\begin{align}\label{eq:shimura_varieties1}
\Sh(\rG^\sharp,\rh^\sharp_\Phi)_{K\times L_0}\simeq \(\Sh(\rG,\rh^\flat_{\rV,\Phi})_K\otimes_{E^\flat_{\rV,\Phi}}E^\sharp_{\rV,\Phi}\)
\times_{E^\sharp_{\rV,\Phi}}\(\Sh(\rH_0,\rh_{\rW_0,\Phi^\tc})_{L_0}\otimes_{E_\Phi}E^\sharp_{\rV,\Phi}\)
\end{align}
in $\Sch_{/E^\sharp_{\rV,\Phi}}$, functorial in $K$, $L_0$, and under Hecke translations.

The Shimura variety $\Sh(\rG^\sharp,\rh^\sharp_\Phi)_{K\times L_0}$ has a moduli interpretation as well.

\begin{definition}\label{de:moduli_connection}
For open compact subgroups $K\subseteq\rG(\bA^\infty)$ and $L\subseteq\bH_1^\infty(\bA^\infty)$, we define a presheaf $\rM(\rV,\bW_1^\infty,\Phi)_{K,L_1}$ on $\Sch'_{/E^\sharp_{\rV,\Phi}}$ as follows: For every object $S\in\Sch'_{/E^\sharp_{\rV,\Phi}}$, we let $\rM(\rV,\bW_1^\infty,\Phi)_{K,L_1}(S)$ be the set of equivalence classes of octuples $(A_0,i_0,\lambda_0,\eta_0;A,i,\lambda,\eta)$, where
\begin{itemize}
  \item $(A_0,i_0)$ is an $(E,\Phi^\tc)$-abelian scheme over $S$,

  \item $\lambda_0$ is a polarization of $(A_0,i_0)$,

  \item $\eta_0$ is an $L_0$-level structure for $(A_0,i_0,\lambda_0)$,

  \item $(A,i)$ is an $(E,\sig_{\rV,\Phi})$-abelian scheme over $S$,

  \item $\lambda$ is a polarization of $(A,i)$,

  \item for chosen geometric point $s$ on every connected component of $S$, $\eta$ is a $\pi_1(S,s)$-invariant $K$-orbit of isometries
      \[
      \rV\otimes_\dQ\bA^\infty\xrightarrow{\sim}\Hom_{E\otimes_\dQ\bA^\infty}(\rH^{\et}_1(A_{0s},\bA^\infty),\rH^{\et}_1(A_s,\bA^\infty))
      \]
      of hermitian spaces over $\bA_E^\infty$. Here, the hermitian pairing on the latter space is given by the formula
      \[
      (x,y)\mapsto i_0^{-1}\((\lambda_{0*})^{-1}\circ y^\vee\circ\lambda_*\circ x\)\in i_0^{-1}\End_{E\otimes_\dQ\bA^\infty}(\rH^{\et}_1(A_{0s},\bA^\infty))=\bA^\infty_E.
      \]
\end{itemize}
Two octuples $(A_0,i_0,\lambda_0,\eta_0;A,i,\lambda,\eta)$ and $(A'_0,i'_0,\lambda'_0,\eta'_0;A',i',\lambda',\eta')$ are equivalent if there are isogenies $\varphi_0\colon A_0\to A'_0$ and $\varphi\colon A\to A'$ such that
\begin{itemize}
  \item there exists $c\in\dQ^\times$ such that $\varphi_0^\vee\circ\lambda'_0\circ\varphi_0=c\lambda_0$ and $\varphi^\vee\circ\lambda'\circ\varphi=c\lambda$,

  \item for every $e\in E$, we have $\varphi_0\circ i_0(e)=i'_0(e)\circ\varphi_0$ and $\varphi\circ i(e)=i'(e)\circ\varphi$,

  \item the $K$-orbit of maps $x\mapsto\varphi_*\circ\eta(x)\circ(\varphi_{0*})^{-1}$ for $x\in\rV\otimes_\dQ\bA^\infty$ coincides with $\eta'$.
\end{itemize}
\end{definition}

\begin{remark}
The Shimura variety $\Sh(\rG^\sharp,\rh^\sharp_\Phi)_{K\times L_1}$ and its moduli interpretation were first introduced in \cite{BHKRY} when $F=\dQ$, and in \cite{RSZ} for more general CM extension $E/F$.
\end{remark}

\begin{lem}\label{le:shimura_variety}
Let the notation be as above. We have a canonical isomorphism
\[
\rM(\rV,\bW_0^\infty,\Phi)_{K,L_0}\simeq\(\Sh(\rG,\rh^\flat_{\rV,\Phi})_K\otimes_{E^\flat_{\rV,\Phi}}E^\sharp_{\rV,\Phi}\)
\times_{E^\sharp_{\rV,\Phi}}\(\rM(\bW^\infty_0,\Phi^\tc)_{L_0}\otimes_{E_\Phi}E^\sharp_{\rV,\Phi}\)
\]
in $\Sch_{/E^\sharp_{\rV,\Phi}}$, functorial in $K$, $L_0$, and under Hecke translations.
\end{lem}

\begin{proof}
We have canonical morphisms
\begin{align*}
\tq&\colon\rM(\rV,\bW_0^\infty,\Phi)_{K,L_0}\to\rM(\bW^\infty,\Psi)_L\otimes_{E_{\rV,\Phi}}E^\sharp_{\rV,\Phi}\\
\tq_0&\colon\rM(\rV,\bW_0^\infty,\Phi)_{K,L_0}\to\rM(\bW^\infty_0,\Phi^\tc)_{L_0}\otimes_{E_\Phi}E^\sharp_{\rV,\Phi}
\end{align*}
of functors obtained from the moduli interpretation. Since $(\tq,\tq_0)$ induces a closed embedding, the functor $\rM(\rV,\bW_0^\infty,\Phi)_{K, L_0}$ is representable. Moreover, we have a canonical isomorphism
\begin{align}\label{eq:shimura_varieties2}
\rM(\rV,\bW_0^\infty,\Phi)_{K,L_0}\simeq\coprod_{\rW_0\in\cW(\bW_0^\infty,\Phi^\tc)}\Sh(\rG^\sharp,\rh^\sharp_\Phi)_{K\times L_0}
\end{align}
functorial in $K$, $L_0$, and under Hecke translations. The morphisms $\tq$ and $\tq_0$ are compatible with $\tq_\rW$ and $\tq_{\rW_0}$ in \eqref{eq:shimura_varieties}, respectively. Combining with \eqref{eq:shimura_varieties1}, we have
\begin{align*}
\eqref{eq:shimura_varieties2}&\simeq
\(\Sh(\rG,\rh^\flat_{\rV,\Phi})_K\otimes_{E^\flat_{\rV,\Phi}}E^\sharp_{\rV,\Phi}\)
\times_{E^\sharp_{\rV,\Phi}}\(\coprod_{\rW_0\in\cW(\bW_0^\infty,\Phi)}
\Sh(\rH_0,\rh_{\rW_0,\Phi^\tc})_{L_0}\otimes_{E_\Phi}E^\sharp_{\rV,\Phi}\)\\
&\simeq\(\Sh(\rG,\rh^\flat_{\rV,\Phi})_K\otimes_{E^\flat_{\rV,\Phi}}E^\sharp_{\rV,\Phi}\)
\times_{E^\sharp_{\rV,\Phi}}\(\rM(\bW^\infty_0,\Phi^\tc)_{L_0}\otimes_{E_\Phi}E^\sharp_{\rV,\Phi}\).
\end{align*}
The lemma follows.
\end{proof}

\subsection{Integral models and uniformization}
\label{ss:integral_models}

In this subsection, we study integral models and uniformization of the Shimura varieties introduced previously, which are only used in Subsection \ref{ss:invariant} and Subsection \ref{ss:orbital} for the main part of the article. We identify $E$ as a subfield of $\dC$ via an element $\tau'\in\Phi_E$. We fix a hermitian space $\rV$ over $E$ that has signature $(n-1,1)$ at $\tau\coloneqq\tau'\res F$ and $(n,0)$ at other places.

We first review the integral models of $\rM(\rV,\bW_0^\infty,\Phi)_{K,L_0}$ in Definition \ref{de:moduli_connection} at good primes, where we assume $\tau'\in\Phi$. Let $\fp$ be a prime of $F$ such that
\begin{itemize}
  \item $\fp$ is inert $E$,

  \item the underlying rational prime $p$ is odd and unramified in $E$,

  \item we may choose a self-dual lattice $\Lambda_\fq$ in $\rV\otimes_FF_\fq$ for every $\fq\in\underline\fp$, where ${\underline\fp}$ denotes the set of all primes of $F$ above $p$ that are inert in $E$,

  \item $L_0=L_0^p\times (L_0)_p$ in which $(L_0)_p$ is the stabilizer of a self-dual lattice in $\bW_0^\infty\otimes_{\bA^\infty}\dQ_p$, and $L_0^p$ is sufficiently small.
\end{itemize}
Fix an isomorphism between $E$-extensions $\dC$ and $E_\fp^\ac$. We denote by $\r{Spl}_p$ the set of primes of $F$ above $p$ that are split in $E$. We also assume that elements in $\Phi$ inducing the same prime in $\r{Spl}_p$ induce the same prime of $E$ (under the fixed isomorphism between $\dC$ and $E_\fp^\ac$).

Denote by $E^\sharp_{\rV,\Phi,\fp}$ the completion of $E^\sharp_{\rV,\Phi}$ in $E_\fp^\ac$. We now consider subgroups $K$ of the form $K=K^p\times K_p^{\underline\fp}\times K_{\underline\fp}$, where $K_{\underline\fp}=\prod_{\fq\in\underline\fp}K_\fq$ in which $K_\fq$ is the stabilizer of $\Lambda_\fq$, and $K^p$ is sufficiently small. For $\fq\in\r{Spl}_p$, we denote by $\fq^-$ the unique prime of $E$ that is in $\Phi$ and regard $K_p^{\underline\fp}$ a subgroup of $\prod_{\fq\in\r{Spl}_p}\GL_{E_{\fq^-}}(\rV\otimes_EE_{\fq^-})$.

The following definition is a special case of the discussion in \cite{RSZ}*{Section~4.1} (but with a slightly finer level structure at $\r{Spl}_p$).

\begin{definition}\label{de:integral_connection}
We define a presheaf $\cM(\rV,\bW_0^\infty,\Phi)_{K,L_0}$ on $\Sch'_{/O_{E^\sharp_{\rV,\Phi,\fp}}}$ as follows: For every object $S\in\Sch'_{/O_{E^\sharp_{\rV,\Phi,\fp}}}$, we let $\cM(\rV,\bW_0^\infty,\Phi)_{K,L_0}(S)$ be the set of equivalence classes of nonuples $(A_0,i_0,\lambda_0,\eta_0^p;A,i,\lambda,\eta^p,\eta_p^\spl)$, where
\begin{itemize}
  \item $(A_0,i_0)$ is an $(E,\Phi^\tc)$-abelian scheme over $S$ (in the sense of Remark \ref{re:integral_similitude}),

  \item $\lambda_0$ is a $\underline\fp$-principal polarization of $(A_0,i_0)$,

  \item $\eta_0^p$ is an $L_0^p$-level structure for $(A_0,i_0,\lambda_0)$,

  \item $(A,i)$ is an $(E,\sig_{\rV,\Phi})$-abelian scheme over $S$ (in the sense of Remark \ref{re:integral_similitude}),

  \item $\lambda$ is a $p$-principal polarization of $(A,i)$,

  \item for chosen geometric point $s$ on every connected component of $S$,
    \begin{itemize}
      \item $\eta^p$ is a $\pi_1(S,s)$-invariant $K^p$-orbit of isometries
         \[
         \rV\otimes_\dQ\bA^{\infty,p}\xrightarrow{\sim}
         \Hom_{E\otimes_\dQ\bA^{\infty,p}}(\rH^{\et}_1(A_{0s},\bA^{\infty,p}),\rH^{\et}_1(A_s,\bA^{\infty,p}))
         \]
         of hermitian spaces over $E\otimes_\dQ\bA^{\infty,p}$. Here, the hermitian pairing is defined similarly as in Definition \ref{de:moduli_connection},

      \item $\eta_p^\spl$ is a $\pi_1(S,s)$-invariant $K_p^{\underline\fp}$-orbit of isomorphisms
         \[
         \prod_{\fq\in\r{Spl}_p}\rV\otimes_E E_{\fq^-}\xrightarrow{\sim}
         \prod_{\fq\in\r{Spl}_p}\Hom_{O_{E_{\fq^-}}}\(A_{0s}[(\fq^-)^\infty],A_s[(\fq^-)^\infty]\)\otimes_{O_{E_{\fq^-}}}E_{\fq^-}
         \]
         of $\prod_{\fq\in\underline\fp}E_{\fq^-}$-modules. Note that due to the signature condition in Definition \ref{de:moduli_similitude}, both $A_{0s}[(\fq^-)^\infty]$ and $A_s[(\fq^-)^\infty]$ are \'{e}tale $O_{E_{\fq^-}}$-modules.
    \end{itemize}
\end{itemize}
The equivalence relation is defined in a similar way as in Definition \ref{de:moduli_connection} except that we require the isogeny $\varphi_0$ (resp.\ $\varphi$) to be coprime to $p$ (resp.\ $\underline\fp$), and $c\in\dZ_{(p)}^\times$.
\end{definition}

The presheaf $\cM(\rV,\bW_0^\infty,\Phi)_{K,L_0}$ is a separated scheme in $\Sch'_{/O_{E^\sharp_{\rV,\Phi,\fp}}}$, which is proper if and only if $\rV$ is anisotropic. By Definition \ref{de:integral_connection} and Remark \ref{re:integral_similitude}, we have a canonical morphism
\begin{align}\label{eq:integral_canonical}
\bq_0\colon \cM(\rV,\bW_0^\infty,\Phi)_{K,L_0}\to\cM(\bW_0^\infty,\Phi^\tc)_{L_0}\otimes_{O_{E_\Phi,(p)}}O_{E^\sharp_{\rV,\Phi,\fp}}
\end{align}
extending the projection to the second factor in Lemma \ref{le:shimura_variety}.

\begin{proposition}\label{pr:integral_canonical}
Let $\rV$ be as in the beginning of this subsection. Let $\fp$ be a prime of $F$ inert in $E$ such that its underlying rational prime is unramified in $E$. Denote by ${\underline\fp}$ the set of all primes of $F$ with the same residue characteristic of $\fp$ that are inert in $E$. We fix a subgroup $K_\fq\subseteq\rU(\rV)(F_\fq)$ that is the stabilizer of a self-dual lattice in $\rV\otimes_FF_\fq$ for every $\fq\in{\underline\fp}$, and put $K_{\underline\fp}\coloneqq\prod_{\fq\in{\underline\fp}}K_\fq$. Then the Shimura variety
\[
\Sh(\rG,\rh_{\rV,\tau'})_{K_{\underline\fp}}\coloneqq\varprojlim_{K^{\underline\fp}}\Sh(\rG,\rh_{\rV,\tau'})_{K^{\underline\fp} K_{\underline\fp}}
\]
(see Remark \ref{re:picard} for the notation) over $E$ has a (smooth) integral canonical model over $O_{E_\fp}$ in the sense of \cite{Mil92}*{Definition~2.9}.
\end{proposition}

\begin{proof}
Let $p$ be the underlying rational prime of $\fp$. Choose auxiliary data $\Phi$, $\bW_0^\infty$ and $L_0$ as in the previous discussion, such that $L_0=L_0^p\times (L_0)_p$ in which $(L_0)_p$ is the stabilizer of a self-dual lattice in $\bW_0^\infty\otimes_{\bA^\infty}\dQ_p$ and $L_0^p$ is sufficiently small. Write $K$ for $K^{\underline\fp}\times K_{\underline\fp}$. It suffices to consider those $K^{\underline\fp}$ that are of the form $K^p\times K_p^{\underline\fp}$ with $K^p$ sufficiently small.

Put $\cM\coloneqq\cM(\bW_0^\infty,\Phi^\tc)_{L_0}\otimes_{O_{E_\Phi,(p)}}O_{E^\sharp_{\rV,\Phi,\fp}}$ as in \eqref{eq:integral_canonical}, which is a finite \'{e}tale scheme over $O_{E_\fp}$. Put $M\coloneqq\cM\otimes_{\dZ_p}\dQ_p$. Then we have canonical isomorphisms
\begin{align*}
\cM(\rV,\bW_0^\infty,\Phi)_{K,L_0}\times_{\cM}M
&\simeq\cM(\rV,\bW_0^\infty,\Phi)_{K,L_0}\otimes_{O_{E^\sharp_{\rV,\Phi,\fp}}}E^\sharp_{\rV,\Phi,\fp}\\
&\simeq\Sh(\rG,\rh_{\rV,\tau'})_K\times_{E}\(\rM(\bW^\infty_0,\Phi^\tc)_{L_0}\otimes_{E_\Phi}E^\sharp_{\rV,\Phi,\fp}\)\\
&\simeq\Sh(\rG,\rh_{\rV,\tau'})_K\times_{E}M
\end{align*}
by Lemma \ref{le:shimura_variety}. Take a connected component $\cM^0$ of $\cM$, which is isomorphic to $\Spec O_{E'}$ for some unramified finite extension $E'/E_\fp$, with the generic fiber $M^0\coloneqq\cM^0\otimes_{\dZ_p}\dQ_p$. Put $\cM(\rV,\bW_0^\infty,\Phi)_{K,L_0}^0\coloneqq \bq_0^{-1}\cM^0$. Then we have a canonical isomorphism
\[
\cM(\rV,\bW_0^\infty,\Phi)_{K,L_0}^0\times_{\cM^0}M^0\simeq\Sh(\rG,\rh_{\rV,\tau'})_K\times_{E}M^0.
\]
Thus, it suffices to show that $\varprojlim_{K^{\underline\fp}}\cM(\rV,\bW_0^\infty,\Phi)_{K^{\underline\fp} K_{\underline\fp},L_0}^0$ is an integral canonical model over $\cM^0$. We now modify the proof of \cite{Mil92}*{Theorem~2.10}. Take an integral regular scheme $Y$ over $\cM^0$ such that $U\coloneqq Y\times_{\cM^0}M^0$ is dense in $Y$, with a morphism $\alpha\colon U\to\varprojlim_{K^{\underline\fp}}\cM(\rV,\bW_0^\infty,\Phi)_{K^{\underline\fp} K_{\underline\fp},L_0}^0\times_{\cM^0}M^0$. This is equivalent to giving data $(A_0,i_0,\lambda_0,\eta_0^p;A,i,\lambda,\eta^p,\eta_p^\spl)$ as in Definition \ref{de:integral_connection}, but with
\[
\eta^p\colon\rV\otimes_\dQ\bA^{\infty,p}\xrightarrow{\sim}
\Hom_{E\otimes_\dQ\bA^{\infty,p}}(\rH^{\et}_1(A_{0\eta},\bA^{\infty,p}),\rH^{\et}_1(A_\eta,\bA^{\infty,p}))
\]
being a $\pi_1(U,\eta)$-invariant isometry, and
\[
\eta_p^\spl\colon\prod_{\fq\in\underline\fp}\rV\otimes_E E_{\fq^-}\xrightarrow{\sim}\prod_{\fq\in\underline\fp}
\Hom_{O_{E_{\fq^-}}}\(A_{0\eta}[(\fq^-)^\infty],A_\eta[(\fq^-)^\infty]\)\otimes_{O_{E_{\fq^-}}}E_{\fq^-}
\]
being a $\pi_1(U,\eta)$-invariant isomorphism, where $\eta$ is a geometric generic point of $Y$; and with the partial data $(A_0,i_0,\lambda_0,\eta_0^p)$ extending uniquely to $Y$. In particular, the action of $\pi_1(U,\eta)$ on $\rH^{\et}_1(A_{0\eta},\bA^{\infty,p})$ factors through $\pi_1(Y,\eta)$, and that on $\Hom_{E\otimes_\dQ\bA^{\infty,p}}(\rH^{\et}_1(A_{0\eta},\bA^{\infty,p}),\rH^{\et}_1(A_\eta,\bA^{\infty,p}))$ is trivial. Thus, the action of $\pi_1(U,\eta)$ on $\rH^{\et}_1(A_\eta,\bA^{\infty,p})$ factors through $\pi_1(Y,\eta)$. By \cite{Mil92}*{Propositions~2.11, 2.13, 2.14}, the triple $(A,i_A,\theta_A)$ extends uniquely to $Y$. Then it is clear that $(\eta^p,\eta_p^\spl)$ extends uniquely as well.

We then conclude that $\varprojlim_{K^{\underline\fp}}\cM(\rV,\bW_0^\infty,\Phi)_{K^{\underline\fp} K_{\underline\fp},L_0}^0$ is an integral canonical model over $\cM^0$. The proposition follows.
\end{proof}

\begin{definition}\label{de:integral_canonical}
We denote by $\cS(\rG,\rh_{\rV,\tau'})_{K_{\underline\fp}}$ the integral canonical model of $\Sh(\rG,\rh_{\rV,\tau'})_{K_{\underline\fp}}$ over $O_{E_\fp}$ in Proposition \ref{pr:integral_canonical}, on which the action of $\rU(\rV)({\bA_F^\infty}^{,\underline\fp})$ extends uniquely by the extension property. For an open compact subgroup $K\subseteq\rG(\bA^\infty)=\rU(\rV)(\bA_F^\infty)$ of the form $K=K^{\underline\fp}\times K_{\underline\fp}$, we put
\[
\cS(\rG,\rh_{\rV,\tau'})_K\coloneqq \cS(\rG,\rh_{\rV,\tau'})_{K_{\underline\fp}}/K^{\underline\fp}
\]
which we refer as the \emph{canonical integral model} of $\Sh(\rG,\rh_{\rV,\tau'})_K$ over $O_{E_\fp}$. It is proper/smooth if $\Sh(\rG,\rh_{\rV,\tau'})_K$ is.
\end{definition}

\begin{remark}\label{re:integral_canonical}
The extension property of integral canonical models together with Lemma \ref{le:shimura_variety} implies that we have a canonical isomorphism
\[
\cM(\rV,\bW_0^\infty,\Phi)_{K,L_0}\simeq
\cS(\rG,\rh_{\rV,\tau'})_K\times_{O_{E_\fp}}\(\cM(\bW_0^\infty,\Phi^\tc)_{L_0}\otimes_{O_{E_\Phi,(p)}}O_{E^\sharp_{\rV,\Phi,\fp}}\)
\]
under which $\bq_0$ \eqref{eq:integral_canonical} corresponds to the projection to the second factor.
\end{remark}

\begin{remark}
Proposition \ref{pr:integral_canonical} is slightly stronger than the main result in \cite{Kis10}, as the latter has to assume that $K_p$ is hyperspecial maximal.
\end{remark}

At last, we review the uniformization of $\cM(\rV,\bW_0^\infty,\Phi)_{K,L_0}$ along the basic locus, which is only used in Subsection \ref{ss:orbital}. Let $E_\fp^\ur$ be the maximal unramified extension of $E_\fp$ inside $E_\fp^\ac$. Let $k\coloneqq O_{E_\fp^\ur}\otimes_\dZ\dF_p$ be the residue field of $E_\fp^\ur$. Put
\[
\cM(\rV,\bW_0^\infty,\Phi)_{K,L_0}^\ur\coloneqq\cM(\rV,\bW_0^\infty,\Phi)_{K,L_0}\otimes_{O_{E^\sharp_{\rV,\Phi,\fp}}}O_{E_\fp^\ur}
\]
and
\[
\cM(\rV,\bW_0^\infty,\Phi)_{K_{\underline\fp},L_0}^\ur\coloneqq
\varprojlim_{K^p}\varprojlim_{K_p^{\underline\fp}}\cM(\rV,\bW_0^\infty,\Phi)_{K^pK_p^{\underline\fp}K_{\underline\fp},L_0}^\ur.
\]

\begin{definition}\label{de:supersingular}
For an algebraically closed field $k'$ containing $k$, we say that a $k'$-point
\[
(A_0,i_0,\lambda_0,\eta_0^p;A,i,\lambda,\eta^p,\eta_p^\spl)\in\cM(\rV,\bW_0^\infty,\Phi)_{K_{\underline\fp},L_0}^\ur(k')
\]
is \emph{supersingular} if the $p$-divisible group $A[\fp^\infty]$ is supersingular.
\end{definition}

Denote by $\cM(\rV,\bW_0^\infty,\Phi)_{K_{\underline\fp},L_0}^\ssl$ the supersingular locus of $\cM(\rV,\bW_0^\infty,\Phi)_{K_{\underline\fp},L_0}^\ur\otimes_{O_{E_\fp^\ur}}k$, which is a Zariski closed subset. Denote by $\cM(\rV,\bW_0^\infty,\Phi)_{K_{\underline\fp},L_0}^{\ssl,\wedge}$ the completion of $\cM(\rV,\bW_0^\infty,\Phi)_{K,L_0}^\ur$ along $\cM(\rV,\bW_0^\infty,\Phi)_{K_{\underline\fp},L_0}^\ssl$, which is a formal scheme over $O_{E_\fp^\ur}^\wedge$, where $O_{E_\fp^\ur}^\wedge$ is the completion of $O_{E_\fp^\ur}$. The description of the uniformization of $\cM(\rV,\bW_0^\infty,\Phi)_{K_{\underline\fp},L_0}^{\ssl,\wedge}$ depends on the choice of a point
\begin{align}\label{eq:frame}
\bbP=(\bbA_0,\bbi_0,\bblambda_0,\bbeta_0^p;\bbA,\bbi,\bblambda,\bbeta^p,\bbeta_p^\spl)
\in\cM(\rV,\bW_0^\infty,\Phi)_{K_{\underline\fp},L_0}^\ur(O_{E_\fp^\ur})
\end{align}
such that $\bbP_k$ is supersingular. In particular, we have the induced section
\begin{align}\label{eq:uniformization1}
\bbP^\wedge\colon\Spf O_{E_\fp^\ur}^\wedge\to\cM(\rV,\bW_0^\infty,\Phi)_{K_{\underline\fp},L_0}^{\ssl,\wedge}.
\end{align}
We denote the base change of $(\bbA_0,\bbi_0,\bblambda_0;\bbA,\bbi,\bblambda)$ to $k$ by $(\bbA_{0k},\bbi_{0k},\bblambda_{0k};\bbA_k,\bbi_k,\bblambda_k)$. Moreover, we use $\Spec k$ as the reference point in the level structures $(\bbeta_0^p;\bbeta^p,\bbeta_p^\spl)$. We now attach to $\bbP$ two objects: a formal scheme $\cN$ over $O_{E_\fp^\ur}^\wedge$, and a new hermitian space $\bar\rV$ over $E$.

\begin{itemize}
  \item By a slight abuse of notation, let $(\bbX,\bbi,\bblambda)$ be the supersingular unitary $O_{F_\fp}$-module induced from $(\bbA[\fp^\infty],\bbi[\fp^\infty],\bblambda[\fp^\infty])$ via \cite{Mih}*{Theorem~3.3}. Similarly, we have $(\bbX_0,\bbi_0,\bblambda_0)$ obtained from $(\bbA_0,\bbi_0,\bblambda_0)$. Let $\cN$ be the relative Rapoport--Zink space parameterizing quasi-isogenies of the supersingular unitary $O_{F_\fp}$-module $(\bbX_k,\bbi_k,\bblambda_k)$ of signature $(n-1,1)$ as introduced in Subsection \ref{ss:afl}, which is a formal scheme over $O_{E_\fp^\ur}^\wedge$. In particular, the point $\bbP$ induces a section
      \begin{align}\label{eq:uniformization2}
      \bbP^\wedge_\loc\colon\Spf O_{E_\fp^\ur}^\wedge\to\cN.
      \end{align}

  \item Now we define the new hermitian space. Put
      \[
      \bar\rV\coloneqq\Hom_k((\bbA_{0k},\bbi_{0k}),(\bbA_k,\bbi_k))_\dQ,
      \]
      which is an $E$-vector space through $\bbi_{0k}$. We define a map
      \[
      (\;,\;)_{\bar\rV}\colon\bar\rV\times\bar\rV\to E
      \]
      given by the formula
      \begin{align}\label{eq:nearby_form}
      (x,y)_{\bar\rV}=\bbi_{0k}^{-1}\(\bblambda_{0k}^\vee\circ y^\vee\circ\bblambda_k\circ x\)\in\bbi_{0k}^{-1}\End_k((\bbA_{0k},\bbi_{0k}))=E,
      \end{align}
      which is a hermitian form on $\bar\rV$.
\end{itemize}

\begin{lem}\label{le:definite_nearby}
The hermitian space $\bar\rV,(\;,\;)_{\bar\rV}$ has the following properties:
\begin{enumerate}
  \item $\bar\rV$ is of dimension $n$ over $E$.

  \item $\bar\rV$ is totally positive definite.

  \item The composite map
    \begin{align*}
    \bar\rV\otimes_\dQ\bA^{\infty,p}
    \to\Hom_{E\otimes_\dQ\bA^{\infty,p}}(\rH^{\et}_1(\bbA_{0k},\bA^{\infty,p}),\rH^{\et}_1(\bbA_k,\bA^{\infty,p}))
    \xrightarrow{(\bbeta^p)^{-1}}\rV\otimes_\dQ\bA^{\infty,p}
    \end{align*}
    is an isomorphism of hermitian spaces over $F\otimes_\dQ\bA^{\infty,p}$.

  \item The composite map
    \begin{align*}
    \prod_{\fq\in\r{Spl}_p}\bar\rV\otimes_E E_{\fq^-}
    \to\prod_{\fq\in\r{Spl}_p}\Hom_{O_{E_{\fq^-}}}\(\bbA_{0k}[(\fq^-)^\infty],\bbA_k[(\fq^-)^\infty]\)\otimes_{O_{E_{\fq^-}}}E_{\fq^-}
    \xrightarrow{(\bbeta_p^\spl)^{-1}}\prod_{\fq\in\r{Spl}_p}\rV\otimes_E E_{\fq^-}
    \end{align*}
    is an isomorphism of $\prod_{\fq\in\r{Spl}_p}E_{\fq^-}$-modules.

  \item For every $\fq\in\underline\fp$, the canonical map
      \[
      \bar\rV\otimes_FF_\fq\to
      \Hom_k((\bbA_{0k}[\fq^\infty],\bbi_{0k}[\fq^\infty]),(\bbA_k[\fq^\infty],\bbi_k[\fq^\infty]))\otimes_{O_{F_\fq}}F_\fq
      \]
      is an isomorphism of $E_\fq$-vector spaces.

  \item For every $\fq\in\underline\fp\setminus\{\fp\}$, $\bar\Lambda_\fq\coloneqq\Hom_k((\bbA_{0k}[\fq^\infty],\bbi_{0k}[\fq^\infty]),(\bbA_k[\fq^\infty],\bbi_k[\fq^\infty]))$ is a self-dual lattice in $\bar\rV\otimes_FF_\fq$.

  \item $\bar\rV\otimes_FF_\fp$ does not admit a self-dual lattice.
\end{enumerate}
\end{lem}

\begin{proof}
We first show that the canonical map
\begin{align}\label{eq:definite_nearby}
\bar\rV\otimes_\dQ\dQ_p\to\Hom_k((\bbA_{0k}[p^\infty],\bbi_{0k}[p^\infty]),(\bbA_k[p^\infty],\bbi_k[p^\infty]))\otimes_{\dZ_p}\dQ_p
\end{align}
is an isomorphism. Let $O_D$ be the $O_{F,(p)}$-algebra of endomorphisms of $(\bbA_{0k},\bbi_{0k}\res O_{F,(p)},\bblambda_{0k})$, and put $D\coloneqq O_D\otimes_{O_{F,(p)}}F$. Then $D$ is a totally definite (division) quaternion algebra over $F$ which contains $E$ via $\bbi_{0k}$. We write $D=E\oplus Ej$ for some element $j\in O_D\setminus pO_D$ such that $j^{-1}ej=e^\tc$ for every $e\in E$. Choose an element $f\in O_F$ such that $f\in\fp$ but $f\not\in\fq$ for every other prime $\fq$ of $F$ above $p$; and put $j'\coloneqq j+f$. We define a new action $\bbi'_{0k}$ of $O_{E,(p)}$ on $\bbB_k$ via the formula $\bbi'_{0k}(e)=j'^{-1}\circ\bbi_{0k}(e)\circ j'$. Then $(\bbA_{0k},\bbi'_{0k})$ is an $(E,\Phi')$-abelian scheme over $k$, where $\Phi'\coloneqq(\Phi^\tc\setminus\{\tau'^\tc\})\cup\{\tau'\}$, with a polarization $\bblambda'_{0k}\coloneqq(j')^*\bblambda_{0k}$. Now by \cite{RZ96}*{Proposition~6.29}, $(\bbA_k,\bbi_k)$ is quasi-isogenous to $(\bbA_{0k},\bbi_{0k})^{n-1}\times(\bbA_{0k},\bbi'_{0k})$. In particular, \eqref{eq:definite_nearby} is an isomorphism. From this, (1), (3), (4), and (5) follow immediately. Part (2) can be proved in the same way as \cite{KR14}*{Lemma~2.7}. Part (7) is a consequence of (1--6) and the Hasse principle.

It remains to show (6). By the above discussion, $(\bbA_k[\fq^\infty],\bbi_k[\fq^\infty])$ is quasi-isogenous to $(\bbA_{0k}[\fq^\infty],\bbi_{0k}[\fq^\infty])^{\oplus n}$ for $\fq\in\underline\fp\setminus\{\fp\}$. Then they must be isomorphic. In particular, the induced hermitian form on $\bar\Lambda_\fq$ is given by the identity matrix under some basis. Thus, we obtain (6).
\end{proof}

Lemma \ref{le:definite_nearby}(3,4) gives rise to an isomorphism
\begin{align}\label{eq:nearby_definite}
\iota_{\bbP}\colon\bar\rV\otimes_F{\bA_F^\infty}^{,\underline\fp}\to\rV\otimes_F{\bA_F^\infty}^{,\underline\fp}
\end{align}
of hermitian spaces over ${\bA_F^\infty}^{,\underline\fp}$. Let $\bar{K}_\fq$ be the stabilizer of $\bar\Lambda_\fq$ in Lemma \ref{le:definite_nearby}(6) for every $\fq\in\underline\fp\setminus\{\fp\}$, which is a hyperspecial maximal subgroup of $\rU(\bar\rV)(F_\fq)$.

Let $\cM(\bW_0^\infty,\Phi^\tc)_{L_0}^\wedge$ be the completion of $\cM(\bW_0^\infty,\Phi^\tc)_{L_0}\otimes_{O_{E_\Phi,(p)}}O_{E_\fp^\ur}$ along the special fiber, which is isomorphic to a finite disjoint union of $\Spf O_{E_\fp^\ur}^\wedge$. Then \eqref{eq:integral_canonical} induces a morphism
\[
\bq_0^\wedge\colon\cM(\rV,\bW_0^\infty,\Phi)_{K_{\underline\fp},L_0}^{\ssl,\wedge}\to\cM(\bW_0^\infty,\Phi^\tc)_{L_0}^\wedge
\]
of formal schemes over $O_{E_\fp^\ur}^\wedge$.

\begin{proposition}\label{pr:uniformization}
The chosen point $\bbP$ \eqref{eq:frame} induces the following Cartesian diagram
\[
\xymatrix{
\rU(\bar\rV)(F)\backslash\(\cN\times\rU(\bar\rV)({\bA_F^\infty}^{,\fp})/\prod_{\fq\in\underline\fp\setminus\{\fp\}}\bar{K}_\fq\)
\ar[r]\ar[d]_-{\tu_{\bbP}} & \Spf O_{E_\fp^\ur}^\wedge \ar[d]^-{\bq_0^\wedge\circ\bbP^\wedge} \\
\cM(\rV,\bW_0^\infty,\Phi)_{K_{\underline\fp},L_0}^{\ssl,\wedge} \ar[r]^-{\bq_0^\wedge} & \cM(\bW_0^\infty,\Phi^\tc)_{L_0}^\wedge
}
\]
of formal schemes over $O_{E_\fp^\ur}^\wedge$, satisfying
\begin{itemize}
  \item $\tu_{\bbP}\circ(\bbP^\wedge_\loc,1)=\bbP^\wedge$ (see \eqref{eq:uniformization1} and \eqref{eq:uniformization2}), and

  \item $\tu_{\bbP}\circ\tT_{\bar{g}}=\tT_g\circ\tu_{\bbP}$ for every $g\in\rU(\rV)({\bA_F^\infty}^{,\underline\fp})$ and $\bar{g}\in\rU(\bar\rV)({\bA_F^\infty}^{,\underline\fp})$ that correspond under $\iota_{\bbP}$ \eqref{eq:nearby_definite}, where $\tT_g$ (resp.\ $\tT_{\bar{g}}$) denotes the Hecke translation on the target (resp.\ source) of $\tu_{\bbP}$.
\end{itemize}
\end{proposition}

\begin{proof}
The proof is very similar to \cite{RZ96}*{Theorem~6.30}. For readers' convenience, we will describe the morphism
\begin{align}\label{eq:uniformization3}
\tv_{\bbP}\colon\cM_{\bbP}\to
\rU(\bar\rV)(F)\backslash\(\cN\times\rU(\bar\rV)({\bA_F^\infty}^{,\fp})/\prod_{\fq\in\underline\fp\setminus\{\fp\}}\bar{K}_\fq\),
\end{align}
where $\cM_{\bbP}$ is the pullback of $\bq_0^\wedge$ along $\bq_0^\wedge\circ\bbP^\wedge$, for which $\tu_{\bbP}$ is the inverse. This is the hardest step; and in particular, we will see how this morphism depends on $\bbP$.

Let $S$ be a connected scheme in $\Sch'_{/O_{E_\fp^\ur}^\wedge}$ on which $p$ is locally nilpotent, with a chosen geometric point $s\in S(k)$. Take a point
$P=(A_0,i_0,\lambda_0,\eta_0^p;A,i,\lambda,\eta^p,\eta_p^\spl)\in\cM_{\bbP}(S)$, where $(A_0,i_0,\lambda_0,\eta_0^p)$ is the base change of $(\bbA_0,\bbi_0,\bblambda_0,\bbeta_0^p)$ to $S$. By \cite{RZ96}*{Proposition~6.29}, we can choose an $O_E$-linear quasi-isogeny
\[
\rho\colon A\times_SS_k\to\bbA_k\times_kS_k
\]
such that $\rho^*\bblambda_k=\lambda_k$. Then $(A[\fp^\infty],i[\fp^\infty],\lambda[\fp^\infty];\rho[\fp^\infty])$ can be regarded as an element in $\cN(S)$ by \cite{Mih}*{Theorem~3.3}. The composite map
\begin{align*}
\bar\rV\otimes_\dQ\bA^{\infty,p}&\xrightarrow{\iota_{\bbP}}\rV\otimes_\dQ\bA^{\infty,p}
\xrightarrow{\eta^p}
\Hom_{E\otimes_\dQ\bA^{\infty,p}}(\rH^{\et}_1(\bbA_{0k},\bA^{\infty,p}),\rH^{\et}_1(A_s,\bA^{\infty,p}))\\
&\xrightarrow{\rho_{s*}\circ}
\Hom_{E\otimes_\dQ\bA^{\infty,p}}(\rH^{\et}_1(\bbA_{0k},\bA^{\infty,p}),\rH^{\et}_1(\bbA_k,\bA^{\infty,p}))
=\bar\rV\otimes_\dQ\bA^{\infty,p}
\end{align*}
is an isometry, which gives rise to an element $g^p_P\in\rU(\bar\rV)(\bA_F^{\infty,p})$. The same process will produce an element $g_{P,p}^\spl\in\prod_{\fq\in\r{Spl}_p}\rU(\bar\rV)(F_\fq)$. For every $\fq\in\underline\fp\setminus\{\fp\}$, the image of the map
\begin{align*}
\rho_{s*}\circ &\colon\Hom_k((\bbA_{0k}[\fq^\infty],\bbi_{0k}[\fq^\infty]),(A_s[\fq^\infty],i_s[\fq^\infty])) \\
&\to
\Hom_k((\bbA_{0k}[\fq^\infty],\bbi_{0k}[\fq^\infty]),(\bbA_k[\fq^\infty],\bbi_k[\fq^\infty]))\otimes_{O_{F_\fq}}F_\fq=\bar\rV\otimes_FF_\fq
\end{align*}
is a self-dual lattice, say $\Lambda_{P,\fq}$. Therefore, there exists a unique element $g_{P,\fq}\in\rU(\bar\rV)(F_\fq)/\bar{K}_\fq$ such that $g_{P,\fq}\Lambda_{P,\fq}=\bar\Lambda_\fq$. Together, we obtain an element
\[
\((A[\fp^\infty],i[\fp^\infty],\lambda[\fp^\infty];\rho[\fp^\infty]),g_P^p,g_{P,p}^\spl,(g_{P,\fq})_\fq\)\in
\cN(S)\times\rU(\bar\rV)({\bA_F^\infty}^{,\fp})/\prod_{\fq\in\underline\fp\setminus\{\fp\}}\bar{K}_\fq
\]
depending on the choice of $\rho$. However, changing $\rho$ will result the left multiplication by an element in $\rU(\bar\rV)(F)$. Thus, the element
\[
\tv_{\bbP}(P)\coloneqq\((A[\fp^\infty],i[\fp^\infty],\lambda[\fp^\infty];\rho[\fp^\infty]),g_P^p,g_{P,p}^\spl,(g_{P,\fq})_\fq\)
\]
is a well-defined element in the right-hand side of \eqref{eq:uniformization3}. The construction of the inverse of $\tv_{\bbP}$, which is nothing but $\tu_{\bbP}$, is easy by Dieudonn\'{e} theory. We leave the details to the readers; it is the same argument in \cite{RZ96}.
\end{proof}

\begin{remark}
In fact, the morphism $\tu_{\bbP}$ in Proposition \ref{pr:uniformization} is compatible with more Hecke operators. Consider a prime $\fq\in\underline\fp\setminus\{\fp\}$. For every double coset $K_\fq g K_\fq\subseteq\rU(\rV)(F_\fq)$, we have the Hecke correspondence $\tT_{K_\fq g K_\fq}$ on the target of $\tu_{\bbP}$ which is simply the Zariski closure of the usual Hecke correspondence on the generic fiber; it is in fact \'{e}tale. Then we have $\tu_{\bbP}^*\tT_{K_\fq g K_\fq}=\tT_{\bar{K}_\fq\bar{g}\bar{K}_\fq}$ if $K_\fq g K_\fq=\bar{K}_\fq\bar{g}\bar{K}_\fq$ under the canonical isomorphism $K_\fq\backslash\rU(\rV)(F_\fq)/K_\fq\simeq\bar{K}_\fq\backslash\rU(\bar\rV)(F_\fq)/\bar{K}_\fq$. Here, $\tT_{\bar{K}_\fq\bar{g}\bar{K}_\fq}$ denotes the set-theoretical Hecke correspondence on the source of $\tu_{\bbP}$.
\end{remark}

\section{Cohomology of unitary Shimura curves}
\label{ss:d}

In this appendix, we compute the cohomology of Shimura curves associated to \emph{isometry} groups of hermitian spaces of rank $2$, as Galois--Hecke modules. In Subsection \ref{ss:local_theta}, we collect some results about local oscillator representations of unitary groups of general rank. In Subsection \ref{ss:setup}, we recall some facts and introduce some notation about cohomology of Shimura varieties in general.
The last two subsections concern the cohomology of unitary Shimura curves, for the statements and for the proof, respectively. These statements are only used in the proof of Theorem \ref{th:cm_albanese_pre} and Theorem \ref{th:cm_albanese} in the main part of the article.

\subsection{Oscillator representations of local unitary groups}
\label{ss:local_theta}

Let $F$ be a local field whose characteristic is not $2$. Let $E$ be an \'{e}tale $F$-algebra of rank $2$. Denote by $\tc$ the unique nontrivial involution on $E$ that fixes $F$, and put $E^-\coloneqq\{x\in E\res x+x^\tc=0\}$ and $E^1\coloneqq\{x\in E\res xx^\tc=1\}$. Let $\rV,(\;,\;)_\rV$ be a (non-degenerate) hermitian space over $E$ (with respect to $\tc$) of rank $n\geq 2$.

We recall the construction of oscillator representations of $\rU(\rV)$ in three steps.
\begin{description}
  \item[Step 1] Choose an element $\varepsilon\in E^{-\times}/\Nm_{E/F}E^\times$. Let $\rV_\varepsilon$ be the underlying $F$-vector space of $\rV$ equipped with the form $\Tr_{E/F}\varepsilon(\;,\;)_\rV$, which becomes a symplectic space.\footnote{More precisely, we have to choose an element in $E^{-\times}$ in the coset $\varepsilon$; and it is known that the resulting oscillator representation depends only on $\varepsilon$.} Let $\Mp(\rV_\varepsilon)$ be the metaplectic group of $\rV_\varepsilon$ with center $\dC^1$. Then we have the oscillator representation $\omega(\varepsilon)$ of $\Mp(\rV_\varepsilon)$ using the standard additive character $\psi_F$.

  \item[Step 2] Choose a character $\mu\colon E^\times\to\dC^1$ such that $\mu\res{F^\times}$ is the unique character whose kernel is exactly $\Nm_{E/F}E^\times$. Then we have the induced homomorphism $\iota_\mu\colon\rU(\rV)\to\Mp(\rV_\varepsilon)$ (see, for example, \cite{HKS}*{Section~1}). Put $\omega(\mu,\varepsilon)\coloneqq\omega(\varepsilon)\circ\iota_\mu$.

  \item[Step 3] Choose a character $\chi\colon E^1\to\dC^1$. Let $\omega(\mu,\varepsilon,\chi)$ be the maximal quotient of the representation $\omega(\varepsilon,\mu)$ of $\rU(\rV)$ with central character $\chi$.
\end{description}

For $\chi$ in Step 3, we define a character $\check\chi$ of $E^\times$ via the formula $\check\chi(x)=\chi(x/x^\tc)$.

\begin{lem}\label{le:weil_nonarch}
Suppose that $F$ is nonarchimedean. Then $\omega(\mu,\varepsilon,\chi)$ is irreducible and admissible. Moreover,
\begin{enumerate}
  \item $\omega(\mu,\varepsilon,\chi)$ is zero if and only if $E$ is a field, $\rV$ is anisotropic (in particular $n=2$), and $\check\chi=\mu^2$.

  \item The contragredient representation of $\omega(\mu,\varepsilon,\chi)$ is isomorphic to $\omega(\mu^\tc,-\varepsilon,\chi^{-1})$, where $\mu^\tc\coloneqq\mu\circ\tc$ as usual.

  \item If $n\geq 3$, then $\omega(\mu',\varepsilon',\chi')$ is isomorphic to $\omega(\mu,\varepsilon,\chi)$ if and only if $(\mu',\varepsilon',\chi')=(\mu,\varepsilon,\chi)$.

  \item If $n=2$ and $\omega(\mu,\varepsilon,\chi)$ is nonzero, then $\omega(\mu',\varepsilon',\chi')$ is isomorphic to $\omega(\mu,\varepsilon,\chi)$ if and only if either $(\mu',\varepsilon',\chi')=(\mu,\varepsilon,\chi)$, or $\mu'=\mu^\tc\check\chi$, $\chi'=\chi$, and $\varepsilon'=\varepsilon$ (resp.\ $\varepsilon'\neq\varepsilon$) when $\rV$ is isotropic (resp.\ anisotropic).
\end{enumerate}

\end{lem}

\begin{proof}
We consider first the case where $E=F\times F$. We identify $\rU(\rV)$ with $\GL_n(F)$ and $E^-$ with $F$ through the first factor; and write $\mu=\nu\boxtimes\nu^{-1}$. Note that the first component of $\check\chi$ is simply $\chi$. Let $\rQ_{n-1,1}$ be the standard parabolic subgroup of $\GL_n$ whose Levi is $\GL_{n-1}\times\GL_1$. Then $\omega(\mu,\varepsilon,\chi)$ is isomorphic to the unitary induction from $\rQ_{n-1,1}(F)$ to $\GL_n(F)$ of the (unitary) character $(\nu\circ\det)\boxtimes\chi\nu^{1-n}$ of $\GL_{n-1}(F)\times\GL_1(F)$ (hence of $\rQ_{n-1,1}(F)$). See for example \cite{GR90}*{2.6}. The lemma follows from such description.

Now we assume that $E$ is a field. The fact that $\omega(\mu,\varepsilon,\chi)$ is irreducible is a special case of the Howe duality; see for example \cite{GT16}*{Theorem~1.1(1)}.

For (1), the fact that $\omega(\mu,\varepsilon,\chi)$ is nonzero unless in the exceptional case in (1) follows from the persistence property \cite{HKS}*{Proposition~5.1(iii)}, and the first occurrence speculation \cite{HKS}*{Speculation~7.5 \& Speculation~7.6} (which has been proved as \cite{SZ15}*{Theorem~1.10}). Note that in the exceptional case, the first occurrence of the theta lifting of the trivial character in the split even tower is $0$; therefore its first occurrence in the nonsplit even tower is $4$. See \cite{HKS}*{p.986} for more details.

Note that, since $E^1$ is compact, we have a canonical isomorphism of representations of $\rU(\rV)$
\[
\omega(\mu,\varepsilon)\simeq\bigoplus_{\chi}\omega(\mu,\varepsilon,\chi).
\]

For (2), note that under the canonical isomorphism $\Mp(\rV_{\varepsilon})\simeq\Mp(\rV_{-\varepsilon})$, the contragredient of $\omega(\varepsilon)$ is isomorphic to $\omega(-\varepsilon)$. Moreover, under such isomorphism, $\iota_{\mu}$ coincides with $\iota_{\mu^\tc}$ by \cite{HKS}*{Lemma~1.1 \& (1.8)}. Therefore, $\omega(\mu,\varepsilon)$ is contragredient to $\omega(\mu^\tc,-\varepsilon)$. Since $E^1$ is compact, we have a canonical isomorphism $\omega(\mu,\varepsilon)\simeq\bigoplus_{\chi}\omega(\mu,\varepsilon,\chi)$. Thus, $\omega(\mu,\varepsilon,\chi)$ is contragredient to $\omega(\mu^\tc,-\varepsilon,\chi^{-1})$ as both are irreducible with inverse central character, or both are zero.

For (3), it is known when $n=3$ by \cite{GR90}*{Proposition~5.1.4}. In fact, the same proof also works for $n>3$.

For (4), we first have $\chi=\chi'$. By the description of endoscopic packets for in \cite{GGP2}*{Section~8}, we must have either $\mu'=\mu$ or $\mu'=\mu^\tc\check\chi$. There are two cases.

Suppose that $\mu^\tc\check\chi=\mu$. Then $\omega(\mu,\varepsilon,\chi)=\{0\}$ when $\rV$ is anisotropic; and $\omega(\mu,\varepsilon,\chi)$ is not isomorphic to $\omega(\mu,\varepsilon',\chi)$ when $\rV$ is isotropic and $\varepsilon'\neq\varepsilon$. Thus, (4) follows.

Suppose that $\mu^\tc\check\chi\neq\mu$. Then the packet has four members, and we need to show that $\omega(\mu,\varepsilon,\chi)\simeq\omega(\mu^\tc\check\chi,\varepsilon',\chi)$ for $\varepsilon'=\varepsilon$ (resp.\ $\varepsilon'\neq\varepsilon$) when $\rV$ is isotropic (resp.\ anisotropic). We adopt the notation in \cite{GGP2}*{Section~8}. Let $M$ be the two-dimensional conjugate symplectic representation associated to the packet. Then $M$ has two non-isomorphic one dimensional conjugate symplectic representations. We write $M=M_1^\bullet\oplus M_2^\bullet=M_1^\circ\oplus M_2^\circ$ for the two different ways of ordering of direct summands. Thus, we obtain two ways of labelling for the four members in the packet, say $\{\pi^{++}_\bullet,\pi^{--}_\bullet,\pi^{+-}_\bullet,\pi^{-+}_\bullet\}$ and $\{\pi^{++}_\circ,\pi^{--}_\circ,\pi^{+-}_\circ,\pi^{-+}_\circ\}$, respectively. Then (4) is equivalent to the isomorphisms $\pi^{++}_\bullet\simeq\pi^{++}_\circ$, $\pi^{--}_\bullet\simeq\pi^{--}_\circ$, $\pi^{+-}_\bullet\simeq\pi^{-+}_\circ$, and $\pi^{-+}_\bullet\simeq\pi^{+-}_\circ$. However, these isomorphisms are consequences of \cite{GGP2}*{Theorem~10.2}.
\end{proof}

Now we take $F=\dR$ and $E=\dC$. Let $(p,q)$ be the signature of $\rV$. Then we may identify $\rU(\rV)$ with $\rU(p,q)_\dR$, the subgroup of $\Res_{\dC/\dR}\GL_n$ of elements preserving the hermitian form given by the matrix $\(\begin{smallmatrix}\rI_p&\\&-\rI_q\end{smallmatrix}\)$. Denote by $\fu_{p,q}$ the Lie algebra of $\rU(p,q)_\dR$ and fix a maximal compact subgroup $\rK_{p,q}$ of $\rU(p,q)_\dR(\dR)$. In the construction of $\omega(\mu,\varepsilon,\chi)$, the three parameters have the following possibilities:
\[
\mu_m(z)=\arg(z)^m,\; m\text{ odd integer};
\qquad\varepsilon=\pm i;
\qquad\chi_l(z)=z^l,\; l\in\dZ.
\]
To shorten notation, we denote by $\omega_{p,q}^{m,\pm,l}$ the representation $\omega(\mu_m,\pm i,\chi_l)$ of $\rU(p,q)_\dR$. It is well-known (see, for example, \cite{SW78}*{Section~4}) that $\omega_{p,q}^{m,\pm,l}$ is irreducible.

By the computation in \cite{BMM}*{Section~5}, up to equivalence, there are only two irreducible unitary representations $\pi$ of $\rU(n-1,1)_\dR$ such that $\rH^1(\fu_{n-1,1},\rK_{n-1,1};\pi)\neq\{0\}$, in which case the cohomology has dimension $1$ for both representations. Let us label them by
$\pi_{n-1,1}^{1,0}$ and $\pi_{n-1,1}^{0,1}$ in the way that $\rH
^1(\fu_{n-1,1},\rK_{n-1,1};\pi_{n-1,1}^{1,0})$ and $\rH^1(\fu_{n-1,1},\rK_{n-1,1};\pi_{n-1,1}^{0,1})$ have Hodge types
$(1,0)$ and $(0,1)$, respectively.

\begin{lem}\label{le:weil_arch}
Let the notation be as above.
\begin{enumerate}
  \item Among the representations $\omega_{n,0}^{m,\pm,l}$, only $\omega_{n,0}^{1,+,0}$ and $\omega_{n,0}^{-1,-,0}$ are the trivial character.

  \item If $n\geq 3$, then in the set $\{\omega_{n-1,1}^{m,\pm,l}\}$, only $\omega_{n-1,1}^{-1,-,0}$ (resp.\ $\omega_{n-1,1}^{1,+,0}$) is isomorphic to $\pi_{n-1,1}^{1,0}$ (resp.\ $\pi_{n-1,1}^{0,1}$).

  \item If $n=2$, then in the set $\{\omega_{1,1}^{m,\pm,l}\}$, only $\omega_{1,1}^{-1,-,0}$ and $\omega_{1,1}^{1,-,0}$ (resp.\ $\omega_{1,1}^{1,+,0}$ and $\omega_{1,1}^{-1,+,0}$) are isomorphic to $\pi_{1,1}^{1,0}$ (resp.\ $\pi_{1,1}^{0,1}$).
\end{enumerate}
\end{lem}

\begin{proof}
The explicit formulae for the $\rK_{p,q}$-type of $\omega_{p,q}^{m,\pm,l}$ can be found in, for example, \cite{KK07}*{Theorem~5.4} with $p'+q'=1$. In particular, (1) follows directly.

For (2) and (3), it is shown in \cite{BMM}*{Section~5} that both $\pi_{n-1,1}^{1,0}$ and $\pi_{n-1,1}^{0,1}$ are isomorphic to some $\omega_{n-1,1}^{m,\pm,l}$. Comparing the formula for the highest weights in \cite{BMM}*{5.7} with $p=n-1,q=1,a+b=1(\leq p)$ with \cite{KK07}*{Theorem~5.4}, we obtain the assertions.
\end{proof}

\subsection{Setup for cohomology of Shimura varieties}
\label{ss:setup}

Let us recall some general facts about cohomology of Shimura varieties. Let $(\rG,\rh)$ be a Shimura data with $E\subseteq\dC$ its reflex field. In particular, $\rG$ is a reductive group over $\dQ$. Let $\xi$ be an algebraic complex representation of $\rG$.\footnote{In this article, we only need the case where $\xi$ is the trivial representation.} Then it induces a complex local system $\sL_\xi$ on $\{\Sh(\rG,\rh)_K\otimes_E\dC\}$. Let $\rH^i_{(2)}(\Sh(\rG,\rh)_K(\dC),\sL_\xi)$ be the $i$-th $\rL^2$-cohomology of the complex manifold $\Sh(\rG,\rh)_K(\dC)$ with coefficients in $\sL_\xi$. Put
\[
\rH^i_{(2)}(\Sh(\rG,\rh),\sL_\xi)\coloneqq\varinjlim_{K}\rH^i_{(2)}(\Sh(\rG,\rh)_K(\dC),\sL_\xi),
\]
which is a smooth representation of $\rG(\bA^\infty)$. By the Matsushima formula for $\rL^2$-cohomology, we have an isomorphism
\begin{align}\label{eq:matsushima}
\rH^i_{(2)}(\Sh(\rG,\rh),\sL_\xi)\simeq\bigoplus_{\pi}m_\disc(\pi)\rH^i(\fg,\rK_\rG;\xi_\infty\otimes\pi_\infty)\otimes\pi^\infty
\end{align}
of $\rG(\bA^\infty)$-modules, where
\begin{itemize}
  \item $\fg\coloneqq\Lie\rG_\dR$, and $\rK_\rG$ is a maximal connected compact subgroup of $\rG(\dR)$,

  \item $\xi_\infty$ is the associated $(\fg,\rK_\rG)$-module of $\xi$, and

  \item $\pi=\pi_\infty\otimes\pi^\infty$ runs through isomorphism classes of irreducible admissible representations of $\rG(\bA)$, where $m_\disc(\pi)$ is the discrete multiplicity of $\pi$ (Definition \ref{de:realization}).
\end{itemize}
Here, we have to use \cite{BC83}*{Section~4} to conclude that the continuous part of $\rL^2(\rG(\dQ)\backslash\rG(\bA),\chi)$ does not contribute to the $\rL^2$-cohomology in the case of Shimura varieties. By Zucker's conjecture (proved independently by Looijenga \cite{Loo88} and Saper--Sturn \cite{SS90}), we have a canonical isomorphism
\begin{align}\label{eq:zucker}
\rH^i_{(2)}(\Sh(\rG,\rh),\sL_\xi)\simeq\IH^i(\Sh(\rG,\rh),\sL_\xi)
\end{align}
of $\rG(\bA^\infty)$-modules, where
\[
\IH^i(\Sh(\rG,\rh),\sL_\xi)\coloneqq\varinjlim_K\IH^i(\ol\Sh(\rG,\rh)_K\otimes_E\dC,\sL_\xi)
\]
is the direct limit over $K$ of the complex analytic intersection cohomology of $\ol\Sh(\rG,\rh)_K\otimes_E\dC$, where $\ol\Sh(\rG,\rh)_K$ is the Baily--Borel compactification of $\Sh(\rG,\rh)_K$ (over $E$).

Now let $\ell$ be a rational prime and choose an isomorphism $\iota_\ell\colon\dC\xrightarrow{\sim}\dQ_\ell^\ac$. Then the $\dQ_\ell^\ac$-local system $\sL_\xi\otimes_{\dC,\iota_\ell}\dQ_\ell^\ac$ descends to an (\'{e}tale) $\dQ_\ell^\ac$-local system $\sL_{\xi,\iota_\ell}$ on $\{\Sh(\rG,\rh)_K\}$.
We then have a comparison isomorphism
\[
\IH^i_{\et}(\Sh(\rG,\rh),\sL_{\xi,\iota_\ell})\simeq\IH^i(\Sh(\rG,\rh),\sL_\xi)\otimes_{\dC,\iota_\ell}\dQ_\ell^\ac,
\]
where
\[
\IH^i_{\et}(\Sh(\rG,\rh),\sL_{\xi,\iota_\ell})\coloneqq\varinjlim_K\IH^i_{\et}(\ol\Sh(\rG,\rh)_K\otimes_E\dC,\sL_{\xi,\iota_\ell}).
\]
For an irreducible admissible representation $\pi^\infty$ of $\rG(\bA^\infty)$, put
\[
\IH^i_{\xi,\iota_\ell}(\pi^\infty)\coloneqq\Hom_{\dQ_\ell^\ac[\rG(\bA^\infty)]}
\(\iota\circ\pi^\infty,\IH^i_{\et}(\Sh(\rG,\rh),\sL_{\xi,\iota_\ell})\),
\]
which is a finite dimensional representation of $\Gal(\dC/E)$, whose dimension is equal to
\[
\sum_{\pi_\infty}m_\disc(\pi_\infty\otimes\pi^\infty)\dim_\dC\rH^i(\fg,\rK_\rG;\xi_\infty\otimes\pi_\infty),
\]
where $\pi_\infty$ runs through all irreducible admissible representations of $\rG(\dR)$. We suppress $\xi$ in the notation if it is the trivial representation.

\subsection{Statements for cohomology of unitary Shimura curves}
\label{ss:statements_curve}

We fix a CM number field $E$ and regard $E$ as a subfield of $\dC$ via a fixed complex embedding $\tau'_1\colon E\hookrightarrow\dC$. Let $\tc\in\Gal(E/\dQ)$ be the induced complex conjugation and put $F\coloneqq E^{\tc=1}$. Write $\Phi_F=\{\tau_1,\dots,\tau_d\}$ with $d=[F:\dQ]$ as the set of real embeddings of $F$, in which $\tau_1$ is the restriction of $\tau'_1$.

Let $\rV$ be a hermitian space over $E$ of rank $2$ of signature $(1,1)$ at $\tau_1$ and $(2,0)$ elsewhere. As in Subsection \ref{ss:appendix_isometry} especially Remark \ref{re:picard}, we have the Hodge map $\rh\coloneqq\rh_{\rV,\tau'_1}$, the Shimura varieties $\{\Sh(\rG,\rh)_K\}$ defined over $E$, and their Baily--Borel compactification $\ol\Sh(\rG,\rh)_K$, all of which are smooth curves over $E$. By the discussion from Subsection \ref{ss:setup}, we have an isomorphism
\[
\rH^1_\rB(\ol\Sh(\rG,\rh),\dC)\simeq\bigoplus_{\pi}m_\disc(\pi)\rH^1(\fg,\rK_\rG;\pi_\infty)\otimes\pi^\infty
\]
of $\rG(\bA^\infty)$-modules, where $\rH^1_\rB(\ol\Sh(\rG,\rh),\dC)\coloneqq\varinjlim_K\rH^1_\rB(\ol\Sh(\rG,\rh)_K,\dC)$. By Lemma \ref{le:weil_arch}, up to equivalence, there are only two representations $\pi_\infty$ of $\rG(\dR)$ with $\rH^1(\fg,\rK_\rG;\pi_\infty)\neq\{0\}$, namely,
\[
\pi_\infty^{(1,0)}\coloneqq\pi_{1,1}^{1,0}\otimes1\otimes\cdots\otimes 1,\qquad
\pi_\infty^{(0,1)}\coloneqq\pi_{1,1}^{0,1}\otimes1\otimes\cdots\otimes 1.
\]

\begin{definition}\label{de:cohomological_curve}
Let $\pi^\infty$ be an irreducible admissible representation of $\rG(\bA^\infty)$.
\begin{itemize}
  \item We say that $\pi^\infty$ is \emph{stable cohomological} if both $\pi_\infty^{(1,0)}\otimes\pi^\infty$ and $\pi_\infty^{(0,1)}\otimes\pi^\infty$ have positive cuspidal multiplicity.

  \item We say that $\pi^\infty$ is \emph{endoscopic cohomological} if exactly one of $\pi_\infty^{(1,0)}\otimes\pi^\infty$ and $\pi_\infty^{(0,1)}\otimes\pi^\infty$ has positive cuspidal multiplicity.
\end{itemize}
Denote $\cC_\rV^{\r{st}}$ (resp.\ $\cC_\rV^{\r{end}}$) the set of isomorphism classes of stable (resp.\ endoscopic) cohomological irreducible admissible representations of $\rG(\bA^\infty)$. Put $\cC_\rV\coloneqq\cC_\rV^{\r{st}}\coprod\cC_\rV^{\r{end}}$.
\end{definition}

\begin{proposition}\label{pr:endoscopic_curve}
Let $\pi^\infty$ be an irreducible admissible representation of $\rG(\bA^\infty)$.
\begin{enumerate}
  \item If $\pi^\infty$ is endoscopic cohomological, then there exists a unique ad\`{e}lic oscillator triple $(\mu,\varepsilon,\chi)$ (Definition \ref{de:oscillator_triple}) with $\mu$ of weight one and satisfying $\tau'_1\in\Phi_\mu$, such that $\pi^\infty$ is isomorphic to $\omega(\mu,\varepsilon,\chi)$. Moreover, $\Hom_{\dC[\rG(\bA^\infty)]}(\pi^\infty,\rH^1_\rB(\ol\Sh(\rG,\rh),\dC))$ has dimension $1$.

  \item If $\pi^\infty$ is stable cohomological, then $\Hom_{\dC[\rG(\bA^\infty)]}(\pi^\infty,\rH^1_\rB(\ol\Sh(\rG,\rh),\dC))$ has dimension $2$.
\end{enumerate}
\end{proposition}

\begin{proof}
Let $\rV^*$ be an isotropic skew-hermitian space over $E$ of rank $2$, which is unique up to isomorphism. The global inner transfer from $\rU(\rV)$ to $\rU(\rV^*)$ is known; see, for example, \cite{Har93}. More precisely, let $V_\pi$ be an irreducible $\rU(\rV)(\bA_F)$-submodule of $\rL^2_\cusp(\rU(\rV))$ and denote $\ol{V_\pi}$ its complex conjugate space. We may choose an automorphic character $\xi\colon E^1\backslash(\bA_E^\infty)^1\to\dC^\times$ such that the global theta lifting $\Theta^{\rV^*}_{\psi_F,(1,1),\rV}(\ol{V_\pi}\otimes\xi)$ is nonzero. Then $\r{JL}(V_\pi)\coloneqq\Theta^{\rV^*}_{\psi_F,(1,1),\rV}(\ol{V_\pi}\otimes\xi)\otimes\xi$ is a subspace of $\rL^2_\cusp(\rU(\rV^*))$, which is an irreducible $\rU(\rV^*)(\bA_F)$-module and is independent of the choice of $\xi$. Denote by $\r{JL}(\pi)$ the representation of $\rU(\rV^*)(\bA_F)$ on $\r{JL}(V_\pi)$. Since the complement of $\rL^2_\cusp$ in $\rL^2_\disc$ consists of automorphic characters, we have $m_\disc(\pi)=m_\disc(\r{JL}(\pi))$.

The Langlands--Arthur classification for $\rU(\rV^*)$ is known by \cite{Rog90}*{Section~11}. Let $\pi^*$ be an irreducible cuspidal automorphic representation of $\rU(\rV^*)(\bA_F)$. We have the (standard) base change $\Pi$ of $\pi^*$, which is an irreducible isobaric automorphic representation of $\GL_2(\bA_E)$. We say that $\pi^*$ is stable (resp.\ endoscopic) if $\Pi$ is cuspidal (resp.\ $\Pi\simeq\Pi_1\boxplus\Pi_2$ for two conjugate symplectic automorphic characters $\Pi_1$ and $\Pi_2$).

Suppose that $\pi^\infty$ is stable cohomological. Then both $\r{JL}(\pi_\infty^{(1,0)}\otimes\pi^\infty)$ and $\r{JL}(\pi_\infty^{(0,1)}\otimes\pi^\infty)$ have positive multiplicity and the same base change $\Pi$. By Arthur's multiplicity formula, $\Pi$ has to cuspidal, and $m_\disc(\pi_\infty^{(1,0)}\otimes\pi^\infty)=m_\disc(\pi_\infty^{(0,1)}\otimes\pi^\infty)=1$. In particular, (2) follows.

Suppose that $\pi^\infty$ is endoscopic cohomological. Then by the same reasoning, we have $\Pi\simeq\Pi_1\boxplus\Pi_2$, and $m_\disc(\pi_\infty^{(1,0)}\otimes\pi^\infty)+m_\disc(\pi_\infty^{(0,1)}\otimes\pi^\infty)=1$. Let $\pi$ be the unique member in $\{\pi_\infty^{(1,0)}\otimes\pi^\infty,\pi_\infty^{(0,1)}\otimes\pi^\infty\}$ such that $m_\disc(\pi)=1$. Since $\r{JL}(\pi)$ is endoscopic, both $\Pi_1$ and $\Pi_2$ are conjugate symplectic automorphic characters of weight one. Thus, there exists a conjugate symplectic automorphic character $\mu$ of weight one such that $L(s,\Pi\otimes\mu)$ has a simple pole at $s=1$. By Theorem \ref{th:pole}, we have a skew-hermitian space $\rW$ over $E$ of rank $1$ of determinant $e\in E^{-\times}/\Nm_{E/F}E^\times$ and an automorphic character $\chi'$ of $\rU(\rW)(\bA_F)$, such that $\pi$ is realized in the space of global theta lifting $\Theta^\rV_{\psi_F,(\mu,\nu),\rW}(\chi')$. Let $\chi$ be the central character of $\pi$. Then it is trivial at infinity. Thus, by Lemma \ref{le:weil_nonarch}(4), there exist exactly two ad\`{e}lic oscillator triples, which are $(\mu,\varepsilon,\chi)$ and $(\mu^\tc\check\chi,\varepsilon',\chi)$, such that $\pi^\infty$ is isomorphic to the associated oscillator representation. In particular, the condition that $\mu$ is of weight one and satisfies $\tau'_1\in\Phi_\mu$ determines exactly one of the two triples. Therefore, (1) follows.
\end{proof}

\begin{remark}\label{re:galois_curve}
The proof of Proposition \ref{pr:endoscopic_curve}(1) implies that for $\pi^\infty\simeq\omega(\mu,\varepsilon,\chi)$ that is endoscopic cohomological, we have $m_\cusp(\pi_\infty^{(1,0)}\otimes\pi^\infty)=1$ (resp.\ $m_\cusp(\pi_\infty^{(0,1)}\otimes\pi^\infty)=1$) if and only if there exists some $e\in E^{\times-}$ such that
\begin{itemize}
  \item $\varepsilon_v=e\Nm_{E_v/F_v}E^\times_v$ for every nonarchimedean place $v$ of $F$,

  \item $\tau'_i(e)$ has negative imaginary part for $i=2,\dots,d$, where $\tau'_i$ is the unique element in $\Phi_\mu$ above $\tau_i$, and

  \item $\tau'_1(e)$ has negative (resp.\ positive) imaginary part.
\end{itemize}
\end{remark}

Now we study the $\ell$-adic cohomology of $\{\ol\Sh(\rG,\rh)_K\}_K$. Take a rational prime $\ell$ and an isomorphism $\iota_\ell\colon\dC\xrightarrow{\sim}\dQ_\ell^\ac$. Put
\[
\rH^1_{\et}(\ol\Sh(\rG,\rh),\dQ_\ell^\ac)\coloneqq
\varinjlim_K\rH^1_{\et}(\ol\Sh(\rG,\rh)_K\otimes_E\dC,\dQ_\ell^\ac).
\]
By the comparison theorem, we have a canonical isomorphism
\[
\rH^1_{\et}(\ol\Sh(\rG,\rh),\dQ_\ell^\ac)\simeq
\rH^1_\rB(\ol\Sh(\rG,\rh),\dC)\otimes_{\dC,\iota_\ell}\dQ_\ell^\ac
\]
of $\rG(\bA^\infty)$-modules. For an irreducible admissible representation $\pi^\infty$ of $\rG(\bA^\infty)$, the $\dQ_\ell^\ac$-vector space
\[
\rH^1_{\iota_\ell}(\pi^\infty)\coloneqq\Hom_{\dQ_\ell^\ac[\rG(\bA^\infty)]}
\(\iota_\ell\circ\pi^\infty,\rH^1_{\et}(\ol\Sh(\rG,\rh),\dQ_\ell^\ac)\)
\]
is a representation of $\Gal(\dC/E)$, which we denote by $\rho_{\iota_\ell}(\pi^\infty)$.

Suppose that $\pi^\infty$ is endoscopic cohomological. Then by Proposition \ref{pr:endoscopic_curve}, we obtain an $\ell$-adic character $\rho_{\iota_\ell}(\pi^\infty)\colon\Gal(\dC/E)\to(\dQ_\ell^\ac)^\times$. It induces, via the isomorphism $\iota_\ell$, an automorphic character $\rho_\ell(\pi^\infty)\colon E^\times\backslash\bA_E^\times\to\dC^\times$. It is easy to see that the character $\rho_\ell(\pi^\infty)$ does not depend on the choice of the isomorphism $\iota_\ell$, which justifies its notation.

\begin{theorem}\label{th:galois_curve}
Let $\pi^\infty$ be an irreducible admissible representation of $\rG(\bA^\infty)$, and $\ell$ a rational prime.
\begin{enumerate}
  \item Suppose that $\pi^\infty$ is endoscopic cohomological, which is isomorphic to $\omega(\mu,\varepsilon,\chi)$ with $\mu$ of weight one and satisfying $\tau'_1\in\Phi_\mu$ as in Theorem \ref{pr:endoscopic_curve}(1). Then
     \[
     \rho_\ell(\pi^\infty)=
     \begin{dcases}
     \mu\cdot|\;|_E^{-1/2} & \text{if }m_\cusp(\pi_\infty^{(1,0)}\otimes\pi^\infty)=1;\\
     \mu^\tc\check\chi\cdot|\;|_E^{-1/2} & \text{if }m_\cusp(\pi_\infty^{(0,1)}\otimes\pi^\infty)=1.
     \end{dcases}
     \]

  \item Suppose that $\pi^\infty$ is stable cohomological. Then for every $\iota_\ell\colon\dC\xrightarrow{\sim}\dQ_\ell^\ac$, we have
     \begin{enumerate}
       \item $\rho_{\iota_\ell}(\pi^\infty)$ is an irreducible two-dimensional representation of $\Gal(\dC/E)$;

       \item $\rho_{\iota_\ell}(\pi^\infty)^\vee\simeq\rho_{\iota_\ell}(\pi^\infty)^\tc(1)$;

       \item if we let $\Pi^\infty$ be the irreducible admissible representation of $\GL_2(\bA_E^\infty)$ that is the standard base change of $\pi^\infty$, then for every nonarchimedean place $w$ of $E$ coprime to $\ell$,
           \[
           \r{WD}(\rho_{\iota_\ell}(\pi^\infty)\res\Gal(E_w^\ac/E_w))^{\rF\text{-}\r{ss}}
           \simeq\iota_\ell\circ\sL_{2,E_w}(\Pi^\infty_w|\det|_w^{-1/2})
           \]
           holds, where $\sL_{2,E_w}$ denotes the local Langlands correspondence for $\GL_{2,E_w}$.
     \end{enumerate}
\end{enumerate}
\end{theorem}

The proof of the theorem will be given in Subsection \ref{ss:proof_curve}.

The theorem reveals some information about the Albanese variety (Jacobian) $A_K$ of $\ol\Sh(\rG,\rh)_K$. We have a homomorphism $\sC^\infty_c(K\backslash\bG(\bA^\infty)/K,\dQ)\to\End(A_K)_\dQ$ of $\dQ$-algebras induced by the Hecke actions. Note that $\Gal(\dC/\dQ)$ acts on $\cC_\rV$ through the coefficients, which preserves the two subsets $\cC_\rV^{\r{st}}$ and $\cC_\rV^{\r{end}}$. Therefore, we obtain an isogeny decomposition
\begin{align}\label{eq:endoscopic_albanese}
A_K\sim A_K^{\r{st}}\times A_K^{\r{end}}
\end{align}
(over $E$) such that under the canonical isomorphism in Lemma \ref{le:albanese}(1), we have isomorphisms
\begin{align*}
\rH^1_\rB(A_K^{\r{st}},\dC)&\simeq
\bigoplus_{\pi^\infty\in\cC_\rV^{\r{st}}}\rH^1_\rB(\ol\Sh(\rG,\rh)_K,\dC)[(\pi^\infty)^K],\\
\rH^1_\rB(A_K^{\r{end}},\dC)&\simeq
\bigoplus_{\pi^\infty\in\cC_\rV^{\r{end}}}\rH^1_\rB(\ol\Sh(\rG,\rh)_K,\dC)[(\pi^\infty)^K]
\end{align*}
of $\sC^\infty_c(K\backslash\bG(\bA^\infty)/K,\dQ)$-modules.

Put $\underline\cC_\rV^{\r{st}}\coloneqq\cC_\rV^{\r{st}}/\Gal(\dC/\dQ)$, the set of $\Gal(\dC/\dQ)$-orbits in $\cC_\rV^{\r{st}}$. For every orbit $\underline\pi^\infty$, denote by $M(\underline\pi^\infty)\subseteq\dC$ its field of definition, namely, the fixed field of the stabilizer of $\underline\pi^\infty$ in $\Gal(\dC/\dQ)$; it is a number field, either totally real or CM. By Theorem \ref{th:galois_curve}(2) and a standard argument, we may associate to $\underline\pi^\infty$ a (simple) abelian variety $A(\underline\pi^\infty)$ over $E$, which satisfies
\begin{itemize}
  \item $\dim A(\underline\pi^\infty)=[M(\underline\pi^\infty):\dQ]$,

  \item $\End_E(A(\underline\pi^\infty))_\dQ\simeq M(\underline\pi^\infty)$, and

  \item $A(\underline\pi^\infty)\otimes_{E,\tc}E$ is isogenous to $A(\underline\pi^\infty)^\vee$.
\end{itemize}
In fact, $A(\underline\pi^\infty)$ is of strict $\GL(2)$-type in the terminology of \cite{YZZ}*{Section~3.2.1}. Finally, note that for every open compact subgroup $K$ of $\rG(\bA^\infty)$, the dimension of $K$-fixed vectors in a representations in $\underline\pi^\infty$ depends only on the orbit, which we denote by $\dim_\dC(\underline\pi^\infty)^K$. Theorem \ref{th:galois_curve}(2) has the following corollary.

\begin{corollary}
For every sufficiently small open compact subgroup $K$ of $\rG(\bA^\infty)$, we have an isogeny decomposition
\[
A_K^{\r{st}}\sim\prod_{\underline\pi^\infty\in\underline\cC_\rV^{\r{st}}}A(\underline\pi^\infty)^{\dim_\dC(\underline\pi^\infty)^K}
\]
compatible with changing $K$ in the obvious way. In particular, $A_K^{\r{st}}$ does not have factors that are of CM type. Moreover, $A(\underline\pi_1^\infty)$ is isogenous to $A(\underline\pi_2^\infty)$ for $\underline\pi_1^\infty,\underline\pi_2^\infty\in\underline\cC_\rV^{\r{st}}$ if and only if there exist $\pi_1^\infty\in\underline\pi_1^\infty$ and $\pi_2^\infty\in\underline\pi_2^\infty$ that have the same standard base change to $\GL_2(\bA_E^\infty)$.
\end{corollary}

The isogeny decomposition of $A_K^{\r{end}}$ is a special case of Corollary \ref{co:cm_albanese}.

\subsection{Proof of Theorem \ref{th:galois_curve}}
\label{ss:proof_curve}

We prove Theorem \ref{th:galois_curve} by first establishing a congruence relation for the Shimura curve $\Sh(\rG,\rh)_K$ over a set of primes of $E$ of density $1$.

To state the congruence relation, we fix a prime $\fq$ of $E$, with the underlying rational prime $p$, such that
\begin{itemize}
  \item $\rG\otimes_\dQ\dQ_p$ is unramified (in particular, $p$ is unramified in $E$), and

  \item $\fq\neq\fq^\tc$, that is, $\fq$ has degree $1$ over $F$.
\end{itemize}
Denote by $\fp$ the prime of $F$ underlying $\fq$. We identify $F_\fp$ with $E_\fq$. Choose a uniformizer $\varpi$ of $F_\fp$. Put $\sO_\fp\coloneqq O_{F_\fp}$, $\kappa\coloneqq\sO_\fp/\varpi\sO_\fp$, and $q\coloneqq\#\kappa$. Fix a maximal unramified extension $F_\fp^\ur$ of $F_\fp$ with $\sO_\fp^\ur$ the ring of integers and $\kappa^\ac\coloneqq\sO_\fp^\ur/\varpi\sO_\fp^\ur$ the residue field. Let $\sigma\colon\sO_\fp^\ur\to\sO_\fp^\ur$ be the $q$-th Frobenius map.

Fix a basis of the $E_\fq$-vector space $\rV\otimes_EE_\fq$ under which we identify $\rU(\rV\otimes_FF_\fp)$ with $\GL_{2,F_\fp}$. Let $\Iw_\fp\coloneqq\(\begin{smallmatrix}\sO_\fp&\sO_\fp\\\varpi\sO_\fp&\sO_\fp\end{smallmatrix}\)\subseteq \GL_2(\sO_\fp)$ be an Iwahori subgroup. We consider open compact subgroups $K\subseteq\rG(\bA^\infty)$ of the form $\GL_2(\sO_\fp)\times K_p^\fp\times K^p$ where $\GL_2(\sO_\fp)\times K_p^\fp$ is a hyperspecial maximal subgroup of $\rG(\dQ_p)$ and $K^p$ is a sufficiently small open compact subgroup of $\rG(\bA^{\infty,p})$. For such $K$, we put $K_\Iw\coloneqq\Iw_\fp\times K_p^\fp\times K^p$. We have the projection morphism $\pi\colon\ol\Sh(\rG,\rh)_{K_\Iw}\to\ol\Sh(\rG,\rh)_K$, and an isomorphism
\[
\rt_\varpi\colon\ol\Sh(\rG,\rh)_{K_\Iw}\xrightarrow{\sim}\ol\Sh(\rG,\rh)_{K_\Iw}
\]
induced by the Hecke translation of the element $\(\begin{smallmatrix}&\varpi\\1&\end{smallmatrix}\)$. In view of the reciprocity map in Remark \ref{re:picard}, the morphism
\[
\rt_\varpi\otimes\sigma\colon\ol\Sh(\rG,\rh)_{K_\Iw}\otimes_{F_\fp}F_\fp^\ur\to\ol\Sh(\rG,\rh)_{K_\Iw}\otimes_{F_\fp}F_\fp^\ur
\]
preserves every connected component.

We will show in Proposition \ref{pr:congruence} that $\ol\Sh(\rG,\rh)_K$ (resp.\ $\ol\Sh(\rG,\rh)_{K_\Iw}$) admits a smooth model (resp.\ a stable model) $\cS_K$ (resp.\ $\cS_{K_\Iw}$) over $\sO_\fp$. By \cite{LL99}*{Proposition~4.4(a)}, the morphism $\rt_\varpi$ extends (uniquely) to a morphism $\rt_\varpi\colon\cS_{K_\Iw}\to\cS_{K_\Iw}$, which has to be an isomorphism; and $\pi$ extends (uniquely) to a morphism $\pi\colon\cS_{K_\Iw}\to\cS_K$. Finally, to ease notation, we put $\cT_K\coloneqq\cS_K\otimes_{\sO_\fp}\kappa$ and $\cT_{K_\Iw}\coloneqq\cS_{K_\Iw}\otimes_{\sO_\fp}\kappa$ for the special fibers.

\begin{proposition}\label{pr:congruence}
Let the notation be as above. We have
\begin{enumerate}
  \item The smooth projective $F_\fp$-curve $\ol\Sh(\rG,\rh)_K$ admits a smooth model $\cS_K$ over $\sO_\fp$.

  \item The smooth projective $F_\fp$-curve $\ol\Sh(\rG,\rh)_{K_\Iw}$ admits a stable model $\cS_{K_\Iw}$ over $\sO_\fp$.

  \item The $\kappa$-scheme $\cT_{K_\Iw}$ has two irreducible components $\cT_{K_\Iw}^+$ and $\cT_{K_\Iw}^-$, satisfying that
     \begin{enumerate}
       \item $\pi^+\coloneqq\pi\res\cT_{K_\Iw}^+\colon\cT_{K_\Iw}^+\to\cT_K$ is an isomorphism;

       \item $\pi^-\coloneqq\pi\res\cT_{K_\Iw}^-\colon\cT_{K_\Iw}^-\to\cT_K$ is a finite flat morphism of degree $q$;

       \item $\rt_\varpi\otimes\sigma$ induces an isomorphism between $\cT_{K_\Iw}^+\otimes_\kappa\kappa^\ac$ and $\cT_{K_\Iw}^-\otimes_\kappa\kappa^\ac$;

       \item the morphism $(\pi^-\otimes\id)\circ(\rt_\varpi\otimes\sigma)\circ(\pi^+\otimes\id)^{-1}$ coincides with the absolute $q$-th Frobenius morphism of $\cT_K\otimes_\kappa\kappa^\ac$.
     \end{enumerate}
\end{enumerate}
\end{proposition}

\begin{proof}
We first assume $F\neq\dQ$. Then we have $\ol\Sh(\rG,\rh)_K=\Sh(\rG,\rh)_K$. We will reduce the proposition to (a weak form of) the congruence relation in \cite{Car86}*{Proposition~10.3} by changing the Shimura datum.\footnote{In \cite{Car86}, the initial Shimura variety is a quaternionic Shimura curve, and the auxiliary Shimura variety is a (quaternionic) unitary Shimura curve of PEL type. However, in our case, the initial Shimura variety is unitary Shimura curve of non-PEL type, and the auxiliary Shimura variety we introduce below is a quaternionic Shimura curve, which is the initial Shimura variety of Carayol. So strictly speaking, to obtain this proposition, we have to change Shimura data twice but the second step has already been carried out by Carayol. Such consideration was also used in \cite{Liu12}.} Choose a quaternion algebra over $B$ together with an embedding $E\hookrightarrow B$ of $F$-algebras, such that the induced hermitian form on $B$ is isomorphic to $\rV$. In particular, $B$ is indefinite at $\tau$ and definite at all other places of $F$; $B$ is division at a nonarchimedean place $v$ of $F$ if and only if $\rV_v$ is anisotropic. We identify $\rV$ with $B$ as hermitian spaces. Let $(B^\times\times E^\times)^1$ be the subgroup of $B^\times\times E^\times$ consisting of elements $(b,e)$ such that $\Nm_{B/F}b\cdot\Nm_{E/F}e=1$, viewing as a reductive group over $F$. Then we have a short exact sequence
\[
1 \to \dG_{\rm,F} \to (B^\times\times E^\times)^1 \to \rU(\rV) \to 1,
\]
where the homomorphism $\dG_{\rm,F}\to(B^\times\times E^\times)^1$ is given by $e\mapsto (e,e^{-1})$. The fixed basis of $\rV\otimes_EE_\fq$ identifies $B\otimes_FF_\fp$ with $\Mat_2(F_\fp)$, and further $(B\otimes_FF_\fp)^\times$ with $\rU(\rV)(F_\fp)$.

Put $\rG'\coloneqq\Res_{F/\dQ}B^\times$ and let $\rh'$ be the Hodge map that is \emph{inverse}\footnote{This is to ensure that the actions of $\sigma$ on the connected components of $\Sh(\rG,\rh)\otimes_E\dC$ and $\Sh(\rG',\rh')\otimes_F\dC$ are compatible with the composite homomorphism $F_\fp^\times\simeq E_\fq^\times\xrightarrow{e\mapsto(e,e^{-1})}E_\fp^1$.} to the one given in \cite{Car86}*{0.1}. We have the Shimura curve $\Sh(\rG',\rh')_{K'}$ defined over $F$. Here, the open compact subgroup $K'\subseteq\rG'(\bA_\infty)$ is of the form $K'_p\times K^{\prime p}$, where $K'_p$ is hyperspecial maximal of the form $\GL_2(\sO_\fp)\times K_p^{\prime\fp}$. Replacing $\GL_2(\sO_\fp)$ by $\Iw_\fp$, we obtain $K'_\Iw$, hence the Shimura curve $\Sh(\rG',\rh')_{K'_\Iw}$. Applying the constructions for the Shimura data $(\rG,\rh)$ to $(\rG',\rh')$, we obtain $\Sh(\rG',\rh')_{K'}$, $\Sh(\rG',\rh')_{K'_\Iw}$, $\pi'$, and $\rt'_\varpi$. By Deligne's theory of connected Shimura varieties \cite{Del79} (or see \cite{Car86}*{Section~4}), for every connected component $\Sh(\rG,\rh)_K^\dag$ of $\Sh(\rG,\rh)_K\otimes_{F_\fp}F_\fp^\ur$, there exists some $K'^p$ and a connected component $\Sh(\rG',\rh')_{K'}^\dag$ of $\Sh(\rG',\rh')_{K'}\otimes_{F_\fp}F_\fp^\ur$ such that there is a commutative diagram
\[
\xymatrix{
\Sh(\rG,\rh)_{K_\Iw}^\dag \ar[r]^-{\simeq}\ar[d]_-{\pi} & \Sh(\rG',\rh')_{K'_\Iw}^\dag \ar[d]^-{\pi'} \\
\Sh(\rG,\rh)_K^\dag \ar[r]^-{\simeq} & \Sh(\rG',\rh')_{K'}^\dag
}
\]
where $\Sh(\rG,\rh)_{K_\Iw}^\dag\coloneqq\pi^{-1}\Sh(\rG,\rh)_K^\dag$ and $\Sh(\rG',\rh')_{K'_\Iw}^\dag\coloneqq\pi'^{-1}\Sh(\rG',\rh')_{K'}^\dag$, under which the automorphism $\rt_\varpi\otimes\sigma$ of $\Sh(\rG,\rh)_{K_\Iw}^\dag$ coincide with the automorphism $\rt'_\varpi\otimes\sigma$ of $\Sh(\rG',\rh')_{K'_\Iw}^\dag$ respectively. Therefore, the proposition will follow from the version for $(\rG',\rh')$.\footnote{Here, we have to use the fact that constructing smooth (resp. stable) models of smooth projective curves over $\sO_\fp$ is equivalent to constructing them after the base change to $\sO_\fp^\ur$; see, for example, \cite{DM69}*{Section~1}.}

To release ourselves from the clumsy notation, we will now suppress the ``prime'' in all superscripts; in particular, the group $\rG$ now is $\Res_{F/\dQ}B^\times$. Then (1) follows from \cite{Car86}*{Proposition~6.1}. For the remaining claims, we need some preparation.

For $n\geq 1$, put $K_n\coloneqq(\rI_2+\varpi^n\GL_2(\sO_\fp))\times K_p^\fp\times K^p$. In \cite{Car86}*{1.4.4}, Carayol constructed an $\sO_\fp$-divisible group $\rE_\infty$ over $\Sh(\rG,\rh)_K$, such that the pullback of $\rE_\infty[\fp^n]$ to $\Sh(\rG,\rh)_{K_n}$ is trivial. By the construction, the subgroup $\Sh(\rG,\rh)_{K_1}\times\(\begin{smallmatrix}\ast\\0\end{smallmatrix}\)\subseteq\Sh(\rG,\rh)_{K_1}\times(\fp^{-1}/\sO_\fp)^2$ is stable under the action (given in \cite{Car86}*{1.4.2}) of $\Iw_\fp$. In particular, it defines an $\sO_\fp$-stable subgroup $\rC_\Iw$ of $\rE_\infty[\fp]$ over $\Sh(\rG,\rh)_{K_\Iw}$ of rank $q$. By \cite{Car86}*{Proposition~6.4}, the $\sO_\fp$-divisible group $\rE_\infty$ extends uniquely to an $\sO_\fp$-divisible group $\cE_\infty$ over $\cS_K$ such that $\cE_\infty\res{\cT_K}$ is of dimension $1$ and $\sO_\fp$-height $2$.

We define a functor $\cS_{K_\Iw}$ over $\cS_K$ such that for every $\cS_K$-scheme $u\colon S\to\cS_K$, the set $\cS_{K_\Iw}(S)$ consists of $\sO_\fp$-stable finite flat $S$-subgroups of $u^*\cE_\infty[\fp]$ of rank $q$. As pointed out in \cite{Car86}*{Section~6.7}, the supersingular locus of $\cE_\infty$ is discrete. Thus, it follows from \cite{Car86}*{Proposition~6.6} and the Grothendieck--Messing theory that the above functor is represented by a finite flat morphism $\pi\colon\cS_{K_\Iw}\to\cS_K$ of schemes (of degree $q+1$), satisfying that $\cS_{K_\Iw}$ is a semi-stable curve over $\sO_\fp$. Moreover, since the special fiber $\cT_{K_\Iw}\coloneqq\cS_{K_\Iw}\otimes_{\sO_\fp}\kappa$ does not contain genus zero curves as irreducible components, $\cS_{K_\Iw}$ is a stable curve over $\sO_\fp$. The subgroup $\rC_\Iw$ constructed above induces a morphism $\iota\colon\Sh(\rG,\rh)_{K_\Iw}\to\cS_{K_\Iw}\otimes_{\sO_\fp}F_\fp$ of schemes over $\Sh(\rG,\rh)_K$. By the construction of $\rE_\infty$, it is easy to see that the morphism $\pi\colon\cS_{K_\Iw}\otimes_{\sO_\fp}F_\fp\to\cS_K\otimes_{\sO_\fp}F_\fp=\Sh(\rG,\rh)_K$ is \'{e}tale and generically irreducible. Thus, $\iota$ is an isomorphism since both sides are finite \'{e}tale of degree $q+1$ and generically irreducible over $\Sh(\rG,\rh)_K$. Thus, (2) follows, and we will identify $\cS_{K_\Iw}\otimes_{\sO_\fp}F_\fp$ with $\Sh(\rG,\rh)_{K_\Iw}$ via $\iota$.

Let $(\pi^*\cE_\infty,\cC_\Iw)$ be the universal object over $\cS_{K_\Iw}$. Denote by $\cT_{K_\Iw}^+$ (resp.\ $\cT_{K_\Iw}^-$) the Zariski closure of the locus in $\cT_{K_\Iw}$ where $\cC_\Iw$ is continuous (resp.\ \'{e}tale). Then $\cT_{K_\Iw}^\pm$ are union of irreducible components and they cover $\cT_{K_\Iw}$. To prove (3), we have to consider full Drinfeld level structures at $\fp$. For $n\geq 1$, let $\cS_{K_n}$ be the functor over $\cS_K$ such that for every $\cS_K$-scheme $u\colon S\to\cS_K$, the set $\cS_{K_n}(S)$ consists of Drinfeld level structures $\varphi\colon(\fp^{-n}/\sO_\fp)^2\to\Mor_S(S,u^*\cE_\infty[\fp^n])$ (see \cite{Car86}*{Section~7.2} for more details). By \cite{Car86}*{Proposition~7.4}, it is represented by a finite flat morphism $\pi_n\colon\cS_{K_n}\to\cS_K$ of schemes (of degree $\#\GL_2(\sO_\fp/\fp^n)$), such that $\pi_n\otimes_{\sO_\fp}F_\fp$ is canonically isomorphic to the projection $\Sh(\rG,\rh)_{K_n}\to\Sh(\rG,\rh)_K$. Now we take $n=1$, we define a morphism $\pi_\Iw\colon\cS_{K_1}\to\cS_{K_\Iw}$ by sending a Drinfeld level structure $\varphi$ to the subgroup $\sum_{\alpha\in A^+}[\varphi(\alpha)]$ where $A^+\subseteq(\fp^{-1}/\sO_\fp)^2$ is the line with the second coordinate zero. Then $\pi_\Iw\otimes_{\sO_\fp}F_\fp$ is canonically isomorphic to the projection $\Sh(\rG,\rh)_{K_1}\to\Sh(\rG,\rh)_{K_\Iw}$. Let $\cT_{K_1}^\red$ be the induced reduced subscheme of $\cS_{K_1}\otimes_{\sO_\fp}\kappa^\ac$. Then by \cite{Car86}*{9.4.1}, the morphism $\cT_{K_1}^\red\to\cT_K\otimes_\kappa\kappa^\ac$ is finite flat of degree $(q-1)q(q+1)$. For every line $A$ in $(\fp^{-1}/\sO_\fp)^2$, let $\cT_{K_1,A}^\red$ be the locus where $\varphi\res{A}=0$. Then by \cite{Car86}*{Proposition~9.4.4}, $\{\cT_{K_1,A}^\red\}_A$ is the set of all irreducible components of $\cT_{K_1}^\red$. Since $\GL_2(\kappa)$ acts transitively on $\{\cT_{K_1,A}^\red\}_A$, each $\cT_{K_1,A}^\red$ is of degree $q(q-1)$ over $\cT_K\otimes_\kappa\kappa^\ac$. By definition, the image of $\cT_{K_1,A}^\red$ under $\pi_\Iw$ is contained in $\cT_{K_\Iw}^+$ (resp.\ $\cT_{K_\Iw}^-$) if and only if $A=A^+$ (resp.\ $A\neq A^+$). If $A\neq A^+$, then $\pi_\Iw\colon\cT_{K_1,A}^\red\to\cT^-_{K_\Iw}\otimes_\kappa\kappa^\ac$ is \'{e}tale of degree $q-1$ since to recover the Drinfeld level structure is equivalent to choosing a basis of $A^+$. Thus, $\deg(\pi\res{\cT^-_{K_\Iw}})\geq q$. Since $\deg(\pi\res{\cT^+_{K_\Iw}})\geq 1$, we must have $\deg(\pi\res{\cT^-_{K_\Iw}})=q$ and $\deg(\pi\res{\cT^+_{K_\Iw}})=1$, and both $\cT^-_{K_\Iw}$ and $\cT^+_{K_\Iw}$ are irreducible. Thus, (3a) has been verified as a finite flat morphism of degree $1$ must be an isomorphism, and (3b) also follows. For (3c,3d), put $\cS_{K_\infty}\coloneqq\varprojlim_n\cS_{K_n}$. Let $A_\infty^+$ (resp.\ $A_\infty^-$) be the subspace of $F_\fp^2$ with the second (resp.\ first) coordinate zero. In view of the notation of \cite{Car86}*{Section~10.3},\footnote{Our $\cS_{K_\infty}$ is Carayol's \textbf{\emph{M}}.} we have subschemes $(\cS_{K_\infty}\ol\otimes\kappa^\ac)_{A_\infty^\pm}$ of $\cS_{K_\infty}\ol\otimes\kappa^\ac$, which map surjectively to $\cT_{K_\Iw}^\pm\otimes_\kappa\kappa^\ac$ under the composite map $\cS_{K_\infty}\to\cS_{K_1}\xrightarrow{\pi_\Iw}\cS_{K_\Iw}$, respectively. Note that the endomorphism $\rt_\varpi$ lifts to $\cS_{K_\infty}$ by the Hecke translation. By \cite{Car86}*{Proposition~10.3}, the morphism $\rt_\varpi\otimes\sigma$\footnote{Here, our $\sigma$ is the (arithmetic) Frobenius, which is inverse to the one that should appear in \cite{Car86}*{Proposition~10.3}. Such difference is due to the fact that our choice of the Hodge map for $\Res_{F/\dQ}B^\times$ is inverse to Carayol's.} and the Hecke translation by $\(\begin{smallmatrix}&1\\1&\end{smallmatrix}\)$ induce the same map on the underlying set of $(\cS_{K_\infty}\ol\otimes\kappa^\ac)_{A^+_\infty}$. Since the Hecke translation by $\(\begin{smallmatrix}&1\\1&\end{smallmatrix}\)$ maps $(\cS_{K_\infty}\ol\otimes\kappa^\ac)_{A^+_\infty}$ to $(\cS_{K_\infty}\ol\otimes\kappa^\ac)_{A^-_\infty}$, we obtain (3c). For (3d), since $\(\begin{smallmatrix}&1\\1&\end{smallmatrix}\)$ acts trivially on $\cT_K$, we know, again by \cite{Car86}*{Proposition~10.3}, that $(\pi_\infty\otimes\id)\circ(\rt_\varpi\otimes\sigma)$ coincides with $\pi_\infty\otimes\id$ on the underlying set of $(\cS_{K_\infty}\ol\otimes\kappa^\ac)_{A_\infty^+}$ where $\pi_\infty\colon\cS_{K_\infty}\to\cS_K$ is the obvious projection. This implies that $\bt\coloneqq(\pi^-\otimes\id)\circ(\rt_\varpi\otimes\sigma)\circ(\pi^+\otimes\id)^{-1}$ induces the identity map on the underlying set of $\cT_K\otimes_\kappa\kappa^\ac$, which has to be purely inseparable. We factors $\bt$ as the composite map
\[
\cT_K\otimes_\kappa\kappa^\ac\xrightarrow{\bt'}(\cT_K\otimes_\kappa\kappa^\ac)^{(q)}
\xrightarrow{\id\otimes\sigma}\cT_K\otimes_\kappa\kappa^\ac.
\]
Now $\bt'$ is $\kappa^\ac$-linear, purely inseparable, inducing the identity map on the underlying set, and of degree $q$ by (3b,3c), so it has to be the relative $q$-th Frobenius morphism by \cite{SP}*{0CCZ}. Thus, $\bt$ is the absolute $q$-th Frobenius morphism. The proposition is finally proved in the case where $F\neq\dQ$.

When $F=\dQ$, we can still deduce the proposition to the one for $\ol\Sh(\rG',\rh')_{K'}$, which is either: (i) a Shimura curve associated to a division rational quaternion algebra, or (ii) a compactified modular curve. In both cases, $\ol\Sh(\rG',\rh')_K$ is already a moduli space. In case (i), the conclusions of the proposition can be found in \cite{Buz97}. In case (ii), the proposition is well-known (see \cites{DR72,KM85}).
\end{proof}

\begin{corollary}\label{co:congruence}
For every rational prime $\ell\neq p$, the action $(\sigma^{-1})^*$ of the \emph{geometric} Frobenius at $\fq$ on $\rH^1_{\et}(\ol\Sh(\rG,\rh)_K\otimes_E\dC,\dQ_\ell^\ac)$ satisfies the equation
\[
X^2-\rt_\varpi^*X + q\langle\varpi\rangle^*=0,
\]
where $\langle\varpi\rangle\colon\ol\Sh(\rG,\rh)_K\to\ol\Sh(\rG,\rh)_K$ is the Hecke translation given by $\(\begin{smallmatrix}\varpi&\\&\varpi\end{smallmatrix}\)$. Here, we regard $\rt_\varpi$ as a correspondence on $\ol\Sh(\rG,\rh)_K$.
\end{corollary}

\begin{proof}
It suffices to show $\rt_\varpi^*=(\sigma^{-1})^*+q\langle\varpi\rangle^*\circ\sigma^*$ for actions on $\rH^1_{\et}(\ol\Sh(\rG,\rh)_K\otimes_E\dC,\dQ_\ell^\ac)$. By comparison, it suffices to prove this identity on $\cT_K$. However, by Proposition \ref{pr:congruence}(3), the correspondence $\rt_\varpi$ on $\cT_K$ decomposes as the sum of
\[
\cT_K\xleftarrow{\pi^+}\cT_{K_\Iw}^+\xrightarrow{\rt_\varpi^+}\cT_{K_\Iw}^-\xrightarrow{\pi^-}\cT_K,\quad
\cT_K\xleftarrow{\pi^-}\cT_{K_\Iw}^-\xrightarrow{\rt_\varpi^-}\cT_{K_\Iw}^+\xrightarrow{\pi^+}\cT_K,
\]
where $\rt_\varpi^\pm$ is the restriction of $\rt_\varpi$ to $\cT_{K_\Iw}^\pm$, respectively. By Proposition \ref{pr:congruence}(3d), the action of $(\pi^-\circ\rt_\varpi^+\circ(\pi^+)^{-1})^*$ coincides with the action of $(\sigma^{-1})^*$ on $\rH^1_{\et}(\ol\Sh(\rG,\rh)_K\otimes_E\dC,\dQ_\ell^\ac)$; and the action of $(\pi^-\circ\rt_\varpi^+\circ(\pi^+)^{-1})_*$ coincides with the action of $q(\sigma^{-1})_*=q\sigma^*$ on $\rH^1_{\et}(\ol\Sh(\rG,\rh)_K\otimes_E\dC,\dQ_\ell^\ac)$.

For the first part, we have on $\rH^1_{\et}(\cT_K\otimes_\kappa\kappa^\ac,\dQ_\ell^\ac)$ that
\begin{align*}
\pi^+_*\circ(\rt_\varpi^+)^*\circ(\pi^-)^*
&=\pi^+_*\circ(\pi^+)^*\circ((\pi^+)^{-1})^*\circ(\rt_\varpi^+)^*\circ(\pi^-)^*\\
&=(\pi^+_*\circ(\pi^+)^*)\circ(\pi^-\circ\rt_\varpi^+\circ(\pi^+)^{-1})^*
=(\pi^+_*\circ(\pi^+)^*)\circ(\sigma^{-1})^*=(\sigma^{-1})^*
\end{align*}
as $\pi^+$ is an isomorphism by Proposition \ref{pr:congruence}(3a).

For the second part, we have on $\rH^1_{\et}(\cT_K\otimes_\kappa\kappa^\ac,\dQ_\ell^\ac)$ that
\begin{align*}
\pi^-_*\circ(\rt_\varpi^-)^*\circ(\pi^+)^*
&=\pi^-_*\circ((\rt_\varpi^+)^{-1})^*\circ\langle\varpi\rangle^*\circ(\pi^+)^*
=\pi^-_*\circ(\rt_\varpi^+)_*\circ\langle\varpi\rangle^*\circ((\pi^+)^{-1})_*\\
&=\pi^-_*\circ(\rt_\varpi^+)_*\circ((\pi^+)^{-1})_*\circ\langle\varpi\rangle^*
=(\pi^-\circ\rt_\varpi^+\circ(\pi^+)^{-1})_*\circ\langle\varpi\rangle^*\\
&=q\sigma^*\circ\langle\varpi\rangle^*=q\langle\varpi\rangle^*\circ\sigma^*.
\end{align*}

Adding the two parts, we obtain the desired identity.
\end{proof}

\begin{proof}[Proof of Theorem \ref{th:galois_curve}]
Let $\pi^\infty$ be an irreducible admissible representation of $\rG(\bA^\infty)$. Denote by $\Sigma(\pi^\infty)$ the set of primes $\fq$ of $E$ such that $\fq$ has degree $1$ over $F$ and $\pi^\infty_p$ is an unramified representation of $\rG(\dQ_p)$, where $p$ is the underlying rational prime of $\fq$. It is clear that $\Sigma(\pi^\infty)$ Chebotarev density $1$ among all primes of $E$.

We consider (2) first. Let $\ell$ be a rational prime and let $\iota_\ell\colon\dC\xrightarrow{\sim}\dQ_\ell^\ac$ be an isomorphism. Let $\Pi$ be the standard base change of either $\pi^{(1,0)}_\infty\otimes\pi^\infty$ or $\pi^{(0,1)}_\infty\otimes\pi^\infty$. Then $\Pi$ is cuspidal. In \cite{BR93}*{Section~4}, the authors constructed an irreducible Galois representation $\rho_{\Pi,\iota_\ell}\colon\Gal(\dC/E)\to\GL_2(\dQ_\ell^\ac)$ that satisfies (2c) at all but finitely many nonarchimedean places $w$ of $E$ coprime to $\ell$. On the other hand, Corollary \ref{co:congruence} already implies (2c) for $\rho_{\iota_\ell}(\pi^\infty)$ at places $w=\fq\in\Sigma(\pi^\infty)$ that is coprime to $\ell$. By the Chebotarev density theorem, $\rho_{\iota_\ell}(\pi^\infty)$ and $\rho_{\Pi,\iota_\ell}$ are isomorphic, which implies (2a). Moreover, (2b) also follows from the Chebotarev density theorem; and (2c) follows from \cite{Car12}*{Theorem~1.1}.

Now we consider (1). Put $\tilde\mu\coloneqq\rho_\ell(\pi^\infty)\colon E^\times\backslash\bA_E^\times\to\dC^\times$ for simplicity. We also
put $\mu_1\coloneqq\mu|\;|_E^{-1/2}$ and $\mu_2\coloneqq\mu^\tc\check\chi|\;|_E^{-1/2}$. Then Corollary \ref{co:congruence} implies that for every $\fq\in\Sigma(\pi^\infty)$ that is coprime to $\ell$, we have $\tilde\mu_\fq\in\{\mu_{1\fq},\mu_{2\fq}\}$. We claim that $\tilde\mu\in\{\mu_1,\mu_2\}$. For $i=1,2$, let $\Sigma_i$ be the set of primes $v$ of $E$ such that $\tilde\mu_v=\mu_{iv}$, and let $\delta_i$ be the upper density of $\Sigma_i$. Then we have $\delta_1+\delta_2\geq 1$. Without lost of generality, we assume that $\delta_1>0$. Then by \cite{Raj00}*{Theorem~1}, there exists a Dirichlet character $\eta_1$ of $E$ such that $\tilde\mu=\mu_1\eta_1$. If $\eta_1=1$, then we are done. Otherwise, $\delta_1<1$, and then $\delta_2>0$. By the same argument, we have another Dirichlet character $\eta_2$ of $E$ such that $\tilde\mu=\mu_2\eta_2$. Thus, $\mu_1\mu_2^{-1}$ is a Dirichlet character, which is not true. Therefore, we must have $\tilde\mu\in\{\mu_1,\mu_2\}$. We are left to determine which one $\tilde\mu$ is.

Fix an open compact subgroup $K\subseteq\rG(\bA^\infty)$ such that $(\pi^K)^\infty\neq\{0\}$. Let $A_K$ be the Jacobian of $\ol\Sh(\rG,\rh)_K$. Let $\underline\pi^\infty$ be the $\Gal(\dC/\dQ)$-orbit of $\pi^\infty$. Using Hecke operators, we may find a surjective homomorphism $\varphi\colon A_K\to B$ of abelian varieties over $E$ such that the induced map $\phi^*\colon\rH^1_\rB(B,\dQ)\to\rH^1_\rB(A_K,\dQ)[\underline\pi^\infty]$ is an isomorphism. Let $B_0$ be some simple factor of $B$ over $E$. Then $B_0$ has complex multiplications by some subfield $M_0\subseteq\dC$, which has to contain $M'_\mu$ (Definition \ref{de:conjugate2}). There are two cases.

If $m_\cusp(\pi_\infty^{(1,0)}\otimes\pi^\infty)=1$, then $\rH^1_\rB(X,\dC)[\pi^\infty]$ has Hodge type $(1,0)$. Thus, $\tilde\mu$ is the associated CM character of $B_0$. In particular, we have $\tilde\mu_{\tau_1}(z)=1/z$, where we have identified $\dC$ with $E\otimes_{\tau_1}\dR$ through the embedding $\tau'_1$, which implies that $\tilde\mu=\mu|\;|_E^{-1/2}$. 

If $m_\cusp(\pi_\infty^{(0,1)}\otimes\pi^\infty)=1$, then $\rH^1_\rB(X,\dC)[\pi^\infty]$ has Hodge type $(0,1)$. Thus, $\tilde\mu^\tc$ is the associated CM character of $B_0$. In particular, we have $\tilde\mu_{\tau_1}(z)=1/\ol{z}$, which implies that $\tilde\mu=\mu^\tc\check\chi|\;|_E^{-1/2}$.

Theorem \ref{th:galois_curve} is all proved.
\end{proof}

\end{document}